\documentclass[11pt]{article}
\usepackage[margin=1in]{geometry}

\usepackage{amsfonts}
\usepackage{amscd}
\usepackage{amssymb}
\usepackage{amsthm}
\usepackage{amsmath, xspace}
\usepackage{blkarray}
\usepackage[dvipsnames]{xcolor}
\usepackage{enumitem}
\usepackage{fancyhdr}
\usepackage{mathdots}
\usepackage{mathtools}
\usepackage{mathrsfs}
\usepackage{multicol}
\usepackage{stmaryrd}
\usepackage{ytableau}
\usepackage{comment}
\usepackage{cancel}

\usepackage{arydshln}
\usepackage[plainpages,backref]{hyperref}
\usepackage{lscape}

\theoremstyle{plain}
\newtheorem{thm}{Theorem}[section]
\newtheorem{lemma}[thm]{Lemma}
\newtheorem{prop}[thm]{Proposition}
\newtheorem*{prop*}{Proposition}

\newtheorem*{cor*}{Corollary}
\theoremstyle{definition}
	
\newtheorem{remark}[thm]{Remark}
\newtheorem*{remark*}{Remark}
\newtheorem{exampleinner}[thm]{Example}
\newenvironment{example}
  {\pushQED{\qed}\exampleinner}
  {\popQED\endexampleinner}

\theoremstyle{remark}

\numberwithin{equation}{section}

\providecommand{\keywords}[1]{\textbf{\textit{Key words---}} #1}

\newcommand{\dg}{\mathrm{dg}}
\newcommand{\dgplus}{\mathrm{dg}^{+}}
\newcommand{\sgn}{\mathrm{sgn}}

\def\cB{\mathcal{B}}

\def\CC{\mathbb{C}}

\def\KK{\mathbb{K}}

\def\ZZ{\mathbb{Z}}

\def\fa{\mathfrak{a}}

\def\fh{\mathfrak{h}}

\def\Card{\mathrm{Card}}

\def\ev{\mathrm{ev}}

\def\wt{\mathrm{wt}}

\usepackage{etex} 
\usepackage{pictexwd}
\usepackage{tikz}
	\usetikzlibrary{patterns} 
	\usetikzlibrary{shapes.misc}
\usetikzlibrary{arrows}

\newcounter{r}
\newcounter{s}

\newcommand\Part[1]{
        \setcounter{r}{1}
	 \foreach \x in {#1}{
 	{\ifnum\value{r}=1
		\draw (0,\value{r}-1)--(\x,\value{r}-1); 
		\fi}
	\draw (0,\value{r}) to (\x,\value{r});
   	\foreach \y in {0, ..., \x} {\draw (\y,\value{r})--(\y,\value{r}-1);}
	\addtocounter{r}{1}
 }}

\newcommand\Tableau[1]{
        \foreach \x [count = \c from 1] in {#1} {
		\foreach \y [count = \d from 1] in \x{
			\node at (\d-.5,\c-.5) {\scriptsize$\y$}; 
			\draw (\d,\c) to (\d,\c-1);
			{\ifnum\d=1
				\draw (0,\c) to (0,\c-1);
				\fi}
			\setcounter{r}{\d}
		}
		{\ifnum\c=1
			\draw (0,0)--(\value{r},0);
			\fi}
		\draw(0,\c) to (\value{r},\c);
		\setcounter{s}{\c}}}

\newcommand\sTableau[1]{
        \foreach \x [count = \c from 1] in {#1} {
		\foreach \y [count = \d from 1] in \x{
			\node at (\d-.5,\c-.5) {\tiny$\y$}; 
			\draw (\d,\c) to (\d,\c-1);
			{\ifnum\d=1
				\draw (0,\c) to (0,\c-1);
				\fi}
			\setcounter{r}{\d}
		}
		{\ifnum\c=1
			\draw (0,0)--(\value{r},0);
			\fi}
		\draw(0,\c) to (\value{r},\c);
		\setcounter{s}{\c}}}

\tikzstyle{V}=[draw, fill =black, circle, inner sep=0pt, minimum size=1.5pt]
\tikzstyle{wV}=[draw, fill =white, circle, inner sep=0pt, minimum size=4.5pt]
\tikzstyle{bV}=[draw, fill =black, circle, inner sep=0pt, minimum size=4.5pt]
\tikzstyle{over}=[draw=white,double=black,line width=2pt, double distance=.5pt]

\def\Over[#1,#2][#3,#4]{ 
	\draw[style=over]   (#2,#1) .. controls ++(#4*.5-#2*.5,0) and ++(-#4*.5+#2*.5,0) .. (#4,#3);}
\def\Under[#1,#2][#3,#4]{ 
	\draw  (#2,#1) .. controls ++(#4*.5-#2*.5,0) and ++(-#4*.5+#2*.5,0) .. (#4,#3);}
\def\Cross[#1,#2][#3,#4]{
	\Under[#3,#2][#1,#4]\Over[#1,#2][#3,#4]}

\def\tops[#1][#2][#3]{
	\foreach\x in {#1}{
		\draw (#2,\x+.15) -- (#2+.1, \x+.15) (#2, \x-.15) -- (#2+.1, \x-.15) ;
		\draw (#2+.1,\x) arc (0:360:.75mm and 1.5mm);}
	\foreach \x in {1,...,#3} {\draw (#2,\x)  to (#2+.05,\x); \node[V] at (#2+.05,\x){};}
	}
\def\Bottoms[#1][#2][#3]{
	\foreach\x in {#1}{
		\draw (#2, \x+.15) -- (#2-.1, \x+.15) (#2, \x-.15) -- (#2-.1, \x-.15) ;
		\draw (#2-.1, \x+.15) arc (90:270:.75mm and 1.5mm);}
	\foreach \x in {1,...,#3} {\draw (#2, \x)  to (#2-.05, \x); \node[V] at (#2-.05, \x){};}
	}
\def\Caps[#1][#2,#3][#4]{
	\longrightarrow ps[#1][#3][#4]
	\Bottoms[#1][#2][#4]
	}
\def\Pole[#1][#2,#3]{
	\shade[left color=white,right color=white] (#2,#1+.15) rectangle (#3,#1-.15);
	\draw[over] (#2,#1+.15) to (#3,#1+.15) (#2,#1-.15) to (#3,#1-.15) ;}
\def\Label[#1,#2][#3][#4]{
	\node[right] at (#2+.1,#3) {#4};
	\node[left] at (#1-.1,#3) {#4};		}
\def\Nodes[#1][#2]{
	 \foreach \x in {1,...,#2} {\node[V] at (#1,\x){};	}
	}
\def\PoleCaps[#1][#2,#3]{
	\foreach\x in {#1}{
		\draw (#2,\x+.15) -- (#2-.1,\x+.15) (#2,\x-.15) -- (#2-.1,\x-.15) ;
		\draw (#2-.1,\x+.15) arc (0:-180:1.5mm and .75mm);}
	\foreach\x in {#1}{
		\draw (#3,\x+.15) -- (#3+.1,\x+.15) (#3,\x-.15) -- (#3+.1,\x-.15) ;
		\draw (#3+.1,\x+.15) arc (0:360:1.5mm and .75mm);}
	}
\def\PoleTwist[#1,#2]{
	\foreach \x/\y in {-1/1L, -.7/1R, 0/2L, .3/2R}{\coordinate(T\y) at (#2,\x); \coordinate(B\y) at (#1,\x);}
	\draw[thin] (B1R) .. controls ++(#2*.5-#1*.5-.1,0) and ++(-#2*.5+#1*.5-.1,0) ..  (T2R)
			(B1L)   .. controls ++(#2*.5-#1*.5+.1,0) and ++(-#2*.5+#1*.5+.1,0) ..    (T2L) ;
	\draw[line width=2pt, white]
			(#1,.15)  .. controls +(#2*.5-#1*.5,0) and +(-#2*.5+#1*.5,0) ..   (#2,-.85) ;
	\draw[thin,over] 
		(B2R) .. controls ++(#2*.5-#1*.5+.1,0) and ++(-#2*.5+#1*.5+.1,0) ..  (T1R) 
			(B2L)  .. controls +(#2*.5-#1*.5-.1,0) and +(-#2*.5+#1*.5-.1,0) ..   (T1L) ;
			}

\def\SymPolesCaps[#1,#2][#3]{
	\draw (#1,.3) -- (#1-.1,.3) (#1,.15) -- (#1-.1, .15) ;
	\draw (#1-.1, .3) arc (0:-180:2pt and 1.5pt);
	\draw (#1,#3+.7) -- (#1-.1,#3+.7) (#1,#3+.85) -- (#1-.1,#3+.85) ;
	\draw (#1-.1,#3+.85)  arc (0:-180:2pt and 1.5pt);
	\draw (#2,.3) -- (#2+.1, .3) (#2, .15) -- (#2+.1, .15) ;
	\draw (#2+.1, .3) arc (0:360:2pt and 1.5pt);
	\draw (#2, #3+.7) -- (#2+.1, #3+.7) (#2, #3+.85) -- (#2+.1, #3+.85) ;
	\draw (#2+.1, #3+.85) arc (0:360:2pt and 1.5pt);}

\newcommand{\posleq}[1]{
	\hspace{0.1cm}
	\begin{tikzpicture}
	\draw (-0.8ex, -0.5ex) -- (0.8ex, -0.5ex);
	\draw (-0.8ex, 0.4ex) -- (0.7ex, -0.2ex);
	\draw (-0.8ex, 0.4ex) -- (0.7ex, 1ex);
	\draw (0.4ex,0.4ex) --(1.1ex, 0.4ex);
	\draw (0.75ex,0.75ex) --(0.75ex, 0.05ex);
	\end{tikzpicture}
	\hspace{0.1cm}
	}
\newcommand{\negleq}[1]{
	\hspace{0.1cm}
	\begin{tikzpicture}
	\draw (-0.8ex, -0.5ex) -- (0.8ex, -0.5ex);
	\draw (-0.8ex, 0.4ex) -- (0.7ex, -0.2ex);
	\draw (-0.8ex, 0.4ex) -- (0.7ex, 1ex);
	\draw (0.4ex,0.4ex) --(1.1ex, 0.4ex);
	\end{tikzpicture}
	\hspace{0.1cm}
	}
	
\newcommand{\zeroleq}[1]{
	\hspace{0.1cm}
	\begin{tikzpicture}
	\draw (-0.8ex, -0.5ex) -- (0.8ex, -0.5ex);
	\draw (-0.8ex, 0.4ex) -- (0.7ex, -0.2ex);
	\draw (-0.8ex, 0.4ex) -- (0.7ex, 1ex);
	\draw  (0.75ex,0.4ex) ellipse (0.2ex and 0.35ex);
	\end{tikzpicture}
	\hspace{0.1cm}
	}
	
\newcommand{\posgeq}[1]{
	\hspace{0.1cm}
	\begin{tikzpicture}
	\draw (-0.8ex, -0.5ex) -- (0.8ex, -0.5ex);
	\draw (0.8ex, 0.4ex) -- (-0.7ex, -0.2ex);
	\draw (0.8ex, 0.4ex) -- (-0.7ex, 1ex);
	\draw (-0.4ex,0.4ex) --(-1.1ex, 0.4ex);
	\draw (-0.75ex,0.75ex) --(-0.75ex, 0.05ex);
	\end{tikzpicture}
	\hspace{0.1cm}
	}
\newcommand{\neggeq}[1]{
	\hspace{0.1cm}
	\begin{tikzpicture}
	\draw (-0.8ex, -0.5ex) -- (0.8ex, -0.5ex);
	\draw (0.8ex, 0.4ex) -- (-0.7ex, -0.2ex);
	\draw (0.8ex, 0.4ex) -- (-0.7ex, 1ex);
	\draw (-0.4ex,0.4ex) --(-1.1ex, 0.4ex);
	\end{tikzpicture}
	\hspace{0.1cm}
	}
	
\newcommand{\zerogeq}[1]{
	\hspace{0.1cm}
	\begin{tikzpicture}
	\draw (-0.8ex, -0.5ex) -- (0.8ex, -0.5ex);
	\draw (0.8ex, 0.4ex) -- (-0.7ex, -0.2ex);
	\draw (0.8ex, 0.4ex) -- (-0.7ex, 1ex);
	\draw  (-0.75ex,0.4ex) ellipse (0.2ex and 0.35ex);
	\end{tikzpicture}
	\hspace{0.1cm}
	}

\newcommand{\posl}[1]{
	\hspace{0.1cm}
	\begin{tikzpicture}
	\draw (-0.8ex, 0.4ex) -- (0.7ex, -0.2ex);
	\draw (-0.8ex, 0.4ex) -- (0.7ex, 1ex);
	\draw (0.4ex,0.4ex) --(1.1ex, 0.4ex);
	\draw (0.75ex,0.75ex) --(0.75ex, 0.05ex);
	\end{tikzpicture}
	\hspace{0.1cm}
	}
\newcommand{\negl}[1]{
	\hspace{0.1cm}
	\begin{tikzpicture}
	\draw (-0.8ex, 0.4ex) -- (0.7ex, -0.2ex);
	\draw (-0.8ex, 0.4ex) -- (0.7ex, 1ex);
	\draw (0.4ex,0.4ex) --(1.1ex, 0.4ex);
	\end{tikzpicture}
	\hspace{0.1cm}
	}
	
\newcommand{\zerol}[1]{
	\hspace{0.1cm}
	\begin{tikzpicture}
	\draw (-0.8ex, 0.4ex) -- (0.7ex, -0.2ex);
	\draw (-0.8ex, 0.4ex) -- (0.7ex, 1ex);
	\draw  (0.75ex,0.4ex) ellipse (0.2ex and 0.35ex);
	\end{tikzpicture}
	\hspace{0.1cm}
	}
	
\newcommand{\posg}[1]{
	\hspace{0.1cm}
	\begin{tikzpicture}
	\draw (0.8ex, 0.4ex) -- (-0.7ex, 1ex);
	\draw (0.8ex, 0.4ex) -- (-0.7ex, -0.2ex);
	\draw (-0.4ex,0.4ex) --(-1.1ex, 0.4ex);
	\draw (-0.75ex,0.75ex) --(-0.75ex, 0.05ex);
	\end{tikzpicture}
	\hspace{0.1cm}
	}
\newcommand{\negg}[1]{
	\hspace{0.1cm}
	\begin{tikzpicture}
	\draw (0.8ex, 0.4ex) -- (-0.7ex, -0.2ex);
	\draw (0.8ex, 0.4ex) -- (-0.7ex, 1ex);
	\draw (-0.4ex,0.4ex) --(-1.1ex, 0.4ex);
	\end{tikzpicture}
	\hspace{0.1cm}
	}
	
\newcommand{\zerog}[1]{
	\hspace{0.1cm}
	\begin{tikzpicture}
	\draw (0.8ex, 0.4ex) -- (-0.7ex, -0.2ex);
	\draw (0.8ex, 0.4ex) -- (-0.7ex, 1ex);
	\draw  (-0.75ex,0.4ex) ellipse (0.2ex and 0.35ex);
	\end{tikzpicture}
	\hspace{0.1cm}
	}

\makeatletter
\renewcommand{\@makefnmark}{\mbox{\textsuperscript{}}}
\makeatother

\title{Formulas for Koornwinder polynomials}
\author{
Laura Colmenarejo\quad\ \ email:\ lcomen@ncsu.edu \\
Lucas Gagnon\quad\ \ email:\ lgagnon@usc.edu \\
Arun Ram\quad\quad\ \,\ email:\ aram@unimelb.edu.au \\
\\
}
\date{\today}

\usetikzlibrary{arrows.meta}

\begin{document}

\maketitle

\vspace{-2em}
\begin{center}
{\sl Dedicated to Vic Reiner}
\end{center}

\begin{abstract}
\noindent

This paper provides formulas for Koornwinder polynomials in analogy with the creation formula, the alcove walk formula and the non-attacking fillings formula for the type $GL_n$ Macdonald polynomials.  We state the creation formula in
terms of the divided-difference operators used in Schubert calculus, and we use a box-greedy reduced word to reformulate the alcove walk formula in terms of uncompressed set-valued tableaux.  
Then two types of compression, ``around-the-end compression'' and ``across-the-$0$-gap compression'', are used to derive a formula for Koornwinder polynomials in terms of compressed set-valued tableaux.  
Throughout we work in the full generality of relative Koornwinder polynomials,
which are the analogues of the permuted basement Macdonald polynomials used in the type $GL_n$ case.
\end{abstract}


\tableofcontents

\section*{Introduction}

\emph{Koornwinder polynomials}, also known as \emph{Macdonald--Koornwinder polynomials},  are the Macdonald polynomials for the affine root system $CC_{n}$.  (In Macdonald's notation~\cite[(1.3.18)]{Mac03}, $CC_{n}$ is $(C^\vee_n, C_n)$.)  
For $n = 1$, Koornwinder polynomials are the Askey-Wilson polynomials.
More generally,  Koornwinder polynomials are the `universal' Macdonald polynomials of classical type, in that the Macdonald polynomials for any affine root system of classical type can be obtained by specializing Koornwinder polynomials (see e.g.~\cite[Fig. 1.0.1]{CR24}).

The electronic (`non-symmetric') Koornwinder polynomial $E_{\mu}$, for $\mu \in \ZZ^{n}$, is a Laurent polynomial in variables $x_1, \ldots, x_n$ with coefficients that involve six additional parameters: $q^{\frac12}, t^{\frac12}, t_0^{\frac12}, t_n^{\frac12}, u_0^{\frac12}, u_n^{\frac12}$.  
The purpose of this paper is to describe the monomial expansion in $x_1^{\pm1}, \ldots, x_n^{\pm1}$ of all Koornwinder polynomials $E_{\mu}$.  

For the sake of comparison, three well-known expansion formulas for the `usual' (type $GL_n$) non-symmetric Macdonald polynomials are:

\begin{enumerate}
\item[(a)] The creation formula (see~\cite[(5.6)]{Ch96} and~\cite[(5.10.13)]{Mac03} and~\cite[(3.11)]{CR24}),
\item[(b)] The alcove walk formula (see~\cite[Theorem 3.1]{RY08} and~\cite[Theorem 1.1]{GR21}),
\item[(c)] The non-attacking fillings formula (see~\cite[Theorem 3.5.1]{HHL06} and~\cite[Theorem 1.1]{GR21}).
\end{enumerate}
The creation and alcove walk formulas are known for every affine root system.  
However, the general formulation uses the double affine Hecke algebra (DAHA) and may be hard to parse.  
Our first goal is to give a simplified description of the creation and alcove walk formulas for Koornwinder polynomials, emphasizing parallels with the type $GL_n$ case where possible.  
Concretely,
\begin{itemize}
\item in Proposition~\ref{cformsteps} we state the creation formula for Koornwinder polynomials using the divided-difference operators used in Schubert calculus,
and 
\item in Theorem~\ref{thm:USVformula}, we reparametrize the alcove walk formula for Koornwinder polynomials using set-valued tableaux, generalizing the set-valued tableaux formula for type $GL_n$ Macdonald polynomials that is given in~\cite{DR22}.
\end{itemize}
Notable applications of the alcove walk formula in the Koornwinder case appear in~\cite{OS13} and~\cite{Yam23}.

Our second aim is to understand how the non-attacking fillings formula generalizes to Koornwinder polynomials.  
In type $GL_n$, the non-attacking fillings formula can be obtained from the alcove walk formula by \emph{compression}, a heuristic for consolidating coefficients in order to reduce the number of summands in the final expansion (see~\cite{Len08} and~\cite{GR21}).  
In Theorem~\ref{CSVformula}, we state a new formula for Koornwinder polynomials using two kinds of compression:
\begin{enumerate}
\item[(a)] across-the-$0$-gap compression, which replaces a sum of $2^{k}$ terms by $k+1$ terms, and
\item[(b)] around-the-end compression, which replaces a sum of $2^{k}$ terms by $k^2$ terms,
\end{enumerate}
where $k$ is the size of the relevant compression section.
Around-the-end compression is a straightforward generalization of the compression used in~\cite[\S1.2]{GR21} to go from the alcove walk formula to the non-attacking fillings formula.  
Across-the-$0$-gap compression is, less transparently, a generalization of the compression from non-attacking fillings to queue tableaux found in~\cite[Appendix A]{CMW18}, see also~\cite{CdGW15}.  
Combining both kinds of compression reduces the alcove walk formula to a sum over a distinguished subset of set-valued tableaux that we call \emph{CSV-tableaux}; see Section~\ref{section:compression}.  
Despite having more compression than in type $GL_n$, our formula for Koornwinder polynomials has a significantly larger number of terms than the type $GL_n$ non-attacking fillings formula, which is a sum over tableaux with only one entry per box.  

Throughout the paper, we work in the generality of the relative Koornwinder polynomials $E^z_\mu$, where $z$ is a signed permutation.  
These are the type $CC_{n}$ analogues of the permuted basement Macdonald polynomials studied by~\cite{Al16}.  
Each of our formulas extends to compute all $E^z_\mu$, and we suspect that after appropriate specialization, our formulas are related to the tableaux constructions for type B/C Hall-Littlewood polynomials
given in~\cite{Len10}.

We have also created several tools for future investigations into Koornwinder polynomials.  
First, the \emph{box-greedy reduced word} defined in~\cite{GR21} is a key tool for obtaining the set-valued tableaux formulation of the alcove walk for type $GL_n$, and in Section~\ref{section:bgrw} we generalize the box-greedy reduced word to the Koornwinder case.  
Second, our compression result relies on several identities involving $c$-functions in the DAHA, in the sense of~\cite{CR24}, which may simplify future work.  
Lastly, we have written {\color{purple}sage code~\cite{GagnonGitHub}} that implements many of our constructions, including all combinatorial algorithms, the alcove walk formula, and the uncompressed set-valued tableaux formula.  

The paper is organized as follows: In Section~\ref{sec: Koornwinder poly} we introduce the electronic, relative and bosonic (`symmetric') Koornwinder polynomials. Section~\ref{sec: c-functions} focuses on the needed notions related to the affine and finite Weyl groups, the coroots and the coroot sequence, evaluations and the $c$-functions and fold functions. 
Section~\ref{subsec:weight functions} describes the family of functions that will be weights in our compressed formula for Koornwinder polynomials.  
In Section~\ref{sec: box combinatorics} we introduce the box combinatorics, that is, the box-greedy reduced word (Section~\ref{section:bgrw}) and the coroot sequence (Section~\ref{section:coroot}).  In effect, this combinatorics is the Koornwinder analogue of the arm and leg statistics used in the type $GL_n$ case (see~\cite[Remark 2.3]{GR21} for the relationship).

The final two sections focus on formulas for Koornwinder polynomials.  
Section~\ref{sec: step-by-step} states two recursive definitions of relative Koornwinder polynomials: the creation formula using divided difference operations and the set-valued tableaux formulation of the alcove walk formula.  
Section~\ref{section:compression} applies compression to the set-valued tableaux formula from Section~\ref{sec: step-by-step} in order to prove our compressed formula, Theorem~\ref{CSVformula}, with important definitions given in Sections~\ref{set-valued tab} and~\ref{weights}. The proof of Theorem~\ref{CSVformula} is also contained in Section~\ref{section:compression}, and follows the same inductive structure as the proof of the uncompressed set-valued tableaux formula.

As in~\cite{Mac03} and~\cite{GR21}, the double affine Hecke algebra for type $CC_n$ is our most essential tool for deriving formulas for Koornwinder polynomials.  
In Appendix~\ref{section:DAHA} we give a brief but thorough exposition of the DAHA for type $CC_n$.
Finally in Appendix~\ref{sec: examples} we provide a result with a small extension of the 0-gap compression of
Proposition~\ref{0gap} and explain how it is used to compute the cases of Koornwinder polynomials that are computed by the rhombic staircase tableaux given in~\cite[Theorem 1.10]{CMW18}.  At this time, we do not understand how to relate our combinatorial construction to the combinatorics of
rhombic staircase tableaux.

\bigskip\noindent
\textbf{Acknowledgments.} It is a pleasure to dedicate this paper to Vic Reiner, who has consistently and exceptionally contributed to, 
supported and served our algebraic combinatorics community with sustained energy for so many years. L.\ Colmenarejo is partially supported by the Simons Foundation. A.\ Ram thanks Alexandr Garbali, Jules Lamers and Weiying Guo for many helpful conversations and enthusiastic help computing one-box Koornwinder polynomials before we understood how to go about getting compressed formulas for Koornwinder polynomials.  We also thank A.~Aggarwal, H.~Ben Dali, J.~de Gier, P.~Diaconis, A.~Garbali, W.~Guo, J.~Lamers, D.~Orr, and A.~Schilling for their helpful comments on a preliminary version of this paper.

\section{Macdonald-Koornwinder polynomials}\label{sec: Koornwinder poly}
As mentioned in the introduction, the Macdonald-Koornwinder polynomials are the Macdonald polynomials for the affine root system of type $CC_n$. For short, in the rest of the paper, we refer to them simply as \emph{Koornwinder polynomials}.

Although there is an extensive body of literature on Macdonald polynomials, our exposition here is self-contained and no background on the topic is assumed. 

\subsection{Operators on polynomials}

Let
$$
q, t^{\frac12}, t_0^{\frac12}, u_0^{\frac12}, t_n^{\frac12}, u_n^{\frac12} \quad\hbox{be independent parameters,}
$$
and let
$$\KK = \CC(q, t^{\frac12}, t_0^{\frac12}, u_0^{\frac12}, t_n^{\frac12}, u_n^{\frac12} )
\quad\hbox{be the field of fractions of $\CC[q, t^{\frac12}, t_0^{\frac12}, u_0^{\frac12}, t_n^{\frac12}, u_n^{\frac12} ]$.}  
$$
Let $\KK[x^{\pm1}] = \KK[x_1^{\pm1}, \ldots, x_n^{\pm1}]$.  
Let $\ZZ^n$ denote the set of sequences $\mu = (\mu_1, \ldots, \mu_n)$ of integers.
The ring 
$$\hbox{$\KK[x^{\pm1}] \quad $ has basis}\quad
\{ x^\mu\ |\ \mu = (\mu_1 ,\ldots, \mu_n) \in \ZZ^n\},
\qquad\hbox{where
\quad
$x^\mu = x_1^{\mu_1}\cdots x_n^{\mu_n}$.}
$$
For $f\in \KK[x^{\pm1}]$ define the following sets of operators on $\KK[x^{\pm1}]$. \newline
Define operators $\xi_{s_0}, \xi_{s_1}, \ldots, \xi_{s_n}$ on $\KK[x^{\pm1}]$ by
\begin{align}
(\xi_{s_0}f)(x_1, \ldots, x_n) &= f(qx^{-1}_1, x_2, \ldots, x_n), \nonumber \\
(\xi_{s_i}f)(x_1, \ldots, x_n) &= f(x_1, \ldots, x_{i-1}, x_{i+1}, x_i, x_{i+2}, \ldots, x_n), \quad \hbox{for  $i\in \{1, \ldots, n\}$} 
\label{xisiopdefn}
\\
(\xi_{s_n}f)(x_1, \ldots, x_n) &= f(x_1, \ldots, x_{n-1}, x^{-1}_n),  \nonumber
\end{align} 
Define operators $X_1, \ldots, X_n$ on $\KK[x^{\pm1}]$ by
\begin{equation}
(X_j f)(x_1, \ldots, x_n) = x_j f(x_1, \ldots, x_n), \qquad \hbox{for  $j\in \{1, \ldots, n\}$.} 
\label{polyaction}
\end{equation}

Define operators $T_0, T_1, \ldots, T_n$ on $\KK[x^{\pm1}]$ by
\begin{align}
t_0^{\frac12}T_0 &= t_0  - \left( \frac{(x_1+q^{\frac12}t^{\frac12}_0 u^{-\frac12}_0)
(x_1-q^{\frac12}t^{\frac12}_0u^{\frac12}_0)}
{x_1^2-q}\right) (1-\xi_{s_0}),
\nonumber \\
t^{\frac12}T_i &= t - \frac{x_{i+1}-tx_i}{x_{i+1}-x_i} (1-\xi_{s_i}),
\qquad\hbox{for $i\in \{1, \ldots, n-1\}$},
\label{DGaction}
\\
t_n^{\frac12}T_n &= t_n  - \left( \frac{(1+t^{\frac12}_n u^{-\frac12}_n x_n)
(1- t^{\frac12}_n u^{\frac12}_n x_n)}
{1-x_n^2}\right) (1-\xi_{s_n}),
\nonumber
\end{align}
(see~\cite[\S3]{Nou95} and~\cite[(13)]{Sah99} and~\cite[(73)]{CGdGW16}).
The operators $T_0, T_1, \ldots, T_n$ and $X_1, \ldots, X_n$ satisfy all the relations of the
double affine Hecke algebra of type $CC_n$, which implies that the material in the appendix in Section~\ref{section:DAHA} of this paper applies to these operators.

\begin{remark}
In Schubert calculus one encounters (see~\cite[\S4]{FL06}) operators  of the form
\begin{equation}
\partial_0 = \frac{1-\xi_{s_0}}{x_1^2-q}, \qquad
\partial_i = \frac{1-\xi_{s_i}}{x_{i+1}-x_i}, \qquad
\partial_n = \frac{1-\xi_{s_n}}{1-x_n^2},
\qquad\hbox{for $i\in \{1, \ldots, n-1\}$.}
\label{ddivopdefn}
\end{equation}
By~\eqref{DGaction}, 
$$
\begin{array}{l}
t_0^{\frac12}T_0 = t_0 - (x_1+q^{\frac12}t_0^{\frac12}u_0^{-\frac12})(x_1-q^{\frac12}t_0^{\frac12}u_0^{\frac12})\partial_0, \\[7pt]
t^{\frac12}T_i = t - (x_{i+1}-tx_i)\partial_i, \qquad \hbox{for $i\in \{1, \ldots, n-1\}$,} \\[7pt]
t_n^{\frac12}T_n = t_n - (1+t_n^{\frac12}u_n^{-\frac12}x_n)(1-t_n^{\frac12}u_n^{\frac12}x_n)\partial_n. 
\end{array}
$$

\qed
\end{remark}

\begin{remark} \textbf{Askey-Wilson parameters.}
In Type $CC_1$, the bosonic Macdonald polynomials $P_\lambda(q, t_1,u_1,t_0,u_0)$
are also known as the Askey-Wilson polynomials. Following~\cite[\S3]{Nou95}, the correspondence to the original Askey-Wilson parameters is given by
\begin{equation}
q=q, \quad
a = q^{\frac12}t_0^{\frac12}u_0^{\frac12}, \quad
b = - q^{\frac12}t_0^{\frac12}u_0^{-\frac12}, \quad
c = t_n^{\frac12}u_n^{\frac12}, \quad
d = -t_n^{\frac12}u_n^{-\frac12}.
\label{AWparams}
\end{equation} 
These conversions are equivalent to
$$t_0 = -q^{-1}ab, \qquad t_n = -cd, \qquad u_0 = -ab^{-1}, \qquad u_n = -cd^{-1},$$
and it is useful to note that $t_0^{\frac12}t_n^{\frac12} = q^{-\frac12}(abcd)^{\frac12}$, 
$$
\begin{array}{lcl}
a+b = q^{\frac12}t_0^{\frac12}(u_0^{\frac12}-u_0^{-\frac12}),
& \qquad & 
c+d = t_n^{\frac12}(u_n^{\frac12}-u_n^{-\frac12}), \\
q^{-\frac12}a+q^{\frac12}b^{-1} = u_0^{\frac12}(t_0^{\frac12}-t_0^{\frac12}), 
& & 
c+d^{-1} = u_n^{\frac12}(t_n^{\frac12}-t_n^{-\frac12}).
\end{array}
$$
Up to permutations of $a,b,c,d$, these parameters are used 
in~\cite[(1)]{Sah00}, 
in~\cite[(17)]{CGdGW16}, 
in~\cite[Def. 2.2]{CMW23} 
and, with different notation, in ~\cite[(5.1.14)]{Mac03}.  
\qed
\end{remark}

\subsection{Electronic, relative and bosonic Koornwinder polynomials}

Define operators $Y_1, \ldots, Y_n$ on $\KK[x^{\pm1}]$ by
\begin{equation}
Y_j = T_{j-1}^{-1}\cdots T_1^{-1}T_0T_1\cdots T_n\cdots T_j,
\qquad\hbox{for $j\in \{1, \ldots, n\}$.}
\label{Yjdefn}
\end{equation}
The group $W_{\mathrm{fin}}$ of signed permutations (generated by $s_1, \ldots, s_n$)
acts on $\ZZ^n$ by 
\begin{align*}
s_i(\mu_1, \ldots, \mu_n) &= (\mu_1, \ldots, \mu_{i-1}, \mu_{i+1}, \mu_i, \mu_{i+2}, \ldots, \mu_n), 
\quad\hbox{for $i\in \{1, \ldots, n-1\}$, \quad and } \\
s_n(\mu_1, \ldots, \mu_n) &= (\mu_1, \ldots, \mu_{n-1}, -\mu_n).
\end{align*}
For $\mu\in \ZZ^n$ let $v_\mu\in W_{\mathrm{fin}}$ be the minimal length signed permutation such that $v_\mu \mu$ is weakly increasing
with all entries $\le 0$.
(See~\eqref{vmuvals} for an explicit formula for $v_\mu$.)
The \emph{electronic Koornwinder polynomials}
$$E_\mu(x_1,\ldots, x_n;q,t, t_0^{\frac12}, u_0^{\frac12}, t_n^{\frac12}, u_n^{\frac12})
\in \KK[x^{\pm1}]
\quad\hbox{are indexed by $\quad \mu=(\mu_1, \ldots, \mu_n)\in \ZZ^n$}
$$
and are defined by the conditions that 
\begin{equation}
Y_j E_\mu = q^{-\mu_j}t^{-v_\mu(j)} (t^{\frac12}_0t^{\frac12}_n t^n)^{\mathrm{sgn}(v_\mu(j))}
E_\mu, \quad \hbox{for $j\in \{1, \ldots, n\}$,}
\label{Eeig}
\end{equation}
and the coefficient of $x^{\mu}$ in $E_\mu$ is $1$.

Let $\mu = (\mu_1, \ldots, \mu_n)\in \ZZ^n$ and $z\in W_{\mathrm{fin}}$.  For a reduced word
$z = s_{i_1}\cdots s_{i_r}$ define $T_z = T_{i_1}\cdots T_{i_r}$.
\begin{equation}
\hbox{The \emph{relative Koornwinder polynomial $E^z_\mu$} is} \qquad
E^z_\mu = 
\frac{t^{\frac12\ell_s(v_\mu^{-1})}t_n^{\frac12\ell_d(v_\mu^{-1})}  }{ t^{\frac12\ell_s(zv_\mu^{-1}) }t_n^{\frac12\ell_d(zv_\mu^{-1})}  } 
T_z E_\mu.
\label{relMac}
\end{equation}
where the normalization constant 
$t^{\frac12(\ell_s(v_\mu^{-1})-\ell_s(zv_\mu^{-1}))}t_n^{\frac12(\ell_d(v_\mu^{-1})-\ell_d(zv_\mu^{-1}))}$ 
is determined by requiring the coefficient of $x^{z\mu}$ to be 1.

The $E^z_\mu$ are sometimes called \emph{open boundary ASEP polynomials} (see~\cite[Theorem 1.10 and Remark 1.13]{CMW23}).
In the type $GL_n$ case the $E^z_\mu$ are sometimes called \emph{permuted basement Macdonald polynomials} (see~\cite{Al16} and~\cite[(3.7)]{GR21}).
It is not difficult to generalize~\cite[Proposition 1.2]{GR21b} from the $S_n$ case to the case of $W_{\mathrm{fin}}$
to show that the $E^z_\mu$ for $\mu$ in a single $W_{\mathrm{fin}}$-orbit form a qKZ-family (see~\cite[Def.\ 2.4 and Prop.\ 2.6]{CMW23}).

Let
$$(\ZZ_{\ge0}^n)^+ = \{
\lambda=(\lambda_1, \ldots, \lambda_n)\in \ZZ^n \ |\ \lambda_1\ge \cdots \ge \lambda_n\ge 0\}.
$$
For $\lambda\in (\ZZ^n_{\ge0})^+$
the \emph{bosonic Koornwinder polynomials} are 
\begin{equation}
P_\lambda
= \sum_{w\in W_{\mathrm{fin}}} E^w_\lambda.
\label{bosMac}
\end{equation}
From~\cite[(53) and the first formula in Proposition 3.2]{CR24} or~\cite[(5.5.14) and \S5.7]{Mac03}, an alternate formula for $P_\lambda$ is  
\begin{equation*}
P_\lambda = 
\frac{1}{W_\lambda(t,t_n)} 
\sum_{w\in W_{\mathrm{fin}}} 
w\Big(E_\lambda  \Big(\prod_{i<j} \frac{(1-tx_i^{-1}x_j)(1-tx_i^{-1}x_j^{-1})}{(1-x_i^{-1}x_j)(1-x_i^{-1}x_j^{-1})}\Big)
\Big(\prod_{i=1}^n \frac{(1-t_n^{\frac12}u_n^{\frac12}x_i^{-1})(1+t_n^{\frac12}u_n^{-\frac12} x_i^{-1})}{(1-x_i^{-2})}\Big)\Big).
\end{equation*}
This is an analogue of the formula for the Hall-Littlewood polynomial given in~\cite[(2.1)]{Mac}.  
The normalization constant $\frac{1}{W_\lambda(t,t_n)}$ makes the coefficient of $x^\lambda$ in $P_\lambda$ equal to $1$.
In the notation of~\eqref{lengthdef}, an explicit formula for $W_\lambda(t,t_n)$ is 
\begin{equation}
W_\lambda(t,t_n) = \sum_{z\in W_\lambda} t^{\ell_s(z)}t_n^{\ell_d(z)},
\qquad\hbox{where}\quad
W_\lambda = \{ z\in W_{\mathrm{fin}}\ |\ z \lambda = \lambda\}.
\label{Ppoly}
\end{equation}

To get a feel for what Koornwinder polynomials look like, it is useful to look at explicit formulas for the
case when $\mu= (\mu_1, \ldots, \mu_n)$ is such that $\vert \mu\vert = \mu_1+\cdots+\mu_n = 1$ as provided by the following
Proposition.  These formulas can be derived directly
by using the recursion formulas in Section~\ref{Erecursion} to compute, 
in succession, $E_{\varepsilon_1}, E_{\varepsilon_2},
\ldots, E_{\varepsilon_n}, E_{-\varepsilon_n}, \ldots, E_{-\varepsilon_2}, E_{-\varepsilon_1}$.

\begin{prop} \label{Emu1box}  Let $\varepsilon_j = (0,\ldots, 0, 1, 0, \ldots, 0)$ be the sequence
with $1$ in the $j$th component.
\item[(a)] 
If $j\in \{1, \ldots, n\}$ then
$$E_{\varepsilon_j} 
= x_j +\frac{1-t}{1-qt^{2n-j-1}t_0t_n}  (x_{j-1}+\cdots + x_1)  
+ \frac{ q^{\frac12}t_0^{\frac12}(u_0^{-\frac12} - u_0^{\frac12})
+qt^{n-1} t_0 t_n^{\frac12}(u_n^{-\frac12} - u_n^{\frac12}) }{1-qt^{2n-j-1}t_0t_n},
$$
\item[(b)] 
If $j\in \{1, \ldots, n\}$ then
\begin{align*}
E_{-\varepsilon_j} 
&= x^{-1}_j + \frac{1-t}{1-qt^2} (x_{j+1}^{-1}+\cdots +x_n^{-1} + x_n+\cdots +x_{j+1}) \\
&\qquad +  \frac{ qt^{n-1}t^{n-i}(1-qt^i)t_0t_n +(1-qt^n) t^{n-i}(1-qt^{i-1}) t_n 
-t^{n-i+1}(1-qt^{i-1}) +(1-t) }
{ (1-qt^i) (1- q^2t^{2(n-1)} t_0t_n)} x_j \\
&\qquad + \Big(\frac{1-qt^n}{1-qt^i}\Big)\Big( \frac{
(1-t)(qt^{n-1}t_n+1)}{1-q^2t^{2(n-1)}t_0t_n} \Big) (x_{j-1}+\cdots+x_1) \\
&\qquad + \Big(\frac{1-qt^n}{1-qt^i}\Big)\Big( \frac{
q^{\frac12}t_0^{\frac12} (qt^{n-1}t_n+1)(u_0^{-\frac12} - u_0^{\frac12}) 
+(qt^{n-1}t_0 + 1) t_n^{\frac12} (u_n^{-\frac12}-u_n^{\frac12}) }
{1-q^2t^{2(n-1)}t_0t_n}\Big).
\end{align*}
\end{prop}

\section{Coroots, evaluations and fold functions}\label{sec: c-functions}

In this section we introduce the needed notions related to the affine and finite Weyl groups, the coroots and the coroot sequence, evaluations, the $c$-functions and the fold functions. 
These objects provide the type $CC_{n}$ generalization of the content of~\cite[\S 2]{GR21}.  See also~\cite[\S2]{CR24} 
for further information about the type $CC_{n}$ case.  

\subsection{The affine Weyl group $W$ and the finite Weyl group $W_{\mathrm{fin}}$}
\label{sec:WeylGroups}

Let
\begin{equation}
Q = \hbox{$\ZZ$-span}\left\{ \varepsilon_1, \ldots, \varepsilon_n, \hbox{$\frac12$} K \right\}
\label{eq:Qdef}
\end{equation}
be the free $\ZZ$-module spanned by symbols $\varepsilon_1, \ldots, \varepsilon_n$ and $\frac12 K$.
The affine Weyl group $W$ is the group of $\ZZ$-linear transformations of $Q$
generated by the transformations $s_0, s_1, \ldots, s_n$ that act on an element $\lambda = \lambda_1 \varepsilon_1 + \cdots + \lambda_n \varepsilon_n + \frac{k}{2}K$ in $Q$ by
\begin{align}
s_0 \lambda &= - \lambda_1\varepsilon_1 + \lambda_2\varepsilon_2
+\cdots+\lambda_n \varepsilon_n + \big(\hbox{$\frac{k}{2}$}+\lambda_1) K,
\nonumber
\\
s_n \lambda &= \lambda_1\varepsilon_1+\cdots +\lambda_{n-1}\varepsilon_{n-1}
-\lambda_n \varepsilon_n + \hbox{$\frac{k}{2}$}K, \quad\hbox{and}
\label{lvl0onaZ}
\\
s_i \lambda &= \lambda_1\varepsilon_1 + \cdots + \lambda_{i-1}\varepsilon_{i-1}
+\lambda_{i+1}\varepsilon_i+\lambda_i \varepsilon_{i+1} + \lambda_{i+2}\varepsilon_{i+2}+\cdots +
\lambda_n\varepsilon_n + \hbox{$\frac{k}{2}$} K,
\nonumber
\end{align}
for $i\in \{1, \ldots, n-1\}$.

By~\cite[Ch.\ V, \S4.9 Prop.\ 10]{Bou} or~\cite[\S11-3]{Kane}, 
$W$ is presented by generators $s_0,s_1, \ldots, s_n$ and relations
\begin{equation}
(s_i)^2 = 1
\qquad\hbox{and}\qquad
\begin{matrix}\begin{tikzpicture}[every node/.style={inner sep=1}, scale=1]
\foreach \x in {0,5}{\foreach \y in {-1,1}{\draw (\x,\y*.05) to (\x+1,\y*.05);}}
\node[wV, label=above:{$s_0$}] (0) at (0,0) {};
\node[wV, label=above:{$s_n$}] (6) at (6,0) {};
\foreach \x/\y in {1/1,2/2,4/n\!-\!2,5/n\!-\!1}{
    \node[wV] (\x) at (\x,0) {};}
\draw (1) to (2) (4) to (5);
\draw[dashed] (2) to (4);
\node[wV, label=above:{$s_1$}] (1) at (1,0) {};
\node[wV, label=above:{$s_2$}] (2) at (2,0) {};
\node[wV, label=above:{$s_{n-2}$}] (4) at (4,0) {};
\node[wV, label=above:{$s_{n-1}$}] (5) at (5,0) {};
\end{tikzpicture}
\end{matrix}
\label{Wrels}
\end{equation}
where we use the following graphical notation for relations:
$$
\begin{array}{ccccc}
\begin{tikzpicture}[every node/.style={inner sep=1}, scale=1]
	\node[wV, label=above:{$g_i$} ] (0) at (0,0) {};
	\node[wV, label=above:{$g_j$} ] (1) at (1,0) {};
	\end{tikzpicture} & \qquad & 
\begin{tikzpicture}[every node/.style={inner sep=1}, scale=1]
	\foreach \x in {0}{\foreach \y in {0}{\draw (\x,\y*.05) to (\x+1,\y*.05);}}
	\node[wV, label=above:{$g_i$} ] (0) at (0,0) {};
	\node[wV, label=above:{$g_j$} ] (1) at (1,0) {};
	\end{tikzpicture} & \qquad & 
\begin{tikzpicture}[every node/.style={inner sep=1}, scale=1]
	\foreach \x in {0}{\foreach \y in {-1,1}{\draw (\x,\y*.05) to (\x+1,\y*.05);}}
	\node[wV, label=above:{$g_i$} ] (0) at (0,0) {};
	\node[wV, label=above:{$g_j$} ] (1) at (1,0) {};
	\end{tikzpicture} \\
g_ig_j =g_jg_i & & 
g_ig_jg_i = g_jg_ig_j & & 
g_ig_jg_ig_j = g_jg_ig_jg_i.
\end{array}
$$
The \emph{finite Weyl group} is the subgroup $W_{\mathrm{fin}}$ of $W$ generated by $s_{1}, \ldots, s_{n}$.   
There is a group homomorphism 
\begin{equation}
\overline{\phantom{T}}\colon W \longrightarrow  W_{\mathrm{fin}}
\quad\hbox{given by}\quad
\hbox{$\overline{s_0} = s_1\cdots s_{n-1}s_ns_{n-1}\cdots s_1$
and $\overline{s_i} = s_i$ for $i\in \{1, \ldots, n\}$.}
\label{Wtofin}
\end{equation}

Let $w \in W$.  A \emph{reduced word for $w$} is a minimal-length expression $w= s_{i_1}\ldots s_{i_\ell}$.   
For any reduced word $w= s_{i_1}\ldots s_{i_\ell}$, define
\begin{equation}
\begin{array}{lcl}
\ell_g(w) = \hbox{\# of factors $s_0$ in $s_{i_1} \ldots s_{i_\ell}$}, 
& \qquad &
\ell_d(w) = \hbox{\# of factors $s_n$ in $s_{i_1} \ldots s_{i_\ell}$}, 
\nonumber \\
\ell(w) = \ell = \hbox{\# of factors in $s_{i_1} \ldots s_{i_\ell}$}, 
& &
\ell_s(w) = \ell(w)-\ell_g(w)-\ell_d(w).
\label{lengthdef}
\end{array}
\end{equation}
Each of $\ell(w), \ell_d(w)$, $\ell_g(w)$, and $\ell_s(w)$ are independent of the choice of reduced word for $w$.
As in~\cite{CR24}, the subscripts in $\ell_{g}$, $\ell_{d}$, and $\ell_{s}$ serve to remind us that the generators $s_{0}$ and $s_{n}$ live on the left (``\underline{g}auche'') and right (``\underline{d}roite'') sides of the graph in~\eqref{Wrels}, while the generators $s_{1}, \ldots, s_{n-1}$ generate the \underline{s}ymmetric group $S_{n} \subseteq W$.

\begin{remark} \label{sgnorder}
The finite Weyl group is isomorphic to the group of signed permutations, where a \emph{signed permutation} is a bijection 
$$w\colon \{1, \ldots, n, -n, \ldots, -1\} \longrightarrow  \{1, \ldots, n, -n, \ldots, -1\} \quad\hbox{such that}\quad
w(-i) = -w(i).
$$
Using the total order on $\{1,\ldots, n, -n, \ldots, -1\}$ given by
\[
1 \prec 2 \prec \cdots \prec n \prec -n \prec -(n-1) \prec \cdots \prec -1,
\]
and defining for $w\in W_{\mathrm{fin}}$
$$
\mathrm{Inv}(w) = \big\{i \in \{1, \ldots, n\}, j \in \{\pm 1, \ldots, \pm n\} \;|\; \text{$i \prec j$ and $w(i) \succ w(j)$}\big\},
$$
then $\ell(w) = \ell_{s}(w) + \ell_{d}(w) = \#\mathrm{Inv}(w)$.
\qed
\end{remark}

\subsection{The elements $h_\mu$, $u_\mu$ and $v_\mu$}

Define $h_1, \ldots, h_n\in W$ by
$$h_1 = s_0s_1s_2\cdots s_{n-1}s_ns_{n-1}\cdots s_2s_1
\quad\hbox{and}\quad
h_j = s_jh_{j-1}s_j\ \quad \hbox{for $j\in \{2, \ldots, n\}$.}
$$
For $\mu=(\mu_1, \ldots, \mu_n)\in \ZZ^n$ define the \emph{translation} $h_\mu$ by
$$h_\mu = h_1^{\mu_1}\cdots h_n^{\mu_n}$$
and define $u_\mu\in W$ and $v_\mu\in W_{\mathrm{fin}}$ by the equation
\begin{equation}
h_\mu = u_\mu v_\mu, \qquad
\hbox{where $v_\mu\in W_{\mathrm{fin}}$ and $u_\mu$ is minimal length in the coset $h_\mu W_{\mathrm{fin}}$.}
\label{uvmudef}
\end{equation}
Alternatively for $\mu\in \ZZ^n$ the element $v_\mu\in W_{\mathrm{fin}}$ is the minimal length signed permutation such that $v_\mu \mu$ is weakly increasing with all entries $\le 0$.  Then
\begin{equation}
v_\mu(i) = - \mathrm{sgn}(\mu_i)
\left(\begin{array}{l}
1
+ \#\{j \in \{1, \ldots, n\} \ |\ \vert \mu_{j} \vert > \vert \mu_i\vert \}\big) \\
+\#\{j \in \{1, \ldots, i-1\} \ |\ \hbox{$\vert \mu_{j}\vert = \vert \mu_i\vert$ and $\mu_j\le0$}  \}  \\
+\#\{j \in \{i+1, \ldots, n\} \ |\ \hbox{$\vert \mu_{j}\vert = \vert \mu_i\vert$ and $\mu_i>0$}  \}
\end{array}\right).
\label{vmuvals}
\end{equation}
Under the action of $W$ on $Q$ given in~\eqref{lvl0onaZ},
$$h_\mu K = K \quad\hbox{and}\quad
h_\mu(\varepsilon_j)
= \varepsilon_j - \mu_j K, \quad \hbox{for $j\in \{1, \ldots ,n\}$.}
$$

\subsection{Coroots and coroot sequences}

Define $\alpha_0, \alpha_1, \ldots, \alpha_n$ in $Q$ by
\begin{equation}
\alpha_0 = -\varepsilon_1 + \hbox{$\frac12$}K, \qquad
\alpha_n = \varepsilon_n, \qquad\hbox{and}\qquad
\alpha_i = \varepsilon_i - \varepsilon_{i+1} \quad\hbox{for $i\in \{1, \ldots, n-1\}$.}
\label{scoroots}
\end{equation}

The set of coroots for type $CC_n$ is the union of the five $W$-orbits given by
\begin{align*}
O_1 
&= W\cdot\alpha_n = W\cdot \varepsilon_n 
= \{ \pm \varepsilon_i+rK \ |\  i\in \{1, \ldots, n\}, r\in \ZZ\}, \\
O_2 &= W\cdot 2\alpha_n = W\cdot 2\varepsilon_n
= \{ \pm 2\varepsilon_i+2rK \ |\  i\in \{1, \ldots, n\}, r\in \ZZ\} , \\
O_3 
&= W\cdot \alpha_0 = W\cdot (-\varepsilon_1+\hbox{$\frac12$}K) 
= \{ \pm(\varepsilon_i+\hbox{$\frac12$}(2r+1)K \ |\  i\in \{1, \ldots, n\}, r\in \ZZ\}, \\
O_4 
&= W\cdot 2\alpha_0 = W_Y\cdot (-2\varepsilon_1+K)
= \{ \pm 2\varepsilon_i+(2r+1)K \ |\  i\in \{1, \ldots, n\}, r\in \ZZ\} , \\
O_5  
&= W\cdot \alpha_1 = W\cdot (\varepsilon_1-\varepsilon_2)
= 
\left\{ \begin{array}{l}
\pm(\varepsilon_i+\varepsilon_j)+rK,  \\
\pm(\varepsilon_i-\varepsilon_j)+rK  
\end{array}\ \Big\vert\ 
i,j\in \{1, \ldots, n\}, i<j, r\in \ZZ \right\}.
\end{align*}

Let $w\in W$ and let $w=s_{i_\ell} \ldots s_{i_1}$ be a reduced word for $w \in W$.
The \emph{coroot sequence} of the reduced word $w=s_{i_\ell}\cdots s_{i_1}$ is the sequence
$(\beta_\ell,\ldots, \beta_1)$ given by 
\begin{equation}
\beta_\ell = 
s_{i_1}s_{i_2}\cdots s_{i_{\ell-1}}\alpha_{i_\ell}, \quad
\beta_{\ell-1} = 
s_{i_1}s_{i_2}\cdots s_{i_{\ell-2}}\alpha_{i_{\ell-1}}, \quad \ldots \quad
\beta_2 = s_{i_1}\alpha_{i_2}, \quad
\beta_1 = \alpha_{i_1}.
\label{crtseq}
\end{equation}

\begin{example}
\label{ex:rootseq}
For $n = 3$, represent the generators $s_{0}, s_{1}, s_{2}$, and $s_{3}$ diagrammatically by
$$
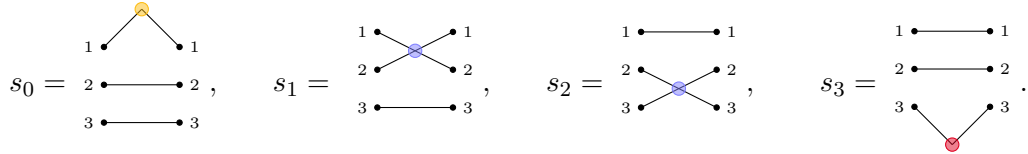
\begin{array}{cccccc}
s_0=
\begin{matrix}
\begin{tikzpicture}[scale=0.5, baseline=-55, xscale=-1,yscale=-1]
\foreach \x in {1,2,3}
\filldraw[black]  (0,\x) circle [radius=2pt] node[anchor=west]{\tiny $\x$};
\foreach \y in {1,2,3}
\filldraw[black]  (2,\y) circle [radius=2pt] node[anchor=east]{\tiny $\y$};
\draw (0,1) -- (1,0);
\draw (1,0) -- (2,1);
\draw (0,2) -- (2,2);
\draw (0,3) -- (2,3);
\filldraw [color=yellow!45!orange, fill, fill opacity=0.5]  (1,0) circle (5pt);
\end{tikzpicture}
\end{matrix}
,\quad
&
s_1= \begin{matrix}\ 
\begin{tikzpicture}[scale=0.5, baseline=-55, xscale=-1,yscale=-1]
\foreach \x in {1,2,3}
\filldraw[black]  (0,\x) circle [radius=2pt] node[anchor=west]{\tiny $\x$};
\foreach \y in {1,2,3}
\filldraw[black]  (2,\y) circle [radius=2pt] node[anchor=east]{\tiny $\y$};
\draw (0,1) -- (2,2);
\draw (0,2) -- (2,1);
\draw (0,3) -- (2,3);
\filldraw [color=blue!50, fill, fill opacity=0.5] (1,1.5) circle (5pt);
\end{tikzpicture}
\end{matrix}
,\quad
&s_2=\begin{matrix}
\begin{tikzpicture}[scale=0.5, baseline=-55, xscale=-1,yscale=-1]
\foreach \x in {1,2,3}
\filldraw[black]  (0,\x) circle [radius=2pt] node[anchor=west]{\tiny $\x$};
\foreach \y in {1,2,3}
\filldraw[black]  (2,\y) circle [radius=2pt] node[anchor=east]{\tiny $\y$};
\draw (0,1) -- (2,1);
\draw (0,2) -- (2,3);
\draw (0,3) -- (2,2);
\filldraw [color=blue!50, fill, fill opacity=0.5] (1,2.5) circle (5pt);
\end{tikzpicture}
\end{matrix}, \quad
&\ 
s_3 = \begin{matrix}
\begin{tikzpicture}[scale=0.5, baseline=-55, xscale=-1,yscale=-1]
\foreach \x in {1,2,3}
\filldraw[black]  (0,\x) circle [radius=2pt] node[anchor=west]{\tiny $\x$};
\foreach \y in {1,2,3}
\filldraw[black]  (2,\y) circle [radius=2pt] node[anchor=east]{\tiny $\y$};
\draw (0,1) -- (2,1);
\draw (0,2) -- (2,2);
\draw (0,3) -- (1,4);
\draw (1,4) -- (2,3);
\filldraw [color=purple!50!red, fill, fill opacity=0.5](1,4) circle (5pt);
\end{tikzpicture}
\end{matrix}.
\end{array}
$$
Then the reduced word
$$
s_{0}s_{3}s_{2}s_{1}s_{0}s_{2}s_{3}s_{2}s_{1}s_{0}s_{1}s_{2}s_{3}s_{2}s_{1}s_{0}
=
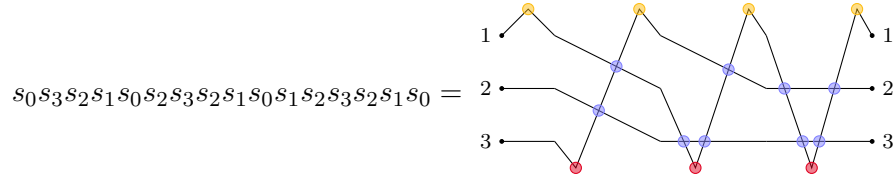
\begin{matrix}
\begin{tikzpicture}[scale = 0.7]
\foreach \y in {1, 2, 3}{
    \draw[fill] (3, -\y) circle (1pt) node (1\y) {};
    \node[left] at (1\y) {$\scriptstyle \y$};
    \draw[fill] (10, -\y) circle (1pt) node (5\y) {};
    \node[right] at (5\y) {$\scriptstyle \y$};
    \foreach \x in {2,...,4}{
        \draw[fill] (2*\x, -\y) circle (0pt) node (\x\y) {};
    }
}
\draw (12.center) -- (22.center);
\draw (13.center) -- (23.center);
\draw (21.center) -- (32.center);
\draw (22.center) -- (33.center);
\draw (33.center) -- (43.center);
\draw (31.center) -- (42.center);
\draw (42.center) -- (52.center);
\draw (43.center) -- (53.center);
\draw (3.5, -0.5) node[inner sep = 1.2pt] (1s0) {};
\draw (11.center) -- (1s0.center) -- (21.center);
\draw (2*2.2, -3.5) node (2s3) {};
\draw (2*2.41667, -2.41667) node (2s2) {};
\draw (2*2.58333, -1.58333) node (2s1) {};
\draw (2*2.8, -0.5) node (2s0) {};
\draw (23.center) -- (2s3.center) -- (2s2.center) -- (2s1.center) -- (2s0.center) -- (31.center);
\draw (2*3.22, -3) node[inner sep = 1.2pt] (3s2) {};
\draw (2*3.33, -3.5) node[inner sep = 1.2pt] (3s3) {};
\draw (2*3.41667, -3) node[inner sep = 1.2pt] (3s2p) {};
\draw (2*3.64286, -1.64286) node[inner sep = 1.2pt] (3s1) {};
\draw (2*3.83333, -0.5) node[inner sep = 1.2pt] (3s0) {};
\draw (32.center) -- (3s2.center) -- (3s3.center) -- (3s2p.center) -- (3s1.center) -- (3s0.center) -- (41.center);
\draw (2*4.17143, -2) node[inner sep = 1.2pt] (4s1) {};
\draw (2*4.34286, -3) node[inner sep = 1.2pt] (4s2) {};
\draw (2*4.42857, -3.5) node[inner sep = 1.2pt] (4s3) {};
\draw (2*4.5, -3) node[inner sep = 1.2pt] (4s2p) {};
\draw (2*4.64286, -2) node[inner sep = 1.2pt] (4s1p) {};
\draw (2*4.85714, -0.5) node[inner sep = 1.2pt] (4s0) {};
\draw (41.center) -- (4s1.center) -- (4s2.center) -- (4s3.center) -- (4s2p.center) -- (4s1p.center) -- (4s0.center) -- (51.center);
\draw[color=yellow!45!orange, fill, fill opacity=0.5] (1s0) circle (3pt) {};
\draw[color=yellow!45!orange, fill, fill opacity=0.5] (2s0) circle (3pt);
\draw[color=yellow!45!orange, fill, fill opacity=0.5] (3s0) circle (3pt);
\draw[color=yellow!45!orange, fill, fill opacity=0.5=] (4s0) circle (3pt);
\draw[color=blue!50, fill, fill opacity=0.5] (2s2)  circle (3pt);
\draw[color=blue!50, fill, fill opacity=0.5] (2s1)  circle (3pt);
\draw[color=blue!50, fill, fill opacity=0.5] (3s2)  circle (3pt);
\draw[color=blue!50, fill, fill opacity=0.5] (3s2p)  circle (3pt);
\draw[color=blue!50, fill, fill opacity=0.5] (3s1)  circle (3pt);
\draw[color=blue!50, fill, fill opacity=0.5] (4s1)  circle (3pt);
\draw[color=blue!50, fill, fill opacity=0.5] (4s2)  circle (3pt);
\draw[color=blue!50, fill, fill opacity=0.5] (4s2p)  circle (3pt);
\draw[color=blue!50, fill, fill opacity=0.5] (4s1p)  circle (3pt);
\draw[color=purple!50!red, fill, fill opacity=0.5] (2s3) circle (3pt);
\draw[color=purple!50!red, fill, fill opacity=0.5] (3s3) circle (3pt);
\draw[color=purple!50!red, fill, fill opacity=0.5] (4s3) circle (3pt);
\end{tikzpicture}
\end{matrix}
$$
has coroot sequence
$$
\left(\begin{array}{l}
-\varepsilon_{1} + \frac{5}{2}K,  \,
- \varepsilon_{2} + K, \,
-\varepsilon_{2} - \varepsilon_{3} + K,  \,
-\varepsilon_{1} - \varepsilon_{2} + 3K, \,
-\varepsilon_{2} + \frac{1}{2} K,  \,
\\
- \varepsilon_{1} + \varepsilon_{3} + 2K, \,
-\varepsilon_{1} + 2K,
-\varepsilon_{1} - \varepsilon_{3} + 2K, \,
-\varepsilon_{1} - \varepsilon_{2} + 2K, \,
-\varepsilon_{1} + \frac{3}{2} K,  \,
\\
-\varepsilon_{1} + \varepsilon_{2} + K, \,
-\varepsilon_{1} + \varepsilon_{3} + K, \,
- \varepsilon_{1} + K, \,
-\varepsilon_{1} - \varepsilon_{3} + K, \,
-\varepsilon_{1}  - \varepsilon_{2} + K , \,
-\varepsilon_{1} + \frac{1}{2} K
\end{array}
\right).
$$
\end{example}

Each colored dot in the diagrammatic expression of the reduced word $w=s_{i_1}\cdots s_{i_\ell}$
corresponds to a single factor in the reduced word and the corresponding element of the coroot sequence 
is determined by following the edges from this dot to the right end of a diagram. 
For an edge with endpoint $a$, let
\begin{align*}
l(a) &= \text{the number of yellow dots on the path to $a$},\,\text{and} \\
d(a) &= \text{the number of pink dots on the path to $a$}.
\end{align*}
(not counting the dot we start at).
Then the coroot corresponding to the colored dot is determined by the formulas
\begin{align*}
\beta_{{\color{blue!50}{\bullet}}{<}^a_b} 
&= (-1)^{l(a)+d(a)}\varepsilon_a-(-1)^{l(b)+d(b)}\varepsilon_b+(l(a)+l(b))K, \\
\beta_{{\color{yellow!45!orange}{\bullet}}-a} 
&=  -(-1)^{l(a)+d(a)}\varepsilon_a + (l(a)+\hbox{$\frac12$})K, \\
\beta_{{\color{purple!50!red}{\bullet}}-a} 
&= (-1)^{l(a)+d(a)}\varepsilon_a + l(a)K.
\end{align*}

\subsection{Evaluations}

Identify the Laurent polynomial ring $\KK[Y] = \KK[Y_1^{\pm1}, \ldots, Y_n^{\pm1}]$
with the group algebra of $Q$ from~\eqref{eq:Qdef} 
via the notations
\begin{equation}
q^{\frac12} = Y^{-\frac12 K} \quad\hbox{and}\quad
Y_i = Y^{\varepsilon_i},
\quad\hbox{and}\quad
q^{-\frac{k}{2}} Y_1^{\lambda_1}\cdots Y_n^{\lambda_n}
=Y^{\frac{k}{2}K+\lambda_1\varepsilon_1+\cdots+\lambda_n\varepsilon_n}=Y^{\lambda},
\label{Ynot}
\end{equation}
for $i\in \{1, \ldots, n\}$ and $\lambda 
= \lambda_1\varepsilon_1+\cdots+\lambda_n\varepsilon_n+\frac{k}{2}K
\in  Q$. 
Then
\begin{equation}
Y^{\alpha_0} = q^{-\frac12} Y_1^{-1}, \qquad
Y^{\alpha_n} = Y_n, \qquad \hbox{and}\qquad
Y^{\alpha_i} = Y_iY^{-1}_{i+1}\quad\hbox{for $i\in \{1, \ldots, n-1\}$.}
\label{Ysroots}
\end{equation}

For $\mu\in \ZZ^n$ consider the element $v_\mu\in W_{\mathrm{fin}}$ defined in~\eqref{uvmudef} and
define the homomorphism $\ev_\mu \colon \KK[Y] \longrightarrow  \KK$ by
\begin{equation}
\ev_\mu(Y_j) = q^{-\mu_j} t^{-v_\mu(j)} (t_0^{\frac12}t_n^{\frac12}t^n)^{\mathrm{sgn}(v_\mu(j))},
\qquad\hbox{for $j\in \{1, \ldots, n\}$,}
\label{evhomdefn}
\end{equation}
and extend the definition to elements of $\KK(Y)$ such that the denominator does not evaluate to $0$.
If $f(Y)\in \KK(Y)$ and $\ev_\mu(f(Y))$ is defined then
$$f(Y)E_\mu = \ev_\mu(f(Y))E_\mu.$$

\begin{example} \label{CFshrink}
A short computation that will be used in the proof of Lemma~\ref{endcombgen} is the following.  
Let $i\in \{1, \ldots, n-2\}$ and assume that $\nu=(\nu_1, \ldots, \nu_n)\in \ZZ^n$ with $\nu_{i+1}=\nu_{i+2}$.  
Then using notation as in~\eqref{cfcndefn} and~\eqref{Fdefn},
\begin{align*}
\ev_\nu(Y_{i+2}) &= \ev_\nu(T_{i+1}^{-1}Y_{i+1}T_{i+1}^{-1})=t^{-1}\ev_\nu(Y_{i+1})\qquad\hbox{and} \\
\ev_\nu(C_{\varepsilon_i+\varepsilon_{i+1}}F^+_{\varepsilon_i+\varepsilon_{i+2}})
&= \ev_\nu\Big(\frac{t^{-\frac12}(1-tY_i^{-1}Y_{i+1}^{-1})}{(1-Y_i^{-1}Y_{i+1}^{-1})}\cdot \frac{t^{-\frac12}(1-t)}{1-Y_i^{-1}Y_{i+2}^{-1}}\Big)
\\
&= \ev_\nu\Big(\frac{t^{-\frac12}(1-Y_i^{-1}Y_{i+2}^{-1})}{(1-Y_i^{-1}Y_{i+1}^{-1})}\cdot \frac{t^{-\frac12}(1-t)}{1-Y_i^{-1}Y_{i+2}^{-1}}\Big)
\\
&= t^{-\frac12}\ev_\nu\Big(\frac{t^{-\frac12}(1-t)}{1-Y_i^{-1}Y_{i+1}^{-1}}\Big)
= t^{-\frac12}\ev_\nu(F^+_{\varepsilon_i+\varepsilon_{i+1}}).
\end{align*}
\end{example}

\subsection{$c$-functions and fold functions}

Introduce the following notation that assigns parameters to each $\alpha_j$ by
\begin{equation}
\begin{array}{llcl}
t_{\alpha_0}=u_n, \quad &t_{\alpha_n} = t_n, &\quad\hbox{and}\quad &t_{\alpha_i} = t\ \quad \hbox{for $i\in \{1, \ldots, n-1\}$,} \\
u_{\alpha_0}=u_0, \quad &u_{\alpha_n} = t_0, &\quad\hbox{and}\quad &u_{\alpha_i} = t\ \quad \hbox{for $i\in \{1, \ldots, n-1\}$},
\end{array}
\label{cortparams}
\end{equation}
and extend this parameter notation to all coroots by setting
\begin{equation}
t_{w\alpha_i} = t_{\alpha_i}
\quad\hbox{and}\quad
u_{w\alpha_i} = u_{\alpha_i}, \qquad\hbox{for $i\in \{0, \ldots, n\}$ and $w\in W$.}
\label{crtparext}
\end{equation}
Following~\cite[(4.2.2)]{Mac03} and~\cite[(16)]{CR24}, the (local) \emph{$c$-functions} are defined by 
\begin{equation}
C_{-\alpha} = 
t^{-\frac12}_{\alpha}
\frac{(1-t_{\alpha}^{\frac12}u_{\alpha}^{\frac12}Y^{\alpha})
(1+t_{\alpha}^{\frac12}u_{\alpha}^{-\frac12}Y^{\alpha})}
{(1-Y^{2\alpha})}
\label{cfcndefn}
\end{equation}
and the \emph{fold functions} $F^+_{\alpha_i}$ and $F^-_{\alpha_i}$ are given by
\begin{equation}
F_{\alpha}^+ = t_{\alpha}^{-\frac12} - C_{-\alpha}
\qquad\hbox{and}\qquad
F_{\alpha}^- = t_{\alpha}^{\frac12} - C_{-\alpha}.
\label{Fdefn}
\end{equation}

A short direct computation (see~\cite[(17)]{CR24}) gives
\begin{align}
C_\alpha+C_{-\alpha}  
= t_{\alpha}^{\frac12}+t_{\alpha}^{-\frac12}
\label{cplusc}
\end{align}
and it follows that
\begin{align}
t^{\frac12}_{\alpha} + F^+_{\alpha} 
&= t^{\frac12}_{\alpha} + t^{-\frac12}_{\alpha} - C_{-\alpha} 
= C_{-\alpha} + C_\alpha - C_{-\alpha} = C_\alpha
\qquad\hbox{and}
\nonumber
\\
t^{-\frac12}_{\alpha} + F^-_{\alpha}
&= t^{-\frac12}_{\alpha}+t^{\frac12}_{\alpha} - C_{-\alpha} 
=C_{-\alpha} + C_\alpha - C_{-\alpha} = C_\alpha
\label{bigC}
\end{align}

If $t_{\alpha} = u_{\alpha}=t$ then
$$
C_{-\alpha} = t^{-\frac12}\frac{(1-tY^{\alpha})}{(1-Y^{\alpha})}, \qquad
F_{\alpha}^+ = t^{-\frac12}\frac{(1-t)}{1-Y^{-\alpha}}
\qquad\hbox{and}\qquad
F_{\alpha}^- = t^{-\frac12}\frac{(1-t)Y^{-\alpha_i}}{1-Y^{-\alpha}}.
$$
More generally,
\begin{align*}
F_{\alpha}^+ &= t_{\alpha}^{-\frac12} - C_{-\alpha}
= t_{\alpha}^{-\frac12} -  t^{-\frac12}_{\alpha}
\frac{(1-t_{\alpha}^{\frac12}u_{\alpha}^{\frac12}Y^{\alpha})
(1+t_{\alpha}^{\frac12}u_{\alpha}^{-\frac12}Y^{\alpha})}
{(1-Y^{2\alpha})}  \\
&= \frac{(t_{\alpha}^{-\frac12}-t_{\alpha}^{\frac12}) }{ 1-Y^{-\alpha} }
\frac{\Big(
1+\frac{(u_{\alpha}^{-\frac12}-u_{\alpha}^{\frac12})} {(t_{\alpha}^{-\frac12}-t_{\alpha}^{\frac12})}
Y^{-\alpha}\Big) }{1+Y^{-\alpha}}
\end{align*}
and
\begin{align*}
F_{\alpha}^- 
&= t_{\alpha}^{\frac12}  -  C_{-\alpha} 
= t_{\alpha}^{\frac12} -  t^{-\frac12}_{\alpha}
\frac{(1-t_{\alpha}^{\frac12}u_{\alpha}^{\frac12}Y^{\alpha})
(1+t_{\alpha}^{\frac12}u_{\alpha}^{-\frac12}Y^{\alpha})}
{(1-Y^{2\alpha})}  \\
&= \frac{ (t_{\alpha}^{-\frac12}-t_{\alpha}^{\frac12})Y^{-\alpha} }{1-Y^{-\alpha} }
\frac{\Big( \frac{(u_{\alpha}^{-\frac12}-u_{\alpha}^{\frac12}) }{(t_{\alpha}^{-\frac12}-t_{\alpha}^{\frac12})} 
+Y^{-\alpha} \Big) }
{1+Y^{-\alpha} }.
\end{align*}

\subsection{Weight functions}\label{subsec:weight functions}

The following weight functions will be used in Section~\ref{section:compression} for the CSV-tableaux formula for relative Koornwinder polynomials. It is conceptually convenient to define these weight functions here because, like the $c$-functions $C_\alpha$ and the fold functions $F_\alpha^\pm$, they are elements of $\KK(Y_1, \ldots,Y_n)$.
These weight functions $A^{(i,j,m)}_\ell$, which are linear combinations of products of the $C_\alpha$ and $F_\alpha^\pm$, are the extensions of the $C_\alpha$ and $F_\alpha^\pm$ that are needed to package the coefficients that occur in the expansion of relative Macdonald polynomials $E^z_\mu$ in their compressed form.
The weight functions $A^{(i,j,m)}_\ell$ depend on sequences of coroots. For the purpose of defining these functions, these sequences of coroots can be arbitrary, but it will be helpful to use the same labeling for the sequences that will be used in Propositions~\ref{0gap} and~\ref{0end}.

\smallskip\noindent
\begin{enumerate}
\item[(CS2)] Let $i,m\in \{1, \ldots, m\}$ with $i< m$.
Let $\beta=(\beta_m, \ldots, \beta_{i+1})$ be a sequence of coroots.  For $j\in \{i, \ldots, m\}$ define
\begin{equation}
A^{(i,j,m)}_\ell(\beta)
= A^{(i,j,m)}_\ell(\beta_m, \ldots, \beta_{i+1}) = \begin{cases}
t^{-\frac12(m-\ell-1)}F^+_{\beta_{i+1}}, &\hbox{if $\ell\in \{i,\ldots, j-1\}$,}  \\
t^{-\frac12(m-\ell-1)}F^-_{\beta_{i+1}}, &\hbox{if $\ell\in \{j, \ldots, m-1\}$,}  \\
1, &\hbox{if $\ell=m$.}
\end{cases}
\label{AEwts}
\end{equation}

\item[(CS1pre)]
Let $i\in \{1, \ldots, n\}$ and let $\beta = (\beta_{-i}, \ldots, \beta_{-n}, \beta_n, \ldots, \beta_{i+1} )$ be a sequence of coroots.
Define $A^{(k)}_\ell(\beta)$ for $\ell\in \{-i, \ldots, -n, n, \ldots, i\}$ by
\begin{equation}
A^{(i)}_{-i}(\beta) = 1 \qquad\hbox{and}\qquad
A^{(i)}_{-(i+1)}(\beta) = F^+_{\beta_{-i}}
\label{Adef1A}
\end{equation}
and
\begin{align}
A^{(i)}_{i+1}(\beta) 
&= 
\Big(\prod_{m=i+1}^{n-1} C_{\beta_{-m}}\Big) C_{\beta_{-n}}  t^{-\frac12 (n-i+1)} F_{\beta_n}^+
\nonumber \\
&\qquad
+F_{\beta_{-i}}^-
\Big(
F_{\beta_{-n}}^+ 
+ \sum_{s=i+1}^{n-1} F^+_{\beta_{-s}}
\Big(\prod_{r=s+1}^{n-1} C_{\beta_{-r}}\Big) C_{\beta_{-n}}  t^{-\frac12 (n-s-1)} F_{\beta_n}^+
 \Big)
\nonumber \\
A^{(i)}_i(\beta)
&=
F_{\beta_{-n}}^+ 
+ \sum_{w=i}^{n-1} F^+_{\beta_{-w}}
\Big(\prod_{r=w+1}^{n-1} C_{\beta_{-r}}\Big) C_{\beta_{-n}} t^{-\frac12 (n-s-1)} F_{\beta_n}^+.
\label{Adef1B}
\end{align}
For $\ell\in \{i+2, \ldots, n\}$ define
\begin{align}
A_{-\ell}^{(i)}(\beta) &= t^{-\frac12 (\ell-i-1)} F^+_{\beta_{-i}}, 
\nonumber \\
A_\ell^{(i)}(\beta) &= \Big(\prod_{m=i}^{\ell-2} C_{\beta_{-m}}\Big) \Big(\prod_{r=\ell}^{n-1} C_{\beta_{-r}}\Big) C_{\beta_{-n}}
t^{-\frac12 (n-\ell)} F^+_{\beta_n}
+ t^{\frac12 (\ell-i-1)} 
F^-_{\beta_{-i}}F_{\beta_{-n}}^+  
\nonumber  \\
&\qquad +t^{\frac12 (\ell-1-i)}
F^-_{\beta_{-i}}
\Big(\sum_{s=\ell}^{n-1} F^+_{\beta_{-s}}
\Big(\prod_{r=s}^{n-1} C_{\beta_{-r}}\Big) C_{\beta_{-n}}
t^{-\frac12 (n-s-1)} F^+_{\beta_n}
\Big).
\label{Adef1C}
\end{align}

\item[(CS1gen)] Define functions
$$A^{(i,m,-i)}_{\pm\ell}(\beta)
=A^{(i,m,-i)}_{\pm\ell}(\beta_{-i}, \ldots, \beta_{-n}, \beta_n, \ldots, \beta_{i+1})
\quad\hbox{for}\quad \begin{array}{l}
\hbox{$i\in \{1, \ldots, n\}$,} \\
\hbox{$m\in \{i, \ldots, n, -n, \ldots, -i\}$ \quad and} \\
\hbox{$\ell\in \{i+1, \ldots, n, -n, \ldots, -k\}$,}
\end{array}
$$
as follows:

\smallskip\noindent
\emph{Case 1:} For $m=k$ define
\begin{align}
A^{(i,i,-i)}_{-\ell}(\beta) &= A^{(i)}_{-\ell}(\beta)&&\hbox{if $-\ell\in \{-i, \ldots, -n\}$,} 
\nonumber \\
A^{(i,i,-i)}_{\ell}(\beta) &= A^{(i)}_{\ell}(\beta)&&\hbox{if $\ell\in \{i, \ldots, n\}$.} 
\label{AgencA}
\end{align}
\emph{Case 2:} For $m\in \{i+1, \ldots, n\}$ define
\begin{align*}
A^{(i,m,-i)}_{-\ell}(\beta) &= A^{(i)}_{-\ell}(\beta) + (t^{\frac12}-t^{-\frac{1}{2}})t^{-\frac{1}{2}(\ell-1-i)} A^{(i)}_{-i}(\beta), 
&&\hbox{if $-\ell\in \{-i, \ldots, -m\}$,} 
\\
A^{(i,m,-i)}_{-\ell}(\beta) &= A^{(i)}_{-\ell}(\beta), &&\hbox{if $-\ell\in \{-(m+1), \ldots, -n\}$,} 
\\
A^{(i,m,-i)}_{\ell}(\beta) &= A^{(i)}_{\ell}(\beta), &&\hbox{if $\ell\in \{i+1, \ldots, n\}$,} 
\end{align*}
and
\begin{equation}
A^{(i,m,-i)}_i(\beta) = A^{(i)}_i(\beta) + (t^{\frac12}-t^{-\frac12})
\sum_{j=i+1}^m t^{\frac12(j-1-i)} A^{(i)}_j(\beta).
\label{AgencB}
\end{equation}
\emph{Case 3:}  For $-m = -n$ define
\begin{align*}
A^{(i,-n,-i)}_{-\ell}(\beta) &= A^{(i)}_{-\ell}(\beta) + (t^{\frac12}-t^{-\frac{1}{2}})t^{-\frac{1}{2}(\ell-1-i)}A^{(i)}_{-i}(\beta), 
&&\hbox{if $-\ell\in \{-i, \ldots, -n\}$,} 
\\
A^{(i,-n,-i)}_{\ell}(\beta) &= A^{(i)}_{\ell}(\beta), &&\hbox{if $\ell\in \{i+1, \ldots, n\}$,} 
\end{align*}
and
  \begin{equation}
A^{(i,-n,-i)}_i(\beta) = A^{(i)}_i(\beta) +  (t^{\frac12}-t^{-\frac12})
\Big(\sum_{j=i+1}^n t^{\frac12(j-i-1)} A^{(i)}_j(\beta) \Big) + t(t_n^{\frac12}-t_n^{-\frac12}) A^{(i)}_{-i}(\beta).
\label{AgencC}
\end{equation}
\emph{Case 4:}  For $-m\in \{-i, \ldots, -(n-1)\}$ define
\begin{align*}
A^{(i,-m,-i)}_{-\ell}(\beta) &= A^{(i)}_{-\ell}(\beta) + (t^{\frac12}-t^{-\frac{1}{2}})t^{-\frac{1}{2}(\ell-1-i)} A^{(i)}_{-i}(\beta), 
&&\hbox{if $-\ell\in \{-i, \ldots, -n\}$,} 
\\
A^{(i,-m,-i)}_{\ell}(\beta) &= A^{(i)}_\ell(\beta) 
+ (t^{\frac12}-t^{-\frac12})t^{-\frac12(2n-4-\ell-i)}t_n^{-\frac12} A^{(i)}_{-i}(\beta), 
&&\hbox{if $\ell\in \{m, \ldots, n\}$,}
\\
A^{(i,-m,-i)}_{\ell}(\beta) &= A^{(i)}_{\ell}(\beta), &&\hbox{if $\ell\in \{i+1, \ldots, m-1\}$,} 
\end{align*}
and 
\begin{align}
A^{(i,-m,-i)}_i(\beta) 
&= A^{(i)}_i(\beta) +  (t^{\frac12}-t^{-\frac12})
\Big(\sum_{j=i+1}^n t^{\frac12(j-i)} A^{(i)}_j(\beta)\Big) 
\nonumber \\
&\qquad + (t^{\frac12} - t^{-\frac12})t_n^{\frac12} \Big(\sum_{j=m+1}^n t^{\frac12(2n-j-1-i)} A^{(i)}_{-j}(\beta) \Big)
\nonumber \\
&\qquad + \Big(
t(t_n^{\frac12}-t_n^{-\frac12}) + t^{\frac12}t_n^{\frac12}(t^{\frac12}-t^{-\frac12}) (t^{n-m}-1)\Big) A^{(i)}_{-i}(\beta).
\label{AgencD}
\end{align}
\end{enumerate}

\section{Box Combinatorics}\label{sec: box combinatorics}

Fix $n\in \ZZ_{>0}$.
A \emph{box} is an element of $\{1,\ldots, n\}\times \ZZ$ so that
$$\{\hbox{boxes}\} = \{ (i,j)\ |\ i\in \{1, \ldots, n\},\ j\in \ZZ \}.$$
To conform to~\cite[p.2]{Mac}, arrange the boxes $(i,j)$ as for matrix entries, where the first coordinate indicates the row 
(from top to bottom) and the second coordinate indicates the column (from left to right). 

Pictorially for $n=5$ the set of boxes looks like
\[
\begin{array}{ccc|c|ccc}
\cdots
&\boxed{{}_{25}\ (1,-2)} 
&\boxed{{}_{15}\ (1,-1)} 
&\boxed{{}_{5}\ (1,0)} 
&\boxed{{}_{6\ }\ (1,1)} 
&\boxed{{}_{16}\ (1,2)} 
&\cdots 
\\
\cdots
&\boxed{{}_{24}\ (2,-2)} 
&\boxed{{}_{14}\ (2,-1)} 
&\boxed{{}_{4}\ (2,0)} 
&\boxed{{}_{7\ }\ (2,1)} 
&\boxed{{}_{17}\ (2,2)} 
&\cdots 
\\
\cdots
&\boxed{{}_{23}\ (3,-2)} 
&\boxed{{}_{13}\ (3,-1)} 
&\boxed{{}_{3}\ (3,0)} 
&\boxed{{}_{8\ }\ (3,1)} 
&\boxed{{}_{18}\ (3,2)} 
&\cdots 
\\
\cdots
&\boxed{{}_{22}\ (4,-2)} 
&\boxed{{}_{12}\ (4,-1)} 
&\boxed{{}_{2}\ (4,0)} 
&\boxed{{}_{9\ }\ (4,1)} 
&\boxed{{}_{19}\ (4,2)} 
&\cdots 
\\
\cdots
&\boxed{{}_{21}\ (5,-2)} 
&\boxed{{}_{11}\ (5,-1)} 
&\boxed{{}_{1}\ (5,0)} 
&\boxed{{}_{10}\ (5,1)} 
&\boxed{{}_{20}\ (5,2)} 
&\cdots 
\end{array}
\]
The labels in the bottom-left corner of each box correspond to
\begin{equation}
\hbox{the \emph{spiral coordinate} of the box $(i,j)$ \quad given by}\ 
\begin{cases}
n+i+2n(j-1), &\hbox{if $j>0$,} \\
n-i+1, &\hbox{if $j=0$,} \\
n-i-2nj, &\hbox{if $j<0$.}
\end{cases}
\label{spiral}
\end{equation}

Let $\mu=(\mu_1, \ldots, \mu_n)\in \ZZ^n$.  
The \emph{diagram} of $\mu$ is the collection of boxes
$$
\dg(\mu) = \left\{(r, \sgn(\mu_{r})c) \;|\; \hbox{$r\in \{1, \ldots, n\}$ and $c\in \{1, \ldots, \vert \mu_r\vert \}$}\}
\cup 
\{(r, 0) \;|\; \mu_{r} \le 0 \right\}.
$$
Throughout the paper we identify $\mu$ and $\dg(\mu)$.

\begin{example}
We illustrate the diagrams for $\mu = (-3,5,-1,4)$ (left) and $\nu = (0,2,3,-1,1)$ (right). 
$$
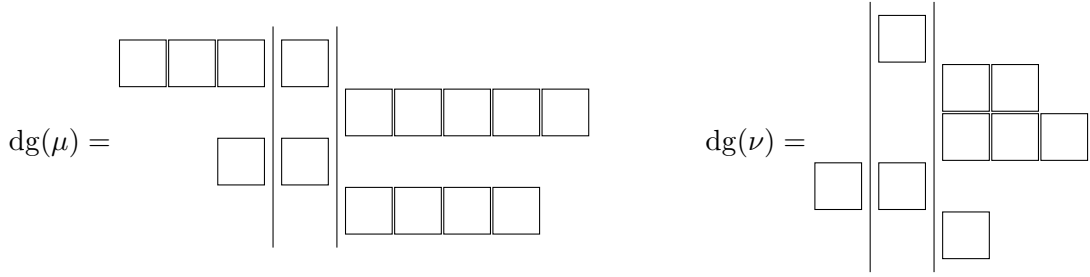
\dg(\mu)= \begin{matrix}
\begin{tikzpicture}[scale = 0.65]
\foreach \x\y in {1/-3, 1/-2, 1/-1, 1/0, 2/1, 2/2, 2/3, 2/4, 2/5, 3/-1, 3/0, 4/1, 4/2, 4/3, 4/4}{
    \pgfmathtruncatemacro{\shift}{(\y>0) - (\y<0)}
    \begin{scope}[xshift = 0.3*\shift cm]
    \draw[thin] (\y+0.025, -\x-0.025) rectangle (0.975+\y, -\x-0.975);
    \draw (0.5+\y,-\x-0.5) node (\x\y) {};
    \end{scope}
}
\draw (-0.15, -0.75) -- (-0.15, -5.25);
\draw (1.15, -0.75) -- (1.15, -5.25);
\end{tikzpicture}
\end{matrix}
\qquad \qquad 
\dg(\nu)=\begin{matrix}
\begin{tikzpicture}[scale = 0.65]
\foreach \x\y in {1/0, 2/1, 2/2,  3/1, 3/2, 3/3, 4/0, 4/-1, 5/1}{
    \pgfmathtruncatemacro{\shift}{(\y>0) - (\y<0)}
    \begin{scope}[xshift = 0.3*\shift cm]
    \draw[thin] (\y+0.025, -\x-0.025) rectangle (0.975+\y, -\x-0.975);
    \draw (0.5+\y,-\x-0.5) node (\x\y) {};
    \end{scope}
}
\draw (-0.15, -0.75) -- (-0.15, -6.25);
\draw (1.15, -0.75) -- (1.15, -6.25);
\end{tikzpicture}
\end{matrix}
$$
\end{example}

\subsection{The box-greedy reduced word for $u_\mu$}
\label{section:bgrw}

The affine Weyl group $W$ acts on $\ZZ^n$ by
\begin{align}
s_0\mu &=   (-\mu_1+1, \mu_2, \ldots, \mu_n), \nonumber \\
s_i\mu &= (\mu_{1}, \ldots, \mu_{i-1}, \mu_{i+1}, \mu_{i}, \mu_{i+2}, \ldots, \mu_n), \qquad  \hbox{if $i\in \{1, \ldots, n\}$,}  
\label{WonZn} 
\\
s_n \mu &= (\mu_1, \ldots, \mu_{n-1}, -\mu_n). \nonumber
\end{align}
For $\mu = (\mu_{1}, \ldots, \mu_{n}) \in \ZZ^n$ the element $u_\mu\in W$ defined in~\eqref{uvmudef} is the 
unique minimal-length element in $W$ such that
\begin{equation}
u_\mu(0,\ldots, 0) = (\mu_1, \ldots, \mu_n),
\label{umudefn}
\end{equation}
The box greedy reduced word $u_\mu^\square$ (defined in~\eqref{boxgreedyredwd})
is a specially chosen reduced word for the element $u_\mu\in W$.

Let $(r, c)$ be a box of $\mu$ and let $b$ be the spiral coordinate of $(r,c)$.  Define 
the \emph{sequence of attacking boxes for $(r,c)$} in two cases:

\smallskip\noindent
\emph{Case $c=1$:} 
$\mathrm{attack}_\mu(r,1)$  is the set of boxes
$$\{ (i,0)\ |\ \hbox{$i\in \{1, \ldots, r-1\}$ and $(i,1)\not\in \mu$}\}\cup \{ (i,1)\ |\  \hbox{$i\in \{1, \ldots, r-1\}$ and $(i,1)\in \mu$}\}$$
arranged in decreasing order by spiral coordinate.

\smallskip\noindent
\emph{Case $c\ne1$:} 
$\mathrm{attack}_\mu(r,c)$  is the set of boxes with spiral coordinates in
$$\{ b-i\ |\ \hbox{$i\in \{1, \ldots, 2n-1\}$ and $b-i\in \mu$}\}$$
arranged in decreasing order by spiral coordinate.

\begin{example}
If $\mu=(0,2,3,-1,1)$ and $b = (3, 1)$, and $b' = (4, -1)$ then
$$
\operatorname{attack}(3,1) = ((2,1),(1,0)), \qquad\hbox{and} \qquad
\operatorname{attack}(4,-1) = ((5,1), (3,1),(2,1), (1,0)) 
$$
$$
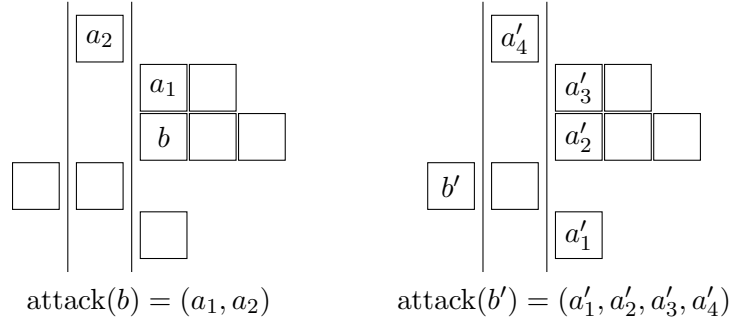
\begin{matrix}
\begin{matrix}
\begin{tikzpicture}[scale = 0.65]
\foreach \x\y in {1/0, 2/1, 2/2,  3/1, 3/2, 3/3, 4/0, 4/-1, 5/1}{
    \pgfmathtruncatemacro{\shift}{(\y>0) - (\y<0)}
    \begin{scope}[xshift = 0.3*\shift cm]
    \draw[thin] (\y+0.025, -\x-0.025) rectangle (0.975+\y, -\x-0.975);
    \draw (0.5+\y,-\x-0.5) node (\x\y) {};
    \end{scope}
}
\node at (10) {$a_{2}$};
\node at (21) {$a_1$};
\node at (31) {$b$};
\draw (-0.15, -0.75) -- (-0.15, -6.25);
\draw (1.15, -0.75) -- (1.15, -6.25);
\end{tikzpicture}
\end{matrix}
&\qquad
&\begin{matrix}
\begin{tikzpicture}[scale = 0.65]
\foreach \x\y in {1/0, 2/1, 2/2,  3/1, 3/2, 3/3, 4/0, 4/-1, 5/1}{
    \pgfmathtruncatemacro{\shift}{(\y>0) - (\y<0)}
    \begin{scope}[xshift = 0.3*\shift cm]
    \draw[thin] (\y+0.025, -\x-0.025) rectangle (0.975+\y, -\x-0.975);
    \draw (0.5+\y,-\x-0.5) node (\x\y) {};
    \end{scope}
}
\node at (10) {$a'_4$};
\node at (21) {$a'_3$};
\node at (31) {$a'_2$};
\node at (51) {$a'_1$};
\node at (4-1) {$b'$};
\draw (-0.15, -0.75) -- (-0.15, -6.25);
\draw (1.15, -0.75) -- (1.15, -6.25);
\end{tikzpicture}
\end{matrix}
\\
\mathrm{attack}(b) = (a_1,a_2)
&&\mathrm{attack}(b') = (a'_1,a'_2, a'_3,a_4')
\end{matrix}
$$
\end{example}

For $\ell\in \{1, \ldots, n\}$ define
\begin{equation}
d_\ell = s_{\ell-1}\cdots s_2s_1s_0
\qquad\hbox{and}\qquad
d_{-\ell} = s_\ell\cdots s_{n-1}s_n s_{n-1}\cdots s_2s_1s_0.
\label{didefn}
\end{equation}
Then define
\begin{equation}
u_\mu(r,c) = \begin{cases} 
1+\#\mathrm{attack}(r,c), &\hbox{if $\mu_r>0$ and $c=1$,} \\
-\Big(1+\#\mathrm{attack}(r,c)\Big), &\hbox{otherwise}.
\end{cases}
\label{boxfactor}
\end{equation}
The \emph{box-greedy reduced word for $u_\mu$} is 
\begin{equation}
u_\mu^\square = \prod_{(r,c)\in \dgplus(\mu) } d_{u_\mu(r,c)},
\label{boxgreedyredwd}
\end{equation}
where the product is taken over the boxes of $\mu$ in decreasing order by spiral coordinates. Note that if we read the product from right to left, then the entries are in increasing order by spiral coordinate.

In Proposition~\ref{prop:boxgreedycoroot} we compute the inversion set of $u_{\mu}$.  
This set has the same size as the number of letters in the box-greedy reduced word, which proves that it is a reduced word for $u_{\mu}$.

\subsection{The coroot sequence of $u_\mu^\square$}
\label{section:coroot}

Fix $\mu\in \ZZ^n$ and let $b=(r,c)$ be a box in $\mu$. The \emph{leg length of $\mu$ at $b$} is 
\begin{equation}
l(b) = l(r,c) = \vert \mu_r\vert -\vert c\vert.
\label{legdefn}
\end{equation}
Define a sequence of coroots $\beta(b) = \beta(r,c)$ as follows.
Let
$$
a_1 = (r_{1}, c_{1}), \quad
a_2 = (r_{2}, c_{2}), \quad\ldots,\quad
a_k = (r_{k}, c_{k})
$$
be the sequence of attacking boxes for the box $b=(r,c)$ coordinates (in decreasing spiral order).

\smallskip\noindent
\emph{Case $c=1$:} Then $\beta(b) = (\beta(b)_{k+1}, \ldots, \beta(b)_1)$, where
\begin{align}
\beta(b)_1 &= -\varepsilon_{|v_\mu(r)|} + \left(l(b) + \textstyle\frac{1}{2}\right) K ,  \nonumber \\
\beta(b)_j &= -\varepsilon_{|v_\mu(r)|} - \varepsilon_{|v_\mu(r_{j-1})|} + \left(l(b) + l(a_{j-1})+1\right)K \ \hbox{for $j \in \{k+1, k, \ldots, 2\}$}.
\label{rtseqa}
\end{align}

\smallskip\noindent
\emph{Case $c\ne1$:} Then $u_\mu(b) = -k<0$ and $\beta(b) = (\beta(b)_{-k}, \ldots, \beta(b)_{-n}, \beta(b)_{n-1}, \ldots, \beta(b)_1)$, where
\begin{align}
\beta(b)_{-j} &= -\varepsilon_{|v_\mu(r)|} + \varepsilon_j + ( l(b)+1)K,
&&\text{for $-j\in \{-k, -(k+1), \ldots, -(n-1)\}$,}
\nonumber
\\
\beta(b)_{-n} &= - \varepsilon_{|v_\mu(r)|} + (l(b)+1) K,
\nonumber  \\
\beta(b)_j &= - \varepsilon_{|v_\mu(r)|} - \varepsilon_j + (l(b)+1)K, 
&& \text{for $j\in \{n-1, \ldots, k+2\}$,} 
\label{rtseqb}
\\
\beta(b)_j &= - \varepsilon_{|v_\mu(r)|} - \varepsilon_{|v_\mu(r_{j-1})|} + (l(b)+l(a_{j-1})+1)K, 
&& \text{for $j \in \{k+1, k, \ldots, 2\}$,}
\nonumber  \\
\beta(b)_1 &= - \varepsilon_{|v_\mu(r)|} + (l(b)+\textstyle \frac{1}{2} \displaystyle)K, 
\nonumber
\end{align}

Recall from~\eqref{crtseq} that the \emph{coroot sequence} of the reduced word $w=s_{i_\ell}\cdots s_{i_1}$ is the sequence
$(\beta_\ell,\ldots, \beta_1)$ given by
\begin{equation*}
\beta_\ell = 
s_{i_1}s_{i_2}\cdots s_{i_{\ell-1}}\alpha_{i_\ell}, \quad
\beta_{\ell-1} = 
s_{i_1}s_{i_2}\cdots s_{i_{\ell-2}}\alpha_{i_{\ell-1}}, \quad \ldots \quad
\beta_2 = s_{i_1}\alpha_{i_2}, \quad
\beta_1 = \alpha_{i_1}.
\end{equation*}
The following result explicitly determines the coroot sequence of the box-greedy reduced word $u_\mu^\square$ in terms of the
combinatorics of the diagram of $\mu$.

\begin{prop}
\label{prop:boxgreedycoroot}
Let $\mu \in \ZZ^{n}$ and let $u_\mu^\square = s_{i_\ell}\cdots s_{i_1}$ be the box-greedy reduced word for the element $u_\mu$.
Then the coroot sequence $(\beta_\ell, \ldots, \beta_1)$ for $u_\mu^\square$ is 
$$\hbox{the concatenation (in increasing spiral order) of the sequences}\quad
\beta(r,c), \quad\hbox{for $(r,c)\in \mu$.}$$
\end{prop}

\begin{proof}
The proof is by induction on the number of boxes of $\mu$.
The base case is when $\mu$ has a single box.   Let
\begin{align*}
\begin{array}{lcl}
\gamma_1 =\alpha_0 = -\varepsilon_1+\hbox{$\frac12$}K, 
& \qquad &
\gamma_{-1} = s_0s_1\cdots s_n\cdots s_2\alpha_1 = -\varepsilon_1+\varepsilon_2+K, \\
\gamma_2 =s_0\alpha_1 = -\varepsilon_1-\varepsilon_2+K, 
& & 
\gamma_{-2} = s_0s_1\cdots s_n\cdots s_3\alpha_2 = -\varepsilon_1+\varepsilon_3+K, \\
\vdots & & \vdots \\
\gamma_n = s_0\cdots s_{n-2}\alpha_{n-1}= -\varepsilon_1-\varepsilon_n+K, 
& & 
\gamma_{-n} = s_0\cdots s_{n-1}\alpha_n = -\varepsilon_1+K, 
\end{array}
\end{align*}

\smallskip\noindent
If $\mu = \varepsilon_i$ then $u_{\varepsilon_i}^\square = s_{i-1}\cdots s_1s_0$ and 
$\beta_{\varepsilon_i}((i,1)) = ( \gamma_i, \ldots, \gamma_2, \gamma_1)$.

\smallskip\noindent
If $\mu = -\varepsilon_i$ then $u_{-\varepsilon_i}^\square = s_i\ldots s_n \cdots s_1s_0$ and
$\beta_{-\varepsilon_i}((i,-1)) = ( \gamma_{-i}, \ldots, \gamma_{-n}, \gamma_n, \ldots, \gamma_1)$.

The induction step is as follows.  Let $b=(r,c)$ be the first box in $\mu$ (in spiral order) and, 
with $u_\mu(r,c)$ as defined in~\eqref{boxfactor},  let $k=u_\mu(r,c)$ so that
$$u_\mu^\square = d_k u_\nu^\square,
\qquad\hbox{where}\qquad \nu = d_k^{-1}\mu.
$$
Then $c\in \{1,-1\}$ and $\nu$ has one fewer box than $\mu$.  Except for the first box of $\mu$,
the boxes in $\nu$ are in correspondence to the boxes in $\mu$ (simply by taking the boxes of $\mu$ and
of $\nu$ in increasing order by spiral coordinate).
Under this correspondence the leg length and the attacking boxes are preserved.
Using the homomorphism $\overline{\phantom{T}}\colon W \longrightarrow  W_{\mathrm{fin}}$ defined in~\eqref{Wtofin},
$v_\mu^{-1} = \overline{h_\mu v_\mu^{-1}} = \overline{u_\mu} = \overline{d_ku_\nu} = \overline{d_k h_\nu v_\nu^{-1}} 
= \overline{d_k}v_\nu^{-1}$
giving
$v_\nu = v_\mu \overline{d_k}$.
Thus
$$\beta_\mu(b') = \beta_\nu(b')\quad \hbox{if $b'$ is not the first box in $\mu$.}
$$

If $b=(r,c)$ is the first box in $\mu$ then the elements of the root sequence of $u_\mu^\square$ that correspond to the
factors in $d_k$ are $(\beta(b)_k, \ldots, \beta(b)_1)$ where
$$\beta_\mu(b)_\ell =u_\nu^{-1} \gamma_\ell = u_\mu^{-1}d_k\gamma_\ell
= v_\mu h_\mu^{-1}d_k \gamma_\ell
$$
because $d_{k}u_{\nu} = u_{\mu}$ and 
$u_\mu^{-1} = v_\mu h_\mu^{-1}$. 

\smallskip\noindent
\emph{Case $b=(r,1)$.}
In this case $\mu= (-\mu_1, \ldots, -\mu_{r-1}, \mu_r, \ldots, \mu_n)$ with $-\mu_1, \ldots, -\mu_{r-1}\in \ZZ_{\le 0}$ and $\mu_r\in \ZZ_{>0}$.
The leg length of $b$ is $l(b)=\mu_r-1$, the sequence of attacking boxes is 
$$(a_1, \ldots, a_{r-1}) = ((1,0), \ldots (r-1,0))\quad\hbox{with leg lengths}\quad
(l(a_1), \ldots, l(a_{r-1})) = (\mu_1, \ldots, \mu_{r-1})$$
and $k=r-1$.
If $j\in \{2, \ldots, k+1\}$ then
\begin{align*}
v_\mu h_\mu^{-1}d_r(\gamma_{j})
&= v_\mu h_\mu^{-1}s_{r-1}\cdots s_1 s_0(-\varepsilon_1-\varepsilon_j+K) 
= v_\mu h_\mu^{-1}(\varepsilon_r-K-\varepsilon_{j-1}+K) \\
&= v_\mu (\varepsilon_r+\mu_r K-K-\varepsilon_{j-1}+\mu_{j-1}K +K) \\
&= (\varepsilon_{v_{\mu}(r) } + l(b)K)
+(-\varepsilon_{v_{\mu}(j-1)} + l(a_{j-1}) K)
+ K \\
&= - \varepsilon_{\vert v_{\mu}(r)\vert } - \varepsilon_{\vert v_{\mu}(r_j)\vert } + (l(b) + l(a_{j-1}) + 1)K,
\end{align*}
where 
we use that since $\mu_r>0$ then $v_\mu(r)<0$ and since $-\mu_{j-1}<0$ then $v_\mu(j-1)>0$.
Similarly,
\begin{align*}
v_\mu h_\mu^{-1}d_r(\gamma_{1}) 
&= v_\mu h_\mu^{-1}s_{r-1}\cdots s_1 s_0 (-\varepsilon_1+\hbox{$\frac12$} K) 
= -\varepsilon_{\vert v_{\mu}(r)\vert } +(l(b) + \hbox{$\frac{1}{2})$}K.
\end{align*}

\medskip\noindent
\emph{Case $b=(r,-1)$.}
Then $\mu=(-\mu_1, \ldots, -\mu_r, 0, \ldots, 0)$ with $-\mu_1, \ldots, -\mu_{r-1}\in \ZZ_{\le0}$ and $-\mu_r\in \ZZ_{<0}$.
The leg length of $b$ is $l(b) = \mu_r-1$, the attacking boxes are 
$$(a_1, \ldots, a_{r-1}) = ((0,1), (0,2), \ldots, (0,r-1))
\quad\hbox{with leg lengths}\quad
(l(a_1), \ldots, l(a_{r-1}) = (\mu_1, \ldots, \mu_{r-1})
$$
and $k=r-1$.
If $j \in \{2, \ldots, k+1\}$, then, using that $-\mu_r<0$ so that $v_\mu(r)>0$ and $-\mu_{j-1}\le 0$ so that $v_\mu(j-1)>0$ gives
\begin{align*}
v_\mu h_\mu^{-1}d_{-r}(\gamma_{j}) 
&=v_\mu h_\mu^{-1}s_r\cdots s_n\cdots s_1s_0(-\varepsilon_{1} - \varepsilon_{j} + K)  
=v_\mu h_\mu^{-1}(-\varepsilon_{r} -K - \varepsilon_{j-1} + K)  \\
&=v_\mu (-\varepsilon_{r} +\mu_rK -K - \varepsilon_{j-1} +\mu_{j-1} K+ K)  \\
&= (-\varepsilon_{|v_{\mu}(r)|} + l(b))K)
+(-\varepsilon_{|v_{\mu}(j-1)|} + l(a_{j-1}) K)
+ K \\
&= -\varepsilon_{|v_{\mu}(r)|} -\varepsilon_{|v_{\mu}(j-1)|} + (l(b) + l(a_{j-1}) + 1)K.
\end{align*}
If $j \in \{k+2, \ldots, n-1\}$ then using that $v_\mu(j)=j$ gives
\begin{align*}
v_\mu h_\mu^{-1}d_{-r}(\gamma_{j}) 
&=v_\mu h_\mu^{-1}s_r\cdots s_n\cdots s_1s_0(\varepsilon_1-\varepsilon_j+K) 
=v_\mu h_\mu^{-1}(-\varepsilon_r-K-\varepsilon_j+K) \\
&=v_\mu (-\varepsilon_r+\mu_rK-\varepsilon_j-\mu_j K+K) 
= -\varepsilon_{|v_{\mu}(r)\vert} -\varepsilon_j + (l(b) + 1)K.
\end{align*}
If $-j \in \{-k, -(k+1), \ldots, -(n-1)\}$ then
\begin{align*}
v_\mu h_\mu^{-1}d_{-r}(\gamma_{-j}) 
&= v_\mu h_\mu^{-1}s_r\cdots s_n\cdots s_1s_0(-\varepsilon_{1} + \varepsilon_{j+1} + K)  
= v_\mu h_\mu^{-1}(-\varepsilon_{r} -K+ \varepsilon_{j+1} + K)  \\
&= v_\mu (-\varepsilon_{r} +\mu_r K-K+ \varepsilon_{j+1} -\mu_{j+1}K+ K)  
= - \varepsilon_{\vert v_{\mu}(r)\vert }-\varepsilon_{j+1} + (l(b) + 1)K.
\end{align*}
Similarly,
\begin{align*}
v_\mu h_\mu^{-1}d_{-r}(\gamma_{1}) 
&=v_\mu h_\mu^{-1}s_r\cdots s_n\cdots s_1s_0(-\varepsilon_1+\hbox{$\frac12$}K))  
= -\varepsilon_{\vert v_{\mu}(r)\vert } + (l(b) + \textstyle\frac{1}{2})K \qquad\hbox{and} \\
v_\mu h_\mu^{-1}d_{-r}(\gamma_{-n}) 
&= v_\mu h_\mu^{-1}d_{-r}(-\varepsilon_{1} + K) = -\varepsilon_{v_{\mu}(r)} + (- \mu_{r}-1 )K + K
= -\varepsilon_{v_{\mu}(r)} + (l(b) + 1)K.
\end{align*}
\end{proof}

\begin{example} \label{ex023m11}
Let $\mu = (0,2,3,-1,1)$. Then $n=5$ and $v_\mu = (5\overline{2}\overline{1}3\overline{4})$ and the box-greedy reduced word for $u_\mu$ is
\begin{align*}
u_{(0,2,3,-1,1)}^\square 
&= (s_1s_0)(s_2s_1s_0)(s_4s_3s_2s_1s_0)(s_5s_4s_3s_2s_1s_0) 
\\ &\qquad\qquad\cdot 
(s_4s_5s_4s_3s_2s_1s_0)(s_4s_5s_4s_3s_2s_1s_0)(s_1s_2s_3s_4s_3s_2s_1s_0),
\end{align*}
which is obtained by concatenating the sequences of the diagram below in the spiral order
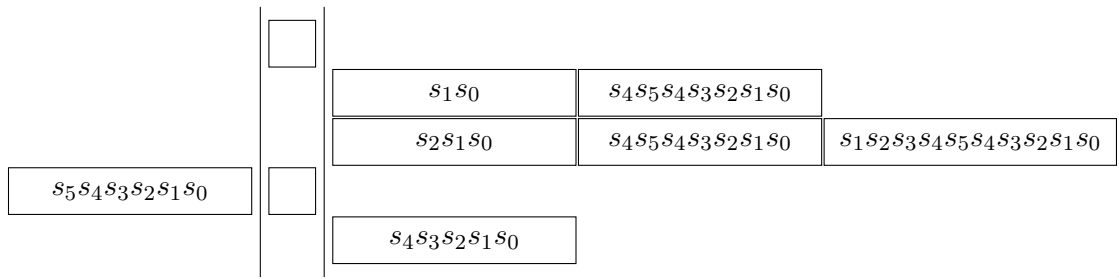
\[
\begin{matrix}\begin{tikzpicture}[scale = 0.65]
\foreach \x\y in {1/0, 4/0}{
    \draw[thin] (\y+0.025, -\x-0.025) rectangle (0.975+\y, -\x-.975);
    \draw (0.5+\y,-\x-0.5) node (\x\y) {};
}
\foreach \x\y in {4/-1}{
    \draw[thin] (5*\y+0.025-0.3, -\x-0.025) rectangle (4.975+5*\y-0.3, -\x-.975);
    \draw (2.5+5*\y-0.3,-\x-0.5) node (\x\y) {};
}
\foreach \x\y in {2/1, 2/2, 3/1, 3/2, 5/1}{
    \draw[thin] (5*\y+0.025 + 0.3 - 4, -\x-0.025) rectangle (4.975+5*\y + 0.3 - 4, -\x-.975);
    \draw (2.5+5*\y + 0.3 - 4,-\x-0.5) node (\x\y) {};
}
\foreach \x\y in {3/3}{
    \draw[thin] (5*\y+0.025 + 0.3 - 4, -\x-0.025) rectangle (4.975+5*\y + 0.3 - 3, -\x-.975);
    \draw (2.5+5*\y + 0.3 - 3.5,-\x-0.5) node (\x\y) {};
}
\node at (21) {$ s_{1}s_{0}$};
\node at (22) {$ s_{4}s_{5}s_{4}s_{3}s_{2}s_{1}s_{0}$};
\node at (31) {$ s_{2}s_{1}s_{0}$};
\node at (32) {$ s_{4}s_{5}s_{4}s_{3}s_{2}s_{1}s_{0}$};
\node at (33) {$ s_{1}s_{2}s_{3}s_{4}s_{5}s_{4}s_{3}s_{2}s_{1}s_{0}$};
\node at (4-1) {$ s_{5}s_{4}s_{3}s_{2}s_{1}s_{0}$};
\node at (51) {$ s_{4}s_{3}s_{2}s_{1}s_{0}$};
\draw (-0.15, -0.75) -- (-0.15, -6.25);
\draw (1.15, -0.75) -- (1.15, -6.25);
\end{tikzpicture}
.
\end{matrix}
\]
Note that $\ell(u^\square_{(0,2,3,-1,1)}) = 10+7+7+6+5+3+2=40$.
The root sequence for $u_{(0,2,3,-1,1)}^\square$ (read top to bottom in each box according to the spiral order) is
$$\beta_{(0,2,3,-1,1)} = (\beta_{40}, \ldots, \beta_2, \beta_1) = (\beta(2,1)_2, \ldots, \beta(3,3)_2, \beta(3,3)_1).$$
In Figure~\ref{ex: root sequence} we present the computation for all the boxes in $\mu$.

The first box in spiral order is $b = (2,1)$ and $d_{u_\mu(2,1)} = d_1=s_1s_0$ and
\begin{align*}
\beta(2,1)_1 &= u_\nu^{-1}\gamma_1 = v_\mu h_\mu^{-1}s_1s_0\alpha_0 
= v_\mu h_{(0,-2,-3,1,-1)}^{-1} s_1(\varepsilon_1-\hbox{$\frac12$}K)
\\
&= v_\mu h_{(0,-2,-3,1,-1)}^{-1} (\varepsilon_2-\hbox{$\frac12$}K)
= v_\mu (\varepsilon_2+(-\hbox{$\frac12$}+2)K)
= \varepsilon_{-2}+\hbox{$\frac32$}K
= -\varepsilon_2+\hbox{$\frac32$}K
\\
\hbox{and}\quad
\beta(2,1)_2 &= u_\nu^{-1}\gamma_2 
= v_\mu h_\mu^{-1}s_1s_0s_0\alpha_1 
= v_\mu h_\mu^{-1}(-\alpha_1)
= v_\mu h_{(0,-2,-3,1,-1)}^{-1} (-\varepsilon_1+\varepsilon_2)
\\
&= v_\mu (-\varepsilon_1+\varepsilon_2+2K)
= -\varepsilon_5-\varepsilon_2+2K,
\end{align*}
illustrating the induction step of the proof of Proposition~\ref{prop:boxgreedycoroot}.
\end{example}

\begin{figure}[h]
$$
\begin{matrix}
\begin{tikzpicture}[scale = 0.65]
\foreach \x\y in {1/0}{
    \draw[thin] (\y+0.025, -\x-0.025) rectangle (0.975+\y, -\x-.975);
    \draw (0.5+\y,-\x-0.5) node (\x\y) {};
}
\foreach \x\y in {2/1, 2/2}{
    \draw[thin] (5*\y+0.025 + 0.3 - 4, 10-6*\x-0.025) rectangle (4.975+5*\y + 0.3 - 4, 10-6*\x-3.5);
    \draw (2.5+5*\y + 0.3 - 4,10-6*\x-1.75) node (\x\y) {};
}
\foreach \x\y in {3/1, 3/2, 3/3}{
    \draw[thin] (5*\y+0.025 + 0.3 - 4, 18.5-8*\x-0.025) rectangle (4.975+5*\y + 0.3 - 4, 18.5-8*\x-4.7);
    \draw (2.5+5*\y + 0.3 - 4,18.5-8*\x-2.3) node (\x\y) {};
}
\foreach \x\y in {4/0}{
    \draw[thin] (\y+0.025, 9.8-5*\x-0.025) rectangle (0.975+\y, 9.8-5*\x-2.975);
    \draw (0.5+\y,9.8-5*\x-2) node (\x\y) {};
}
\foreach \x\y in {4/-1}{
    \draw[thin] (5*\y+0.025-0.3, 9.8-5*\x-0.025) rectangle (4.975+5*\y-0.3, 9.8-5*\x-2.975);
    \draw (2.5+5*\y-0.3,9.8-5*\x-1.5) node (\x\y) {};
}
\foreach \x\y in {5/1}{
    \draw[thin] (5*\y+0.025 + 0.3 - 4, 9.3-4.5*\x-0.025) rectangle (4.975+5*\y + 0.3 - 4, 9.3-4.5*\x-2.6);
    \draw (2.5+5*\y + 0.3 - 4,9.3-4.5*\x-1.35) node (\x\y) {};
}
\node at (21) {\tiny$\begin{array}{l}
\beta(2,1)_{2} =   -\varepsilon_{2}-\varepsilon_5 + 2K \\
\beta(2,1)_{1} =   -\varepsilon_{2}+\frac{3}{2} K 
\end{array}$};
\node at (22) {\tiny$\begin{array}{l}
\beta(2,2)_{-4} =  -\varepsilon_{2}+\varepsilon_5+K  \\
\beta(2,2)_{-5} =  -\varepsilon_{2}+K \\
\beta(2,2)_{5} =  -\varepsilon_{2}-\varepsilon_5+K \\
\beta(2,2)_{4} =  -\varepsilon_{2}-\varepsilon_{1}+3K \\
\beta(2,2)_{3} =  -\varepsilon_{2}-\varepsilon_{4}+K \\
\beta(2,2)_{2} =  -\varepsilon_{2}-\varepsilon_3+K \\
\beta(2,2)_{1} =  -\varepsilon_{2}+\frac{1}{2} K 
\end{array}$};
\node at (31) {\tiny$\begin{array}{l}
\beta(3,1)_{3} =  -\varepsilon_{1}-\varepsilon_5+3K \\
\beta(3,1)_{2} =  -\varepsilon_{1}-\varepsilon_{2}+4K \\
\beta(3,1)_{1} =  -\varepsilon_{1}+\frac52 K 
\end{array}$};
\node at (32) {\tiny$\begin{array}{l}
\beta(3,2)_{-4} = -\varepsilon_{1}+\varepsilon_5+2K \\
\beta(3,2)_{-5} = -\varepsilon_{1}+2K \\
\beta(3,2)_{5} = -\varepsilon_{1}-\varepsilon_5+2K \\
\beta(3,2)_{4} = -\varepsilon_{1}-\varepsilon_4+2K \\
\beta(3,2)_{3} = -\varepsilon_{1}-\varepsilon_3+2K \\
\beta(3,2)_{2} = -\varepsilon_{1}-\varepsilon_{2}+2K \\
\beta(3,2)_{1} = -\varepsilon_{1}+\frac{3}{2} K 
\end{array}$};
\node at (33) {\tiny$\begin{array}{l}
\beta(3,3)_{-1} = -\varepsilon_{1}+\varepsilon_2+K \\
\beta(3,3)_{-2} = -\varepsilon_{1}+\varepsilon_3+K \\
\beta(3,3)_{-3} = -\varepsilon_{1}+\varepsilon_4+K \\
\beta(3,3)_{-4} = -\varepsilon_{1}+\varepsilon_5+K \\
\beta(3,3)_{-5} = -\varepsilon_{1} + K \\
\beta(3,3)_{5} = -\varepsilon_{1}-\varepsilon_5+K \\
\beta(3,3)_{4} = -\varepsilon_{1}-\varepsilon_4+K \\
\beta(3,3)_{3} = -\varepsilon_{1}-\varepsilon_3+K \\
\beta(3,3)_{2} = -\varepsilon_{1}-\varepsilon_2+K \\
\beta(3,3)_{1} = -\varepsilon_{1}+\frac{1}{2} K 
\end{array}$};
\node at (4-1) {\tiny$\begin{array}{l}
\beta(4,-1)_{-5} =  - \varepsilon_3+K \\
\beta(4,-1)_{5} =  -\varepsilon_3-\varepsilon_5+K \\
\beta(4,-1)_{4} =  -\varepsilon_3-\varepsilon_2+2K \\
\beta(4,-1)_{3} =  -\varepsilon_3-\varepsilon_1+3K \\
\beta(4,-1)_{2} =  -\varepsilon_3-\varepsilon_4+K \\
\beta(4,-1)_{1} =  -\varepsilon_3+\frac{1}{2} K 
\end{array}$};
\node at (51) {\tiny$\begin{array}{l}
\beta(5,1)_{5} =  -\varepsilon_{4} -\varepsilon_{3}+2K \\ 
\beta(5,1)_{4} =  -\varepsilon_{4} -\varepsilon_5+K \\
\beta(5,1)_{3} =  -\varepsilon_{4} -\varepsilon_{2}+2K \\
\beta(5,1)_{2} =  -\varepsilon_{4} -\varepsilon_{1}+3K \\
\beta(5,1)_{1} =  -\varepsilon_{4} +\frac12 K 
\end{array}$};
\draw (-0.15, -0.75) -- (-0.15, -16.25);
\draw (1.15, -0.75) -- (1.15, -16.25);
\end{tikzpicture}
\end{matrix}
$$
\caption{Entries of the root sequence for $u_{(0,2,3,-1,1)}^\square$.}
\label{ex: root sequence}
\end{figure}

\newpage

\section{Step-by-step formulas for relative Koornwinder polynomials}\label{sec: step-by-step}

In this section we give a simplified description of the creation and alcove walk formulas for Koornwinder polynomials, emphasizing parallels with the type $GL_n$ case where possible.  Section~\ref{Erecursion} gives the recursive method for computing Koornwinder polynomials using creation operators.
This is the same as the very general method given in~\cite[\S5.10]{Mac03} except that the notations
are specific to the type $CC_n$ case, giving the recursion for creating Koornwinder polynomials.
Section~\ref{divided different op} rewrites the creation operators in terms of the divided-difference operators used in Schubert calculus for classical type root systems.
Section~\ref{alcove walk steps} derives the formulas for single alcove walk steps and Section~\ref{section:usvtableaux} assembles the alcove walk step formulas to give a (uncompressed) set-valued tableaux formula for Koornwinder polynomials.  The result in Theorem~\ref{thm:USVformula}
completes the Koornwinder case generalization of the type $GL_n$ results in~\cite{DR22}.

\subsection{The creation formula} \label{Erecursion}

While the Koornwinder polynomials are completely characterized by the eigenvalue and normalization conditions, this does not immediately give an efficient formula for each $E_\mu$. 
Fortunately, there is a family of operators $\tau_{0}, \tau_{1}, \ldots, \tau_{n}$ which recursively compute $E_{\mu}$.  
This is analogous to the way that, in Schubert calculus, the Schubert polynomials are constructed recursively using divided-difference operators.
In this section we describe how these recursions can be packaged nicely as a single formula for creating the Koornwinder polynomial
$E_\mu$. 

Recalling the operators $T_0, T_1, \ldots, T_n$ and $X_1$ on Laurent polynomials
$\KK[x^{\pm1}]$ defined in~\eqref{DGaction}, 
define an operator $T_0^{\vee}$
on $\KK[x^{\pm1}]$ by
\begin{equation}
(T_0^{\vee})^{-1} = X_1T_1\cdots T_n\cdots T_1.
\label{TiveeonKX}
\end{equation}
Using the fold functions $F^+_{\alpha_i}$ and $F^-_{\alpha_i}$ defined in~\eqref{Fdefn},  define
\begin{equation}
\begin{array}{r@{\;}c@{\;}l} 
\tau_{0} &=&  T_{0}^{\vee} + F^+_{\alpha_0}\\
&=& (T_{0}^{\vee})^{-1} +  F^-_{\alpha_0},
\end{array}
\qquad\text{and}\qquad
\begin{array}{r@{\;}c@{\;}l}
\tau_i &=& T_i + F^+_{\alpha_i} \\
&=& (T_i)^{-1} +  F^-_{\alpha_i},
\end{array}
\qquad\hbox{for $i\in \{1, \ldots, n\}$.}
\label{tauvee}
\end{equation}
By Proposition~\ref{taupastY}, if $\lambda = \lambda_1\varepsilon_1+\cdots \lambda_n \varepsilon_n+\frac{k}{2}K$
as in~\eqref{Ynot}, then
\begin{equation}
\tau_i Y^{\lambda} = Y^{s_i\lambda} \tau_i,
\qquad\hbox{for $i\in \{0,1,\ldots, n\}$,}
\label{taupY}
\end{equation}
where the action of $W$ on $Q = \ZZ\varepsilon_1+\cdots\ZZ\varepsilon_n + \ZZ\frac12 K$ is as defined in
\eqref{lvl0onaZ}. 
Recall also the action of the group $W$ (generated by $s_0,\ldots, s_n$) on $\ZZ^n$ given in~\eqref{WonZn}. 

The relation~\eqref{taupY} is the reason that the electronic Macdonald polynomials $E_\mu$ can be defined by recursive relations:
For $\mu\in \ZZ^n$ define $\widehat{E}_\mu$ by
\begin{enumerate}
\item[(E0)] $\widehat{E}_{(0,\ldots, 0)} = 1$,
\item[(E1)] If $\mu_1 \le 0$ then define $\widehat{E}_{s_0\mu} = 
\tau_0 \widehat{E}_\mu$,
\item[(E2)] If $i\in \{1, \ldots, n-1\}$ and $\mu_i>\mu_{i+1}$ then define
$\widehat{E}_{s_i\mu} = 
\tau_i \widehat{E}_\mu$,
\item[(E3)] 
If $\mu_n>0$ then define $\widehat{E}_{s_n\mu} =  
\tau_n \widehat{E}_\mu$.
\end{enumerate}

Let $\mu = (\mu_1, \ldots, \mu_n)\in \ZZ^n$.
Let $v_\mu\in W_{\mathrm{fin}}$ be the minimal length signed permutation such that
$v_\mu\mu$ is weakly increasing with all entries $\le0$ (see $\eqref{vmuvals}$).  Let $u_\mu$ be the minimal
length element of $W$ such that $u_\mu(0,\ldots,0) = (\mu_1,\ldots, \mu_n)$ (see~\eqref{umudefn}).
The creation formula for $\widehat{E}_\mu$ is  (see also~\cite[\S4.5]{CR24} and~\cite[(5.101.3)]{Mac03})
\begin{equation}
\widehat{E}_\mu = \tau_{u_\mu} \cdot 1 \qquad\hbox{and}\qquad
E_\mu 
= t^{-\frac12\ell_s(v_\mu^{-1})}t_n^{-\frac12\ell_d(v_\mu^{-1})}
\widehat{E}_\mu,
\label{Ecr}
\end{equation}
where
$\tau_{u_\mu} = \tau_{i_1}\cdots \tau_{i_\ell}$ if $u_\mu = s_{i_1}\cdots s_{i_\ell}$ is a reduced word for $u_\mu$.

\begin{remark} (\textbf{Stabilizers $W_\mu$})
As in 
\eqref{relMac}, let $W_{\mathrm{fin}}$ be the group of signed permutations (generated by $s_1, \ldots, s_n$)
and
$$
E^z_\mu = 
t^{\frac12(\ell_s(v_\mu^{-1})-\ell_s(zv_\mu^{-1}))}t_n^{\frac12(\ell_d(v_\mu^{-1})-\ell_d(zv_\mu^{-1}))}
T_z E_\mu.
$$
where the normalization constant 
$t^{\frac12(\ell_s(v_\mu^{-1})-\ell_s(zv_\mu^{-1}))}t_n^{\frac12(\ell_d(v_\mu^{-1})-\ell_d(zv_\mu^{-1}))}$ 
is determined by requiring the coefficient of $x^{z\mu}$ to be 1.
If $\mu\in \ZZ^n$ is such that $\mu=s_i\mu$ for some $i\in \{1, \ldots, n-1\}$  then 
$0=\tau_i E_\mu = (T_i-t^{\frac12})E_\mu$ so that $T_iE_\mu = t^{\frac12}E_\mu$ (see~\cite[(5.4.2)]{Mac03}).
Similarly, if $\mu\in \ZZ^n$ is such that  $\mu=s_n\mu$ then $T_nE_\mu = t_n^{\frac12}E_\mu$.  
These equalities imply
that if 
$$\mu\in \ZZ^n\qquad\hbox{and}\qquad W_\mu = \{ v\in W_{\mathrm{fin}}\ |\ v\mu=\mu\}$$
and 
$$z\in W_{\mathrm{fin}}\quad\hbox{and}\quad
z=yv\quad\hbox{with}\quad v\in W_\mu\quad\hbox{and}\quad\ell(z)=\ell(y)+\ell(v)$$
then
\begin{equation}
\widehat{E}^z_\mu 
= t^{\frac12\ell_{s}(v)} t_{n}^{\frac{1}{2}\ell_{d}(v)} \widehat{E}^{y}_\mu.
\label{stabfactor}
\end{equation}
\qed
\end{remark}

\subsection{Creation via divided-difference operators}\label{divided different op}

Following~\eqref{cfcndefn},~\eqref{Fdefn} and~\eqref{Ysroots}, let
$$
\begin{array}{c} 
C_{\alpha_n} = \dfrac{(1-t_0^{\frac12}t_n^{\frac12}Y_n^{-1})(1+t_n^{\frac12}t_0^{-\frac12}Y_n^{-1})}{(1-Y_n^{-2})},
\qquad
F^-_{\alpha_0}
= \dfrac{((u_n^{-\frac12}-u_n^{\frac12})
+(u_0^{-\frac12}-u_0^{\frac12}) q^{-\frac12}Y_1^{-1}) qY_1^2 }
{1- q Y_1^2 }\\[10pt]
\hbox{and}\qquad
C_{\alpha_i} = t^{-\frac12}\dfrac{1-tY_i^{-1}Y_{i+1} }{ 1-Y_i^{-1}Y_{i+1} },
\qquad\hbox{for $i\in \{1, \ldots, n-1\}$.}
\end{array}
$$

Proposition~\ref{cformsteps} provides a recursive procedure for building $\widehat{E}_\mu$ from $\widehat{E}_0=1$
via the divided-difference operators $\partial_i$ used in Schubert calculus.

\begin{prop} \label{cformsteps}
(Creation formula steps)
Recall the notation for the operators  $\partial_i$ defined in~\eqref{ddivopdefn}, the evaluation homomorphisms $\ev_\mu$ defined in 
\eqref{evhomdefn} and the definition of $\widehat{E}_\mu$ in~\eqref{Ecr}.  Let $\mu = (\mu_1, \ldots, \mu_n)\in \ZZ^n$.
\item[(a)]
If $\mu_i > \mu_{i+1} > 0$ or $0 > \mu_i > \mu_{i+1}$ for some $i\in \{1,\ldots, n-1\}$ then\quad
$$\widehat{E}_{s_i\mu} = \ev_\mu(C_{\alpha_i})\widehat{E}_\mu 
- (t^{-\frac12} x_{i+1}-t^{\frac12} x_i)\partial_i \widehat{E}_\mu.
$$
\item[(b)]
If $\mu_n>0$ then\quad
$$\widehat{E}_{s_n\mu} = \ev_\mu(C_{\alpha_n}) \widehat{E}_\mu
- t_n^{-\frac12} (1-t_n^{\frac12}u_n^{\frac12}x_n)(1+t_n^{\frac12}u_n^{-\frac12}x_n)\partial_n \widehat{E}_\mu.
$$
\item[(c)]
If $\mu_1\le 0$ then 
$$
\widehat{E}_{s_0\mu} =
\ev_\mu(F^-_{\alpha_0})\widehat{E}_\mu  + t_0^{-\frac12} \ev_\mu(Y_1) x_1 \widehat{E}_\mu -
\ev_\mu(Y_1)
\big(t_0^{-\frac12} x_{1}^{3} + 
(u_0^{-\frac12} - u_0^{\frac12}) q^{\frac12}x_{1}^{2} - t_0^{\frac12} q x_1\big) \partial_{0}\widehat{E}_\mu.
$$
\end{prop}
\begin{proof}
To streamline the notation, let
$$c_{\alpha_0}^X = t_0^{-\frac12} 
\frac{(1-t_0^{\frac12}u_0^{\frac12}q^{\frac12}x_1^{-1})
(1+t_0^{\frac12}u_0^{-\frac12}q^{\frac12}x_1^{-1})} {1-qx_1^{-2}},
\qquad
c_{\alpha_n}^X =  t_n^{-\frac12} \frac{(1-t_n^{\frac12}u_n^{\frac12} x_n)(1+t_n^{\frac12}u_n^{-\frac12}x_n)}{(1-x_n^2)},
$$
$$\hbox{and}\qquad
c_{\alpha_i}^X = t^{-\frac12}\frac{1-tx_ix_{i+1}^{-1}}{1-x_ix_{i+1}^{-1}},\qquad\hbox{for $i\in \{1, \ldots, n-1\}$.} 
$$
Letting $t_{\alpha_0} = t_0$, $t_{\alpha_n} = t_n$ and $t_{\alpha_i} = t$ for $i\in \{1, \ldots, n-1\}$, 
the formulas in~\eqref{DGaction} are
$$T_i = t_{\alpha_i}^{\frac12} - c_{\alpha_i}^X(1-\xi_{s_i}), 
\qquad\hbox{for $i\in \{0,1,\ldots, n\}$.}
$$
For $i\in \{1,\ldots, n-1\}$:  
By~\eqref{tauvee},~\eqref{Fdefn}, and~\eqref{cortparams}, 
$\tau_i = T_i+ F^+_{\alpha_i} = T_i + (t^{-\frac12}- C_{-\alpha_i})$.  Then~\eqref{cplusc} gives
\begin{align*}
\tau_i 
&= T_i + (t^{-\frac12} - C_{-\alpha_i})
=t^{\frac12}-c_{\alpha_i}^X (1-\xi_{s_i}) + t^{-\frac12} - C_{-\alpha_i}
\\
&=C_{\alpha_i}+C_{-\alpha_i} -  c_{\alpha_i}^X (1-\xi_{s_i}) - C_{-\alpha_i}
= C_{\alpha_i} -  c_{\alpha_i}^X (1-\xi_{s_i}).
\end{align*}
In view of the recursive formula in (E2) and the definition of the operator $\partial_i$, this gives the formula in (a).

\smallskip\noindent
For $i=n$:
By~\eqref{tauvee} and~\eqref{Fdefn}, 
$\tau_n = T_n+ F^+_{\alpha_n} = T_n + (t^{-\frac12}_{\alpha_n}- C_{-\alpha_n})$.  By~\eqref{cplusc}
\begin{align*}
\tau_n 
&= T_n+(t^{-\frac12}_{\alpha_n} - C_{-\alpha_n}) 
=t^{\frac12}_{\alpha_n}-c_{\alpha_n}^X (1-\xi_{s_n}) + t^{-\frac12}_{\alpha_n} - C_{-\alpha_n}
\\
&=C_{\alpha_n}+C_{-\alpha_n} -  c_{\alpha_n}^X (1-\xi_{s_n}) - C_{-\alpha_n}
= C_{\alpha_n}  -  c_{\alpha_n}^X (1-\xi_{s_n}).
\end{align*}
In view of the recursive formula in (E3) and the definition of the operator $\partial_n$, this gives the formula in (b).

\smallskip\noindent
For $i=0$: By~\eqref{tauvee} and~\eqref{Fdefn} and~\eqref{Yjdefn}, 
\begin{align*}
\tau_0 
&= (T_0^{\vee})^{-1}+F^-_{\alpha_0} 
=  X_1T_1\cdots T_n\cdots T_1 +F^-_{\alpha_0} 
= X_1 (T_0^{\vee})^{-1}Y_1 +F^-_{\alpha_0}.
\end{align*}
By~\eqref{DHeckedefnI}, $(T_0^{\vee})^{-1} = (T_0^{\vee}) - (t_0^{\frac12}-t_0^{-\frac12})$, and so
\begin{align*}
\tau_0 
&= X_1  (T_0-(t_0^{\frac12}-t_0^{-\frac12})) Y_1 + F^-_{\alpha_0}  
= X_1 (t_0^{\frac12}-c_{\alpha_0}^X(1-\xi_{s_0}) -(t_0^{\frac12}-t_0^{-\frac12})) Y_1 + F^-_{\alpha_0}  \\
&= X_1 (t_0^{-\frac12}-c_{\alpha_0}^X(1-\xi_{s_0})) Y_1 + F^-_{\alpha_0}.
\end{align*}
Thus,
$\tau_0\widehat{E}_\mu
= \Big(x_1 (t_0^{-\frac12}-c_{\alpha_0}^X(1-\xi_{s_0})) \ev_\mu(Y_1) + \ev_\mu(F^-_{\alpha_0})\Big) \widehat{E}_\mu$
and
\begin{align*}
x_1(t_0^{-\frac12}-c_{\alpha_0}^X(1-\xi_{s_0})) \ev_\mu(Y_1) 
&= t_0^{-\frac12} \ev_\mu(Y_1) x_1 \Big( 1 -
\frac{(1-t_0^{\frac12}u_0^{\frac12}q^{\frac12}x_1^{-1})
(1+t_0^{\frac12}u_0^{-\frac12}q^{\frac12}x_1^{-1})}
{1-qx_1^{-2}}
(1-\xi_{s_0})\Big)  \\[1ex]
&= t_0^{-\frac12} \ev_\mu(Y_1) x_1 \Big( 1 -
(x_{1}-t_0^{\frac12}u_0^{\frac12}q^{\frac12})
(x_{1}+t_0^{\frac12}u_0^{-\frac12}q^{\frac12})\partial_{0}\Big)   \\
&= t_0^{-\frac12} \ev_\mu(Y_1) x_1 -
t_0^{-\frac12} \ev_\mu(Y_1) x_1 (x_{1}-t_0^{\frac12}u_0^{\frac12}q^{\frac12})
(x_{1}+t_0^{\frac12}u_0^{-\frac12}q^{\frac12}) \partial_{0}  \\
&= t_0^{-\frac12} \ev_\mu(Y_1) x_1 -
\ev_\mu(Y_1)
(t_0^{-\frac12} x_{1}^{3} + 
(u_0^{-\frac12} - u_0^{\frac12}) q^{\frac12}x_{1}^{2} - t_0^{\frac12} q x_1) \partial_{0} 
\end{align*}
So $\tau_0\widehat{E}_\mu = \ev_\mu(F^-_{\alpha_0})\widehat{E}_\mu  + t_0^{-\frac12} \ev_\mu(Y_1) x_1 \widehat{E}_\mu -
\ev_\mu(Y_1)
\big(t_0^{-\frac12} x_{1}^{3} + 
(u_0^{-\frac12} - u_0^{\frac12}) q^{\frac12}x_{1}^{2} - t_0^{\frac12} q x_1\big) \partial_{0}\widehat{E}_\mu$, and this gives the formula in (c).
\end{proof}

\subsection{Alcove walk steps}\label{alcove walk steps}

Proposition~\ref{alcwksteps} provides a recursive method of computing the (unnormalized) relative Macdonald polynomial
$\widehat{E}^z_\mu$, one simple reflection at a time.
Iterating these recursive steps provides the (uncompressed) set-valued tableaux formula for relative Macdonald polynomials given in Theorem~\ref{thm:USVformula}.

To set the stage for the statement of Proposition~\ref{alcwksteps} recall the following notations.
From~\eqref{Fdefn},~\eqref{Ysroots} and~\eqref{cortparams},
\begin{align*}
F^+_{\alpha_0}
=
\frac{(u_n^{-\frac12}-u_n^{\frac12})
+(u_0^{-\frac12}-u_0^{\frac12}) q^{\frac12}Y_1 }
{1-qY_1^2 },
\qquad
F^-_{\alpha_0}
=
\frac{((u_n^{-\frac12}-u_n^{\frac12})
+(u_0^{-\frac12}-u_0^{\frac12}) q^{-\frac12}Y_1^{-1}) qY_1^2 }
{1- q Y_1^2 },
\end{align*}
\begin{align*}
F^+_{\alpha_n}
=\frac{(t_n^{-\frac12}-t_n^{\frac12})+(t_0^{-\frac12}-t_0^{\frac12})Y_n^{-1} }
{1-Y_n^{-2}},
\qquad
F^-_{\alpha_n}
=\frac{( (t_n^{-\frac12}-t_n^{\frac12})+(t_0^{-\frac12}-t_0^{\frac12})Y_n)Y_n^{-2} }
{1-Y_n^{-2}}
\end{align*}
and, for $i\in \{1, \ldots, n-1\}$,
$$F^+_{\alpha_i} = t^{-\frac12}\frac{(1-t)} {1-Y_i^{-1}Y_{i+1}}
\qquad\hbox{and}\qquad
F^-_{\alpha_i} = t^{-\frac12}\frac{(1-t)Y_i^{-1}Y_{i+1}} {1-Y_i^{-1}Y_{i+1}}.
$$
As in~\eqref{evhomdefn}, 
the homomorphism
$\ev_\nu\colon \KK[Y_1^{\pm1}, \ldots, Y_n^{\pm1}] \longrightarrow  \KK$ is given by
\begin{equation*}
\ev_\nu(Y_j) = q^{-\nu_j}t^{-v_\nu(j)} (t^{\frac12}_0t^{\frac12}_n t^n)^{\mathrm{sgn}(v_\nu(j))},
\quad\hbox{for $j\in \{1, \ldots, n\}$.}
\end{equation*}
Let $z\in W_{\mathrm{fin}}$ and let $i\in \{1, \ldots, n\}$.  Viewing $z$ as a signed permutation then
$$\ell(zs_i)=\begin{cases}
\ell(z)+1, \quad\hbox{if $z(i) \prec z(i+1)$,} \\
\ell(z)-1, \quad\hbox{if $z(i) \succ z(i+1)$.}
\end{cases}
$$
where the order $\prec$ on $\{1, \ldots, n, -n, \ldots, -1\}$ is given by 
$1\prec 2\prec \cdots \prec n \prec -n \prec \cdots \prec -2 \prec -1$ (see Remark~\ref{sgnorder}).

\begin{prop}\label{alcwksteps}

Let $z\in W_{\mathrm{fin}}$, $\mu= (\mu_1, \ldots, \mu_n)\in \ZZ^n$ and let $\widehat{E}^z_\mu = T_z \widehat{E}_\mu$.
\begin{itemize}
\item[(a)] If 
$\mu_1\in \ZZ_{>0}$ 
then
$$\widehat{E}^z_\mu = \begin{cases}
x_{\vert z(1)\vert} \widehat{E}^{z\overline{s_0}}_\nu+ \ev_\nu(F_{\alpha_0}^-) \widehat{E}^z_\nu, &\hbox{if $z(1)>0$,} \\
x^{-1}_{\vert z(1)\vert} \widehat{E}^{z\overline{s_0}}_\nu+ \ev_\nu(F_{\alpha_0}^+) \widehat{E}^z_\nu, &\hbox{if $z(1)<0$,} 
\end{cases}
\quad\hbox{where}\quad
\nu = (-(\mu_1-1), \mu_2, \ldots, \mu_n)
$$
and $\overline{s_0} = s_1\cdots s_{n-1} s_n s_{n-1} \cdots s_1\in W_{\mathrm{fin}}$ is the 
signed permutation  $\overline{s_0}\in W_{\mathrm{fin}}$ that switches $1$ and $-1$.
\item[(b)] 
If  $\mu_i < \mu_{i+1}$ for some $i\in \{1, \ldots, n-1\}$ then 
$$\widehat{E}^z_\mu = \begin{cases}
\widehat{E}^{zs_i}_\nu+ \ev_\nu(F_{\alpha_i}^-) \widehat{E}^z_\nu, &\hbox{if $z(i)\succ z(i+1)$,} \\
\widehat{E}^{zs_i}_\nu+ \ev_\nu(F_{\alpha_i}^+) \widehat{E}^z_\nu, &\hbox{if $z(i)\prec z(i+1)$,} 
\end{cases}
\quad\hbox{where}\quad
\nu = (\mu_1, \ldots, \mu_{i-1}, \mu_{i+1}, \mu_i, \mu_{i+2}, \ldots, \mu_n).
$$
\item[(c)] 
If  $\mu = (\mu_1, \ldots, \mu_{n-1}, -\mu_n)$ with $\mu_n>0$
then
$$\widehat{E}^z_\mu = \begin{cases}
\widehat{E}^{zs_n}_\nu+ \ev_\nu(F_{\alpha_n}^-) \widehat{E}^z_\nu, &\hbox{if $z(n)<0$,} \\
\widehat{E}^{zs_n}_\nu+ \ev_\nu(F_{\alpha_n}^+) \widehat{E}^z_\nu, &\hbox{if $z(n)>0$,} 
\end{cases}
\quad\hbox{where}\quad
\nu = (\mu_1, \ldots, \mu_{n-1}, \mu_n).
$$
\end{itemize}
\end{prop}
\begin{proof}
(a) By definition of $\tau_0$,
\begin{align*}
\widehat{E}^z_\mu &= T_z \tau_0 \widehat{E}_\nu
=\begin{cases}
T_z((T_0)^{-1}+F_{\alpha_0}^-)\widehat{E}_\nu, &\hbox{if $z(1)>0$,} \\
T_z(T_0 +F_{\alpha_0}^+)\widehat{E}_\nu, &\hbox{if $z(1)<0$.}
\end{cases}
\end{align*}
Then, by Proposition~\ref{xzsicross},
\begin{align*}
\widehat{E}^z_\mu 
&= \begin{cases}
x_{z(1)}T_{z\overline{s_0}} \widehat{E}_\nu + \ev_\nu(F^-_{\alpha_0}) T_z\widehat{E}_\nu, &\hbox{if $z(1)>0$,} \\
x_{z(1)}T_{z\overline{s_0}} \widehat{E}_\nu + \ev_\nu(F^+_{\alpha_0}) T_z \widehat{E}_\nu, &\hbox{if $z(1)<0$,} 
\end{cases}
\end{align*}
(b) If $T_z T_i = T_{zs_i}$ then $\ell(zs_i) = \ell(z)+1$ and $z(i)\prec z(i+1)$.  
If $T_zT_i^{-1} = T_{zs_i}$  then $\ell(zs_i) = \ell(z)-1$ and 
$z(i)\succ z(i+1)$.  Therefore
\begin{align*}
\widehat{E}^z_\mu 
&= T_z \tau_i \widehat{E}_\nu
=\begin{cases}
T_z((T_i)^{-1}+F_{\alpha_i}^-)\widehat{E}_\nu, &\hbox{if $z(i)\succ z(i+1)$,} \\
T_z(T_i +F_{\alpha_i}^+)\widehat{E}_\nu, &\hbox{if $z(i)\prec z(i+1)$,}
\end{cases}
\end{align*}
which gives
\begin{align*}
\widehat{E}^z_\mu 
&= \begin{cases}
T_{zs_i} \widehat{E}_\nu + \ev_\nu(F^-_{\alpha_i}) T_z\widehat{E}_\nu, &\hbox{if $z(i)\succ z(i+1)$,} \\
T_{zs_i} \widehat{E}_\nu + \ev_\nu(F^+_{\alpha_i}) T_z \widehat{E}_\nu, &\hbox{if $z(i)\prec z(i+1)$.} 
\end{cases}
\end{align*}
(c) If $T_z T_n = T_{zs_n}$ then $\ell(zs_n)=\ell(z)+1$ and $z(n)>0$.  If $T_zT_n^{-1} = T_{zs_n}$  then $\ell(zs_n)=\ell(z)-1$ and $z(n)<0$.  Therefore
\begin{align*}
\widehat{E}^z_\mu 
&= T_z \tau_n \widehat{E}_\nu
=\begin{cases}
T_z((T_n)^{-1}+F_{\alpha_n}^-)\widehat{E}_\nu, &\hbox{if $z(n)<0$,} \\
T_z(T_n +F_{\alpha_n}^+)\widehat{E}_\nu, &\hbox{if $z(n)>0$,}
\end{cases}
\end{align*}
and we conclude that
\[
\widehat{E}^z_\mu 
= \begin{cases}
T_{zs_n} \widehat{E}_\nu + \ev_\nu(F^-_{\alpha_n}) T_z\widehat{E}_\nu, &\hbox{if $z(n)<0$,} \\
T_{zs_n} \widehat{E}_\nu + \ev_\nu(F^+_{\alpha_n}) T_z \widehat{E}_\nu, &\hbox{if $z(n)>0$.} 
\end{cases}
\]
\end{proof}

\subsection{The (uncompressed) set-valued tableaux formula}
\label{section:usvtableaux}

This subsection establishes the alcove walk formula~\cite[Theorem 3.1]{RY08} for the relative Koornwinder polynomial $E^z_\mu$ 
in a set-valued tableaux form.

\subsubsection{Alcove walks and set-valued tableaux}

Let $z\in W_{\mathrm{fin}}$ and let $\mu = (\mu_1, \ldots ,\mu_n)\in \ZZ^n$, and recall the box-greedy reduced word 
$u^\square_\mu = s_{i_\ell}\cdots s_{i_1}$ from Section~\ref{section:bgrw}, 
$$u_\mu^\square = s_{i_\ell}\cdots s_{i_1} = \prod_{b\in \mu} d_{u_\mu(b)}.
$$
An \emph{alcove walk 
of type $(z, u^\square_\mu)$} is 
$$
\hbox{a sequence $p=(p_{\ell+1},p_\ell, \ldots, p_1)$ in $W$ \quad such that\quad
$p_{\ell+1} = z$ and $p_k\in \{ p_{k+1}, p_{k+1}s_{i_k}\}$.}
$$
An \emph{uncompressed set-valued tableau (USV-tableau)} is 
\begin{equation}
\hbox{an assignment of a subset $T(b)\subseteq [0,u_\mu(b)]$ to each box $b=(r,c)$ in $\mu$, }
\label{USVdef}
\end{equation}
where $u_\mu(b)$ is defined in~\eqref{boxfactor} and
\begin{align*}
[0,k] &= \{0,1,\ldots, k\}, &&\hbox{for $k\in \{0, \ldots, n-1\}$, and}
\\
[0, -k] &= \{0,1,\ldots, n-1, -n, -(n-1), \ldots, -k\}, &&\hbox{for $k\in \{1,\ldots, n\}$.}
\end{align*}

Constructing an alcove walk amounts to choosing one of the  $2^\ell$ ways to 
omit ($p_{k} = p_{k+1}$) or include ($p_{k} = p_{k+1}s_{i_{k}}$) each letter of $u_{\mu}^{\square}$.
The set $[0, u_\mu(b)]$ indexes the letters of $d_{u_\mu(b)}$ and the subset $T(b)$ specifies the factors in the subword $d_{u_{\mu}(b)}$ to omit to specify the alcove walk $p$.  Thus the USV-tableau $T$ is the same data as an alcove walk $p$.  We often identify the
alcove walk $p$ and the corresponding uncompressed set-valued tableau $T$.

\begin{example} \label{AWexample}
An example of a USV-tableau of shape $\mu = (0,2,3,-1,1)$ is
$$
T=\begin{matrix}
\begin{tikzpicture}[scale = 0.6, baseline = 0.6*-3.5cm]
\foreach \x\y in {1/0, 4/0}{
    \draw[thin] (\y+0.025, -\x-0.025) rectangle (0.975+\y, -\x-.975);
    \draw (0.5+\y,-\x-0.5) node (\x\y) {};
}
\foreach \x\y in {4/-1}{
    \pgfmathsetmacro{\boxwd}{3};
    \draw[thin] (\boxwd*\y-0.275, -\x-0.025) rectangle (\boxwd + \boxwd*\y - 0.325, -\x-.975);
    \draw (0.5*\boxwd + \boxwd*\y - 0.3,-\x-0.5) node (\x\y) {};
}
\foreach \x\y in {2/1, 2/2, 3/1, 3/2, 3/3, 5/1}{
    \pgfmathsetmacro{\boxwd}{3};
    \draw[thin] (\boxwd*\y - \boxwd + 1.325, -\x-0.025) rectangle (\boxwd*\y + 1.325, -\x-.975);
    \draw (\boxwd*\y + 1.3 - 0.5*\boxwd,-\x-0.5) node (\x\y) {};
}
\node at (21) {$\emptyset$};
\node at (22) {$\{-5, 0\}$};
\node at (31) {$\{1\}$};
\node at (32) {$\{3\}$};
\node at (33) {$\{-3, 2\}$};
\node at (4-1) {$\{0\}$};
\node at (51) {$\{3, 1\}$};
\draw (-0.15, -0.75) -- (-0.15, -6.25);
\draw (1.15, -0.75) -- (1.15, -6.25);
\end{tikzpicture}
\end{matrix}
$$
which is identified with
$$
\begin{matrix}
\begin{tikzpicture}[scale = 0.6, baseline = 0.6*-3.5cm]
\foreach \x\y in {1/0, 4/0}{
    \draw[thin] (\y+0.025, -\x-0.025) rectangle (0.975+\y, -\x-.975);
    \draw (0.5+\y,-\x-0.5) node (\x\y) {};
}
\foreach \x\y in {4/-1}{
    \pgfmathsetmacro{\boxwd}{5};
    \draw[thin] (\boxwd*\y-0.275, -\x-0.025) rectangle (\boxwd + \boxwd*\y - 0.325, -\x-.975);
    \draw (0.5*\boxwd + \boxwd*\y - 0.3,-\x-0.5) node (\x\y) {};
}
\foreach \x\y in {2/1, 2/2, 3/1, 3/2, 5/1}{
    \pgfmathsetmacro{\boxwd}{5};
    \draw[thin] (\boxwd*\y - \boxwd + 1.325, -\x-0.025) rectangle (\boxwd*\y + 1.325, -\x-.975);
    \draw (\boxwd*\y + 1.3 - 0.5*\boxwd,-\x-0.5) node (\x\y) {};
}
\foreach \x\y in {3/3}{
    \pgfmathsetmacro{\boxwd}{5};
    \pgfmathsetmacro{\newboxwd}{6.2};
    \draw[thin] (\boxwd*\y - \boxwd + 1.325, -\x-0.025) rectangle (\boxwd*\y - \boxwd + \newboxwd + 1.325 , -\x-.975);
    \draw (\boxwd*\y - \boxwd + 1.3 + 0.5*\newboxwd,-\x-0.5) node (\x\y) {};
}
\node at (21) {$ s_{1}s_{0}$};
\node at (22) {$ s_{4}\cancel{s_{5}}s_{4}s_{3}s_{2}s_{1}\cancel{s_{0}}$};
\node at (31) {$ s_{2}\cancel{s_{1}}s_{0}$};
\node at (32) {$ s_{4}s_{5}s_{4}\cancel{s_{3}}s_{2}s_{1}s_{0}$};
\node at (33) {$ s_{1}s_{2}\cancel{s_{3}}s_{4}s_{5}s_{4}s_{3}\cancel{s_{2}}s_{1}s_{0}$};
\node at (4-1) {$ s_{5}s_{4}s_{3}s_{2}s_{1}\cancel{s_{0}}$};
\node at (51) {$ s_{4}\cancel{s_{3}}s_{2}\cancel{s_{1}}s_{0}$};
\draw (-0.15, -0.75) -- (-0.15, -6.25);
\draw (1.15, -0.75) -- (1.15, -6.25);
\end{tikzpicture}
\end{matrix}
$$
and the alcove walk (for $z=1$)
\begin{multline*}
p = (1, s_1, s_1s_0, s_1s_0s_2, s_1s_0s_2, s_1s_0s_2s_0, s_1s_0s_2s_0s_4, s_1s_0s_2s_0s_4, s_1s_0s_2s_0s_4s_2, \ldots,  \\
s_{1}s_{0}s_{2}s_{0}s_{4}s_{2}s_{0}s_{5}s_{4}s_{3}s_{2}s_{1}s_{4}s_{4}s_{3}s_{2}s_{1}s_{4}s_{5}s_{4}s_{2}s_{1}s_{0}s_{1}s_{2}s_{4}s_{5}s_{4}s_{3}s_{1}s_{0}).
\end{multline*}
\end{example}
Note that we read the boxes in the set-valued tableau in increasing spiral order, the $s_{i_k}$ in each box from left to right, and we write the entries in the alcove walk $p$ from left to right. Moreover, an entry in $p$ is either equal to the previous entry or contains one more $s_{i_k}$ on the right.

\subsubsection{Weights for set-valued tableaux}

Let $\mu\in \ZZ^n$ and $z\in W_{\mathrm{fin}}$ and let $T$ be a USV-tableau of shape $\mu$ and let
$p=(p_{\ell+1},p_\ell, \ldots, p_1)$ be the corresponding alcove walk with $p_{\ell+1}=z$.
The \emph{permutation sequence of $T$} is the sequence $(z_{\ell+1}, \ldots, z_1)$ of elements of 
$W_{\mathrm{fin}}$ given by
\begin{equation}
(z_{\ell+1}, z_\ell, \ldots, z_1) = (\overline{p_{\ell+1}}, \overline{p_\ell}, \ldots, \overline{p_1}),
\label{permseq}
\end{equation}
where $\overline{\phantom{T}}\colon W \longrightarrow  W_{\mathrm{fin}}$ is the homomorphism defined in~\eqref{Wtofin}.
Recall from~\eqref{crtseq} that the coroot sequence of the box-greedy reduced word $u_\mu^\square=s_{i_\ell}\cdots s_{i_1}$ 
is the sequence
$(\beta_\ell,\ldots, \beta_1)$ given by
\begin{equation*}
\beta_\ell = 
s_{i_1}s_{i_2}\cdots s_{i_{\ell-1}}\alpha_{i_\ell}, \quad
\beta_{\ell-1} = 
s_{i_1}s_{i_2}\cdots s_{i_{\ell-2}}\alpha_{i_{\ell-1}}, \quad \ldots \quad
\beta_2 = s_{i_1}\alpha_{i_2}, \quad
\beta_1 = \alpha_{i_1}.
\end{equation*}
For $k\in \{1, \ldots, \ell\}$ define 
\begin{align}
\wt_T(k) = \wt_T^{z_k}(\cancel{s_{i_k}}) 
&= \begin{cases}
F^-_{\beta_k}, &\hbox{if $i_k=0$ and $z_k(1)>0$,} \\
F^+_{\beta_k}, &\hbox{if $i_k=0$ and $z_k(1)<0$,} \\
F^+_{\beta_k}, &\hbox{if $i_k\in \{1, \ldots, n-1\}$ and $z_k(i_k)\prec z_k(i_{k+1})$}, \\
F^-_{\beta_k}, &\hbox{if $i_k\in \{1, \ldots, n-1\}$ and $z_k(i_k)\succ z_k(i_{k+1})$}, \\
F^+_{\beta_k}, &\hbox{if $i_k=n$ and $z_k(n)>0$}, \\
F^-_{\beta_k}, &\hbox{if $i_k=n$ and $z_k(n)<0$},
\end{cases}
\\
\hbox{and}\qquad
\wt_T(k) =\wt_T^{z_k}(s_{i_k}) &= \begin{cases}
x_{z_k(1)}, &\hbox{if $i_k=0$,} \\
1, &\hbox{otherwise,}
\end{cases}
\label{wtTzk}
\end{align}
depending on whether $s_{i_k}$ is crossed-out or not crossed-out in $T$.
Define the \emph{weight of the USV-tableaux $T$} by 
$$\mathrm{wt}(T) = 
t^{\frac12\ell_s(z_{\mathrm{fin}})} t_n^{\frac12\ell_d(z_{\mathrm{fin}})}
\ev_0\Big(\prod_{k=1}^r\wt_{T}(k)\Big) 
$$
where $z_{\mathrm{fin}}=z_1$ is the last element of the permutation sequence of $T$.

\begin{remark}
Using the action of $W$ on $\ZZ^n$ defined in~\eqref{WonZn}, the \emph{end of $p$} is 
$$
\mathrm{end}(p)= p_{\mathrm{fin}}(0,0,\ldots, 0) \in \ZZ^{n},
$$
where $p_{\mathrm{fin}}=p_1$ is the last element of the alcove walk $p=(p_{\ell+1}, \ldots, p_1)$.
If $\mathrm{end}(p) = (a_1, \ldots, a_n)$ and 
$x^{T} = x^{\mathrm{end}(p)} = x_1^{a_1}\cdots x_n^{a_n}$ then
$$
\mathrm{wt}(T) = 
t^{\frac12\ell_s(z_{\mathrm{fin}})} t_n^{\frac12\ell_d(z_{\mathrm{fin}})} \ev_0\Big(\prod_{\cancel{s_{i_k}}} \wt^{z_k}_{T}(\cancel{s_{i_k}})\Big)x^T, \qquad \hbox{ with } \quad \prod_{s_{i_k}}\wt_T(s_{i_k} ) = x^T,
$$
where the first product is taken over the crossed-out letters in $T$ and the second product is taken over the not crossed-out letters of $T$.
\qed
\end{remark}

Define the \emph{normalization factor} for the pair $(z,\mu)$ by
\begin{equation}
\mathrm{nf}(z,\mu)
=t^{\frac12(\ell_s(v_\mu^{-1})-\ell_s(zv_\mu^{-1}))}t_n^{\frac12(\ell_d(v_\mu^{-1})-\ell_d(zv_\mu^{-1}))}.
\label{nfzmu}
\end{equation}

\begin{thm} 
\label{thm:USVformula}
Let $\mu\in \ZZ^n$ and $z\in W_{\mathrm{fin}}$. 
Then 
$$E^z_\mu = \mathrm{nf}(z,\mu)\Big(
\sum_{T} \wt(T)\Big),
\qquad
\hbox{where the sum is over USV-tableaux $T$ for $\mu$}.
$$
\end{thm}
\begin{proof}
Let $u_\mu^\square = s_{i_\ell}\cdots s_{i_2}s_{i_1}$.
The proof is by induction on $\ell$ (the sequence of steps in the box greedy reduced word $u_\mu^\square$).
The base case is the case $\mu = (0, \ldots, 0)$ for which $u_\mu = 1$ and 
$$\widehat{E}^z_{(0, \ldots, 0)}=t^{\frac12\ell_s(z_{\mathrm{fin}})}t_n^{\frac12\ell_d(z_{\mathrm{fin}})}.$$
The induction step is Proposition~\ref{alcwksteps}. 
The result of the induction is
$$\widehat{E}^z_\mu 
= \sum_T \wt(T)\widehat{E}^{z_{\mathrm{fin}}}_{(0,\ldots,0)}
= \sum_T \wt(T) t^{\frac12\ell_s(z_{\mathrm{fin}})}t_n^{\frac12\ell_d(z_{\mathrm{fin}})}
$$
and the final formula for  $E^z_\mu$ is the result of accounting for the normalization factor in~\eqref{relMac}.
\end{proof}

\begin{example} \label{AWexcont}
The permutation sequence for the USV-tableau from Example~\ref{AWexample} is
given in the diagram below. 
For each $\cancel{s_{i_k}}$ in $T$ the diagram indicates the value of $\wt_T^{z_k}(\cancel{s_{i_k}})$.

The term in $E_\mu$ corresponding to the tableau $T$ is
$$t^{\frac{15}{2}-\frac{21}{2}}t_n^{\frac{2}{2}-\frac32} \ev_0\left( \begin{array}{l}
F^+_{-\varepsilon_{1}-\varepsilon_{2}+4K}
F^+_{-\varepsilon_{4}-\varepsilon_5+K}
F^-_{-\varepsilon_{4}-\varepsilon_{1}+3K}
F^-_{-\varepsilon_3+\hbox{$\frac12$}K}
F^+_{-\varepsilon_{2}+K}
\\
F^-_{-\varepsilon_{2}+\hbox{$\frac12 $}K}
\cdot F^-_{-\varepsilon_{1}-\varepsilon_4+2K}
F^-_{-\varepsilon_{1}+\varepsilon_4+K}
F^+_{-\varepsilon_{1}-\varepsilon_3+K}
\end{array}
\right)
x_2^{-1}x_2x_2^{-1}x_2x_4,
$$
where the normalization is computed by noting that $E_\mu = 
t^{\frac12\ell_s(v_\mu^{-1})}t_n^{\frac12\ell_d(v_\mu^{-1})}\widehat{E}_\mu = t^{\frac{21}{2}}t_n^{\frac32}\widehat{E}_\mu$,
since 
$$v_\mu = (5\overline{2}\overline{1}3\overline{4}) = (s_4s_5s_4)(s_1s_2s_3s_4s_5s_4s_3s_2s_1)
(s_2s_3s_4s_5s_4s_3s_2s_1)(s_4s_3s_2s_1),
$$
and, since $z_{\mathrm{fin}} = (\overline{3}\overline{4}125) = (s_3s_4s_5s_4s_3s_2s_1)(s_4s_5s_4s_3s_2)$ then
\[
\widehat{E}^{z_{\mathrm{fin}}}_{(0,0,0,0,0)} = t^{\frac12\ell_s(z_{\mathrm{fin}})}t_n^{\frac12\ell_d(z_{\mathrm{fin}})}
= t^{\frac{15}{2}}t_n^{\frac22}.
\]
\[
\begin{tikzpicture}[scale = 0.65]
\foreach \x\y in {1/0}{
    \draw[thin] (\y+0.025, -\x-0.025) rectangle (0.975+\y, -\x-.975);
    \draw (0.5+\y,-\x-0.5) node (\x\y) {};
}
\foreach \x\y in {2/1, 2/2}{
    \pgfmathsetmacro{\boxwd}{3.5};
    \draw[thin] (\boxwd*\y+0.025 - \boxwd + 1.3, 10-6*\x-0.025) rectangle (\boxwd*\y + 1.275, 10-6*\x-5.975);
    \draw (\boxwd*\y + 1.3 - 0.5*\boxwd,10-6*\x-3) node (\x\y) {};
}
\foreach \x\y in {3/1, 3/2, 3/3}{
    \pgfmathsetmacro{\boxwd}{3.5};
    \draw[thin] (\boxwd*\y - \boxwd + 1.325, 16-8*\x-0.025) rectangle (\boxwd*\y + 1.275, 16-8*\x-7.975);
    \draw (\boxwd*\y - 0.5*\boxwd + 1.3,16-8*\x-4) node (\x\y) {};
}
\foreach \x\y in {4/0}{
    \pgfmathsetmacro{\boxwd}{3.5};
    \draw[thin] (\y+0.025, 4-5*\x-0.025) rectangle (0.975+\y, 4-5*\x-4.975);
    \draw (0.5+\y,4-5*\x-2.5) node (\x\y) {};
}
\foreach \x\y in {4/-1}{
    \pgfmathsetmacro{\boxwd}{3.5};
    \draw[thin] (\boxwd*\y -0.275, 4-5*\x-0.025) rectangle (\boxwd+\boxwd*\y-0.325, 4-5*\x-4.975);
    \draw (0.5*\boxwd+\boxwd*\y-0.3,4-5*\x-2.5) node (\x\y) {};
}
\foreach \x\y in {5/1}{
    \pgfmathsetmacro{\boxwd}{3.5};
    \draw[thin] (\boxwd*\y+1.325 - \boxwd, 1.5-4.5*\x-0.025) rectangle (\boxwd*\y + 1.325, 1.5-4.5*\x-4.475);
    \draw (\boxwd*\y - 0.5*\boxwd + 1.3,1.5-4.5*\x-2.25) node (\x\y) {};
}
\node at (21) {$\begin{array}{l}
s_1(21345) \\
s_0(\overline{2}1345)
\end{array}$};
\node at (22) {$\begin{array}{l}
s_4(\overline{4}\overline{2}153) \\
\cancel{s_5}\  F^+_{\beta(2,2)_{-5}} \\
s_4(\overline{4}\overline{2}135) \\
s_3(\overline{4}\overline{2}315) \\
s_2(\overline{4}3\overline{2}15) \\
s_1(3\overline{4}\overline{2}15) \\
\cancel{s_0}\ F^-_{\beta(2,2)_1}
\end{array}$};
\node at (31) {$\begin{array}{l}
s_2(\overline{2}3145) \\
\cancel{s_1}\; F^+_{\beta(3,1)_2} \\
s_0(23145)
\end{array}$};
\node at (32) {$ \begin{array}{l}
s_4(3\overline{4}\overline{2}51) \\
s_5(3\overline{4}\overline{2}5\overline{1}) \\
s_4(3\overline{4}\overline{2}\overline{1}5) \\
\cancel{s_3}\ F^+_{\beta(3,2)_4} \\
s_2(3\overline{2}\overline{4}\overline{1}5) \\
s_1(\overline{2}3\overline{4}\overline{1}5) \\
s_0(23\overline{4}\overline{1}5) 
\end{array}$};
\node at (33) {$\begin{array}{l}
s_1(32\overline{4}\overline{1}5)  \\
s_2(3\overline{4}2\overline{1}5)  \\
\cancel{s_3}\ F^+_{\beta(3,3)_{-3}} \\
s_4(3\overline{4}25\overline{1})  \\
s_5(3\overline{4}251)  \\
s_4(3\overline{4}215)  \\
s_3(3\overline{4}125)  \\
\cancel{s_2}\ F^-_{\beta(3,3)_3} \\
s_1(\overline{4}3125) \\
s_0(4\overline{3}125) 
\end{array}$};
\node at (4-1) {$\begin{array}{l}
s_5(\overline{2}135\overline{4}) \\
s_4(\overline{2}13\overline{4}5) \\
s_3(\overline{2}1\overline{4}35) \\
s_2(\overline{2}\overline{4}135) \\
s_1(\overline{4}\overline{2}135) \\
\cancel{s_0}\ F^+_{\beta(4,-1)_1}
\end{array}$};
\node at (51) {$ \begin{array}{l}
s_4(23154) \\
\cancel{s_3}\ F^+_{\beta(5,1)_4} \\
s_2(21354) \\
\cancel{s_1}\ F^-_{\beta(5,1)_2} \\
s_0(\overline{2}1354)
\end{array}$};
\draw (-0.15, -0.75) -- (-0.15, -25.75);
\draw (1.15, -0.75) -- (1.15, -25.75);
\end{tikzpicture}
\]
\end{example}

\section{Compression}
\label{section:compression}

The goal of this section is to state and prove Theorem~\ref{CSVformula}, which is the compressed set-valued (CSV) tableaux formula. 
To do so, we first divide the box-greedy reduced word into compression sections, which are defined in Section~\ref{set-valued tab}.   Section~\ref{weights} describes the weights assigned to each compression section and packages the resulting formula for the Koornwinder polynomials as a sum over compressed set-valued tableaux.
Section~\ref{weights} also contains a high-level proof of the CSV-tableaux formula, which is exactly analogous to the proof of the USV-tableaux formula except that the induction is on the sequence of compression sections instead of  on the sequence of
alcove walk steps.

The remainder of the section verifies, in several cases, that the weight of each compression section is given by one of the formulas in Section~\ref{subsec:weight functions}.
In
Section~\ref{0-gap} we verify the weights for the across-the-$0$-gap compression sections and in Section~\ref{section:around_the_end_z1} and~\ref{subsect: around the end general} we verify the weights the around-the-end compression sections.
The results in these sections are the compression section analogues of the alcove walk steps that are derived in Section~\ref{section:usvtableaux}.

\subsection{Compression sections and compressed set-valued tableaux}\label{set-valued tab}

For $i\in \{1, \ldots, n\}$ and $\ell\in \{i, \ldots, n, -n, \ldots, -i\}$ define
\begin{equation}
d_{\ell,i} = \begin{cases}
s_{\ell-1}\cdots s_i, &\hbox{if $\ell\in \{i, \ldots, n\}$,} \\
s_{-\ell} \cdots s_n\cdots s_i, &\hbox{if $\ell\in \{-n, \ldots, -i\}$.}
\end{cases}
\label{dilfactor}
\end{equation}
This notation generalizes the notation $d_\ell = d_{\ell,0}$ used in~\eqref{didefn} and~\eqref{boxgreedyredwd}.

Let $\mu \in \ZZ^{n}$ and fix a box $b = (r,c)\in \mu$.
A \emph{compression section} is a subword in the word $d_{u_{\mu}(r, c)}$, which is a factor in~\eqref{boxgreedyredwd}, that is of one of the following three types of distinguished forms: 

\medskip\noindent
\begin{itemize}
\item[(CS1)]  If $c\ne1$ and $-i = u_\mu(r,c)$ then 
the \emph{\color{blue} around-the-end compression section} of $d_{u_{\mu}(r, c)}$ is the subword 
$$S={\color{blue} s_i\cdots s_n\cdots s_i} \quad \hbox{ in } \quad d_{u_{\mu}(r, c)} = d_{-i} = {\color{blue} s_i\cdots s_n\cdots s_i} s_{i-1}\cdots s_2 s_1s_0.$$

\item[(CS2)] 
Let  $(l(a_1), \ldots, l(a_{|u_\mu(r,c)|}))$ be the sequence of leg lengths of the attacking boxes
$a_1, \ldots, a_{|u_\mu(r,c)|}$ for the box $b=(r,c)$.
A \emph{\color{red}$0$-gap compression section} of $d_{u_\mu(r,c)}$ is a subword 
$$S={\color{red}s_{m-1} \cdots s_{i}}, \quad\hbox{where}\quad
\begin{array}{l}
\hbox{$i<m$ and $l(a_m)\ne 0$ and $l(a_{i-1})\ne 0$} \\
\hbox{and $l(a_{m-1}) = l(a_{m-2}) = \cdots = l(a_{i}) = 0$.}
\end{array}
$$

\medskip\noindent
\item[(CS3)] In all other cases the compression section consists
of the single factor $S=s_i$. It is convenient to refer to this case as an \emph{uncompressed step}.  
\end{itemize}

In the filled box diagram in Example~\ref{CSVex}, the uncompressed steps are depicted in black, the around-the-end compression sections in blue and the $0$-gap compression sections in red.

\medskip\noindent
With definitions as in (CS1), (CS2), (CS3), the box greedy reduced word is factored as a product of compression sections
\begin{equation}
u_\mu^\square = S_L\cdots S_1.
\label{CSfact}
\end{equation}

A \emph{CSV-tableau (compressed set-valued tableau)} is a USV-tableau (as in Section~\ref{section:usvtableaux}) such that
\begin{equation}
\begin{array}{l}
\hbox{if $S$ is a compression section and $s_j$ is crossed out in $S$}\\
\hbox{then every factor before $s_j$ in $S$ is also crossed out.}
\end{array}
\label{CSVcond}
\end{equation}

\begin{example} \label{CSVex} If $\mu = (0,2,3,-1,1)$ then $u_\mu^\square$ is a product of 25 compression sections.
A CSV-tableau of shape $\mu = (0,2,3,-1,1)$ is 

$$
\begin{matrix}
\begin{tikzpicture}[scale = 0.6, baseline = 0.6*-3.5cm]
\foreach \x\y in {1/0, 4/0}{
    \draw[thin] (\y+0.025, -\x-0.025) rectangle (0.975+\y, -\x-.975);
    \draw (0.5+\y,-\x-0.5) node (\x\y) {};
}
\foreach \x\y in {4/-1}{
    \pgfmathsetmacro{\boxwd}{5};
    \draw[thin] (\boxwd*\y-0.275, -\x-0.025) rectangle (\boxwd + \boxwd*\y - 0.325, -\x-.975);
    \draw (0.5*\boxwd + \boxwd*\y - 0.3,-\x-0.5) node (\x\y) {};
}
\foreach \x\y in {2/1, 2/2, 3/1, 3/2, 5/1}{
    \pgfmathsetmacro{\boxwd}{5};
    \draw[thin] (\boxwd*\y - \boxwd + 1.325, -\x-0.025) rectangle (\boxwd*\y + 1.325, -\x-.975);
    \draw (\boxwd*\y + 1.3 - 0.5*\boxwd,-\x-0.5) node (\x\y) {};
}
\foreach \x\y in {3/3}{
    \pgfmathsetmacro{\boxwd}{5};
    \pgfmathsetmacro{\newboxwd}{6.2};
    \draw[thin] (\boxwd*\y - \boxwd + 1.325, -\x-0.025) rectangle (\boxwd*\y - \boxwd + \newboxwd + 1.325 , -\x-.975);
    \draw (\boxwd*\y - \boxwd + 1.3 + 0.5*\newboxwd,-\x-0.5) node (\x\y) {};
}
\node at (21) {$ s_{1}s_{0}$};
\node at (22) {$ {\color{blue}\cancel{s_{4}} \cancel{s_{5}}s_{4}} s_{3} {\color{red}s_{2} s_{1}} \cancel{s_{0}}$};
\node at (31) {$ s_{2} \cancel{s_{1}} s_{0}$};
\node at (32) {$ {\color{blue} s_{4}s_{5}s_{4}} {\color{red} \cancel{s_{3}}s_{2}s_{1}} s_{0}$};
\node at (33) {$ {\color{blue} \cancel{s_{1}} \cancel{s_{2}} \cancel{s_{3}} \cancel{s_{4}}\cancel{s_{5}} \cancel{s_{4}}\cancel{s_{3}}\cancel{s_{2}}s_{1}} s_{0}$};
\node at (4-1) {$ s_{5}s_{4}s_{3}s_{2}s_{1} \cancel{s_{0}}$};
\node at (51) {$ s_{4} \cancel{s_{3}} s_{2} \cancel{s_{1}} s_{0}$};
\draw (-0.15, -0.75) -- (-0.15, -6.25);
\draw (1.15, -0.75) -- (1.15, -6.25);
\end{tikzpicture}
\end{matrix}
$$
\end{example}

\subsection{Compression section weights and the compressed set-valued tableaux formula}\label{weights}

Let $\mu\in \ZZ^n$ and $z\in W_{\mathrm{fin}}$ and let $T$ be a CSV-tableau of shape $\mu$
and let $p=(p_{\ell+1}, p_\ell, \ldots, p_1)$ be the corresponding alcove walk with $p_{\ell+1}=z$.  The box greedy reduced word
has two forms:
\begin{equation}
u_\mu^\square = s_{i_\ell}\cdots s_{i_1} = S_L\cdots S_1,
\label{Cfact}
\end{equation}
where the first product is a product of simple reflections and the second product is a product of compression sections.
The \emph{compressed permutation sequence} of $T$ is the subsequence $(z_{S_L},\cdots, z_{S_1})$
of the permutation sequence $(z_{\ell+1}, \ldots, z_1)$ of $p$ (see~\eqref{permseq}) given by setting, for a compression section $S$,
$$z_S = z_{k+1},\quad\hbox{if the first letter of $S$ is the $s_{i_{k}}$ in $u_\mu^\square = s_{i_\ell}\cdots s_{i_1}$.}
$$

Each compression section $S$ in $u_\mu^\square$ has a corresponding root sequence
$$\beta_S = (\beta_k, \ldots, \beta_j),
\qquad\hbox{where $S=s_{i_k}\cdots s_{i_j}$ as a subword of $u_\mu^\square = s_{i_\ell}\cdots s_{i_1}$.}
$$
The full root sequence of $u_\mu^\square$ is the concatenation of the root sequences of the compression sections,
$$\beta = (\beta_\ell, \ldots, \beta_1)=(\beta_{S_L},\ldots, \beta_{S_1}),
$$
Each compression section $S$ in $T$ gets a weight $\mathrm{cwt}^{z_S}_T(S)$ defined as follows.

\begin{itemize}
\item[(CS1)] 
Let $S={\color{blue} s_i\cdots s_n\cdots s_{i}}$ be an around-the-end compression section in box $b$ of $\mu$ with permutation $z_S$.
Let $W_{[i+1,-(i+1)]}$ be the subgroup of $W_{\mathrm{fin}}$ generated by $\{s_{i+1}, \ldots, s_n\}$ 
and let $y,v\in W_{\mathrm{fin}}$ be such that
$$\hbox{$z_S = yv$ with $v\in W_{[i+1,-(i+1)]}$ and $\ell(z_S) = \ell(y)+\ell(v).$}$$
Then $0< y(i+1) < \cdots < y(n)$. We define $j$ depending on where $y(i)$ appears with respect to the other entries as follows
$$j = \begin{cases}
i, &\hbox{if $0<{\color{blue} y(i)} < y(i+1)<\cdots < y(n)$, } \\
m, &\hbox{if $0<y(i+1)<\cdots < y(m)  < {\color{blue} y(i) } < y(m+1) < \cdots < y(n)$,} \\
n, &\hbox{if $0<y(i+1)<\cdots  <y(n)< {\color{blue} y(i)}$,} \\
-n, &\hbox{if $0<y(i+1)<\cdots  <y(n)< {\color{blue} -y(i)}$,} \\
-m, &\hbox{if $0<y(i+1)<\cdots < y(m) < {\color{blue} -y(i)} < y(m+1) < \cdots < y(n)$,} \\
-i, &\hbox{if $0 < {\color{blue} -y(i)} < y(i+1)<\cdots <y(n)$.}
\end{cases}
$$
For $k\in \{i, \ldots, m\}$ let $l = v(k)$ and $u^{(k)} = d_{k,i}^{-1}v^{-1}d_{l,i}$ and define
$$\mathrm{covid}(z_S,k) = \ell_s(v)-\ell_s(u^{(k)}). 
$$
Denote the subsequence of the root sequence corresponding to the factors in $S$ by
$$\beta(b)_S = (\beta(b)_{-i}, \ldots, \beta(b)_i).$$
Letting $A^{(i,j,-i)}_{\pm \ell}$ be as defined in~\eqref{AgencA}-\eqref{AgencD}, define 
\begin{align}
\mathrm{cwt}_T(S) = \mathrm{cwt}^{z_S}_T({\color{blue} \cancel{s_i}\cdots\cancel{s_{k}} s_{k+1} \cdots s_n\cdots s_i})
&= t^{\frac12\mathrm{covid}(z_S,k)}t_n^{\frac12} A_{v(k)}^{(i,j,-i)}(\beta(b)_S)
\qquad\hbox{and} \nonumber \\
\mathrm{cwt}_T(S) = \mathrm{cwt}^{z_S}_T({\color{blue} \cancel{s_i}\cdots \cancel{s_n} \cdots \cancel{s_{k}} s_{k-1} \cdots s_i})
&= t^{\frac12\mathrm{covid}(z_S,k)} A_{v(k)}^{(i,j,-i)}(\beta(b)_S).
\label{CS1wt}
\end{align}

\item[(CS2)] 
Let $S={\color{red} s_{m-1}\cdots s_{i}}$ be a $0$-gap compression section in box $b$ of $\mu$ with permutation $z_S$.
Let $W_{[i,m-1]}$ be the subgroup of $W_{\mathrm{fin}}$ generated by $\{s_i, \ldots, s_{m-2} \}$ and let $y,v\in W_{\mathrm{fin}}$ be such that
$$\hbox{$z_S = yv$ with $v\in W_{[i,m-1]}$ and $\ell(z_S) = \ell(y)+\ell(v).$}$$
Then $y(i)\prec y(i+1)\prec \cdots \prec y(m-1)$. Let
$$j \in \{i, i+1, \ldots, m\}\quad\hbox{be such that}\quad
y(j) \prec {\color{blue}y(m)} \prec y(j+1)
$$
with $j=i$ if ${\color{blue} y(m)}\prec y(i)$ and $j=m$ if $y(m-1)\prec {\color{blue} y(m)}$.
For $k\in \{i, \ldots, m\}$ let $l = v(k)$ and $u^{(k)} = d_{k,i}^{-1}v^{-1}d_{l,i}$ and define
$$\mathrm{covid}(z_S,k) = \ell_s(v)-\ell_s(u^{(k)}).$$
Denote the subsequence of the root sequence $\beta(b)$ consisting of the coroots  corresponding to the factors in $S$ by
$$\beta(b)_S = (\beta(b)_{m-1}\cdots \beta(b)_i).$$ 
Letting $A^{(i,j,m)}_\ell(\beta(b)_S)$ be as defined in~\eqref{AEwts}, define
\begin{equation}
\mathrm{cwt}_T(S) = \mathrm{cwt}^{z_S}_T({\color{red} \cancel{s_{m-1}}\cdots \cancel{s_k} s_{k-1} \cdots s_{i}}) = 
t^{\frac12\mathrm{covid}(z_S,k)}A^{(i,j,m)}_{v(k)}(\beta(b)_S).
\label{CS2wt}
\end{equation}

\item[(CS3)] In the case that $S = s_i$ is an uncompressed step define
\begin{equation}
\mathrm{cwt}_T(S) = \mathrm{cwt}_T^{z_S}(s_i) = \wt_T^{z_S}(s_i)
\qquad\hbox{and}\qquad
\mathrm{cwt}_T(S) = \mathrm{cwt}_T^{z_S}(\cancel{s_i}) = \wt_T^{z_S}(\cancel{s_i})
\label{CS3wt}
\end{equation}
where $\wt_T^{z_S}$ is as defined in~\eqref{wtTzk}.
\end{itemize}

Using the weights for each compression section defined~\eqref{CS1wt}-\eqref{CS3wt}, 
the \emph{compressed weight of the CSV-tableau $T$} is defined by
$$\mathrm{cwt}(T) = t^{\frac12\ell_s(z_{\mathrm{fin}}) }t_n^{\frac12\ell_d(z_{\mathrm{fin}})} \ev_0\Big(\prod_{k=1}^L \mathrm{cwt}(S_k)\Big),
$$
where $z_{\mathrm{fin}}=z_{S_L}=z_1$ is the last element of the permutation sequence of $T$.

Theorem~\ref{CSVformula} gives the compressed set-valued tableaux formula for the relative Koornwinder
polynomial $E^z_\mu$.

\begin{thm} 
\label{CSVformula}
Let $z\in W_{\mathrm{fin}}$ and let $\mu\in \ZZ^n$.  Let $\mathrm{nf}(z,\mu)$ be as defined in~\eqref{nfzmu}.  
Then
$$
E^z_\mu = \mathrm{nf}(z,\mu)\Big(\sum_T \mathrm{cwt}(T)\Big),
\qquad\hbox{where the sum is over CSV-tableaux $T$ for $\mu$.}
$$
\end{thm}

\begin{proof}
The proof is similar to the proof of the USV-tableaux formula except that the induction is on the sequence of compression sections (the USV-tableaux case was an induction on the sequence of steps).

The induction step is as follows. 
Using the notation $d_{\ell,i}$ defined in~\eqref{dilfactor},
Proposition~\ref{0gap} (for $m\in \{i+1, \ldots, n\}$) and
Proposition~\ref{proposition:GeneralCaseAroundTheEnd} (for $m=-i$) 
provide
$$A^{(i,j,m)}_\ell\in \CC(Y_1, \ldots, Y_n) \qquad\hbox{such that}\qquad
\widehat{E}^{y}_\mu 
=\sum_{\ell=i}^m \ev_\nu(A_\ell^{(i,j,m)}) \widehat{E}^{yd_{\ell,i}}_\nu.
$$
Now use the notations $z=yv$, $k=v^{-1}(\ell)$ and $u^{(k)} = d_{k,i}^{-1}v^{-1}d_{\ell,i}$ introduced in the 
definition of the compression section weights in~\eqref{CS1wt} and~\eqref{CS2wt}.
By~\eqref{stabfactor}, $\widehat{E}^{yv}_\mu = t^{\frac12\ell(v)} \widehat{E}^{y}_\mu$ 
and $\widehat{E}^{zd_{k,i}u^{(k)}}_\nu = t^{-\frac12\ell(u^{(k)} ) } \widehat{E}^{zd_{k,i} }_\mu$,
so that 
\begin{align*}
\widehat{E}^z_\mu 
&= \widehat{E}^{yv}_\mu 
=t^{\frac12\ell(v)} \widehat{E}^{y}_\mu 
=t^{\frac12\ell(v)} \sum_{\ell=i}^m \ev_\nu(A_\ell^{(i,j,m)}) \widehat{E}^{y d_{\ell,i}}_\nu 
=t^{\frac12\ell(v)} \sum_{k=i}^m \ev_\nu(A_{v(k)}^{(i,j,m)}) \widehat{E}^{z d_{k,i} u^{(k)}}_\nu 
\\
&=\sum_{k=i}^m \ev_\nu(A_{v(k)}^{(i,j,m)}) t^{\frac12(\ell(v)-\ell(u^{(k)}))} \widehat{E}^{z d_{k,i} }_\nu 
=\sum_{k=i}^m \ev_\nu(A_{v(k)}^{(i,j,m)}) t^{\frac12\mathrm{covid}(z,k)} \widehat{E}^{zd_{k,i}}_\nu
\end{align*}
where the sum changes from a sum over $\ell$ to a sum over $k$ by the relation $k=v^{-1}(\ell)$.
Since
$$\ev_\nu(A_{v(k)}^{(i,j,m)})=\ev_0(A_{v(k)}^{(i,j,m)}(\beta(b)_S))$$
then
\begin{align*}
\widehat{E}^{z}_\mu 
&=\sum_{k=i}^m t^{\frac12\mathrm{covid}(z,k)} \ev_\nu(A_{v(k)}^{(i,j,m)})  \widehat{E}^{zd_{k,i}}_\nu 
=\sum_{k=i}^m t^{\frac12\mathrm{covid}(z,k)} \ev_0(A_{v(k)}^{(i,j,m)}(\beta(b)_S))  \widehat{E}^{zd_{k,i}}_\nu 
\end{align*}
Thus the induction gives
$$\widehat{E}^z_\mu 
= \sum_T \mathrm{cwt}(T)\widehat{E}^{z_{\mathrm{fin}}}_{(0,\ldots,0)}
= \sum_T \mathrm{cwt}(T) t^{\frac12\ell_s(z_{\mathrm{fin}})}t_n^{\frac12\ell_d(z_{\mathrm{fin}})}.
$$
The final formula for  $E^z_\mu$ is the result of accounting for the normalization factor in~\eqref{relMac}.

\end{proof}

\begin{example} \label{CSVex2}
Let $\mu = (0,2,3,-1,1)$ and let $Q$ be the CSV-tableau of Example~\ref{CSVex}.  
The length of $u^\square_\mu$ is $r=40$ and
$$z_{\mathrm{fin}}= \overline{5}2\overline{4}\overline{1}3 = 
s_{5} s_{4} s_{5} \; s_{1} s_{2} s_{3} s_{4} s_{5} \; s_{2} s_{3} s_{4} \; s_{1}.
$$
So
$$t^{\frac12 \ell_s(z_{\mathrm{fin}})}t_n^{\frac12\ell_d(z_{\mathrm{fin}})}
=t^{\frac{9}{2}}t_n^{\frac{3}{2}}
\qquad\hbox{and, as in Example~\ref{AWexcont},}\quad
t^{\frac12\ell_s(v_\mu^{-1})}t_n^{\frac12\ell_d(v_\mu^{-1})} = t^{\frac{21}{2}}t_n^{\frac32}.
$$
The term in $E_\mu$ coming from the CSV-tableau $T$ is 
$$t^{\frac{9}{2}-\frac{21}{2}}t_n^{\frac{3}{2}-\frac32}\ev_0\left(
\begin{array}{l}
F^{-}_{-\varepsilon_1-\varepsilon_2+4K}
\cdot  F^+_{-\varepsilon_4-\varepsilon_5+K}
\cdot  F^-_{-\varepsilon_4-\varepsilon_1+3K}
\cdot  F^{+}_{-\varepsilon_3+\frac12K} \\
\cdot  {\color{blue} t^{\frac{0}{2}} t_{n}^{\frac{0}{2}} A^{(4, 4, -4)}_{5}(\beta(2,2)_{-4},\beta(2,2)_{-5},\beta(2,2)_{5})}
 \cdot {\color{red} 1} 
\cdot  F_{-\varepsilon_{2} + \frac{1}{2}K}^{-} \\
\cdot {\color{blue} t^{\frac{0}{2}} t_{n}^{\frac{0}{2}} A^{(4,4,-4)}_{-4}(\beta(3,2)_{-4},\beta(3,2)_{-5},\beta(3,2)_5)}
\cdot {\color{red} t^{\frac{0}{2}}  F_{-\varepsilon_{1} - \varepsilon_{2} + 2K}^{+}} \\
\cdot  {\color{blue} t^{\frac{3}{2}} t_{n}^{\frac{0}{2}} A^{(1,2,-1)}_{5}(\beta(3,3)_{-1},\ldots, \beta(3,3)_2)  }
\end{array}
\right)
x_5.
$$
We show some selected steps in the derivation of this term in $E_\mu$. 
The permutation sequence of $T$ is 
$$\begin{array}{c|c|ccccc}
&\boxed{\phantom{T} } & \\
&\boxed{\phantom{T} } &\boxed
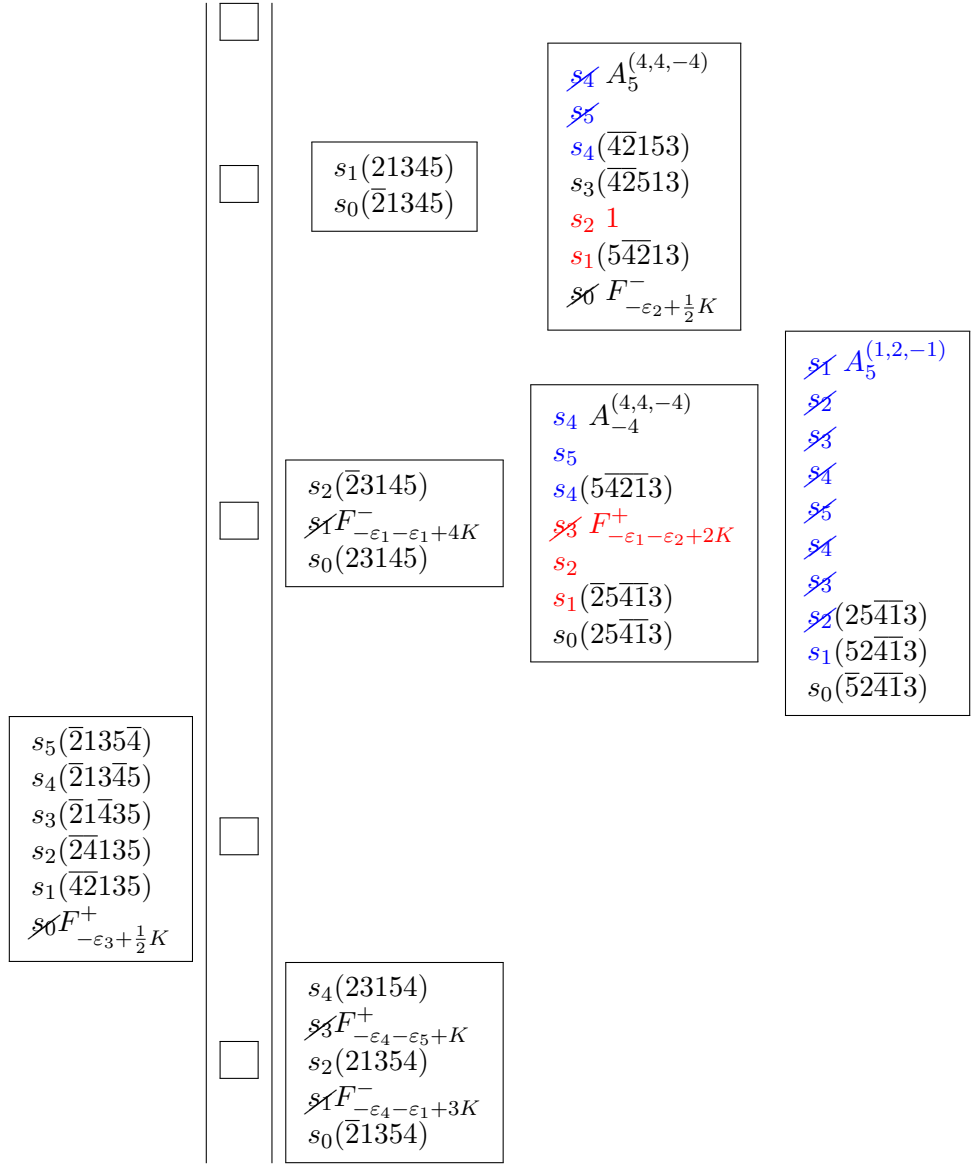
{
\begin{array}{l}
s_1(21345) \\
s_0(\overline{2}1345)
\end{array}
} 
&\boxed{
\begin{array}{l}
{\color{blue}{\cancel{s_4}}}\ A^{(4,4,-4)}_5\\
{\color{blue}{\cancel{s_5}}} \\
{\color{blue}{s_4}}(\overline{4}\overline{2}153) \\
s_3(\overline{4}\overline{2}513) \\
{\color{red}{s_2}\ 1}
\\
{\color{red}{s_1}}(5\overline{4}\overline{2}13) \\
\cancel{s_0}\ F^-_{-\varepsilon_2+\frac12K}
\end{array}
}
\\
&\boxed{\phantom{T} } &\boxed{
\begin{array}{l}
s_2(\overline{2}3145) \\
\cancel{s_1}F^-_{-\varepsilon_1-\varepsilon_1+4K} \\
s_0(23145)
\end{array}
}
&\boxed{
\begin{array}{l}
{\color{blue}{s_4}}\ A_{-4}^{(4,4,-4)}
\\
{\color{blue}{s_5}} 
\\
{\color{blue}{s_4}}(5\overline{4}\overline{2}\overline{1}3) \\
{\color{red}{\cancel{s_3}}\ F^+_{-\varepsilon_1-\varepsilon_2+2K}} \\
{\color{red}{s_2}} 
\\
{\color{red}{s_1}}(\overline{2}5\overline{4}\overline{1}3) \\
s_0(25\overline{4}\overline{1}3) 
\end{array}
}
&\boxed{
\begin{array}{l}
{\color{blue}{\cancel{s_1}}\ A_5^{(1,2,-1)} }  \\
{\color{blue}{\cancel{s_2}}} \\
{\color{blue}{\cancel{s_3}}} \\
{\color{blue}{\cancel{s_4}}}  \\
{\color{blue}{\cancel{s_5}}} \\
{\color{blue}{\cancel{s_4}}}  \\
{\color{blue}{\cancel{s_3}}}  \\
{\color{blue}{\cancel{s_2}}} (25\overline{4}\overline{1}3) \\
{\color{blue}{s_1}} (52\overline{4}\overline{1}3) \\
s_0(\overline{5}2\overline{4}\overline{1}3) 
\end{array}
} 
\\
\boxed{\begin{array}{l}
s_5(\overline{2}135\overline{4}) \\
s_4(\overline{2}13\overline{4}5) \\
s_3(\overline{2}1\overline{4}35) \\
s_2(\overline{2}\overline{4}135) \\
s_1(\overline{4}\overline{2}135) \\
\cancel{s_0}F^+_{-\varepsilon_3+\frac12 K}
\end{array}
}
&\boxed{\phantom{T} }  \\
&\boxed{\phantom{T} } &\boxed{
\begin{array}{l}
s_4(23154) \\
\cancel{s_3}F^+_{-\varepsilon_4-\varepsilon_5+K} \\
s_2(21354) \\
\cancel{s_1}F^-_{-\varepsilon_4-\varepsilon_1+3K} \\
s_0(\overline{2}1354)
\end{array}
}
\\
\end{array}
$$

\noindent
$\bullet$ Start from the first entry of box $(2,1)$ which has $z=(12345)$.
Noting that $1\prec 2$, Proposition~\ref{alcwksteps} gives
$$\widehat{E}_{(0,2,3,-1,1)} = \widehat{E}^{(12345)}_{(0,2,3,-1,1)}
=\widehat{E}^{(12345)\cdot s_1}_{(2,0,3,-1,1)} 
+ \ev_{(2,0,3,-1,1)}(F^+_{\varepsilon_1-\varepsilon_2})\widehat{E}^{(12345)}_{(2,0,3,-1,1)}. 
$$
Since $s_1$ is not canceled in box $(2,1)$, we focus on the first term 
$\widehat{E}^{(12345)\cdot s_1}_{(2,0,3,-1,1)}=\widehat{E}^{(21345)}_{(2,0,3,-1,1)}$.

\noindent
$\bullet$  Now we are at the second step in box $(2,1)$ for which Proposition~\ref{alcwksteps} gives (note $2>0$)
$$\widehat{E}^{(21345)}_{(2,0,3,-1,1)}  = 
x_2\widehat{E}^{(21345)\cdot \overline{s_0}}_{(-1,0,3,-1,1)}
+\ev_{(-1,0,3,-1,1)}(F^-_{-\varepsilon_1+K})\widehat{E}^{(21345)}_{(-1,0,3,-1,1)}.
$$
Since $s_0$ is not canceled in box $(2,1)$, we focus on the first term.

\noindent
$\bullet$ The second step in box $(3,1)$ showcases what happens when an $s_i$ is canceled in $T$. 
At the second step in box $(3,1)$, the permutation in the compressed permutation sequence is $z_S =  (\overline{2}1345)\cdot s_2
=(\overline{2}3145)$ and we are at a term of the form
$x_2\widehat{E}^{(\overline{2}3145)}_{(-1,3,0,-1,1)}$. 
Applying Proposition~\ref{alcwksteps}, this expands to (note $-2\succ 1$)
\begin{align*}
x_2\widehat{E}^{(\overline{2}3145)}_{(-1,3,0,-1,1)}
&=x_2\widehat{E}^{(\overline{2}3145)\cdot s_1}_{(3,-1,0,-1,1)}
+\ev_{(3,-1,0,-1,1)}(F^-_{\varepsilon_1-\varepsilon_2}) x_2\widehat{E}^{(\overline{2}3145)}_{(3,-1,0,-1,1)}
\end{align*}
and the $\cancel{s_1}$ in box $(3,1)$ focuses on the second term, which is
\begin{align*}
\ev_{(3,-1,0,-1,1)}(F^-_{\varepsilon_1-\varepsilon_2}) x_2\widehat{E}^{(\overline{2}3145)}_{(3,-1,0,-1,1)}
&=
\ev_{(0,0,0,0,0)}(F^-_{-\varepsilon_1-\varepsilon_2+4K}) x_2\widehat{E}^{(\overline{2}3145)}_{(3,-1,0,-1,1)},
\end{align*}
since
$\ev_{(3,-1,0,-1,1)}(F^-_{\varepsilon_1-\varepsilon_2})
=\ev_0(F^-_{u_{(3,-1,0,-1,1)}^{-1}(\varepsilon_1-\varepsilon_2)})$
and
\begin{align*}
u_{(3,-1,0,-1,1)}^{-1}(\varepsilon_1-\varepsilon_2)
&=v_{(3,-1,0,-1,1)}h_{(-3,1,0,1,-1)}(\varepsilon_1-\varepsilon_2)
=v_{(3,-1,0,-1,1)}(\varepsilon_1-\varepsilon_2+(3+1)K)
\\
&=(\overline{1}253\overline{4})(\varepsilon_1-\varepsilon_2+4K)
=-\varepsilon_1-\varepsilon_2+4K.
\end{align*}

\noindent
$\bullet$ Continuing the computation step by step, if the corresponding $s_i$ appears in the CSV-tableau we focus on the first term in the
expansion from Proposition~\ref{alcwksteps}, or on the second term, if $\cancel{s_i}$ appears in the CSV-tableau. 
The steps with $\cancel{s_i}$ involve evaluating a fold function $\ev_\nu(F_{\alpha}^{\pm})$, which gets rewritten in the
form $\ev_0(F^\pm_{\beta(r,c)_j})$.

\noindent
$\bullet$ The step in box $(2,2)$ provides an illustration of how to deal with an \textcolor{blue}{around-the-end compression section}. 
At this point we are at a term of the form
$$
x_2\ev_0(F^-_{-\varepsilon_1-\varepsilon_2+4K} F^+_{-\varepsilon_4-\varepsilon_5+K}
F^-_{-\varepsilon_4-\varepsilon_1+3K} F^+_{-\varepsilon_3+\frac12K}) \widehat{E}^{(\overline{4}\overline{2}135)}_{(0,0,-2,-1,0)}
$$
which, using Proposition~\ref{0end}, is equal to (here $z_S=(\overline{4}{\overline{2}135})$ is already minimal length with respect to $W_{[5]}$ and $3<5$ so $j=4$) 
\begin{align*}
x_2\ev_0(&F^-_{-\varepsilon_1-\varepsilon_2+4K} F^+_{-\varepsilon_4-\varepsilon_5+K}
F^-_{-\varepsilon_4-\varepsilon_1+3K} F^+_{-\varepsilon_3+\frac12K})
\\
&\cdot \left( \begin{array}{l}
\ev_{(0,0,-2,1,0)}(A^{(4,4,-4)}_{-4})\widehat{E}^{ys_4s_5s_4}_{(0,0,-2,1,0)}
+\ev_{(0,0,-2,1,0)}(A^{(4,4,-4)}_{-5})\widehat{E}^{ys_5s_4}_{(0,0,-2,1,0)}
\\
+\ev_{(0,0,-2,1,0)}(A^{(4,4,-4)}_{5})\widehat{E}^{ys_4}_{(0,0,-2,1,0)}
+\ev_{(0,0,-2,1,0)}(A^{(4,4,-4)}_{4})\widehat{E}^{y}_{(0,0,-2,1,0)}
\end{array}
\right)
\end{align*}
The compression sequence $\textcolor{blue}{\cancel{s_4}\cancel{s_5}s_4}$ in box $(2,2)$ focuses on the third term which is
$$x_2\ev_0(F^-_{-\varepsilon_1-\varepsilon_2+4K} F^+_{-\varepsilon_4-\varepsilon_5+K}
F^-_{-\varepsilon_4-\varepsilon_1+3K} F^+_{-\varepsilon_3+\frac12K})
\ev_{(0,0,-2,1,0)}(A^{(4,4,-4)}_{5})\widehat{E}^{(\overline{4}\overline{2}153)}_{(0,0,-2,1,0)}.
$$

\noindent
$\bullet$
The \textcolor{red}{$0$-gap compression sections} are dealt with similarly, and so we do not illustrate them. 

\noindent
$\bullet$
The step in box $(2,2)$ corresponding to the compression section ${\color{blue} \cancel{s_{1}} \cancel{s_{2}} \cancel{s_{3}} \cancel{s_{4}}\cancel{s_{5}} \cancel{s_{4}}\cancel{s_{3}}\cancel{s_{2}}s_{1}}$ is a case where $z_S$ is not of minimal length. At this point
the term being considered is
\begin{align*}
&x_2\ev_0(F^-_{-\varepsilon_1-\varepsilon_2+4K} F^+_{-\varepsilon_4-\varepsilon_5+K}
F^-_{-\varepsilon_4-\varepsilon_1+3K} F^+_{-\varepsilon_3+\frac12K})
\ev_{(0,0,-2,1,0)}(A^{(4,4,-4)}_{5})\cdot 1 \cdot \ev_0(F^-_{-\varepsilon_2+\frac12 K})
\\
&\quad\cdot 
\ev_{(0,0,0,2,0)}(A^{(4,4,-4)}_{-4})
\ev_{(2,0,0,0,0)}(A^{(1,4,4)}_3) 
\widehat{E}^{(25\overline{4}\overline{1}3)}_{(-1,0,0,0,0)}
\end{align*}
corresponding to $(\overline{2}5\overline{41}3)\cdot\overline{s_0}$ in box $(3,2)$.

We want to evaluate  
$\widehat{E}^{(25\overline{4}\overline{1}3)}_{(-1,0,0,0,0)}$.
Here $z_S=(25\overline{4}\overline{1}3)$ is not minimal length with respect to
$W_{[2,5]}$ and has
$$z_S = (25\overline{4}\overline{1}3)
= (21345)\cdot(15\overline{4}\overline{2}3) = yv,
\quad\hbox{with $y=s_1$ and $v = s_4s_5s_3s_4s_2s_3s_4s_5s_4$}
$$
and
$$\widehat{E}^{(25\overline{4}\overline{1}3)}_{(-1,0,0,0,0)} = t^{\frac72}t_n^{\frac22} \widehat{E}^{(21345)}_{(-1,0,0,0,0)}.$$
Since $y$ has $1< {\color{blue}2} <3<4<5$ then $j=2$ and Proposition~\ref{0end} gives
\begin{align*}
\widehat{E}^{(21345)}_{(-1,0,0,0,0)}
=
\left(\begin{array}{l}
\ev_{(1,0,0,0,0)}(A^{(1,2,-1)}_{-1}) \widehat{E}^{(21345)\cdot s_1s_2s_3s_4s_5s_4s_3s_2s_1}_{(1,0,0,0,0)}
+\ev_{(1,0,0,0,0)}(A^{(1,2,-1)}_{-2}) \widehat{E}^{(21345)\cdot s_2s_3s_4s_5s_4s_3s_2s_1}_{(1,0,0,0,0)}
\\
\ev_{(1,0,0,0,0)}(A^{(1,2,-1)}_{-3}) \widehat{E}^{(21345)\cdot s_3s_4s_5s_4s_3s_2s_1}_{(1,0,0,0,0)}
+\ev_{(1,0,0,0,0)}(A^{(1,2,-1)}_{-4}) \widehat{E}^{(21345)\cdot s_4s_5s_4s_3s_2s_1}_{(1,0,0,0,0)}
\\
\ev_{(1,0,0,0,0)}(A^{(1,2,-1)}_{-5}) \widehat{E}^{(21345)\cdot s_5s_4s_3s_2s_1}_{(1,0,0,0,0)}
+\ev_{(1,0,0,0,0)}(A^{(1,2,-1)}_{5}) \widehat{E}^{(21345)\cdot s_4s_3s_2s_1}_{(1,0,0,0,0)}
\\
\ev_{(1,0,0,0,0)}(A^{(1,2,-1)}_{4}) \widehat{E}^{(21345)\cdot s_3s_2s_1}_{(1,0,0,0,0)}
+\ev_{(1,0,0,0,0)}(A^{(1,2,-1)}_{3}) \widehat{E}^{(21345)\cdot s_2s_1}_{(1,0,0,0,0)}
\\
\ev_{(1,0,0,0,0)}(A^{(1,2,-1)}_{2}) \widehat{E}^{(21345)\cdot s_1}_{(1,0,0,0,0)}
+\ev_{(1,0,0,0,0)}(A^{(1,2,-1)}_{1}) \widehat{E}^{(21345)}_{(1,0,0,0,0)}
\end{array}
\right)
\end{align*}
Focus on the term of $\widehat{E}^{(25\overline{4}\overline{1}3)}_{(-1,0,0,0,0)}$ with $\ev_{(1,0,0,0,0)}(A^{(1,2,-1)}_{5}) $ which is equal to 
(note $5>0$)
\begin{align*}
t^{\frac72}t_n &\ev_{(1,0,0,0,0)}(A^{(1,2,-1)}_{5}) \widehat{E}^{(52134)}_{(1,0,0,0,0)}
\\
&=t^{\frac72}t_n \ev_{(1,0,0,0,0)}(A^{(1,2,-1)}_{5})
(x_5  \widehat{E}^{(52134)\cdot \overline{s_0}}_{(0,0,0,0,0)}
+\ev_0(F^-_{-\varepsilon_1+\frac12 K})\widehat{E}^{(52134)}_{(0,0,0,0,0)}
\\
&= t^{\frac72}t_n \ev_{(1,0,0,0,0)}(A^{(1,2,-1)}_{5})
(x_5  \widehat{E}^{(\overline{5}21345)}_{(0,0,0,0,0)})
+\ev_0(F^-_{-\varepsilon_1+\frac12 K})t^{\frac52})
\\
&=t^{\frac72}t_n \ev_{(1,0,0,0,0)}(A^{(1,2,-1)}_{5})
(x_5  t^{\frac52}t_n^{\frac12}
+t^{\frac52} \ev_0(F^-_{-\varepsilon_1+\frac12 K}))
\\
&= \ev_{(1,0,0,0,0)}(A^{(1,2,-1)}_{5})
(x_5  t^{\frac{12}{2}}t_n^{\frac32}
+t^{\frac{12}{2}} t_n^{\frac22} \ev_0(F^-_{-\varepsilon_1+\frac12 K})).
\end{align*}

Alternatively, use the formula in~\eqref{CS1wt}:
For this term $z=z_S=(25\overline{4}\overline{1}3)$ and $\ell = 5$ and 
$v = (15\overline{4}\overline{2}3)$ and $k=v^{-1}(\ell) = v^{-1}(5) = 2$ and
$$u^{(2)} = d^{-1}_{2,1}v^{-1}d_{5,1},
\quad\hbox{where
$d^{-1}_{2,1}=s_1$ and $v^{-1}=(1\overline{4}5\overline{3}2)$ and $d_{5,1} = s_4s_3s_2s_1$.}
$$
So $u^{(2)}=s_1\cdot (1\overline{4}5\overline{3}2)\cdot s_4s_3s_2s_1 = 
s_1\cdot (21\overline{4}5\overline{3})= (12\overline{4}5\overline{3})$ and $\ell_s(u^{(2)})=4$ and 
$$\mathrm{covid}_s(z_S,2) = \ell_s(v)-\ell_s(u^{(2)})
=7-4=3.
$$
and, with $\beta(b)_S = \beta(3,3)_S = (\beta(3,3)_{-1}, \ldots, \beta(3,3)_2)$ as in 
Example~\ref{ex023m11},
$$\mathrm{cwt}_T(S) = t^{\frac12\mathrm{covid}_s(z_S,k)}A_{v(k)}^{(i,j,-i)}(\beta(b)_S)
=t^{\frac32}\ev_0(A^{(1,2,-1)}_5)(\beta(3,3)_S).$$
With $\nu=(1,0,0,0,0)$ and $z=z_S=(25\overline{4}\overline{1}3)$ so that $z_Sd_{2,1} = z_Ss_1 = (52\overline{4}\overline{1}3)$ then
\begin{align*}
&\ev_\nu(A^{(1,2,-1)}_{v(2)})t^{\frac12(\ell(v)-\ell(u^{(2)}))}E^{zd_{2,1}}_\nu 
= 
\ev_\nu(A^{(1,2,-1)}_5)t^{\frac32} E^{(52\overline{4}\overline{1}3)}_{(1,0,0,0,0)}
\\
&\qquad
= \ev_0(A^{(1,2,-1)}_5(\beta(3,3)_S) t^{\frac32} (x_5E^{(\overline{5}2\overline{4}\overline{1}3)}_{(0,0,0,0,0)}
+\ev_0(F^-_{-\varepsilon_1+\frac12 K})E^{(52\overline{4}\overline{1}3)}_{(0,0,0,0,0))})
\\
&\qquad
= \ev_0(A^{(1,2,-1)}_5(\beta(3,3)_S) t^{\frac32} (x_5t^{\frac92}t_n^{\frac32}
+\ev_0(F^-_{-\varepsilon_1+\frac12 K})t^{\frac92}t_n^{\frac22}).
\\
&\qquad
= \ev_0(A^{(1,2,-1)}_5(\beta(3,3)_S)  (x_5t^{\frac{12}{2}}t_n^{\frac32}
+\ev_0(F^-_{-\varepsilon_1+\frac12 K})t^{\frac{12}{2}}t_n^{\frac22}).
\end{align*}
\end{example}

\begin{remark}
The CSV formula coarsens the USV formula.  The term in the CSV formula corresponding to a CSV-tableau $T$ is the
sum of the terms $T'$ in the USV formula such that the USV-tableau $T'$ maps to the CSV-tableau $T$ under the conversion in~\eqref{USVtoCSV} below. 

Recalling the factorization $u_\mu^\square = s_{i_\ell}\cdots s_{i_1} = S_L\cdots S_1$
in~\eqref{Cfact}, if $S=S_j$ is a compression section in $u_\mu^\square$ let $\nu_S = S_{j-1}\cdots S_1\cdot(0,\ldots,0)$.
In other words $u_{\nu_S}^\square = S_{j-1}\cdots S_1$.
Let $T'_S$ be the product of the noncrossed out $s_i$ in $S$ in the USV-tableau $T'$ and let
$W_{\nu_S} = \langle s_i\in \{s_1, \ldots, s_n\} \ |\ s_i\nu_S = \nu_S\rangle$ be the subgroup of $W_{\mathrm{fin}}$ generated by the simple reflections that fix $\nu_S$.  The CSV-tableau $T$ corresponding to $T'$ is given by making
\begin{equation}
\hbox{$T_S$ be the minimal length representative of the coset $T'_SW_{\nu_S}$.}
\label{USVtoCSV}
\end{equation}
If $T'_S$ is a compression section in $T'$ which already satisfies~\eqref{CSVcond} then the conversion in~\eqref{USVtoCSV} gives 
$T_S = T'_S$.

To give an example, the USV-tableau
\[
T' = 
\begin{tikzpicture}[scale = 0.6, baseline = 0.6*-3.5cm]
\foreach \x\y in {1/0, 4/0}{
    \draw[thin] (\y+0.025, -\x-0.025) rectangle (0.975+\y, -\x-.975);
    \draw (0.5+\y,-\x-0.5) node (\x\y) {};
}
\foreach \x\y in {4/-1}{
    \pgfmathsetmacro{\boxwd}{5};
    \draw[thin] (\boxwd*\y-0.275, -\x-0.025) rectangle (\boxwd + \boxwd*\y - 0.325, -\x-.975);
    \draw (0.5*\boxwd + \boxwd*\y - 0.3,-\x-0.5) node (\x\y) {};
}
\foreach \x\y in {2/1, 2/2, 3/1, 3/2, 5/1}{
    \pgfmathsetmacro{\boxwd}{5};
    \draw[thin] (\boxwd*\y - \boxwd + 1.325, -\x-0.025) rectangle (\boxwd*\y + 1.325, -\x-.975);
    \draw (\boxwd*\y + 1.3 - 0.5*\boxwd,-\x-0.5) node (\x\y) {};
}
\foreach \x\y in {3/3}{
    \pgfmathsetmacro{\boxwd}{5};
    \pgfmathsetmacro{\newboxwd}{6.2};
    \draw[thin] (\boxwd*\y - \boxwd + 1.325, -\x-0.025) rectangle (\boxwd*\y - \boxwd + \newboxwd + 1.325 , -\x-.975);
    \draw (\boxwd*\y - \boxwd + 1.3 + 0.5*\newboxwd,-\x-0.5) node (\x\y) {};
}
\node at (21) {$ s_{1}s_{0}$};
\node at (22) {$ {\color{blue}s_{4} \cancel{s_{5}} {s_{4}}} s_{3} {\color{red}s_{2} s_{1}} \cancel{s_{0}}$};
\node at (31) {$ s_{2} \cancel{s_{1}} s_{0}$};
\node at (32) {$ {\color{blue} s_{4}s_{5}s_{4}} {\color{red} \cancel{s_{3}}s_{2}s_{1}} s_{0}$};
\node at (33) {$ {\color{blue} {s_{1}} {s_{2}} \cancel{s_{3}} {s_{4}}{s_{5}} {s_{4}}{s_{3}}\cancel{s_{2}}s_{1}} s_{0}$};
\node at (4-1) {$ s_{5}s_{4}s_{3}s_{2}s_{1} \cancel{s_{0}}$};
\node at (51) {$ s_{4} \cancel{s_{3}} s_{2} \cancel{s_{1}} s_{0}$};
\draw (-0.15, -0.75) -- (-0.15, -6.25);
\draw (1.15, -0.75) -- (1.15, -6.25);
\end{tikzpicture}
\]
from Example~\ref{AWexample} is not a CSV-tableau because the around-the-end (blue) compression sections in boxes $(2, 2)$ and $(3, 3)$
do not satisfy~\eqref{CSVcond}.
The corresponding CSV-tableaux is 
\[
T = 
\begin{tikzpicture}[scale = 0.6, baseline = 0.6*-3.5cm]
\foreach \x\y in {1/0, 4/0}{
    \draw[thin] (\y+0.025, -\x-0.025) rectangle (0.975+\y, -\x-.975);
    \draw (0.5+\y,-\x-0.5) node (\x\y) {};
}
\foreach \x\y in {4/-1}{
    \pgfmathsetmacro{\boxwd}{5};
    \draw[thin] (\boxwd*\y-0.275, -\x-0.025) rectangle (\boxwd + \boxwd*\y - 0.325, -\x-.975);
    \draw (0.5*\boxwd + \boxwd*\y - 0.3,-\x-0.5) node (\x\y) {};
}
\foreach \x\y in {2/1, 2/2, 3/1, 3/2, 5/1}{
    \pgfmathsetmacro{\boxwd}{5};
    \draw[thin] (\boxwd*\y - \boxwd + 1.325, -\x-0.025) rectangle (\boxwd*\y + 1.325, -\x-.975);
    \draw (\boxwd*\y + 1.3 - 0.5*\boxwd,-\x-0.5) node (\x\y) {};
}
\foreach \x\y in {3/3}{
    \pgfmathsetmacro{\boxwd}{5};
    \pgfmathsetmacro{\newboxwd}{6.2};
    \draw[thin] (\boxwd*\y - \boxwd + 1.325, -\x-0.025) rectangle (\boxwd*\y - \boxwd + \newboxwd + 1.325 , -\x-.975);
    \draw (\boxwd*\y - \boxwd + 1.3 + 0.5*\newboxwd,-\x-0.5) node (\x\y) {};
}
\node at (21) {$ s_{1}s_{0}$};
\node at (22) {$ {\color{blue} \cancel{s_{4}} \cancel{s_{5}} \cancel{s_{4}}} s_{3} {\color{red}s_{2} s_{1}} \cancel{s_{0}}$};
\node at (31) {$ s_{2} \cancel{s_{1}} s_{0}$};
\node at (32) {$ {\color{blue} s_{4}s_{5}s_{4}} {\color{red} \cancel{s_{3}}s_{2}s_{1}} s_{0}$};
\node at (33) {$ {\color{blue} \cancel{s_{1}} \cancel{s_{2}} \cancel{s_{3}} \cancel{s_{4}} \cancel{s_{5}} \cancel{s_{4}} \cancel{s_{3}} {s_{2}}s_{1}} s_{0}$};
\node at (4-1) {$ s_{5}s_{4}s_{3}s_{2}s_{1} \cancel{s_{0}}$};
\node at (51) {$ s_{4} \cancel{s_{3}} s_{2} \cancel{s_{1}} s_{0}$};
\draw (-0.15, -0.75) -- (-0.15, -6.25);
\draw (1.15, -0.75) -- (1.15, -6.25);
\end{tikzpicture}
\]
In this example, $T'_{{\color{blue}s_{4} \cancel{s_{5}} {s_{4}}}}$ produces $T_{\color{blue}\cancel{s_{4}} \cancel{s_{5}} \cancel{s_{4}}}$
since 
$s_{4}s_{4} = 1$, and 
\[
T'_{{\color{blue} {s_{1}} {s_{2}} \cancel{s_{3}} {s_{4}}{s_{5}} {s_{4}}{s_{3}}\cancel{s_{2}}s_{1}}}
\qquad\hbox{produces}\qquad
T_{{\color{blue}
  \cancel{s_{1}} \cancel{s_{2}} \cancel{s_{3}} \cancel{s_{4}}\cancel{s_{5}} \cancel{s_{4}}\cancel{s_{3}}s_{2}s_{1}
}  }
\]
since, for this compression section, $\nu_S = (1, 0, 0, 0, 0)$ and $W_{\nu_S}= \langle s_2,s_3,s_4,s_5\rangle$ and
$${s_{1}} {s_{2}}  {s_{4}}{s_{5}} {s_{4}}{s_{3}}s_{1}W_{\nu_S} = {s_{1}} {s_{2}}s_{1}W_{\nu_S} = s_{2}s_{1}s_{2}W_{\nu_S} = s_{2}s_{1}W_{\nu_S}.$$
\qed
\end{remark}

\subsection{Across-the-0-gap compression weights}\label{0-gap}

Proposition~\ref{0gap} provides the weight formula for $0$-gap compression sequences.  
It is a translation of~\cite[Lemma 4.2]{GR21} to the Koornwinder case. Example~\ref{0gapexample} illustrates 
Proposition~\ref{0gap}. 

Suppose that $i, m\in \{1, \ldots, n\}$ with $i< m-1$ and let $W_{[i,m-1]}$ be the subgroup of $W_{\mathrm{fin}}$ generated
by $s_i, \ldots, s_{m-2}$.  An element $y \in W_{\mathrm{fin}}$ is the minimal length representative of the coset $yW_{[i,m-1]}$ if and only if
$$y(i) \prec y(i+1) \prec \cdots \prec y(m-1).$$ 
where $\prec$ is the order on $\{1, \ldots, n, -n, \ldots, -1\}$ given by
$1\prec 2\prec \cdots \prec n \prec -n \prec \cdots \prec -2\prec -1$ (see Remark~\ref{sgnorder}).
The coroot sequence for the word $s_{m-1}\cdots s_i$ is the coroot sequence $\beta = (\beta_m, \ldots, \beta_{i+1})$ given by
\begin{equation}
\beta_m = \varepsilon_{m-1}-\varepsilon_m, \quad\ldots\ , \quad
\beta_{i+2} = \varepsilon_{i+1}-\varepsilon_m, \quad \beta_{i+1} = \varepsilon_i-\varepsilon_m.
\label{0gapcrtseq}
\end{equation}

\begin{prop} \label{0gap} 
Let $i, m\in \{1, \ldots, n\}$ with $i< m-1$.
Let 
\begin{align*}
\mu &= (\mu_1, \ldots, \mu_{i-1}, 0, \ldots, 0, \mu_m, \mu_{m+1}, \ldots, \mu_n)\quad\hbox{with $\mu_m\in \ZZ_{>0}$ \quad and let } \\
\nu &= (\mu_1, \ldots, \mu_{i-1}, \mu_m, 0, \ldots, 0, \mu_{m+1}, \ldots, \mu_n),
\quad\hbox{so that $\mu =s_{m-1}\cdots s_{i+1}s_i\nu$.}
\end{align*}
Let $y \in W_{\mathrm{fin}}$ such that $y(i) \prec y(i+1) \prec \cdots \prec y(m-1)$
and let 
$$j \in \{i, i+1, \ldots, m\}\quad\hbox{be such that}\quad
y(j) \prec {\color{blue}y(m)} \prec y(j+1)
$$
with $j=i$ if ${\color{blue} y(m)}\prec y(i)$ and $j=m$ if $y(m-1)\prec {\color{blue} y(m)}$.
Let $A^{(i,j,m)}_\ell=A^{(i,j,m)}_\ell(\beta)$ as defined in~\eqref{AEwts} where
$\beta = (\beta_m, \ldots, \beta_{i+1})$ is the coroot sequence given in~\eqref{0gapcrtseq}.
Then
$$\widehat{E}^y_\mu 
= 
\sum_{\ell = i}^m \ev_\nu(A^{(i,j,m)}_\ell) \widehat{E}^{ys_{\ell-1}\cdots s_i}_\nu,
\quad\hbox{where}\quad
A^{(i,j,m)}_\ell = \begin{cases}
t^{-\frac12(m-\ell-1)}F^+_{\varepsilon_i-\varepsilon_m}, &\hbox{if $\ell\in \{i,\ldots, j-1\}$,}  \\
t^{-\frac12(m-\ell-1)}F^-_{\varepsilon_i-\varepsilon_m}, &\hbox{if $\ell\in \{j, \ldots, m-1\}$,}  \\
1, &\hbox{if $\ell=m$.}
\end{cases}
$$
\end{prop}
\begin{proof}

The proof is by induction on $m - i$.  
The base case $m = i+1$ is Proposition~\ref{alcwksteps}(b):
$$
\widehat{E}^y_\mu = \begin{cases}
\widehat{E}^{ys_{m-1}}_\nu+ \ev_\nu(F_{\varepsilon_{m-1}-\varepsilon_m}^-) \widehat{E}^y_\nu, &\hbox{if $y(m)\prec y(m-1)$,} \\
\widehat{E}^{ys_{m-1}}_\nu+ \ev_\nu(F_{\varepsilon_{m-1}-\varepsilon_m}^+) \widehat{E}^y_\nu, &\hbox{if $y(m-1)\prec y(m)$.} 
\end{cases}
$$
The induction step is as follows. By Proposition~\ref{alcwksteps}(b),
\begin{equation*}
\widehat{E}^y_\mu
= \widehat{E}^{ys_{m-1}}_{s_{m-1}\mu} 
+ \ev_{s_{m-2}\cdots s_i\nu}(F^-_{\varepsilon_{m-1}-\varepsilon_m})\widehat{E}^y_{s_{m-1}\mu}
= \widehat{E}^{ys_{m-1}}_{s_{m-1}\mu} 
+ \ev_\nu(F^-_{\varepsilon_i-\varepsilon_m})\widehat{E}^y_{s_{m-1}\mu}.
\end{equation*}
where the second equality is from
$$
\ev_{s_{m-1}\mu}(F^-_{\varepsilon_{m-1}-\varepsilon_m})
=\ev_{s_{m-2}\cdots s_i\nu}(F^-_{\varepsilon_{m-1} - \varepsilon_m})
= \ev_\nu(s_i\cdots s_{m-2}F^-_{\varepsilon_{m-1} - \varepsilon_m})
= \ev_\nu(F^-_{\varepsilon_i - \varepsilon_m}).
$$
Since
$s_{m-1}\mu = (\mu_1,\ldots, \mu_{i-1}, 0, \ldots, 0, \mu_m, 0, \mu_{m+1}, \ldots, \mu_n)$ with $\mu_m>0$ and
$$
ys_{m-1}(i) \prec \cdots \prec   ys_{m-1}(j) \prec {\color{blue} ys_{m-1}(m-1)} \prec ys_{m-1}(j+1)
\prec \cdots \prec ys_{m-1}(m-2). 
$$
Then, by induction,
\begin{align*}
\widehat{E}^{ys_{m-1}}_{s_{m-1}\mu} 
&= \widehat{E}_\nu^{ys_{m-1}s_{m-2}\cdots s_i}  
+\sum_{\ell = i}^{m-2}
\ev_\nu(A_\ell^{(i,j,m-1)})
\widehat{E}_\nu^{ys_{m-1}s_{\ell-1}\cdots s_i} \qquad\hbox{and}\\
\ev_\nu(F^-_{\varepsilon_i-\varepsilon_m})\widehat{E}^y_{s_{m-1}\mu}
&= \ev_\nu(F^-_{\varepsilon_i-\varepsilon_m}) \widehat{E}_\nu^{ys_{m-2}\cdots s_i}
+\sum_{\ell = i}^{m-2}
\ev_\nu(F^-_{\varepsilon_i-\varepsilon_m} A_\ell^{(i,m-1,m-1)})
\widehat{E}_\nu^{ys_{\ell-1}\cdots s_i},
\end{align*}
since $y(i)\prec \cdots \prec y(m-2)\prec {\color{blue}  y(m-1)}$.

If $\ell\in \{i, \ldots, m-2\}$ then
$\widehat{E}_\nu^{ys_{m-1}s_{\ell-1}\cdots s_i} 
= \widehat{E}_\nu^{ys_{\ell-1}\cdots s_is_{m-1}}
= T_{ys_{\ell-1}\cdots s_is_{m-1}} \widehat{E}_\nu
= T_{ys_{\ell-1}\cdots s_i}T_{s_{m-1}}^{-1}\widehat{E}_\nu$,
and since $\nu_{m-1}=\nu_m$ then $T_{s_{m-1}}^{-1}\widehat{E}_\nu = t^{-\frac12}\widehat{E}_\nu$. 
So
\begin{align*}
\widehat{E}_\nu^{ys_{m-1}s_{\ell-1}\cdots s_i} 
= T_{zs_{\ell-1}\cdots s_i}t^{-\frac12}\widehat{E}_\nu
= t^{-\frac12}\widehat{E}_\nu^{ys_{\ell-1}\cdots s_i}
\end{align*}
and
\begin{align*}
&\widehat{E}^y_\mu
= \widehat{E}^{ys_{m-1}}_{s_{m-1}\mu} 
+ \ev_\nu(F^-_{\varepsilon_i-\varepsilon_m})\widehat{E}^y_{s_{m-1}\mu} \\
&= \widehat{E}_\nu^{ys_{m-1}s_{m-2}\cdots s_i}  +\ev_\nu(F^-_{\varepsilon_i-\varepsilon_m}) \widehat{E}_\nu^{ys_{m-2}\cdots s_i} 
+\sum_{\ell=i}^{m-2} 
\ev_\nu\big(t^{-\frac12}A_\ell^{(i,j,m-1)} +F^-_{\varepsilon_i-\varepsilon_m} A_\ell^{(i,m-1,m-1)} \big)
\widehat{E}_\nu^{ys_{\ell-1}\cdots s_i}.
\end{align*}
If $\ell\in \{i, \ldots, j-1\}$ then $A_\ell^{(i,j,m-1)} = A_\ell^{(i,m-1,m-1)}
= t^{-\frac12(m-\ell-2)} F^+_{\varepsilon_i-\varepsilon_{m-1}}$
and
\begin{align*}
\ev_\nu(t^{-\frac12}&A_\ell^{(i,j,m-1)} +F^-_{\varepsilon_i-\varepsilon_m} A_\ell^{(i,m-1,m-1)})
= t^{-\frac12(m-\ell-1)} \ev_\nu(
F^+_{\varepsilon_i-\varepsilon_{m-1}} + t^{\frac12} F^-_{\varepsilon_i-\varepsilon_m} F^+_{\varepsilon_i-\varepsilon_{m-1}} ).
\end{align*}
If $\ell\in \{j, \ldots, m-2\}$ then $A_\ell^{(i,j,m-1)} =t^{-\frac12(m-\ell-2)} F^-_{\varepsilon_i-\varepsilon_{m-1}}$ and
$A_\ell^{(i,j,m-1)} =t^{-\frac12(m-(\ell-2) ) } F^+_{\varepsilon_i-\varepsilon_{m-1}}$
and
\begin{align*}
\ev_\nu(t^{-\frac12}&A_\ell^{(i,j,m-1)} +F^-_{\varepsilon_i-\varepsilon_m} A_\ell^{(i,m-1,m-1)})
= t^{-\frac12(m-\ell-1)} \ev_\nu(F^-_{\varepsilon_i-\varepsilon_{m-1}}
+t^{\frac12}F^-_{\varepsilon_i-\varepsilon_m} F^+_{\varepsilon_i-\varepsilon_{m-1}}))
\end{align*}
Since $\nu_{m-1}=\nu_m$ then $\ev_\nu(Y_m) 
= \ev_\nu(T_{m-1}^{-1}Y_{m-1}T_{m-1}^{-1}) = \ev_\nu(t^{-1}Y_{m-1})$. So
\begin{align*}
\ev_\nu&(F^-_{\varepsilon_i-\varepsilon_{m-1}}
+t^{\frac12}F^-_{\varepsilon_i-\varepsilon_m}F^+_{\varepsilon_i - \varepsilon_{m-1}}) 
= \ev_\nu\Big(
\frac{t^{-\frac12}(1-t)Y_i^{-1}Y_{m-1}}{(1-Y_i^{-1}Y_{m-1})}+
t^{\frac12}
\frac{t^{-\frac12}(1-t)Y_i^{-1}Y_m}{(1-Y_i^{-1}Y_m)}
\frac{t^{-\frac12}(1-t)}{(1-Y_i^{-1}Y_{m-1})}
\Big)
\\
&= \ev_\nu\Big(
\frac{t^{-\frac12}(1-t)Y_i^{-1}tY_m}{(1-Y_i^{-1}Y_{m-1})}+
\frac{(1-t)Y_i^{-1}Y_m}{(1-Y_i^{-1}Y_m)}
\frac{t^{-\frac12}(1-t)}{(1-Y_i^{-1}Y_{m-1})}
\Big)
=\ev_\nu(F^-_{\varepsilon_i-\varepsilon_m})
\end{align*}
and
\begin{align*}
\ev_\nu(&F^+_{\varepsilon_i - \varepsilon_{m-1}}
+t^{\frac12}F^-_{\varepsilon_i-\varepsilon_m}F^+_{\varepsilon_i - \varepsilon_{m-1}}) 
= \ev_\nu\Big(
\frac{t^{-\frac12}(1-t)}{(1-Y_i^{-1}Y_{m-1})}+
t^{\frac12}
\frac{t^{-\frac12}(1-t)Y_i^{-1}Y_m}{(1-Y_i^{-1}Y_m)}
\frac{t^{-\frac12}(1-t)}{(1-Y_i^{-1}Y_{m-1})}
\Big)
\\
&= \ev_\nu\Big(
\frac{t^{-\frac12}(1-t)}{(1-Y_i^{-1}Y_{m-1})}\Big(
\frac{1-Y_i^{-1}Y_m+Y_i^{-1}Y_m -tY_i^{-1}Y_m}{(1-Y_i^{-1}Y_m)}
\Big)\Big)
=  \ev_\nu(F^+_{\varepsilon_i-\varepsilon_m}).
\end{align*}
So
$\ev_\nu(t^{-\frac12}A_\ell^{(i,j,m-1)} +F^-_{\varepsilon_i-\varepsilon_m} A_\ell^{(i,m-1,m-1)})
= \ev_\nu(A_\ell^{(i,j,m)})$ and the result follows.
\end{proof}

\begin{example} \label{0gapexample}
Let $\mu = (2, 0, 0, 0,4,6)$ 
and $\nu = (2,4,0,0,0,6)$. \hfil\break
In the case that $y(2)\prec y(3)\prec y(4)\prec {\color{blue} y(5)}$ then
\begin{align*}
\widehat{E}^y_\mu 
&=  \ev_\nu(A^{(2,5,5)}_5)\widehat{E}^{ys_4s_3s_2}_\nu
+ \ev_\nu(A^{(2,5,5)}_4) \widehat{E}^{ys_3s_2}_\nu
+ \ev_\nu(A^{(2,5,5)}_3)\widehat{E}^{ys_2}_\nu 
+ \ev_\nu(A^{(2,5,5)}_2)\widehat{E}^{y}_\nu \\
&= \widehat{E}^{ys_4s_3s_2}_\nu
+ \ev_\nu(F^+_{\varepsilon_2-\varepsilon_5}) \widehat{E}^{ys_3s_2}_\nu
+ \ev_\nu(t^{-\frac12} F^+_{\varepsilon_2-\varepsilon_5})\widehat{E}^{ys_2}_\nu 
+ \ev_\nu(t^{-\frac22}F^+_{\varepsilon_2-\varepsilon_5} )\widehat{E}^{y}_\nu.
\end{align*}
In the case that $y(2)\prec y(3)\prec {\color{blue}y(5)} \prec y(4)$ then
\begin{align*}
\widehat{E}^y_\mu 
&=  \ev_\nu(A^{(2,4,5)}_5)\widehat{E}^{ys_4s_3s_2}_\nu
+ \ev_\nu(A^{(2,4,5)}_4) \widehat{E}^{ys_3s_2}_\nu
+ \ev_\nu(A^{(2,4,5)}_3)\widehat{E}^{ys_2}_\nu 
+ \ev_\nu(A^{(2,4,5)}_2)\widehat{E}^{y}_\nu \\
&= \widehat{E}^{ys_4s_3s_2}_\nu
+ \ev_\nu(F^-_{\varepsilon_2-\varepsilon_5}) \widehat{E}^{ys_3s_2}_\nu
+ \ev_\nu( t^{-\frac12}F^+_{\varepsilon_2-\varepsilon_5} )\widehat{E}^{ys_2}_\nu 
+ \ev_\nu( t^{-1}F^+_{\varepsilon_2-\varepsilon_5} )\widehat{E}^{y}_\nu.
\end{align*}
In the case that $y(2)\prec {\color{blue} y(5)}\prec y(3)\prec y(4)$ then
\begin{align*}
\widehat{E}^y_\mu 
&=  \ev_\nu(A^{(2,3,5)}_5)\widehat{E}^{ys_4s_3s_2}_\nu
+ \ev_\nu(A^{(2,3,5)}_4) \widehat{E}^{ys_3s_2}_\nu
+ \ev_\nu(A^{(2,3,5)}_3)\widehat{E}^{ys_2}_\nu 
+ \ev_\nu(A^{(2,3,5)}_2)\widehat{E}^{y}_\nu \\
&= \widehat{E}^{ys_4s_3s_2}_\nu
+ \ev_\nu(F^-_{\varepsilon_2-\varepsilon_5}) \widehat{E}^{ys_3s_2}_\nu
+ \ev_\nu(t^{-\frac12}F^-_{\varepsilon_2-\varepsilon_5})\widehat{E}^{ys_2}_\nu 
+ \ev_\nu( t^{-1}F^+_{\varepsilon_2-\varepsilon_5} )\widehat{E}^{y}_\nu.
\end{align*}
In the case that ${\color{blue} y(5)} \prec y(2) \prec y(3) \prec y(4)$ then
\begin{align*}
\widehat{E}^y_\mu 
&=  \ev_\nu(A^{(2,2,5)}_5)\widehat{E}^{ys_4s_3s_2}_\nu
+ \ev_\nu(A^{(2,2,5)}_4) \widehat{E}^{ys_3s_2}_\nu
+ \ev_\nu(A^{(2,2,5)}_3)\widehat{E}^{ys_2}_\nu 
+ \ev_\nu(A^{(2,2,5)}_2)\widehat{E}^{y}_\nu \\
&= \widehat{E}^{ys_4s_3s_2}_\nu
+ \ev_\nu(F^-_{\varepsilon_2-\varepsilon_5}) \widehat{E}^{ys_3s_2}_\nu
+ \ev_\nu(t^{-\frac12}F^-_{\varepsilon_2-\varepsilon_5})\widehat{E}^{ys_2}_\nu 
+ \ev_\nu( t^{-1}F^-_{\varepsilon_2-\varepsilon_5} )\widehat{E}^{y}_\nu.
\end{align*}
\end{example}

\subsection{Around-the-end compression weights for $y(i)\prec y(i+1)$}
\label{section:around_the_end_z1}

Proposition~\ref{0end} provides the around-the-end compression weights for the case when
$y$ has $0 < y(i) \prec y(i+1)$. Proposition~\ref{0end} 
will be proved in Section~\ref{subsubsect: 0end} after deriving some
computational simplification identities in Section~\ref{subsect: simplication}. Building on Proposition~\ref{0end}, the general case will be treated in 
Proposition~\ref{proposition:GeneralCaseAroundTheEnd} in Section~\ref{subsect: around the end general}.

Suppose that $i\in \{1, \ldots, n\}$ and let $W_{[i+1,-(i+1)]}$ be the subgroup of $W_{\mathrm{fin}}$ generated
by $s_{i+1}, \ldots, s_n$.  An element $y \in W_{\mathrm{fin}}$ is the minimal length representative of the coset
$yW_{[i+1,-(i+1)]}$ if and only if $$0 < y(i+1) < \cdots < y(n).$$ 
The weights $A^{(i)}_{\pm\ell}$ appearing in Proposition~\ref{0end} are the $A^{(i)}_{\pm\ell}(\beta)$ of~\eqref{Adef1A}-\eqref{Adef1C}
for the coroot sequence $\beta = (\beta_{-i}, \ldots, \beta_{-n}, \beta_n, \ldots \beta_{i+1})$ for the word $s_i\cdots s_n\cdots s_i$, given by
\begin{align}
\beta_{-i} = \varepsilon_i+\varepsilon_{i+1}, \ \ldots,\ \beta_{-(n-1)}=\varepsilon_i+\varepsilon_n,
\ \beta_{-n} = \varepsilon_i, \ \beta_n = \varepsilon_i-\varepsilon_n, \ \ldots\ \beta_{i+1}=\varepsilon_i-\varepsilon_{i+1}.
\label{0endcrtseq}
\end{align}

\begin{prop}  \label{0end}
Let $i\in \{1, \ldots, n\}$. Let 
\begin{align*}
\mu &= (\mu_1, \ldots, \mu_{i-1}, -\mu_i, 0, \ldots, 0)
\quad\hbox{with $\mu_1,\ldots, \mu_{i-1} \in \ZZ$ and $\mu_i \in \ZZ_{>0}$, \quad and let} \\
\nu &= (\mu_1, \ldots, \mu_{i-1}, \mu_i, 0, \ldots, 0)
\quad\hbox{so that}\quad
\mu = s_is_{i+1}\cdots s_n \cdots s_{i+1}s_i\nu.
\end{align*}
Let $y\in W_{\mathrm{fin}}$ with $0 < y(i) < y(i+1) < y(i+2) < \cdots < y(n)$ so that $y$ is minimal length in the coset $yW_{[i+1,-(i+1)]}$
and $y(i)<y(i+1)$.
Let $A^{(i)}_{\pm\ell}= A^{(i)}_{\pm\ell}(\beta_{-i}, \ldots, \beta_{-n}, \beta_n, \ldots \beta_{i+1})$ as defined in~\eqref{Adef1A}-\eqref{Adef1C}
where $(\beta_{-i}, \ldots, \beta_{-n}, \beta_n, \ldots \beta_{i+1})$ is the coroot sequence given in~\eqref{0endcrtseq}.
Then
\begin{align*}
\widehat{E}^y_\mu
&= 
\ev_\nu(A^{(i)}_{i})\widehat{E}_\nu^{y}
+
\sum_{\ell=i+1}^n \ev_\nu(A^{(i)}_{\ell})\widehat{E}_\nu^{ys_{\ell-1}\cdots s_i}
+
\sum_{\ell=i}^n \ev_\nu(A^{(i)}_{-\ell})\widehat{E}_\nu^{ys_\ell\cdots s_n\cdots s_i}.
\end{align*}
\end{prop}

The following gives an example of the statement of Proposition~\ref{0end}.
\begin{example} \label{0endexample}
Let $n=7$ and consider 
$\mu = (-4,-1, -5, 0, 0, 0, 0)$ and $\nu = (-4,-1,5, 0,0,0,0)$, for which $i=3$. 
In this case,
\begin{align*}
\widehat{E}^{y}_\mu
&=\ev_\nu(A^{(3)}_{-3}) \widehat{E}^{ys_3\cdots s_7\cdots s_3}_\nu
+\ev_\nu(A^{(3)}_{-4})\widehat{E}^{ys_4\cdots s_7\cdots s_3}_\nu
+\ev_\nu(A^{(3)}_{-5}) \widehat{E}^{ys_5\cdots s_7\cdots s_3}_\nu
+\ev_\nu(A^{(3)}_{-6}) \widehat{E}^{ys_6 s_7\cdots s_3}_\nu
\\
&\qquad
+\ev_\nu(A^{(3)}_{-7}) \widehat{E}^{y s_7\cdots s_3}_\nu
+\ev_\nu(A^{(3)}_{7})
\widehat{E}^{y s_6\cdots s_3}_\nu
+\ev_\nu(A^{(3)}_{6})
\widehat{E}^{y s_5s_4 s_3}_\nu
+\ev_\nu(A^{(3)}_{5})
\widehat{E}^{y s_4 s_3}_\nu
\\
&\qquad
+\ev_\nu(A^{(3)}_{4})
\widehat{E}^{y s_3}_\nu
+\ev_\nu(A^{(3)}_{3})
\widehat{E}^{y}_\nu\qquad\hbox{is}
\end{align*}
\begin{align*}
\widehat{E}^{y}_\mu
&=\widehat{E}^{ys_3\cdots s_7\cdots s_3}_\nu
+\ev_\nu(F^+_{\varepsilon_3+\varepsilon_4}) \widehat{E}^{ys_4\cdots s_7\cdots s_3}_\nu
+\ev_\nu(t^{-\frac12} F^+_{\varepsilon_3+\varepsilon_4}) \widehat{E}^{ys_5\cdots s_7\cdots s_3}_\nu
\\
&\qquad
+\ev_\nu(t^{-\frac22} F^+_{\varepsilon_3+\varepsilon_4})
\widehat{E}^{ys_6 s_7\cdots s_3}_\nu
+\ev_\nu(t^{-\frac32} F^+_{\varepsilon_3+\varepsilon_4})
\widehat{E}^{y s_7\cdots s_3}_\nu
\\
&\qquad
+\ev_\nu(t^{\frac32} F^-_{\varepsilon_3+\varepsilon_4} F^+_{\varepsilon_3} + C_{\varepsilon_3}C_{\varepsilon_3+\varepsilon_4}C_{\varepsilon_3+\varepsilon_5} C_ {\varepsilon_3+\varepsilon_6} F^+_{\varepsilon_3-\varepsilon_7}))
\widehat{E}^{y s_6\cdots s_3}_\nu
\\
&\qquad
+\ev_\nu(t^{\frac22} F^-_{\varepsilon_3+\varepsilon_4} F^+_{\varepsilon_3} + t^{\frac22} F^-_{\varepsilon_3+\varepsilon_4} F^+_{\varepsilon_3+\varepsilon_7} C_{\varepsilon_3} F^+_{\varepsilon_3-\varepsilon_7} + 
t^{-\frac12} C_{\varepsilon_3}C_{\varepsilon_3+\varepsilon_4}
C_{\varepsilon_3+\varepsilon_5} C_{\varepsilon_3+\varepsilon_7} F^+_{\varepsilon_3-\varepsilon_7})
\widehat{E}^{y s_5s_4 s_3}_\nu
\\
&\qquad
+\ev_\nu\left(
\begin{array}{l}
t^{\frac12}F^-_{\varepsilon_3+\varepsilon_4} F^+_{\varepsilon_3} 
+ t^{\frac12}F^-_{\varepsilon_3+\varepsilon_4} F^+_{\varepsilon_3+\varepsilon_7} C_{\varepsilon_3} F^+_{\varepsilon_3-\varepsilon_7} 
+ t^{\frac12}F^-_{\varepsilon_3+\varepsilon_4}  F^+_{\varepsilon_3+\varepsilon_6}  
C_{\varepsilon_3+\varepsilon_7}C_{\varepsilon_3} 
t^{-\frac12} F^+_{\varepsilon_3-\varepsilon_7}  \\ + 
C_{\varepsilon_3} C_{\varepsilon_3+\varepsilon_4} C_{\varepsilon_3+\varepsilon_6} C_{\varepsilon_3+\varepsilon_7} 
t^{-\frac22} F^+_{\varepsilon_3-\varepsilon_7}
\end{array}
\right)
\widehat{E}^{y s_4 s_3}_\nu
\\
&\qquad
+\ev_\nu\left( 
\begin{array}{l}
F^-_{\varepsilon_3+\varepsilon_4} F^+_{\varepsilon_3} + 
F^-_{\varepsilon_3+\varepsilon_4} F^+_{\varepsilon_3+\varepsilon_7} C_{\varepsilon_3} F^+_{\varepsilon_3-\varepsilon_7} +F^-_{\varepsilon_3+\varepsilon_4} F^+_{\varepsilon_3+\varepsilon_6} 
C_{\varepsilon_3+\varepsilon_7} C_{\varepsilon_3} t^{-\frac12} F^+_{\varepsilon_3-\varepsilon_7} \\
+F^-_{\varepsilon_3+\varepsilon_4} F^+_{\varepsilon_3+\varepsilon_5} 
C_{\varepsilon_3+\varepsilon_6} C_{\varepsilon_3+\varepsilon_7} C_{\varepsilon_3} t^{-\frac22}  F^+_{\varepsilon_3-\varepsilon_7}
+C_{\varepsilon_3+\varepsilon_5} C_{\varepsilon_3+\varepsilon_6} C_{\varepsilon_3+\varepsilon_7} 
C_{\varepsilon_3} t^{-\frac32}
F^+_{\varepsilon_3-\varepsilon_7}
\end{array}
\right)
\widehat{E}^{y s_3}_\nu
\\
&\qquad
+\ev_\nu
\left(
\begin{array}{l}
F^+_{\varepsilon_3}
+F^+_{\varepsilon_3+\varepsilon_7} C_{\varepsilon_3} F^+_{\varepsilon_3-\varepsilon_7}
+F^+_{\varepsilon_3+\varepsilon_6} 
C_{\varepsilon_3+\varepsilon_7} C_{\varepsilon_3} 
t^{-\frac12}F^+_{\varepsilon_3-\varepsilon_7}
\\[0.2cm]
+F^+_{\varepsilon_3+\varepsilon_5} 
C_{\varepsilon_3+\varepsilon_6} C_{\varepsilon_3+\varepsilon_7} C_{\varepsilon_3} 
t^{-\frac22} F^+_{\varepsilon_3-\varepsilon_7}
\\[0.2cm]
+F^+_{\varepsilon_3+\varepsilon_4} C_{\varepsilon_3+\varepsilon_5} C_{\varepsilon_3+\varepsilon_6} C_{\varepsilon_3+\varepsilon_7} 
C_{\varepsilon_3}
t^{-\frac32} F^+_{\varepsilon_3-\varepsilon_7}
\end{array}\right) 
\widehat{E}^{y}_\nu.
\end{align*}
\end{example}

\subsubsection{Computational simplification lemmas}\label{subsect: simplication}

\begin{lemma} \label{endcombgen}
Let $i\in \{ 1, \ldots, n-2\}$
and $\nu =(\nu_1, \ldots, \nu_i, 0, \ldots, 0)\in \ZZ^n$ and let $\ell\in \{i+2, \ldots, n\}$.  
Let $A^{(i)}_{\pm\ell}= A^{(i)}_{\pm\ell}(\beta_{-i}, \ldots, \beta_{-n}, \beta_n, \ldots \beta_{i+1})$ as defined in~\eqref{Adef1A}-\eqref{Adef1C}
where $\beta_{-i}, \ldots, \beta_{-n}, \beta_n, \ldots \beta_{i+1}$  are as given in~\eqref{0endcrtseq}.
\item[(a)]
$\displaystyle{
\ev_\nu\Big(A^{(i)}_{-\ell}\Big) = \ev_\nu\Big(C_{\varepsilon_i+\varepsilon_{i+1}}(s_iA^{(i+1)}_{-\ell}) \Big)
\quad\hbox{and}\quad
\ev_\nu\Big(A^{(i)}_{\ell} \Big)
= \ev_\nu\Big(C_{\varepsilon_i+\varepsilon_{i+1}}(s_iA^{(i+1)}_{\ell})\Big).
}$
\item[(b)]
\begin{align*}
\ev_\nu\Big(&\sum_{\ell=i+1}^n \big(
t^{n-\frac{i+\ell+1}{2}} t_n^{\frac12} (s_iA^{(i+1)}_{-\ell})
+ 
t^{\frac{\ell-i-1}{2}}(s_iA^{(i+1)}_\ell) \big)
\Big)
=\ev_\nu\Big(\Big(\prod_{m=i+2}^n C_{\varepsilon_i+\varepsilon_m}\Big) C_{\varepsilon_i}
\Big(\prod_{m=i+2}^n C_{\varepsilon_i-\varepsilon_m}\Big)\Big).
\end{align*}
\end{lemma}

\begin{proof}
Using
$\ev_\nu\Big(C_{\varepsilon_i+\varepsilon_{i+1}} F^+_{\varepsilon_{i}+\varepsilon_{i+2}} \Big)
 = t^{-\frac12} \ev_\nu\Big(F^+_{\varepsilon_{i}+\varepsilon_{i+1}}\Big)$ from Example~\ref{CFshrink} gives
\begin{align*}
\ev_\nu\Big(C_{\varepsilon_i+\varepsilon_{i+1}}(s_iA^{(i+1)}_{-\ell})\Big) 
&= t^{-\frac12 (\ell-i-2)} \ev_\nu\Big(C_{\varepsilon_i+\varepsilon_{i+1}} F^+_{\varepsilon_{i}+\varepsilon_{i+2}} \Big) \\ 
&
=t^{-\frac12 (\ell-i-2)} t^{-\frac12} \ev_\nu\Big(F^+_{\varepsilon_{i}+\varepsilon_{i+1}} \Big) 
= \ev_\nu\Big(A_{-\ell}^{(i)}\Big),
\end{align*}
establishing the first identity in (a).

Using
$\ev_\nu\Big(C_{\varepsilon_i+\varepsilon_{i+1}} F^-_{\varepsilon_{i}+\varepsilon_{i+2}} \Big)
= t^{\frac12} \ev_\nu\Big(F^-_{\varepsilon_{i}+\varepsilon_{i+1}}\Big)$ gives
\begin{align*}
\ev_\nu&\Big(C_{\varepsilon_i+\varepsilon_{i+1}}s_i \left(A_\ell^{(i+1)} \right) \Big) \\
&= \ev_\nu\Big(C_{\varepsilon_i+\varepsilon_{i+1}} \Big(\prod_{m=i+2}^{\ell-1} C_{\varepsilon_{i}+\varepsilon_m}\Big) 
\Big(\prod_{r=\ell+1}^n C_{\varepsilon_{i} +\varepsilon_r}\Big) C_{\varepsilon_{i}}
t^{-\frac12 (n-\ell)} F^+_{\varepsilon_{i}-\varepsilon_n}\Big) \\
&\quad + t^{\frac12 (\ell-1-i-1)} \ev_\nu\Big(C_{\varepsilon_i+\varepsilon_{i+1}} 
F^-_{\varepsilon_{i}+\varepsilon_{i+2}}
F_{\varepsilon_{i}}^+ \Big) \\
&\quad +t^{\frac12  (\ell-1-i-1)} \ev_\nu\Big(C_{\varepsilon_i+\varepsilon_{i+1}}
F^-_{\varepsilon_{i}+\varepsilon_{i+2}}
\Big(\sum_{s=\ell+1}^n F^+_{\varepsilon_{i}+\varepsilon_s}
\Big(\prod_{r=s+1}^n C_{\varepsilon_{i}+\varepsilon_r}\Big) C_{\varepsilon_i}
t^{-\frac12 (n-s)} F^+_{\varepsilon_{i}-\varepsilon_n}
\Big) \Big)\\
&
=\ev_\nu\Big(\Big(\prod_{m=i+1}^{\ell-1} C_{\varepsilon_{i}+\varepsilon_m}\Big) \Big(\prod_{r=\ell+1}^n C_{\varepsilon_{i} +\varepsilon_r}\Big) C_{\varepsilon_{i}}
t^{-\frac12 (n-\ell)} F^+_{\varepsilon_{i}-\varepsilon_n} \Big)\\
&\qquad + t^{\frac12  (\ell-i-1)}  
\ev_\nu\Big(F^-_{\varepsilon_{i}+\varepsilon_{i+1}}
F_{\varepsilon_{i}}^+  \Big) \\
&\qquad +t^{\frac12  (\ell-i-1)} 
\ev_\nu\Big(F^-_{\varepsilon_{i}+\varepsilon_{i+1}}
\Big(\sum_{s=\ell+1}^n F^+_{\varepsilon_{i}+\varepsilon_s}
\Big(\prod_{r=s+1}^n C_{\varepsilon_{i}+\varepsilon_r}\Big) C_{\varepsilon_i}
t^{-\frac12 (n-s)} F^+_{\varepsilon_{i}-\varepsilon_n}
\Big) \Big) \\
&= \ev_\nu\Big(A_\ell^{(i)} \Big),
\end{align*}
which establishes the second identity in (a).

\smallskip\noindent
(b) First note that, by~\eqref{bigC},
\begin{align}
\ev_\nu\Big(&t^{n-i-1}t_n^{\frac12} (s_i A^{(i+1)}_{-(i+1)})
+t^{n-i-1-\frac{1}{2}}t_n^{\frac12} (s_i A^{(i+1)}_{-(i+2)})\Big)
\nonumber \\
&
= \ev_\nu\Big(t^{\frac{2n-2i-3}{2}}t_n^{\frac12}(t^{\frac12}+F_{\varepsilon_i+\varepsilon_{i+2}}^+)\Big)
= \ev_\nu\Big(t^{\frac{2n-2i-3}{2}}t_n^{\frac12}C_{\varepsilon_i+\varepsilon_{i+2}}\Big)
\label{ST1}
\end{align}
and that 
\begin{align*}
\ev_\nu&\Big(t^{\frac12} (s_iA^{(i+1)}_{i+2}) + (s_iA^{(i+1)}_{i+1}) \Big)
\\
&=t^{\frac12} \ev_\nu\Big( s_i
\left( 
\left[\prod_{m=i+3}^n C_{\varepsilon_{i+1}+\varepsilon_m}\right] C_{\varepsilon_{i+1}}
\left[\prod_{m=i+3}^n C_{\varepsilon_{i+1}-\varepsilon_m}\right] F_{\varepsilon_{i+1}-\varepsilon_{i+2}}^+ \right) \Big)\\
&\qquad
+ t^{\frac12} \ev_\nu\Big( s_i
\left( F_{\varepsilon_{i+1}+\varepsilon_{i+2}}^-
\left[
F_{\varepsilon_{i+1}}^+ 
+ \sum_{\ell=i+3}^n F^+_{\varepsilon_{i+1}+\varepsilon_\ell}
\Big(\prod_{m=\ell+1}^n C_{\varepsilon_{i+1}+\varepsilon_m}\Big) C_{\varepsilon_{i+1}}
\Big(\prod_{m=\ell+1}^n C_{\varepsilon_{i+1}-\varepsilon_m}\Big)
F^+_{\varepsilon_{i+1}-\varepsilon_\ell}
\right]\right) \Big)
\\
&\qquad
+ \ev_\nu\Big( s_i \left(
F_{\varepsilon_{i+1}}^+ 
+ \sum_{\ell=i+2}^n F^+_{\varepsilon_{i+1}+\varepsilon_\ell}
\Big(\prod_{m=\ell+1}^n C_{\varepsilon_{i+1}+\varepsilon_m}\Big) C_{\varepsilon_{i+1}}
\Big(\prod_{m=\ell+1}^n C_{\varepsilon_{i+1}-\varepsilon_m}\Big)
F^+_{\varepsilon_{i+1}-\varepsilon_\ell} \right) \Big)
\\
&= t^{\frac12}
 \ev_\nu\Big(
\left(\prod_{m=i+3}^n C_{\varepsilon_i+\varepsilon_m}\right) C_{\varepsilon_i}
\left(\prod_{m=i+3}^n C_{\varepsilon_i-\varepsilon_m}\right)F_{\varepsilon_i-\varepsilon_{i+2}}^+\Big)\\
&\qquad
+t^{\frac12} \ev_\nu\Big(
F_{\varepsilon_i + \varepsilon_{i+2}}^-
\left[
F_{\varepsilon_i}^+ 
+ \sum_{\ell=i+3}^n F^+_{\varepsilon_i+\varepsilon_\ell}
\Big(\prod_{m=\ell+1}^n C_{\varepsilon_i+\varepsilon_m}\Big) C_{\varepsilon_i}
\Big(\prod_{m=\ell+1}^n C_{\varepsilon_i-\varepsilon_m}\Big)
F^+_{\varepsilon_i-\varepsilon_\ell} \right] \Big)
\\
&\qquad
+ 
\ev_\nu\Big( F_{\varepsilon_i}^+ 
+ \sum_{\ell=i+2}^n F^+_{\varepsilon_i+\varepsilon_\ell}
\Big(\prod_{m=\ell+1}^n C_{\varepsilon_i+\varepsilon_m}\Big) C_{\varepsilon_i}
\Big(\prod_{m=\ell+1}^n C_{\varepsilon_i-\varepsilon_m}\Big)
F^+_{\varepsilon_i-\varepsilon_\ell}\Big).
\end{align*}
So
\begin{align*}
\ev_\nu &\Big(t^{\frac12} (s_iA^{(i+1)}_{i+2}) + (s_iA^{(i+1)}_{i+1}) \Big)
\\
&= \ev_\nu\Big(t^{\frac12}(F_{\varepsilon_i + \varepsilon_{i+2}}^-+t^{-\frac12})
\left[
F_{\varepsilon_i}^+ 
+ \sum_{\ell=i+3}^n F^+_{\varepsilon_i+\varepsilon_\ell}
\Big(\prod_{m=\ell+1}^n C_{\varepsilon_i+\varepsilon_m}\Big) C_{\varepsilon_i}
\Big(\prod_{m=\ell+1}^n C_{\varepsilon_i-\varepsilon_m}\Big)
F^+_{\varepsilon_i-\varepsilon_\ell} \right]\Big)
\\
&\quad
+\ev_\nu\Big( (t^{\frac12}+F^+_{\varepsilon_i+\varepsilon_{i+2}})
\Big(\prod_{m=i+3}^n C_{\varepsilon_i+\varepsilon_m}\Big) C_{\varepsilon_i}
\Big(\prod_{m=i+3}^n C_{\varepsilon_i-\varepsilon_m}\Big)
F^+_{\varepsilon_i-\varepsilon_{i+2}} \Big)
\\
&= \ev_\nu\Big( t^{\frac12}C_{\varepsilon_i + \varepsilon_{i+2}}
\Big(
F_{\varepsilon_i}^+ 
+ \sum_{\ell=i+3}^n F^+_{\varepsilon_i+\varepsilon_\ell}
\Big(\prod_{m=\ell+1}^n C_{\varepsilon_i+\varepsilon_m}\Big) C_{\varepsilon_i}
\Big(\prod_{m=\ell+1}^n C_{\varepsilon_i-\varepsilon_m}\Big)
F^+_{\varepsilon_i-\varepsilon_\ell} \Big) \Big)
\\
&\quad
+ \ev_\nu\Big(C_{\varepsilon_i+\varepsilon_{i+2}}
\Big(\prod_{m=i+3}^n C_{\varepsilon_i+\varepsilon_m}\Big) C_{\varepsilon_i}
\Big(\prod_{m=i+3}^n C_{\varepsilon_i-\varepsilon_m}\Big)
F^+_{\varepsilon_i-\varepsilon_{i+2}} \Big).
\end{align*}
Thus
\begin{align}
\ev_\nu\Big(t^{\frac12} &(s_iA^{(i+1)}_{i+2}) + (s_iA^{(i+1)}_{i+1}) \Big)
\label{ST2}
\\
&= \ev_\nu\Big(t^{\frac12} C_{\varepsilon_i+\varepsilon_{i+2}}(s_is_{i+1}A^{(i+2)}_{i+2})
+ C_{\varepsilon_i+\varepsilon_{i+2}} \Big(\prod_{\ell=i+3}^n C_{\varepsilon_i+\varepsilon_\ell}\Big)
C_{\varepsilon_i}
\Big(\prod_{\ell=i+3}^n C_{\varepsilon_i-\varepsilon_\ell}\Big)
F^+_{\varepsilon_i+\varepsilon_{i+2}}\Big).
\nonumber
\end{align}

Let 
\begin{align*}
L 
&=\sum_{\ell=i+1}^n 
t^{n-\frac{i+\ell+1}{2}} t_n^{\frac12} (s_iA^{(i+1)}_{-\ell})
+ \sum_{\ell=i+1}^n t^{\frac{\ell-i-1}{2}}(s_iA^{(i+1)}_\ell).
\end{align*}
so that $\ev_\nu(L)$ is the left-hand side of the identity in (b). 
Plugging in~\eqref{ST1} and~\eqref{ST2} gives
\begin{align*}
\ev_\nu(L) &= \ev_\nu\Big( t^{n-i-1}t_n^{\frac12} (s_i A^{(i+1)}_{-(i+1)}) +t^{n-i-1-\frac{1}{2}}t_n^{\frac12} (s_i A^{(i+1)}_{-(i+2)})
+\sum_{\ell=i+3}^n 
t^{n-\frac{i+\ell+1}{2}} t_n^{\frac12} (s_iA^{(i+1)}_{-\ell}) \Big)
\\
&\quad
+ \ev_\nu\Big(t^{\frac12} (s_iA^{(i+1)}_{i+2}) + (s_iA^{(i+1)}_{i+1}) 
+ \sum_{\ell=i+3}^n t^{\frac{\ell-i-1}{2}}(s_iA^{(i+1)}_\ell) \Big)
\\
&= \ev_\nu\Big( t^{\frac{2n-2i-3}{2}} t_n^{\frac12} C_{\varepsilon_i+\varepsilon_{i+2}}
+ \sum_{\ell=i+3}^n 
t^{n-\frac{i+\ell+1}{2}} t_n^{\frac12} (s_i A^{(i+1)}_{-\ell})
+ \sum_{\ell=i+3}^n t^{\frac{\ell-i-1}{2}} (s_iA^{(i+1)}_\ell) \Big)
\\
&\quad
+ \ev_\nu\Big( t^{\frac12} C_{\varepsilon_i+\varepsilon_{i+2}}(s_is_{i+1}A^{(i+2)}_{i+2})
+ C_{\varepsilon_i+\varepsilon_{i+2}} \Big(\prod_{\ell=i+3}^n C_{\varepsilon_i+\varepsilon_\ell}\Big)
C_{\varepsilon_i}
\Big(\prod_{\ell=i+3}^n C_{\varepsilon_i-\varepsilon_\ell}\Big)
F^+_{\varepsilon_i+\varepsilon_{i+2}}\Big).
\end{align*}
Now note that $A^{(i+2)}_{-(i+2)}=1$ and that if $\ell\in \{i+3, \ldots, n\}$ then
\begin{align*}
\ev_\nu\Big(s_iA^{(i+1)}_{-\ell}\Big)
&= \ev_\nu\Big( C_{\varepsilon_i+\varepsilon_{i+2}}(s_is_{i+1}A^{(i+2)}_{-\ell}) \Big)
\quad\hbox{and}  \\
\ev_\nu\Big(s_iA^{(i+1)}_{\ell}\Big)
&= \ev_\nu\Big( C_{\varepsilon_i+\varepsilon_{i+2}}(s_is_{i+1}A^{(i+2)}_\ell) \Big).
\end{align*}
Thus, 
\begin{align*}
\ev_\nu (L) 
&= \ev_\nu\Big(t^{\frac{2n-2i-3}{2}}t_n^{\frac12} C_{\varepsilon_i+\varepsilon_{i+2}} (s_is_{i+1}A^{(i+2)}_{-(i+2)})\Big) \\
&\quad
+ \ev_\nu\Big(\sum_{\ell=i+3}^n 
t^{n-\frac{i+\ell+1}{2}} t_n^{\frac12} C_{\varepsilon_i+\varepsilon_{i+2}}(s_is_{i+1}A^{(i+2)}_{-\ell})
+ \sum_{\ell=i+3}^n t^{\frac{\ell-i-1}{2}}C_{\varepsilon_i+\varepsilon_{i+2}}(s_is_{i+1}A^{(i+2)}_\ell) \Big)
\\
&\quad
+ \ev_\nu\Big(t^{\frac12} C_{\varepsilon_i+\varepsilon_{i+2}}(s_is_{i+1}A^{(i+2)}_{i+2})
+ C_{\varepsilon_i+\varepsilon_{i+2}} \Big(\prod_{\ell=i+3}^n C_{\varepsilon_i+\varepsilon_\ell}\Big)
C_{\varepsilon_i}
\Big(\prod_{\ell=i+3}^n C_{\varepsilon_i-\varepsilon_\ell}\Big)
F^+_{\varepsilon_i+\varepsilon_{i+2}} \Big)
\\
&=
\ev_\nu\Big(t^{\frac12}C_{\varepsilon_i+\varepsilon_{i+2}} \left[s_is_{i+1} \Big( \sum_{\ell=i+2}^n 
t^{n-\frac{i+\ell+2}{2}} t_n^{\frac12}  A^{(i+2)}_{-\ell}
+ \sum_{\ell=i+2}^n t^{\frac{\ell-i-2}{2}} A^{(i+2)}_\ell \Big)\right]\Big)
\\
&\quad
+ \ev_\nu\Big(C_{\varepsilon_i+\varepsilon_{i+2}} \Big(\prod_{\ell=i+3}^n C_{\varepsilon_i+\varepsilon_\ell}\Big)
C_{\varepsilon_i}
\Big(\prod_{\ell=i+3}^n C_{\varepsilon_i-\varepsilon_\ell}\Big)
F^+_{\varepsilon_i+\varepsilon_{i+2}}\Big).
\end{align*}
Then, by induction,
\begin{align*}
\ev_\nu(L)
&=
\ev_\nu\Big(t^{\frac12} C_{\varepsilon_i+\varepsilon_{i+2}} 
\Big(\prod_{\ell=i+3}^n C_{\varepsilon_i+\varepsilon_\ell}\Big)
C_{\varepsilon_i}
\Big(\prod_{\ell=i+3}^n C_{\varepsilon_i-\varepsilon_\ell}\Big) \Big)
\\
&\qquad
+ \ev_\nu\Big( C_{\varepsilon_i+\varepsilon_{i+2}} 
\Big(\prod_{\ell=i+3}^n C_{\varepsilon_i+\varepsilon_\ell}\Big)
C_{\varepsilon_i}
\Big(\prod_{\ell=i+3}^n C_{\varepsilon_i-\varepsilon_\ell}\Big)
F^+_{\varepsilon_i+\varepsilon_{i+2}} \Big)
\\
&=
\ev_\nu\Big(C_{\varepsilon_i+\varepsilon_{i+2}} 
\Big(\prod_{\ell=i+3}^n C_{\varepsilon_i+\varepsilon_\ell}\Big)
C_{\varepsilon_i}
\Big(\prod_{\ell=i+3}^n C_{\varepsilon_i-\varepsilon_\ell}\Big)
(t^{\frac12}+F^+_{\varepsilon_i+\varepsilon_{i+2}})\Big)
\\
&= \ev_\nu\Big(\Big(\prod_{\ell = i+2}^n C_{\varepsilon_i+\varepsilon_\ell}\Big)C_{\varepsilon_i} 
\Big(\prod_{\ell= i+2}^n C_{\varepsilon_i-\varepsilon_\ell}\Big)\Big).
\end{align*}
\end{proof}

\subsubsection{Proof of Proposition~\ref{0end}}\label{subsubsect: 0end}

This subsection provides the proof of Proposition~\ref{0end}.  Recall the statement.

\begin{prop*} 
Let $i\in \{1, \ldots, n\}$. Let 
\begin{align*}
\mu &= (\mu_1, \ldots, \mu_{i-1}, -\mu_i, 0, \ldots, 0)
\quad\hbox{with $\mu_1,\ldots, \mu_{i-1} \in \ZZ$ and $\mu_i \in \ZZ_{>0}$, \quad and let} \\
\nu &= (\mu_1, \ldots, \mu_{i-1}, \mu_i, 0, \ldots, 0)
\quad\hbox{so that}\quad
\mu = s_is_{i+1}\cdots s_n \cdots s_{i+1}s_i\nu.
\end{align*}
Let $y\in W_{\mathrm{fin}}$ with $0 < y(i) < y(i+1) < y(i+2) < \cdots < y(n)$ so that $y$ is minimal length in the coset $yW_{[i+1,-(i+1)]}$
and $y(i)<y(i+1)$.
Let $A^{(i)}_{\pm\ell}= A^{(i)}_{\pm\ell}(\beta_{-i}, \ldots, \beta_{-n}, \beta_n, \ldots \beta_{i+1})$ as defined in~\eqref{Adef1A}-\eqref{Adef1C}
where $(\beta_{-i}, \ldots, \beta_{-n}, \beta_n, \ldots \beta_{i+1})$ is the coroot sequence given in~\eqref{0endcrtseq}.
Then
\begin{align*}
\widehat{E}^y_\mu
&= 
\ev_\nu(A^{(i)}_{i})\widehat{E}_\nu^{y}
+
\sum_{\ell=i+1}^n \ev_\nu(A^{(i)}_{\ell})\widehat{E}_\nu^{ys_{\ell-1}\cdots s_i}
+
\sum_{\ell=i}^n \ev_\nu(A^{(i)}_{-\ell})\widehat{E}_\nu^{ys_\ell\cdots s_n\cdots s_i}.
\end{align*}
\end{prop*}

\begin{proof}
The proof is by descending induction on $i$.  The base case $i=n$ is covered in Proposition~\ref{alcwksteps}(c) with $y(n)>0$. Thus,
$$\widehat{E}^y_\mu 
= \widehat{E}^{ys_n}_\nu + \ev_\nu(F^+_{\alpha_n})\widehat{E}^y_\nu
= \widehat{E}^{ys_n}_\nu + \ev_\nu(F^+_{\varepsilon_n})\widehat{E}^y_\nu,
$$
so that $A^{(n)}_{-n} = 1$ and $A^{(n)}_n = F_{\varepsilon_n}^+$.

The induction step is as follows.  Let $v = s_{-(i+1),i+1}$.
Then
\begin{align*}
\widehat{E}^y_{s_i v s_i\nu}
&= \widehat{E}^{ys_i}_{vs_i\nu} + \ev_{vs_i\nu}(F^+_{\alpha_i})\widehat{E}^y_{vs_i\nu}
= \widehat{E}^{ys_i}_{vs_i\nu} + \ev_{vs_i\nu}(F^+_{\varepsilon_i-\varepsilon_{i+1}})\widehat{E}^y_{vs_i\nu}
\\
&= \widehat{E}^{ys_i}_{vs_i\nu} + \ev_\nu(F^+_{\varepsilon_i+\varepsilon_{i+1}})\widehat{E}^y_{vs_i\nu}
\\
&= \Big(\sum_{\ell=i+1}^n \ev_{s_i\nu}(A^{(i+1)}_{-\ell})\widehat{E}^{ys_is_\ell\cdots s_n\cdots s_{i+1}}_{s_i\nu}
+\sum_{\ell=n}^{i+1} \ev_{s_i\nu}(A^{(i+1)}_\ell) \widehat{E}^{ys_i s_\ell\cdots s_{i+1}}_{s_i\nu}  \Big)
\\
&\quad
+ \ev_{\nu}(F^+_{\varepsilon_i+\varepsilon_{i+1}}) 
\Big(\sum_{\ell=i+1}^n \ev_{s_i\nu}(A^{(i+1)}_{-\ell})\widehat{E}^{ys_\ell\cdots s_n\cdots s_{i+1}}_{s_i\nu}
+\sum_{\ell=n}^{i+1} \ev_{s_i\nu}(A^{(i+1)}_\ell) \widehat{E}^{y s_\ell\cdots s_{i+1}}_{s_i\nu}  \Big)
\\
&= 
{\color{blue}\sum_{\ell=i+1}^n \ev_{\nu}(s_iA^{(i+1)}_{-\ell})\widehat{E}^{ys_is_\ell\cdots s_n\cdots s_{i+1}}_{s_i\nu}}
+
{\color{red} \sum_{\ell=n}^{i+1} \ev_{\nu}(s_iA^{(i+1)}_\ell) \widehat{E}^{ys_i s_\ell\cdots s_{i+1}}_{s_i\nu}} 
\\
&\quad
+ \ev_{\nu}(F^+_{\varepsilon_i+\varepsilon_{i+1}}) 
\left({\color{brown} \sum_{\ell=i+1}^n \ev_{\nu}(s_kA^{(i+1)}_{-\ell})\widehat{E}^{ys_\ell\cdots s_n\cdots s_{i+1}}_{s_i\nu} }
+ {\color{purple} \sum_{\ell=n}^{i+1} \ev_{\nu}(s_iA^{(i+1)}_\ell) \widehat{E}^{y s_\ell\cdots s_{i+1}}_{s_i\nu} } \right),
\end{align*}
where the color-coding for the terms is introduced as a visual aid to facilitate tracking of the terms in the remainder of the computation.

By applying induction to each term, 
\begin{align*}
\widehat{E}^y_{s_ivs_i\nu}
&= 
{\color{blue} 
\left(\widehat{E}^{ys_i\cdots s_n\cdots s_i}_{\nu}
+\ev_\nu\left(F_{\varepsilon_i-\varepsilon_{i+1}}^+\right) 
\widehat{E}^{ys_i \cdots s_n\cdots s_{i+1}}_{\nu}\right) \ev_\nu\left(s_iA^{(i+1)}_{-(i+1)}\right)  } \\
&\quad
{\color{blue} 
+ \sum_{\ell=i+2}^n 
\left(\widehat{E}^{ys_is_\ell\cdots s_n\cdots s_i}_{\nu}
+\ev_\nu\left(F_{\varepsilon_i-\varepsilon_{i+1}}^+ \right)\widehat{E}^{ys_is_\ell\cdots s_n\cdots s_{i+1}}_{\nu}\right)
\ev_\nu\left(s_iA^{(i+1)}_{-\ell}\right)
}\\
&\quad
{\color{red} 
+\sum_{\ell=i+2}^n 
\left(\widehat{E}^{ys_is_{\ell-1}\cdots s_i}_{\nu}
+\ev_\nu\left(F_{\varepsilon_i-\varepsilon_{i+1}}^+\right) \widehat{E}^{ys_is_{\ell-1}\cdots s_{i+1}}_{\nu}\right)
\ev_\nu \left(s_iA^{(i+1)}_{\ell}\right) }
\\
&\quad
{\color{red} 
+ \left(\widehat{E}^y_{\nu}
+\ev_\nu\left(F_{\varepsilon_i-\varepsilon_{i+1}}^- \right) \widehat{E}^{ys_i}_{\nu}\right)
\ev_\nu \left(s_iA^{(i+1)}_{i+1}\right)
}\\
&\quad
{\color{brown} 
+\sum_{\ell=i+1}^n 
\left(\widehat{E}^{ys_\ell\cdots s_n\cdots s_i}_{\nu}
+\ev_\nu\left(F_{\varepsilon_i-\varepsilon_{i+1}}^+\right) \widehat{E}^{ys_\ell\cdots s_n\cdots s_{i+1}}_{\nu}\right)
\ev_\nu\left(F_{\varepsilon_i+\varepsilon_{i+1}}^+ s_i A^{(i+1)}_{-\ell}\right)
}\\
&\quad
{\color{purple} 
+\sum_{\ell=i+2}^n 
\left(\widehat{E}^{ys_{\ell-1}\cdots s_i}_{\nu}
+\ev_\nu\left(F_{\varepsilon_i-\varepsilon_{i+1}}^+\right) \widehat{E}^{ys_{\ell-1}\cdots s_{i+1}}_{\nu} \right)
\ev_\nu\left(F_{\varepsilon_i+\varepsilon_{i+1}}^+ s_iA^{(i+1)}_{\ell}\right)
 } \\
 &\quad
{\color{purple} 
+\left(\widehat{E}^{ys_i}_{\nu}
+\ev_\nu\left(F_{\varepsilon_i-\varepsilon_{i+1}}^+ \right)\widehat{E}^y_{\nu}\right)
\ev_\nu\left(F_{\varepsilon_i+\varepsilon_{i+1}}^+ s_iA^{(i+1)}_{i+1}\right).
}
\end{align*}
Since
\begin{align*}
ys_is_\ell\cdots s_n\cdots s_i &= ys_\ell\cdots s_n\cdots s_{i+2}s_i s_{i+1}s_i = ys_\ell\cdots s_n\cdots s_i s_{i+1} \text{, \quad and }\\
ys_i s_{\ell-1}\cdots s_i &= ys_{\ell-1}\cdots s_{i+2}s_i s_{i+1}s_i = ys_i s_{\ell-1}\cdots s_i s_{i+1},    
\end{align*} 
then
\begin{align*}
\widehat{E}^y_{s_i vs_i\nu}
&= 
{\color{blue} 
\left(\widehat{E}^{ys_i\cdots s_n\cdots s_i }_{\nu}
+\ev_\nu\left(F_{\varepsilon_i-\varepsilon_{i+1}}^+\right) 
\widehat{E}^{ys_i \cdots s_n\cdots s_{i+1}}_{\nu}\right) \ev_\nu\left(s_iA^{(i+1)}_{-(i+1)}\right)  } \\
&\quad
{\color{blue} 
+ \sum_{\ell=i+2}^n 
\left(\widehat{E}^{ys_\ell\cdots s_n\cdots s_i s_{i+1}}_{\nu}
+\ev_\nu\left(F_{\varepsilon_i-\varepsilon_{i+1}}^+ \right)
\widehat{E}^{ys_i s_\ell\cdots s_n\cdots s_{i+1}}_{\nu}\right) \ev_\nu\left(s_i A^{(i+1)}_{-\ell}\right)
}\\
&\quad
{\color{red} 
+\sum_{\ell = i+2}^n 
\left(\widehat{E}^{ys_{\ell-1}\cdots s_i s_{i+1}}_{\nu}
+\ev_\nu\left(F_{\varepsilon_i-\varepsilon_{i+1}}^+\right) \widehat{E}^{ys_i s_{\ell-1}\cdots s_{i+1}}_{\nu}\right)
\ev_\nu \left(s_i A^{(i+1)}_{\ell}\right) }
\\
&\quad
{\color{red} 
+ \left(\widehat{E}^y_{\nu}
+\ev_\nu\left(F_{\varepsilon_i-\varepsilon_{i+1}}^- \right) \widehat{E}^{ys_i}_{\nu}\right)
\ev_\nu \left(s_i A^{(i+1)}_{i+1}\right)
}\\
&\quad
{\color{brown} 
+\sum_{\ell = i+1}^n 
\left(\widehat{E}^{ys_\ell\cdots s_n\cdots s_i }_{\nu}
+\ev_\nu\left(F_{\varepsilon_i-\varepsilon_{i+1}}^+\right) \widehat{E}^{ys_\ell\cdots s_n\cdots s_{i+1}}_{\nu}\right)
\ev_\nu\left(F_{\varepsilon_i+\varepsilon_{i+1}}^+ s_iA^{(i+1)}_{-\ell}\right)
}\\
&\quad
{\color{purple} 
+\sum_{\ell=i+2}^n 
\left(\widehat{E}^{ys_{\ell-1}\cdots s_i}_{\nu}
+\ev_\nu\left(F_{\varepsilon_i-\varepsilon_{i+1}}^+\right) \widehat{E}^{ys_{\ell-1}\cdots s_{i+1}}_{\nu} \right)
\ev_\nu\left(F_{\varepsilon_i+\varepsilon_{i+1}}^+ s_iA^{(i+1)}_{\ell}\right)
 } \\
 &\quad
{\color{purple} 
+\left(\widehat{E}^{ys_i}_{\nu}
+\ev_\nu\left(F_{\varepsilon_i-\varepsilon_{i+1}}^+ \right)\widehat{E}^{y}_{\nu}\right)
\ev_\nu\left(F_{\varepsilon_i+\varepsilon_{i+1}}^+ s_iA^{(i+1)}_{i+1}\right).
}
\end{align*}

Next, apply~\eqref{stabfactor} so that we are left with terms of the form $\widehat{E}^w_\nu$ where $w=ys_\ell \ldots s_i$ 
with $\ell\in \{i, \ldots, n\}$
or $w=ys_\ell \ldots s_n \ldots s_i$ with $\ell\in \{i+1, \ldots, n-1\}$:
\begin{align}
\widehat{E}^y_{s_ivs_i\nu}
&= 
{\color{blue} 
\left(\widehat{E}^{ys_i\cdots s_n\cdots s_i}_{\nu}
+\ev_\nu\left(F_{\varepsilon_i-\varepsilon_{i+1}}^+\right) t_n^{\frac12}t^{\frac{2(n-1-i)}{2}}
\widehat{E}^{ys_i}_{\nu}\right) \ev_\nu\left(s_iA^{(i+1)}_{-(i+1)}\right)  } 
\nonumber \\
&\quad
{\color{blue} 
+ \sum_{\ell=i+2}^n 
\left(t^{\frac12}\widehat{E}^{ys_\ell\cdots s_n\cdots s_i}_{\nu}
+\ev_\nu\left(F_{\varepsilon_i-\varepsilon_{i+1}}^+ \right)
t_n^{\frac12} t^{\frac{n-1-i+n-1-(\ell-1)}{2}}  \widehat{E}^{ys_i}_{\nu}
\right)\ev_\nu\left(s_iA^{(i+1)}_{-\ell}\right)
}
\nonumber \\
&\quad
{\color{red} 
+\sum_{\ell=i+2}^n 
\left(t^{\frac12}\widehat{E}^{ys_{\ell-1}\cdots s_i}_{\nu}
+\ev_\nu\left(F_{\varepsilon_i-\varepsilon_{i+1}}^+\right) 
t^{\frac{\ell-1-i}{2}}  \widehat{E}^{ys_i}_{\nu}\right)
\ev_\nu \left(s_iA^{(i+1)}_{\ell}\right) }
\nonumber \\
&\quad
{\color{red} 
+ \left(\widehat{E}^y_{\nu}
+\ev_\nu\left(F_{\varepsilon_i-\varepsilon_{i+1}}^- \right) \widehat{E}^{ys_i}_{\nu}\right)
\ev_\nu \left(s_iA^{(i+1)}_{i+1}\right)
}
\nonumber \\
&\quad
{\color{brown} 
+\sum_{\ell=i+1}^n 
\left(\widehat{E}^{ys_\ell\cdots s_n\cdots s_i}_{\nu}
+\ev_\nu\left(F_{\varepsilon_i-\varepsilon_{i+1}}^+\right) t_n^{\frac12} 
t^{\frac{n-1-i+n-1-(\ell-1)}{2}}  \widehat{E}^{y}_{\nu}
\right)
\ev_\nu\left(F_{\varepsilon_i+\varepsilon_{i+1}}^+ s_iA^{(i+1)}_{-\ell}\right)
}
\nonumber \\
&\quad
{\color{purple} 
+\sum_{\ell=i+2}^n 
\left(\widehat{E}^{ys_{\ell-1}\cdots s_i}_{\nu}
+\ev_\nu\left(F_{\varepsilon_i-\varepsilon_{i+1}}^+\right) 
t^{\frac{\ell-1-i}{2}}  \widehat{E}^{y}_{\nu}
\right)
\ev_\nu\left(F_{\varepsilon_i+\varepsilon_{i+1}}^+ s_iA^{(i+1)}_{\ell}\right)
 } 
 \nonumber \\
 &\quad
{\color{purple} 
+\left(\widehat{E}^{ys_i}_{\nu}
+\ev_\nu\left(F_{\varepsilon_i-\varepsilon_{i+1}}^+ \right)\widehat{E}^{y}_{\nu}\right)
\ev_\nu\left(F_{\varepsilon_i+\varepsilon_{i+1}}^+ s_iA^{(i+1)}_{i+1}\right).
}
\label{Prop511pf}
\end{align}

Writing
\begin{align*}
\widehat{E}^y_{s_ivs_i\nu}
&= \sum_{\ell=i}^n \ev_\nu(B^{(i)}_{-\ell})\widehat{E}_\nu^{ys_\ell\cdots s_n\cdots s_i}
+ \sum_{\ell=i}^n \ev_\nu(B^{(i)}_{\ell})\widehat{E}_\nu^{ys_{\ell-1}\cdots s_i},
\end{align*}
we need to check that $\ev_\nu(B^{(i)}_{\pm\ell}) = \ev_\nu(A^{(i)}_{\pm\ell})$. 
This follows for most of the coefficients by~\eqref{Adef1A} and~\eqref{Adef1C}. For $\ell \in \{i+1,\ldots, n\}$,
\begin{align*}
\ev_\nu\left(B^{(i)}_{-\ell}\right)
&= t^{\frac12}\ev_\nu\left(s_iA^{(i+1)}_{-\ell}\right) + \ev_\nu\left(F_{\varepsilon_i+\varepsilon_{i+1}}^+ s_iA^{(i+1)}_{-\ell}\right) 
= \ev_\nu\left(C_{\varepsilon_i+\varepsilon_{i+1}} s_iA^{(i+1)}_{-\ell}\right) 
= \ev_\nu\left(A^{(i)}_{-\ell}\right)
\end{align*}
and for $\ell \in \{i+2,\ldots, n\}$
\begin{align*}
\ev_\nu\left(B^{(i)}_{\ell}\right)
&= t^{\frac12}\ev_\nu \left(s_iA^{(i+1)}_{\ell}\right) + \ev_\nu\left(F_{\varepsilon_i+\varepsilon_{i+1}}^+ s_iA^{(i+1)}_{\ell}\right) 
= \ev_\nu\left(C_{\varepsilon_i+\varepsilon_{i+1}} s_iA^{(i+1)}_{\ell}\right) 
= \ev_\nu\left(A^{(i)}_{\ell}\right).
\end{align*}
For $\ell\in\{i, i+1\}$,
\begin{align*}
\ev_\nu\left(B^{(i)}_{-i}\right) 
&= 1 = \ev_\nu\left(A^{(i)}_{-i}\right) \qquad\hbox{and}\qquad
\ev_\nu\left(B^{(i)}_{-(i+1)} \right) 
=\ev_\nu\left(F_{\varepsilon_i+\varepsilon_{i+1}}^+\right)
=\ev_\nu\left(A^{(i)}_{-(i+1)} \right).
\end{align*}

It remains to show that
$\ev_\nu(B^{(i)}_{i})=\ev_\nu(A^{(i)}_{i})$ and $\ev_\nu(B^{(i)}_{i+1}) = \ev_\nu(A^{(i)}_{i+1})$.
By~\eqref{Prop511pf},
\begin{align*}
\ev_\nu\Big(B^{(i)}_{i+1} \Big)
&= \ev_\nu\Big( F_{\varepsilon_i-\varepsilon_{i+1}}^+ 
\Big(\sum_{\ell = i+1}^n t^{n-\frac{\ell+i+1}{2}} t_n^{\frac12} (s_iA^{(i+1)}_{-\ell})
+\sum_{\ell = n}^{i+2} t^{\frac{\ell-i-1}{2}}  (s_iA^{(i+1)}_{\ell})\Big)\Big)
\\
&\qquad\quad
+ \ev_\nu\Big(\left(F_{\varepsilon_i-\varepsilon_{i+1}}^- +F_{\varepsilon_i+\varepsilon_{i+1}}^+\right)(s_iA^{(i+1)}_{i+1})\Big)
\end{align*}
and 
\begin{align*}
\ev_\nu\Big(B^{(i)}_i \Big)
= \ev_\nu\Big((s_iA^{(i+1)}_{i+1})
+F_{\varepsilon_i-\varepsilon_{i+1}}^+ F_{\varepsilon_i+\varepsilon_{i+1}}^+
\Big(\sum_{\ell=i+1}^n t^{n-\frac{i+\ell+1}{2}}t_n^{\frac12} (s_iA^{(i+1)}_{-\ell}) 
+t^{\frac{\ell-i-1}{2}}(s_iA^{(i+1)}_\ell)\Big)\Big).
\end{align*}

By~\eqref{Fdefn}, $F_{\varepsilon_i-\varepsilon_{i+1}}^+ - F_{\varepsilon_i-\varepsilon_{i+1}}^-
=t^{-\frac12}-t^{\frac12} = F_{\varepsilon_i+\varepsilon_{i+1}}^+ - F_{\varepsilon_i+\varepsilon_{i+1}}^-$,
and so 
$F_{\varepsilon_i-\varepsilon_{i+1}}^- + F_{\varepsilon_i+\varepsilon_{i+1}}^+ 
= F_{\varepsilon_i-\varepsilon_{i+1}}^+ + F_{\varepsilon_i+\varepsilon_{i+1}}^-$.
Using
$F_{\varepsilon_i-\varepsilon_{i+1}}^- + F_{\varepsilon_i+\varepsilon_{i+1}}^+ 
= F_{\varepsilon_i-\varepsilon_{i+1}}^+ + F_{\varepsilon_i+\varepsilon_{i+1}}^-$ 
and then Lemma~\ref{endcombgen}(b) gives
\begin{align*}
\ev_\nu\Big(B^{(i)}_{i+1}\Big)
&= \ev_\nu\Big(F_{\varepsilon_i-\varepsilon_{i+1}}^+ \Big(\sum_{\ell = i+1}^n t^{n-\frac{\ell+i-1}{2}} t_n^{\frac12} (s_iA^{(i+1)}_{-\ell})
+\sum_{\ell = n}^{i+1} t^{\frac{\ell-i-1}{2}}  (s_iA^{(i+1)}_{\ell})\Big)
+F_{\varepsilon_i+\varepsilon_{i+1}}^-(s_iA^{(i+1)}_{i+1}) \Big)\\
&= 
\ev_\nu\Big(\Big(\prod_{m=i+2}^n C_{\varepsilon_i+\varepsilon_m}\Big) C_{\varepsilon_i}
\Big(\prod_{m=i+2}^n C_{\varepsilon_i-\varepsilon_m}\Big)F_{\varepsilon_i-\varepsilon_{i+1}}^+ \Big)
\\
&\quad
+\ev_\nu\Big(F_{\varepsilon_i+\varepsilon_{i+1}}^-
\Big(
F_{\varepsilon_i}^+ 
+ \sum_{\ell=i+2}^n F^+_{\varepsilon_i+\varepsilon_\ell}
\Big(\prod_{m=\ell+1}^n C_{\varepsilon_i+\varepsilon_m}\Big) C_{\varepsilon_i}
\Big(\prod_{m=\ell+1}^n C_{\varepsilon_i-\varepsilon_m}\Big)
F^+_{\varepsilon_i-\varepsilon_\ell}
\Big) \Big)
\\
&= \ev_\nu\Big(A^{(i)}_{i+1}\Big),
\end{align*}
where the last equality is from~\eqref{Adef1B}.

Using Lemma~\ref{endcombgen}(b) gives
\begin{align*}
\ev_\nu\Big(B^{(i)}_i\Big)
&=
\ev_\nu\Big((s_iA^{(i+1)}_{i+1})
+F_{\varepsilon_i-\varepsilon_{i+1}}^+ F_{\varepsilon_i+\varepsilon_{i+1}}^+
\Big(\prod_{m=i+2}^n C_{\varepsilon_i+\varepsilon_m}\Big) C_{\varepsilon_i}
\Big(\prod_{m=i+2}^n C_{\varepsilon_i-\varepsilon_m}\Big) \Big)
\\
&=
\ev_\nu\Big(\Big(
F_{\varepsilon_i}^+ 
+ \sum_{\ell=i+2}^n F^+_{\varepsilon_i+\varepsilon_\ell}
\Big(\prod_{m=\ell+1}^n C_{\varepsilon_i+\varepsilon_m}\Big) C_{\varepsilon_i}
\Big(\prod_{m=\ell+1}^n C_{\varepsilon_i-\varepsilon_m}\Big)
F^+_{\varepsilon_i-\varepsilon_\ell}
\Big) \Big)
\\
&\qquad
+ \ev_\nu\Big(F_{\varepsilon_i-\varepsilon_{i+1}}^+ 
\Big(\prod_{m=i+2}^n C_{\varepsilon_i+\varepsilon_m}\Big) C_{\varepsilon_i}
\Big(\prod_{m=i+2}^n C_{\varepsilon_i-\varepsilon_m}\Big) F_{\varepsilon_i+\varepsilon_{i+1}}^+ \Big)
 \\
&=
\ev_\nu\Big( F_{\varepsilon_i}^+ 
+ \sum_{\ell=i+1}^n F^+_{\varepsilon_i+\varepsilon_\ell}
\Big(\prod_{m=\ell+1}^n C_{\varepsilon_i+\varepsilon_m}\Big) C_{\varepsilon_i}
\Big(\prod_{m=\ell+1}^n C_{\varepsilon_i-\varepsilon_m}\Big)
F^+_{\varepsilon_i-\varepsilon_\ell} \Big) 
\\
&= \ev_\nu\Big(A^{(i)}_i \Big),
\end{align*}
where the last equality is from~\eqref{Adef1B}.
\end{proof}

\subsection{Around-the-end compression weights when $y(j)\prec y(i)\prec y(j+1)$}\label{subsect: around the end general}

Proposition~\ref{proposition:GeneralCaseAroundTheEnd}
generalizes Proposition~\ref{0end} to arbitrary $y$ and expresses the general coefficients 
$A^{(i,j,-i)}_{\pm\ell}$ for around-the-end compression
in terms of the $A^{(i)}_{\pm\ell}$ defined in~\eqref{Adef1A},~\eqref{Adef1B},~\eqref{Adef1C}.

Suppose that $i\in \{1, \ldots, n\}$ and let $W_{[i+1,-(i+1)]}$ be the subgroup of $W_{\mathrm{fin}}$ generated
by $s_{i+1}, \ldots, s_n$.  An element 
$$\hbox{$y \in W_{\mathrm{fin}}$ is minimal length in $yW_{[i+1,-(i+1)]}$
\quad if and only if \quad}
0 < y(i+1) < \cdots < y(n).$$ 

\begin{prop} 
\label{proposition:GeneralCaseAroundTheEnd}
Let $i\in \{1, \ldots, n\}$,
\begin{align*}
\mu &= (\mu_1, \ldots, \mu_{i-1}, -\mu_i, 0, \ldots, 0)
\quad\hbox{with $\mu_1,\ldots, \mu_{i-1} \in \ZZ$ and $\mu_i \in \ZZ_{>0}$, \quad and let} \\
\nu &= (\mu_1, \ldots, \mu_{i-1}, \mu_i, 0, \ldots, 0)
\quad\hbox{so that}\quad \mu = s_i\cdots s_n \cdots s_i\nu.
\end{align*}
Let $y\in W_{\mathrm{fin}}$ with $0 < y(i+1) < y(i+2) < \cdots < y(n)$.
Let
$$j = \begin{cases}
i, &\hbox{if $0<{\color{blue} y(i)} < y(i+1)<\cdots < y(n)$, } \\
m, &\hbox{if $0<y(i+1)<\cdots < y(m)  < {\color{blue} y(i) } < y(m+1) < \cdots < y(n)$,} \\
n, &\hbox{if $0<y(i+1)<\cdots  <y(n)< {\color{blue} y(i)}$,} \\
-n, &\hbox{if $0<y(i+1)<\cdots  <y(n)< {\color{blue} -y(i)}$,} \\
-m, &\hbox{if $0<y(i+1)<\cdots < y(m) < {\color{blue} -y(i)} < y(m+1) < \cdots < y(n)$,} \\
-i, &\hbox{if $0 < {\color{blue} -y(i)} < y(i+1)<\cdots <y(n)$.}
\end{cases}
$$
Let $A^{(i)}_{\pm\ell}= A^{(i)}_{\pm\ell}(\beta_{-i}, \ldots, \beta_{-n}, \beta_n, \ldots \beta_{i+1})$ as defined in~\eqref{AgencA}-\eqref{AgencD}
where $(\beta_{-i}, \ldots, \beta_{-n}, \beta_n, \ldots \beta_{i+1})$ is the coroot sequence given in~\eqref{0endcrtseq}.
Then 
$$\widehat{E}^y_\mu
= \sum_{\ell = i}^n 
\ev_\nu(A^{(i,j,-i)}_{-\ell}) \widehat{E}^{ys_\ell \cdots s_n\cdots s_i}_\nu
+\ev_\nu(A^{(i,j,-i)}_{\ell}) \widehat{E}^{ys_{\ell-1} \cdots s_i}_\nu,
$$
\end{prop}
\begin{proof}

Let $W_{[1,i]}$ be the subgroup of $W_{\mathrm{fin}}$ generated by $s_1, \ldots, s_{i-1}$ and let
$u\in W_{[1,i]}$ be such that
$$y=uv,\quad\hbox{where $u\in W_{[1,i]}$ and}
\quad
v = \begin{cases}
1, &\hbox{if $j=i$,} \\
s_{m-1}s_{m-2}\cdots s_i, &\hbox{if $j=m$,} \\
s_{n-1}s_{n-2}\cdots s_i, &\hbox{if $j=n$,} \\
s_n s_{n-1}\cdots s_i, &\hbox{if $j=-n$,} \\
s_m\cdots s_{n-1}s_n s_{n-1}\cdots s_i, &\hbox{if $j=-m$,} \\
s_i\cdots s_{n-1}s_n s_{n-1}\cdots s_i, &\hbox{if $j=-i$.}
\end{cases}
$$
Assume that the statement has been proved for $v$. That is,
\begin{align}
\widehat{E}^v_\mu 
&=\sum_{\ell = i}^n \ev_\nu(A^{(i,m,-i)}_{-\ell}) \widehat{E}^{vs_\ell \cdots s_n\cdots s_i}_\nu
+\ev_\nu(A^{(i,m,-i)}_{\ell}) \widehat{E}^{vs_{\ell-1} \cdots s_i}_\nu.
\label{vcase}
\end{align}
Then
\begin{align*}
\widehat{E}^{y}_\mu 
&= \widehat{E}^{uv}_\mu 
= T_u \widehat{E}^{v}_\mu 
=\sum_{\ell = i}^n \ev_\nu(A^{(i,m,-i)}_{-\ell}) T_u\widehat{E}^{vs_\ell \cdots s_n\cdots s_i}_\nu
+\ev_\nu(A^{(i,m,-i)}_{\ell}) T_u\widehat{E}^{vs_{\ell-1} \cdots s_i}_\nu
\\
&=\sum_{\ell = i}^n \ev_\nu(A^{(i,m,-i)}_{-\ell}) \widehat{E}^{uvs_\ell \cdots s_n\cdots s_i}_\nu
+\ev_\nu(A^{(i,m,-i)}_{\ell}) \widehat{E}^{uvs_{\ell-1} \cdots s_i}_\nu
\\
&=\sum_{\ell = i}^n \ev_\nu(A^{(i,m,-i)}_{-\ell}) \widehat{E}^{ys_\ell \cdots s_n\cdots s_i}_\nu
+\ev_\nu(A^{(i,m,-i)}_{\ell}) \widehat{E}^{ys_{\ell-1} \cdots s_i}_\nu.
\end{align*}
So it suffices to prove~\eqref{vcase} for each of the possible values of $j$.
The proof uses Lemmas~\ref{Arec} and~\ref{Tred}, which are proved separately below.

\noindent
\emph{Case 1:  $j=i$.}
This is the case covered by Proposition~\ref{0end}.

\noindent
\emph{Case 2:  $j=m$ with $m\in \{i+1, \ldots, n\}$.} 
Set $v = s_{m-1}s_{m-2}\cdots s_i$ and $w= s_{m-1}v = s_{m-2}\cdots s_i$.  
Using $\widehat{E}^v_\mu = T_{m-1} \widehat{E}^w_\mu$ gives
\begin{align*}
\widehat{E}^v_\mu &= \widehat{E}^{s_{m-1}\cdots s_i}_\mu
= T_{m-1}\widehat{E}^{s_{m-2}\cdots s_i}_\mu
\\
&= T_{m-1}\Big(\sum_{\ell = i}^n 
\ev_\nu(A^{(i,m-1,-i)}_{-\ell}) \widehat{E}^{ws_\ell \cdots s_n\cdots s_i}_\nu
+\ev_\nu(A^{(i,m-1,-i)}_{\ell}) \widehat{E}^{ws_{\ell-1} \cdots s_i}_\nu\Big)
\\
&= \sum_{\ell = i}^n 
\ev_\nu(A^{(i,m-1,-i)}_{-\ell}) T_{m-1}\widehat{E}^{ws_\ell \cdots s_n\cdots s_i}_\nu
+\sum_{\ell = i}^n \ev_\nu(A^{(i,m-1,-i)}_{\ell}) T_{m-1}\widehat{E}^{ws_{\ell-1} \cdots s_i}_\nu.
\end{align*}
Plugging in the expressions for $T_{m-1}\widehat{E}^{ws_\ell \cdots s_n\cdots s_i}_\nu$ and $T_{m-1}\widehat{E}^{ws_{\ell-1} \cdots s_i}_\nu$
which are given in Case 2 of Lemma~\ref{Tred} gives
\begin{align*}
\widehat{E}^v_\mu 
&= 
\ev_\nu(A^{(i,m-1,-i)}_{-i}) \big(
\widehat{E}^{vs_i\cdots s_n\cdots s_i}_\nu
+(t^{\frac12}-t^{-\frac12})t^{-\frac12(m-i-1)} \widehat{E}_\nu^{vs_m\cdots s_n \cdots s_i}\big) \\
&\qquad+
\sum_{\ell=i+1}^{n} 
\ev_\nu(A^{(i,m-1,-i)}_{-\ell})  \widehat{E}^{vs_\ell \cdots s_n\cdots s_i}_\nu
+
\sum_{\ell=i+1}^{m-1} \ev_\nu(A^{(i,m-1,-i)}_{\ell})  \widehat{E}^{vs_{\ell-1} \cdots  s_i}_\nu \\
&\qquad+ \ev_\nu(A^{(i,m-1,-i)}_{m})
\big(\widehat{E}^{vs_{m-1} \cdots  s_i}+(t^{\frac12}-t^{-\frac12})t^{\frac12(m-i-1)}\widehat{E}^v_\nu\big) \\
&\qquad+
\sum_{\ell=m+1}^{n} \ev_\nu(A^{(i,m-1,-i)}_{\ell})  \widehat{E}^{vs_{\ell-1} \cdots  s_i}_\nu \\
&= \sum_{\ell=i}^{m-1} \ev_\nu(A_{-\ell}^{(i,m-1,-i)}) \widehat{E}^{vs_\ell \cdots s_n\cdots s_i}_\nu 
+ \sum_{\ell=m+1}^{n} \ev_\nu(A_{-\ell}^{(i,m-1,-i)}) \widehat{E}^{vs_\ell \cdots s_n\cdots s_i}_\nu \\
&\qquad+ \ev_\nu(A_{-m}^{(i,m-1,-i)} + (t^{\frac12}-t^{-\frac12})t^{-\frac12(m-1-i)}A_{-i}^{(i,m-1,-i)}) 
\widehat{E}^{vs_m \cdots s_n\cdots s_i}_\nu \\
&\qquad+ \sum_{\ell=i+1}^{n} \ev_\nu(A_{\ell}^{(i,m-1,-i)}) \widehat{E}^{vs_{\ell-1} \cdots  s_i}_\nu \\
&\qquad+ \ev_\nu\big(A_{i}^{(i,m-1,-i)} + (t^{\frac12}-t^{-\frac12})t^{\frac12(m-1-i)}A_m^{(i,m-1,-i)}\big) \widehat{E}^v_\nu.
\end{align*}
Then, by Lemma~\ref{Arec} and the relations in~\eqref{Agenr2},
\begin{align*}
\widehat{E}^v_\mu 
&=\sum_{\ell = i}^n \ev_\nu(A^{(i,m,-i)}_{-\ell}) \widehat{E}^{vs_\ell \cdots s_n\cdots s_i}_\nu
+\ev_\nu(A^{(i,m,-i)}_{\ell}) \widehat{E}^{vs_{\ell-1} \cdots s_i}_\nu.
\end{align*}

The remaining cases are similar to Case 2. 

\end{proof}

\begin{lemma} \label{Arec}
The $A^{(i,m,-i)}_{\pm \ell}$ are determined by~\eqref{AgencA} and the following recursions:

\noindent 
Case 2: $m\in \{i+1, \ldots, n\}$.
\begin{align}
A^{(i,m,-i)}_{-\ell} &= A^{(i,m-1,-i)}_{-\ell}, \qquad \text{if $-\ell\in \{-i, \ldots, -n\}$ and $-\ell\ne -m$,} \nonumber
\\
A^{(i,m,-i)}_{-m} &= A^{(i,m-1,-i)}_{-m} + (t^{\frac12}-t^{-\frac12})t^{-\frac12(m-1-i)} A^{(i,m-1,-i)}_{-i},\qquad \hbox{(the case $-\ell = -m$),}
\nonumber
\\
A^{(i,m,-i)}_{\ell} &= A^{(i,m-1,-i)}_{\ell}, \qquad\text{if $\ell\in \{i+1, \ldots, n\}$,}
\label{Agenr2}
\\
A^{(i,m,-i)}_i &= A^{(i,m-1,-i)}_i +  (t^{\frac12}-t^{-\frac12})t^{\frac12(m-1-k)} A^{(i,m-1,-i)}_m,  \qquad \text{(the case $\ell= i$).} \nonumber
\end{align}
Case 3: $-m=-n$.
\begin{align}
A^{(i,-n,-i)}_{-\ell} &= A^{(i,n,-i)}_{-\ell}, && \text{if $-\ell \in \{-i, \ldots, -n\}$,}\nonumber
\\
A^{(i,-n,-i)}_{\ell} &= A^{(i,n,-i)}_{\ell}, && \text{if $\ell\in \{i+1, \ldots, n\}$,}\nonumber
\\
A^{(i,-n,-i)}_i &= A^{(i,n,-i)}_i + (t_n^{\frac12}-t_n^{-\frac12}) A^{(i,n,-i)}_{-i},  && \text{(the case $\ell= i$).}
\label{Agenr3}
\end{align}

Case 4. $-m\in \{-i, \ldots, -(n-1)\}$.
\begin{align}
A^{(i,-m,-i)}_{-\ell} &= A^{(i, -(m+1),-i)}_{-\ell},\qquad\hbox{if $-\ell\in \{-i, \ldots, -n\}$.} \nonumber \\
A^{(i,-m,-i)}_\ell &= A^{(i, -(m+1),-i)}_\ell,\qquad\hbox{if $\ell\in \{i+1, \ldots, n\}$ and $\ell\neq m$,} \nonumber
\\
A^{(i,-m,-i)}_{m} 
&= A^{(i,-(m+1),-i)}_{m}+(t^{\frac12}-t^{-\frac12})t^{-\frac12(2n-2-i-m)}t_n^{-\frac12} A^{(i,-(m+1),-i)}_{-i}, \nonumber
\\
A^{(i,-m,-i)}_i
&= A^{(i,-(m+1),-i)}_i+(t^{\frac12}-t^{-\frac12})t^{\frac12(2n-2-m-i)}t_n^{\frac12}
A^{(i,-(m+1),-i)}_{-(m+1)}
\label{Agenr4}
\end{align}
\end{lemma}
\begin{proof}
\emph{Case 2: $m\in \{i+1, \ldots, n\}$.} 

\noindent
If $-\ell\in \{-i, \ldots, -m\}$ then using the first equality in~\eqref{Agenr2}
and then the second equality in~\eqref{Agenr2} and then the first equality again gives
\begin{align*}
A^{(i,m,-i)}_{-\ell} =  A^{(i,\ell,-i)}_{-\ell} 
&=A^{(i,\ell-1,-i)}_{-\ell} + (t^{\frac12}-t^{-\frac12})t^{-\frac12(\ell-1-i)} A^{(i,\ell-1,-i)}_{-i} \\
&=A^{(i,i,-i)}_{-\ell} + (t^{\frac12}-t^{-\frac12})t^{-\frac12(\ell-1-i)} A^{(i,i,-i)}_{-i},
\end{align*}
which is the first relation~\eqref{AgencB}.

\noindent
If $-\ell\in \{-(m+1), \ldots, -n\}$ then first relation in~\eqref{Agenr2} gives $A^{(i,m,-i)}_{-\ell} = A^{(i,i,-i)}_{-\ell}$, 
which is the second relation in~\eqref{AgencB}.

\noindent
If $\ell\in \{i+1, \ldots, n\}$ then third relation in~\eqref{Agenr2} gives $A^{(i,m,-i)}_\ell = A^{(i,i,-i)}_\ell$, which is the
third relation in~\eqref{AgencB}.

\noindent
To prove the last equation in~\eqref{AgencB} (where $\ell=i$) use induction on $m$.
The base case is $m=i+1$ for which
$A^{(i,i+1,-i)}_i = A^{(i,i,-i)}_i+ (t^{\frac12}-t^{-\frac12}) A^{(i,i,-i)}_{i+1},
$
and the induction step is
\begin{align*}
&A^{(i,m,-i)}_i = A^{(i,m-1,-i)}_i + (t^{\frac12}-t^{-\frac12})t^{\frac12(m-1-i)} A^{(i,m-1,-i)}_m \\ 
&= \Big(A^{(i,i,-i)}_i+ (t^{\frac12}-t^{-\frac12})\sum_{j=i+1}^{m-1} t^{\frac12(j-1-i)}\ev_\nu(A^{(i,i,-i)}_j)\Big)
+ (t^{\frac12}-t^{-\frac12})t^{\frac12(m-1-i)} A^{(i,m-1,-i)}_m \\
&\quad= \ev_\nu(A^{(i,i,-i)}_i+ (t^{\frac12}-t^{-\frac12})
\sum_{j=i+1}^m t^{\frac12(j-1-k)}\ev_\nu(A^{(i,i,-i)}_j).
\end{align*}

\smallskip\noindent
\emph{Case 3: $-m= -n$.} 
The first relation in~\eqref{Agenr3} and the first relation in~\eqref{AgencB} give that if $-\ell\in \{-i, \ldots, -n\}$ then
$$A^{(i, -n,-1)}_{-\ell} = A^{(i,n,-i)}_{-\ell} = 
A^{(i)}_{-\ell}+(t^{\frac12}-t^{-\frac12})t^{-\frac12(\ell-1-i)}A^{(i)}_{-i},$$
which is the first relation in~\eqref{AgencC}.  The second relation in~\eqref{Agenr3} and the third relation in~\eqref{AgencB} give
that if $\ell\in \{i+1, \ldots, n\}$ then
$$A^{(i,-n,-i)}_\ell = A^{(i,n,-i)}_\ell = A^{(i)}_\ell,
\quad\hbox{which is the second relation in~\eqref{AgencC}.}
$$
The third relation in~\eqref{Agenr3} and the first and last relations in~\eqref{AgencB} give
\begin{align*}
A^{(i,-n,-i)}_i 
&= A^{(i,n,-i)}_i + (t_n^{\frac12}-t_n^{-\frac12})A^{(i,n,-i)}_{-i}
\\
&= \Big(A^{(i)}_i +(t^{\frac12}-t^{-\frac12})\sum_{j=i+1}^n t^{\frac12(j-1-i)}A^{(i)}_j\Big)
+(t_n^{\frac12}-t_n^{-\frac12})\left[A^{(i)}_{-i}+(t^{\frac12}-t^{-\frac12})t^{\frac12}A^{(i)}_{-i}\right],
\end{align*}
which is the last relation in~\eqref{AgencC}.

\smallskip\noindent
\emph{Case 4: $-m\in \{-i, \ldots, -(n-1)\}$.}

\noindent
If $-\ell\in \{-i, \ldots, -n\}$ then sucessively using the first relation in~\eqref{Agenr4}
and then the first relation in~\eqref{AgencC} gives
\begin{align*}
A^{(i,-m,-i)}_{-\ell}
&= A^{(i,-n,-i)}_{\ell}
=A^{(i)}_{-\ell}+(t^{\frac12}-t^{-\frac12})t^{-\frac12(\ell-1-i)}A^{(i)}_{-i},
\end{align*}
which is the first relation in~\eqref{AgencD}.

\noindent
If $\ell\in \{m, \ldots,n\}$ then using the first relation in~\eqref{Agenr4} and then the third equality in 
\eqref{Agenr4} and then the first relation in~\eqref{Agenr4} again and then the first relation in~\eqref{AgencC} gives
\begin{align*}
A^{(i,-m,-i)}_\ell 
&= A^{(i, -\ell,-i)}_\ell \\
&=  A^{(i,-(\ell+1),-i)}_{\ell}
+ (t^{\frac12}-t^{-\frac12})t^{-\frac12(2n-2-i-\ell)}t_n^{-\frac12}
A^{(i,-(\ell+1),-i)}_{-i}  \\
&= A^{(i,-n,-i)}_{\ell} + (t^{\frac12}-t^{-\frac12})t^{-\frac12(2n-2-i-\ell)}t_n^{-\frac12}
A^{(i,-n,-i)}_{-i} \\
&= A^{(i,i,-i)}_{\ell} + (t^{\frac12}-t^{-\frac12})t^{-\frac12(2n-2-i-\ell)}t_n^{-\frac12} t A^{(i,i,-i)}_{-i}.
\end{align*}
This is the second relation in~\eqref{AgencD}. 

\noindent
If $\ell\in \{i+1,\ldots, m-1\}$ then successively using the second relation in~\eqref{Agenr4} 
and then using the second relation in~\eqref{AgencC} gives 
$$A^{(i,-m,-i)}_\ell = A^{(i,-n,-i)}_\ell = A^{(i)}_\ell,$$
which is the third relation in~\eqref{AgencD}.

\noindent
If $\ell=i$ then using the last equality in~\eqref{Agenr4} as the base case of 
a descending induction on $m$ gives
\begin{align*}
A^{(i,-m,-i)}_i
&= A^{(i,-(m+1),-i)}_i  
+(t^{\frac12}-t^{-\frac12})t^{\frac12(2n-2-m-i)}t_n^{\frac12} A^{(i,-(m+1),-i)}_{-(m+1)}  \\
&= \Big(A^{(i,-n,-i)}_i+(t^{\frac12}-t^{-\frac12})t_n^{\frac12} \sum_{j=m+2}^n t^{\frac12(2n-1-j-i)} A^{(i,-n,-i)}_{-j}\Big)
\\
&\qquad\qquad
+(t^{\frac12}-t^{-\frac12})t^{\frac12(2n-2-m-i)}t_n^{\frac12} A^{(i,-(m+1),-i)}_{-(m+1)}  \\
&= A^{(i,-n,-i)}_i+(t^{\frac12}-t^{-\frac12})t_n^{\frac12} \sum_{j=m+1}^n t^{\frac12(2n-1-j-i)} A^{(i,-n,-i)}_{-j}.
\end{align*}
Then the first and the last relations in~\eqref{AgencC} give
\begin{align*}
A^{(i,-m,-i)}_i 
&= A^{(i,-n,-i)}_i
+ (t^{\frac12} - t^{-\frac12})t_n^{\frac12} \sum_{j=m+1}^n t^{\frac12(2n-j-1-i)} A^{(i,-n,-i)}_{-j} 
\\
&= A^{(i)}_i +  (t^{\frac12}-t^{-\frac12})
\Big(\sum_{j=i+1}^n t^{\frac12(j-i)} A^{(i)}_j\Big) + t(t_n^{\frac12}-t_n^{-\frac12}) A^{(i)}_{-i} 
\\
&\qquad + (t^{\frac12} - t^{-\frac12})t_n^{\frac12} \sum_{j=m+1}^n t^{\frac12(2n-j-1-i)} 
\big( A^{(i)}_{-j} + (t^{\frac12}-t^{-\frac{1}{2}})t^{-\frac{1}{2}(j-1-i)}A^{(i)}_{-i}\big)
\\
&= A^{(i)}_i +  (t^{\frac12}-t^{-\frac12})
\Big(\sum_{j=i+1}^n t^{\frac12(j-i)} A^{(i)}_j\Big) 
+ (t^{\frac12} - t^{-\frac12})t_n^{\frac12} \Big(\sum_{j=m+1}^n t^{\frac12(2n-j-1-i)} A^{(i)}_{-j}\Big)
\\
&\qquad + \Big(
t(t_n^{\frac12}-t_n^{-\frac12}) + t^n t_n^{\frac12}(t^{\frac12}-t^{-\frac12})^2 \frac{(t^{-n}-t^{-m})}{(t^{-1}-1)}\Big) A^{(i)}_{-i} 
\\
&= A^{(i)}_i +  (t^{\frac12}-t^{-\frac12})
\Big(\sum_{j=i+1}^n t^{\frac12(j-i)} A^{(i)}_j\Big) 
+ (t^{\frac12} - t^{-\frac12})t_n^{\frac12} \Big(\sum_{j=m+1}^n t^{\frac12(2n-j-1-i)} A^{(i)}_{-j}\Big)
\\
&\qquad + \Big(
t(t_n^{\frac12}-t_n^{-\frac12}) + t^{\frac12}t_n^{\frac12}(t^{\frac12}-t^{-\frac12}) (t^{n-m}-1)\Big) A^{(i)}_{-i} 
\end{align*}
which provides the last relation in~\eqref{AgencD}.
\end{proof}

\begin{lemma} \label{Tred}
Let $i\in \{1, \ldots, n\}$ and let $\nu = (\nu_1, \ldots, \nu_i, 0,\ldots, 0)\in \ZZ^n$. 

\smallskip\noindent
Case 2: If $m\in \{i+1, \ldots, n\}$ and $v= s_{m-1}s_{m-2}\cdots s_i$ and $w=s_{m-1}v=s_{m-2}\cdots s_i$ then
\begin{align*}
T_{m-1}\widehat{E}^{ws_\ell\cdots s_n\cdots s_i}_\nu
&= \widehat{E}^{vs_\ell \cdots s_n\cdots s_i}_\nu, 
&&\hbox{if $-\ell\in \{-(i+1), \ldots, -n\}$,}
\\
T_{m-1}\widehat{E}^{ws_i\cdots s_n\cdots s_i}_\nu
&=\widehat{E}^{vs_i\cdots s_n\cdots s_i}_\nu
+(t^{\frac12}-t^{-\frac12})t^{-\frac12(m-i-1)} \widehat{E}_\nu^{vs_m\cdots s_n \cdots s_i},
&&\hbox{(the case $-\ell=-i$)}, \\
T_{m-1}\widehat{E}^{ws_{\ell-1}\cdots s_i}_\nu
&= \widehat{E}^{vs_{\ell-1} \cdots  s_i}_\nu, &&\hbox{if $\ell\in \{i, \ldots, n\}$ and $\ell \ne m$},
\\
T_{m-1}\widehat{E}^{ws_{m-1}\cdots s_i}_\nu
&= \widehat{E}^{vs_{m-1} \cdots  s_i}_\nu+(t^{\frac12}-t^{-\frac12})t^{\frac12(m-i-1)}\widehat{E}^v_\nu &&\hbox{(the case $\ell=m$)}.
\end{align*}
Case 3: If $v= s_ns_{n-1}\cdots s_i$ and $w=s_nv=s_{n-1}\cdots s_i$ then
\begin{align*}
T_n\widehat{E}^{ws_\ell\cdots s_n \cdots s_i} &= \widehat{E}^{vs_\ell \cdots s_n \cdots s_i}_\nu, 
&&\hbox{if $-\ell\in \{-(i+1), \ldots, -n\}$,} 
\\
T_n\widehat{E}^{ws_i\cdots s_n\cdots s_i} &= \widehat{E}^{vs_i\cdots s_n \cdots s_i}_\nu
+(t_n^{\frac12}-t_n^{-\frac12})\widehat{E}^v_\nu
&&\hbox{(the case $-\ell=-i$),} 
\\
T_n\widehat{E}^{ws_{\ell-1} \cdots s_i} &= \widehat{E}^{vs_\ell \cdots s_n \cdots s_i}_\nu, 
&&\hbox{if $\ell\in \{i, \ldots, n\}$.} 
\end{align*}
Case 4: If $-m\in \{-i, \ldots, -(n-1)\}$ and $v=s_m\cdots s_n\cdots s_i$ and $w=s_m v=s_{m+1}\cdots s_n\cdots s_i$ then
\begin{align*}
T_m \widehat{E}^{ws_\ell\cdots s_n\cdots s_i}_\nu &= \widehat{E}^{vs_\ell\cdots s_n \cdots s_i}_\nu, 
\qquad\hbox{if $\ell\in \{i, \ldots, n\}$,} 
\\
T_m \widehat{E}^{ws_{\ell-1}\cdots s_i}_\nu &= \widehat{E}^{vs_\ell\cdots s_n \cdots s_i}_\nu, 
\qquad \hbox{if $\ell\in \{i,i+1, \ldots, n\}$ and $\ell\not\in \{ i,m+1\}$,} \\
T_m \widehat{E}^{ws_i\cdots s_n\cdots s_i}_\nu 
&= \widehat{E}^{vs_\ell\cdots s_n \cdots s_i}_\nu + (t^{\frac12}-t^{-\frac12})
t^{\frac12(i+m)-(n-1)}
t_n^{-\frac12}\widehat{E}^{vs_m\cdots s_i}
\qquad\hbox{(the case $\ell=i$),} \\
T_m \widehat{E}^{ws_{m+1}\cdots s_n\cdots s_i}_\nu 
&= \widehat{E}^{vs_{m+1}\cdots s_n \cdots s_i}_\nu 
+ (t^{\frac12}-t^{-\frac12}) t^{(n-1)-\frac12(i+m)}
t_n^{\frac12} \widehat{E}^v,
\qquad \hbox{(the case $\ell=m+1$).}
\end{align*}
\end{lemma}

\begin{proof} \emph{Case 2: $v= s_{m-1}s_{m-2}\cdots s_i$ and $w=s_{m-1}z=s_{m-2}\cdots s_i$.}
\hfil\break
If $-\ell\in \{-i, \ldots, -n\}$ and $-\ell \neq -i$, then 
$v s_\ell \cdots s_n \cdots s_i > w s_\ell\cdots s_n \cdots s_i$, and
\begin{align*}
T_{m-1} \widehat{E}^{ws_\ell \cdots s_n \cdots s_i}_\nu
&= T_{m-1}T_{ws_\ell \cdots s_n \cdots s_i}\widehat{E}_\nu
= T_{s_{m-1}ws_\ell\cdots s_n \cdots s_i}\widehat{E}_\nu
= \widehat{E}^{vs_\ell\cdots s_n \cdots s_i}_\nu.
\end{align*}
If $-\ell\in \{-i, \ldots, -n\}$ and $-\ell = -i$ then 
$vs_i\cdots s_n\cdots s_i < ws_i\cdots s_n\cdots s_i$ and
\begin{align*}
T_{m-1} \widehat{E}^{ws_i \cdots s_n \cdots s_i}_\nu
&= (T_{m-1}^{-1}+(t^{\frac12}-t^{-\frac12})) T_{ws_i \cdots s_n \cdots s_i}\widehat{E}_\nu
\\
&= T_{s_{m-1} w s_i \cdots s_n \cdots s_i}\widehat{E}_\nu
+(t^{\frac12}-t^{-\frac12}) T_{ws_i \cdots s_n \cdots s_i}\widehat{E}_\nu \\
&= T_{v s_i \cdots s_n \cdots s_i}\widehat{E}_\nu
+(t^{\frac12}-t^{-\frac12}) T_{ws_i \cdots s_n \cdots s_i}\widehat{E}_\nu.
\end{align*}
Since
\begin{align*}
v s_{m} \cdots s_{n} \cdots s_i &=
\big( s_{m-1} {\color{red} s_{m-2} \cdots s_i} \big)
\big(s_{m} \cdots s_{n} \cdots s_i \big) \\
&= s_{m-1} 
\big(s_{m} \cdots s_{n} \cdots s_i \big)
{\color{red} s_{m-1} \cdots s_{i+1}}  \\
&= \big( w s_i \cdots s_{n} \cdots s_i  \big)
s_{m-1} \cdots s_{i+1},
\end{align*}
then $w s_i \cdots s_{n} \cdots s_i = v s_{m} \cdots s_n \cdots s_i (s_{i+1}^{-1} \cdots s_{m-1}^{-1})$
and
$$
T_{ws_i\cdots s_n\cdots s_i}\widehat{E}_\nu 
= T_{vs_m\cdots s_n\cdots s_i}T_{i+1}^{-1}\cdots T_{m-1}^{-1}\widehat{E}_\nu 
= t^{-\frac12(m-i-1)}\widehat{E}^{vs_m\cdots s_n\cdots s_i}_\nu.
$$
Thus
$$
T_{m-1} \widehat{E}^{ws_i \cdots s_n \cdots s_i}_\nu
=
 \widehat{E}^{vs_i \cdots s_n \cdots s_i}_\nu
+(t^{\frac12}-t^{-\frac12}) t^{-\frac12(m-i-1)} \widehat{E}^{vs_m\cdots s_n\cdots s_i}_\nu.
$$
If $\ell\in \{i, \ldots, n\}$ and $\ell \neq m$ then $v s_{\ell-1} \cdots s_i > ws_{\ell-1} \cdots s_i$
and
\begin{align*}
T_{m-1} \widehat{E}^{ws_{\ell-1} \cdots  s_i}_\nu
&
= T_{m-1}T_{ws_{\ell-1} \cdots s_i}\widehat{E}_\nu
= T_{s_{m-1}ws_{\ell-1} \cdots s_i}\widehat{E}_\nu
= \widehat{E}^{vs_\ell\cdots s_n \cdots s_i}_\nu.
\end{align*}
If $\ell\in \{i, \ldots, n\}$ and $\ell = m$ then $vs_{m-1}\cdots s_i < ws_{m-1}\cdots s_i$ and
\begin{align*}
T_{m-1}\widehat{E}^{ws_{m-1}\cdots s_i}_\nu
& = (T_{m-1}^{-1}+(t^{\frac12}-t^{-\frac12}))T_{ys_{m-1}\cdots s_i}\widehat{E}_\nu 
= T_{vs_{m-1}\cdots s_i}\widehat{E}_\nu  + (t^{\frac12}-t^{-\frac12}))T_{ws_{m-1}\cdots s_i}\widehat{E}_\nu.
\end{align*}
Since $ws_{m-1}\cdots s_i = vs_{m-1}\cdots s_{i+1}$ then
$$
T_{ws_{m-1}\cdots s_i}\widehat{E}_\nu 
= T_{v}T_{m-1}\cdots T_{i+1}\widehat{E}_\nu
= t^{\frac12(m-i-1)}T_{v}\widehat{E}_\nu
= t^{\frac12(m-i-1)}\widehat{E}^{v}_\nu
$$
giving
$$T_{m-1}\widehat{E}^{ws_{m-1}\cdots s_i}_\nu
= T_{vs_{m-1}\cdots s_i}\widehat{E}_\nu  + (t^{\frac12}-t^{-\frac12}))t^{\frac12(m-i-1)}\widehat{E}^{v}_\nu.
$$

\smallskip\noindent
\emph{Case 3: $v= s_ns_{n-1}\cdots s_i$ and $w=s_nv=s_{n-1}\cdots s_i$.}
\hfil\break
If $-\ell \in \{-i,\ldots, -n\}$ and $-\ell\ne -i$  then  $v s_\ell \cdots s_n \cdots s_i > w s_\ell\cdots s_n \cdots s_i$ and
\begin{align*}
T_{n} \widehat{E}^{ws_\ell \cdots s_n \cdots s_i}_\nu
&= T_{n}T_{ws_\ell \cdots s_n \cdots s_i}\widehat{E}_\nu
= T_{s_{n}ws_\ell\cdots s_n \cdots s_i}\widehat{E}_\nu
= \widehat{E}^{vs_\ell\cdots s_n \cdots s_i}_\nu.
\end{align*}
If $-\ell=-i$ then $ws_i\cdots s_n\cdots s_i = s_n s_{n-1}\cdots s_i=v$ and $s_n ws_i\cdots s_n \cdots s_i = s_{n-1}\cdots s_i$. Thus, 
\begin{align*}
T_n(\widehat{E}_\nu^{ws_i\cdots s_n \cdots s_i}) 
&= (T_n^{-1} + (t_n^{\frac12}-t_n^{-\frac12})) \widehat{E}_\nu^{ws_i\cdots s_n \cdots s_i} 
= T_{s_nws_i\cdots s_n \cdots s_i}\widehat{E}_\nu + (t_n^{\frac12}-t_n^{-\frac12})\widehat{E}_\nu^{ws_i\cdots s_n \cdots s_i} \\
&= \widehat{E}_\nu^{s_{n-1} \cdots s_i} + (t_n^{\frac12}-t_n^{-\frac12})\widehat{E}_\nu^{v} 
= \widehat{E}_\nu^{vs_i \cdots s_n \cdots s_i} + (t_n^{\frac12}-t_n^{-\frac12})\widehat{E}_\nu^{v}.
\end{align*}
If $\ell \in \{i, \ldots, n\}$ then $v s_{\ell-1} \cdots s_i > ws_{\ell-1} \cdots s_i$ and
\begin{align*}
T_n \widehat{E}^{ws_{\ell-1} \cdots  s_i}_\nu
&
= T_nT_{ws_{\ell-1} \cdots s_i}\widehat{E}_\nu
= T_{s_n ws_{\ell-1} \cdots s_i}\widehat{E}_\nu
= \widehat{E}^{vs_\ell\cdots s_n \cdots s_i}_\nu.
\end{align*}

\smallskip\noindent
\emph{Case 4: $v=s_m\cdots s_n\cdots s_i$ and $w=s_m v=s_{m+1}\cdots s_n\cdots s_i$ .}
\hfil\break
If $\ell \in \{i, \ldots, n\}$ then $v s_{\ell-1} \cdots s_i > ws_{\ell-1} \cdots s_i$ and
\begin{align*}
T_m \widehat{E}^{ws_{\ell-1} \cdots  s_i}_\nu
&
= T_mT_{ws_{\ell-1} \cdots s_i}\widehat{E}_\nu
= T_{s_m ws_{\ell-1} \cdots s_i}\widehat{E}_\nu
= \widehat{E}^{vs_\ell\cdots s_n \cdots s_i}_\nu.
\end{align*}
If $\ell \in \{i,\ldots, n\}$ and $\ell\not\in \{i,m+1\}$ then $v s_\ell \cdots s_n \cdots s_i > w s_\ell\cdots s_n \cdots s_i$ and
\begin{align*}
T_m \widehat{E}^{ws_\ell \cdots s_n \cdots s_i}_\nu
&= T_m T_{ws_\ell \cdots s_n \cdots s_i}\widehat{E}_\nu
= T_{s_m ws_\ell\cdots s_n \cdots s_i}\widehat{E}_\nu
= \widehat{E}^{vs_\ell\cdots s_n \cdots s_i}_\nu.
\end{align*}
If $\ell=i$ then
$w s_i \cdots s_{n} \cdots s_i
= (s_{m+1}\cdots s_n \cdots s_i )( s_i \cdots s_n \cdots s_i ) =  s_m \cdots s_i$ and
\begin{align*}
v s_{m} \cdots s_i (s_{i+1} \cdots s_n \cdots s_{m+1}) 
&= s_ m\cdots s_n \cdots s_i {\color{blue} s_m \cdots s_i \cdots s_m } s_{m+1} \cdots s_n \cdots s_{m+1} \\
&= s_ m\cdots s_n \cdots s_i {\color{blue} s_i \cdots s_m \cdots s_i} s_{m+1} \cdots s_n \cdots s_{m+1}\\
&=  s_ m\cdots s_n \cdots s_{m+1} {\color{purple} s_{m-1} \cdots s_i } s_{m+1} \cdots s_n \cdots s_{m+1} \\
&= s_m {\color{purple} s_{m-1} \cdots s_i } s_{m+1} \cdots s_n \cdots s_{m+1}s_{m+1}\cdots s_n \cdots s_{m+1} \\
&=  s_{m} \cdots s_i
=w s_i \cdots s_n \cdots s_i,
\end{align*}
giving
$$
T_{ws_i\cdots s_n\cdots s_i}\widehat{E}_\nu
= T_{vs_m\cdots s_i}T^{-1}_{i+1}\cdots T_n^{-1}\cdots T_{m+1}^{-1}\widehat{E}_\nu
=t^{-\frac12\big((n-1-i) + (n-1-m)\big)}t_n^{-\frac12}\widehat{E}^{vs_m\cdots s_i}_\nu.
$$
So
\begin{align*}
T_m\widehat{E}^{ws_i\cdots s_n\cdots s_i}
&= T_mT_{ws_i\cdots s_n \cdots s_i}\widehat{E}_\nu
= (T_m^{-1}+(t^{\frac12}-t^{-\frac12}))T_{ws_i\cdots s_n \cdots s_i}\widehat{E}_\nu
\\
&=\widehat{E}^{vs_i\cdots s_n\cdots s_i}_\nu
+(t^{\frac12}-t^{-\frac12})t^{-\frac12\big((n-1-i) + (n-1-m)\big)}t_n^{-\frac12}\widehat{E}^{vs_m\cdots s_i}_\nu.
\end{align*}
If $\ell=m+1$ then
\begin{align*}
ws_{m+1}\cdots s_n\cdots s_i &=
\big( s_{m+1} \cdots s_n \cdots s_{m+1}s_m {\color{blue}s_{m-1}\cdots s_i} \big)
\big(s_{m+1} \cdots s_n \cdots s_i \big) \\
&= \big( s_{m+1} \cdots s_n \cdots s_{m+1} {\color{purple} s_m} \big)
\big(s_{m+1} \cdots s_n \cdots s_i \big) {\color{blue}s_{m} \cdots s_{i+1}}  \\
&= \big( s_{m+1} \cdots s_n \cdots s_{m+1} \big)
\big({\color{purple} s_m} s_{m+1} \cdots s_n \cdots s_i \big) 
s_{m} \cdots s_{i+1}  \\
&= \big( s_m s_{m+1} \cdots s_n \cdots s_i \big) 
 \big( s_{m+1} \cdots s_n \cdots s_{m+1} \big)
s_m \cdots s_{i+1}  \\
&= v s_{m+1} \cdots s_n \cdots s_{i+1},
\end{align*}
giving
$$
T_{ws_{m+1}\cdots s_n\cdots s_i }\widehat{E}_\nu 
= T_vT_{m+1}\cdots T_n\cdots T_{i+1}\widehat{E}_\nu
= t^{\frac12((n-1-m)+(n-1-i))}t_n^{\frac12}\widehat{E}^v_\nu.
$$
So
\begin{align*}
T_m\widehat{E}^{ws_{m+1}\cdots s_n\cdots s_i}
&= T_mT_{ws_{m+1}\cdots s_n \cdots s_i}\widehat{E}_\nu
= (T_m^{-1}+(t^{\frac12}-t^{-\frac12}))T_{ws_{m+1}\cdots s_n \cdots s_i}\widehat{E}_\nu
\\
&= T_{vs_{m+1}\cdots s_n \cdots s_i}\widehat{E}_\nu
+(t^{\frac12}-t^{-\frac12})T_{ws_{m+1}\cdots s_n \cdots s_i}\widehat{E}_\nu
\\
&= \widehat{E}^{vs_{m+1}\cdots s_n \cdots s_i}_\nu
+(t^{\frac12}-t^{-\frac12})t^{\frac12((n-1-m)+(n-1-i))}t_n^{\frac12}\widehat{E}^v_\nu.
\end{align*}
\end{proof}

\newpage
\appendix

\section{Appendix: The double affine Hecke algebra (DAHA)}\label{section:DAHA}

\subsection{The double affine Artin group for type $CC_n$}

Use a graphical notation for relations so that
$$
\begin{array}{cl}
\beginpicture
\setcoordinatesystem units <1cm,1cm>       
\setplotarea x from 2.8 to 4.2, y from 1.8 to 2 
{\scriptsize
\multiput {$\circ$} at 3  1.9 *1 1 0 /      
\put {$g_i$}     at 3 2.1   %
\put {$g_j$}     at 4 2.1   
\linethickness=0.5pt                      
}
\endpicture
&\hbox{means $g_ig_j =g_jg_i$,} 
\\ \\
\beginpicture
\setcoordinatesystem units <1cm,1cm>       
\setplotarea x from 2.8 to 4.2, y from 1.8 to 2 
{\scriptsize
\multiput {$\circ$} at 3  1.9 *1 1 0 /      
\put {$g_i$}     at 3 2.1   %
\put {$g_j$}     at 4 2.1   
\linethickness=0.5pt                      
\putrule from 3.05 1.9 to 3.95 1.9   
}
\endpicture
&\hbox{means $g_ig_jg_i = g_jg_ig_j$, and} \\ \\
\beginpicture
\setcoordinatesystem units <1cm,1cm>       
\setplotarea x from 2.8 to 4.2, y from 1.8 to 2 
{\scriptsize
\multiput {$\circ$} at 3  1.9 *1 1 0 /      
\put {$g_i$}     at 3 2.1   %
\put {$g_j$}     at 4 2.1   
\linethickness=0.5pt                      
\putrule from 3.03 1.935 to 3.97 1.935  %
\putrule from 3.03 1.865 to 3.97 1.865  %
}
\endpicture
&\hbox{means $g_ig_jg_ig_j = g_jg_ig_jg_i$.} 
\end{array}
$$
The type $(C_n, C_n)$ \emph{double affine Artin group (DAArt)} is the group $\tilde \cB_n$ is generated by 
$q^{\frac12}$, $T_1,\ldots, T_n$ and $T_0, 
T_0^{\vee}$
with relations
\begin{equation}
q^\frac12 \in Z(\tilde \cB_n),
\quad
T_0^{\vee} T_1^{-1}T_0 T_1 = T_1^{-1}T_0 T_1T_0^{\vee}, 
\quad
T_0^{-1} T_1(T_0^{\vee})^{-1} T_1^{-1} 
= T_1(T_0^{\vee})^{-1} T_1^{-1}T_0^{-1},
\label{BdefrelsI}
\end{equation}
and
\begin{equation}
\begin{matrix}
\beginpicture
\setcoordinatesystem units <1cm,1cm>        
\setplotarea x from -3 to 3, y from -0.5 to 0.5  
\multiput {$\circ$} at 1   0 *2 1 0 /      %
\multiput {$\circ$} at -3   0 *2 1 0 /      
\put {$T_n$}     at 3 0.4   %
\put {$T_{n-1}$}     at 2 0.4   %
\put {$T_{n-2}$} at 1 0.4   %
\put {$T_2$}     at -1 0.4   %
\put {$T_1$}     at -2 0.4   %
\put {$T_0$}     at -3 0.4   
\linethickness=0.75pt                          
\putrule from -2.97 0.045 to -2.03 0.045       
\putrule from -2.97 -0.045 to -2.03 -0.045       %
\putrule from -1.95 0 to -1.05 0              %
\putrule from 1.05 0 to 1.95 0              %
\putrule from 2.03 0.045 to 2.97 0.045       
\putrule from 2.03 -0.045 to 2.97 -0.045       %
\setlinear
\setdashes <2mm,1mm>          %
\putrule from -0.95 0 to 0.95 0  
\endpicture
\\
\beginpicture
\setcoordinatesystem units <1cm,1cm>        
\setplotarea x from -3 to 3, y from -0.5 to 0.5  
\multiput {$\circ$} at 1   0 *2 1 0 /      %
\multiput {$\circ$} at -3   0 *2 1 0 /      
\put {$T_n$}     at 3 0.4   %
\put {$T_{n-1}$}     at 2 0.4   %
\put {$T_{n-2}$} at 1 0.4   %
\put {$T_2$}     at -1 0.4   %
\put {$T_1$}     at -2 0.4   %
\put {$T_0^{\vee}$}     at -3 0.4   
\linethickness=0.75pt                          
\putrule from -2.97 0.045 to -2.03 0.045       
\putrule from -2.97 -0.045 to -2.03 -0.045       %
\putrule from -1.95 0 to -1.05 0              %
\putrule from 1.05 0 to 1.95 0              %
\putrule from 2.03 0.045 to 2.97 0.045       
\putrule from 2.03 -0.045 to 2.97 -0.045       %
\setlinear
\setdashes <2mm,1mm>          %
\putrule from -0.95 0 to 0.95 0  
\endpicture
\end{matrix}
\label{Bdefrels2I}
\end{equation}
There are redundant relations in this presentation of $\tilde \cB_n$.  In particular, the second relation
in~\eqref{BdefrelsI} implies $T_1T_0^{\vee} T_1^{-1}T_0 = T_0T_1T_0^{\vee}T_1^{-1}$, giving 
that $T_0$ commutes with $T_1T_0^{\vee} T_1^{-1}$ and thus that $T_0^{-1}$ commutes with
$T_1(T_0^{\vee})^{-1}T_1^{-1}$.  But this is the last relation in 
\eqref{BdefrelsI}.

For $j=1,\ldots, n$ define $Y^{\varepsilon_j}$ and $X^{\varepsilon_j}$ in $\tilde \cB_n$ by 
\begin{equation}
\begin{array}{lcl}
Y^{\varepsilon_1} = T_0T_1\cdots T_n\cdots T_1
&\qquad
\hbox{and}\qquad 
&Y^{\varepsilon_{j+1}} = T_j^{-1}Y^{\varepsilon_j}T_j^{-1},
\\
X^{\varepsilon_1} = (T_0^{\vee})^{-1}T_1^{-1}\cdots T_n^{-1}\cdots T_1^{-1},
&\quad\hbox{and}\quad
&X^{\varepsilon_{j+1}}
= T_j X^{\varepsilon_j}T_j.
\end{array}
\label{XYdefn}
\end{equation}

\begin{prop}\label{dualityB}  (Duality)~\cite[(3.5.1)]{Mac03} There is an involutive automorphism 
$\iota\colon \tilde \cB_n \longrightarrow  \tilde \cB_n$ given by
$$\iota(q^{\frac12}) = q^{-\frac12},
\quad
\iota(T_0) = (T_0^{\vee})^{-1},
\quad
\iota(T_0^{\vee}) = T_0^{-1},
\quad
\hbox{and}\quad
\iota(T_i) = T_i^{-1}\ \ \hbox{for $i\in\{1,\ldots, n\}$.}
$$
Furthermore
$$\iota(Y^{\varepsilon_i}) = X^{\varepsilon_i}, \qquad\hbox{for  $i\in\{1,\ldots, n\}$}.$$ 
\end{prop}
\begin{proof}
The involution $\iota$ fixes the first relation in~\eqref{BdefrelsI}, switches the second and third relations
in~\eqref{BdefrelsI}, and switches the relations in~\eqref{Bdefrels2I}.
The last statement follows from~\eqref{XYdefn}.
\end{proof}

The elements of the braid group on $n+3$ strands given by
$$T_0 =
\beginpicture
\setcoordinatesystem units <.5cm,.5cm>         
\setplotarea x from -5 to 3.5, y from -2 to 2    
\put{$\bullet$} at -3 0.75      %
\put{$\bullet$} at -2 0.75      %
\put{$\bullet$} at -1 0.75      %
\put{$\bullet$} at  0 0.75      
\put{$\bullet$} at  1 0.75      %
\put{$\bullet$} at  2 0.75      %
\put{$\bullet$} at  3 0.75      %
\put{$\bullet$} at -3 -0.75          %
\put{$\bullet$} at -2 -0.75          %
\put{$\bullet$} at -1 -0.75          %
\put{$\bullet$} at  0 -0.75          
\put{$\bullet$} at  1 -0.75          %
\put{$\bullet$} at  2 -0.75          %
\put{$\bullet$} at  3 -0.75          %
\plot -6.0 1.25 -6.0 -1.25 / 
\plot -5.75 1.25 -5.75  -1.25 / 
\ellipticalarc axes ratio 1:1 360 degrees from -6.0 1.25 center at -5.875 1.25 
\put{$*$} at -5.875 1.25 
\ellipticalarc axes ratio 1:1 180 degrees from -6.0 -1.25 center at -5.875 -1.25
\plot -4.5 1.25 -4.5 -0.13 / 
\plot -4.5 -0.37   -4.5 -1.25 / 
\plot -4.25 1.25 -4.25  -0.13 / 
\plot -4.25 -0.37 -4.25 -1.25 /
\ellipticalarc axes ratio 1:1 360 degrees from -4.5 1.25 center at
-4.375 1.25 \put{$*$} at -4.375 1.25 \ellipticalarc axes ratio 1:1
180 degrees from -4.5 -1.25 center at -4.375 -1.25
\plot 4.5 1.25 4.5 -1.25 / \plot 4.25 1.25 4.25 -1.25 /
\ellipticalarc axes ratio 1:1 360 degrees from 4.5 1.25 center at
4.375 1.25 \put{$*$} at 4.375 1.25 \ellipticalarc axes ratio 1:1 180
degrees from 4.25 -1.25 center at 4.375 -1.25
\plot -2 0.75  -2 -0.75 / \plot -1 0.75  -1 -0.75 / \plot  0 0.75 0
-0.75 / \plot  1 0.75   1 -0.75 / \plot  2 0.75   2 -0.75 / \plot 3
0.75   3 -0.75 / \setlinear
\plot -3.3 0.25  -4.1 0.25 / \ellipticalarc axes ratio 2:1 180
degrees from -4.65 0.25  center at -4.65 0 \plot -4.65 -0.25  -3.3
-0.25 / \setquadratic
\plot  -3.3 0.25  -3.05 .45  -3 0.75 / \plot  -3.3 -0.25  -3.05
-0.45  -3 -0.75 /
\endpicture
\qquad\hbox{and}\qquad 
T_n = \beginpicture
\setcoordinatesystem units <.5cm,.5cm>         
\setplotarea x from -5 to 3.5, y from -2 to 2    
\put{$\bullet$} at -3 0.75      %
\put{$\bullet$} at -2 0.75      %
\put{$\bullet$} at -1 0.75      %
\put{$\bullet$} at  0 0.75      
\put{$\bullet$} at  1 0.75      %
\put{$\bullet$} at  2 0.75      %
\put{$\bullet$} at  3 0.75      %
\put{$\bullet$} at -3 -0.75          %
\put{$\bullet$} at -2 -0.75          %
\put{$\bullet$} at -1 -0.75          %
\put{$\bullet$} at  0 -0.75          
\put{$\bullet$} at  1 -0.75          %
\put{$\bullet$} at  2 -0.75          %
\put{$\bullet$} at  3 -0.75          %
\plot -6.0 1.25 -6.0 -1.25 / 
\plot -5.75 1.25 -5.75  -1.25 / 
\ellipticalarc axes ratio 1:1 360 degrees from -6.0 1.25 center at -5.875 1.25 
\put{$*$} at -5.875 1.25 
\ellipticalarc axes ratio 1:1 180 degrees from -6.0 -1.25 center at -5.875 -1.25
\plot -4.5 1.25 -4.5 -1.25 / \plot -4.25 1.25 -4.25 -1.25 /
\ellipticalarc axes ratio 1:1 360 degrees from -4.5 1.25 center at
-4.375 1.25 \put{$*$} at -4.375 1.25 \ellipticalarc axes ratio 1:1
180 degrees from -4.5 -1.25 center at -4.375 -1.25
\plot 4.5 1.25 4.5 0.37 / 
\plot 4.5 0.13   4.5 -1.25 / 
\plot 4.25 1.25 4.25  0.37 / 
\plot 4.25 0.13 4.25 -1.25 / 
\ellipticalarc axes
ratio 1:1 360 degrees from 4.5 1.25 center at 4.375 1.25 \put{$*$}
at 4.375 1.25 \ellipticalarc axes ratio 1:1 180 degrees from 4.25
-1.25 center at 4.375 -1.25
\plot  -3 0.75  -3 -0.75 / \plot -2 0.75  -2 -0.75 / \plot -1 0.75
-1 -0.75 / \plot  0 0.75   0 -0.75 / \plot  1 0.75   1 -0.75 / \plot
2 0.75   2 -0.75 / \setlinear
\plot 3.3 -0.25  4.1 -0.25 / \ellipticalarc axes ratio 2:1 180
degrees from 4.65 -0.25  center at 4.65 0 \plot 4.65 0.25  3.3 0.25
/ \setquadratic
\plot  3.3 0.25  3.05 .45  3 0.75 / \plot  3.3 -0.25  3.05 -0.45  3
-0.75 /
\endpicture
$$
$$T_0^{\vee} =
\beginpicture
\setcoordinatesystem units <.5cm,.5cm>         
\setplotarea x from -7 to 3.5, y from -2 to 2    
\put{$\bullet$} at -3 0.75      %
\put{$\bullet$} at -2 0.75      %
\put{$\bullet$} at -1 0.75      %
\put{$\bullet$} at  0 0.75      
\put{$\bullet$} at  1 0.75      %
\put{$\bullet$} at  2 0.75      %
\put{$\bullet$} at  3 0.75      %
\put{$\bullet$} at -3 -0.75          %
\put{$\bullet$} at -2 -0.75          %
\put{$\bullet$} at -1 -0.75          %
\put{$\bullet$} at  0 -0.75          
\put{$\bullet$} at  1 -0.75          %
\put{$\bullet$} at  2 -0.75          %
\put{$\bullet$} at  3 -0.75          %
\plot -6.0 1.25 -6.0 -0.13 / 
\plot -6.0 -0.37   -6.0 -1.25 / 
\plot -5.75 1.25 -5.75  -0.13 / 
\plot -5.75 -0.37 -5.75 -1.25 /
\ellipticalarc axes ratio 1:1 360 degrees from -6.0 1.25 center at -5.875 1.25 
\put{$*$} at -5.875 1.25 
\ellipticalarc axes ratio 1:1 180 degrees from -6.0 -1.25 center at -5.875 -1.25
\plot -4.5 1.25 -4.5 -1.25 / 
\plot -4.25 1.25 -4.25  -1.25 / 
\ellipticalarc axes ratio 1:1 360 degrees from -4.5 1.25 center at
-4.375 1.25 \put{$*$} at -4.375 1.25 \ellipticalarc axes ratio 1:1
180 degrees from -4.5 -1.25 center at -4.375 -1.25
\plot 4.5 1.25 4.5 -1.25 / \plot 4.25 1.25 4.25 -1.25 /
\ellipticalarc axes ratio 1:1 360 degrees from 4.5 1.25 center at
4.375 1.25 \put{$*$} at 4.375 1.25 \ellipticalarc axes ratio 1:1 180
degrees from 4.25 -1.25 center at 4.375 -1.25
\plot -2 0.75  -2 -0.75 / 
\plot -1 0.75  -1 -0.75 / 
\plot  0 0.75 0 -0.75 / 
\plot  1 0.75   1 -0.75 / 
\plot  2 0.75   2 -0.75 / 
\plot 3 0.75   3 -0.75 / 
\setlinear
\plot -3.3 0.25  -4.1 0.25 / 
\plot -4.65 0.25 -5.6 0.25 /
\ellipticalarc axes ratio 2:1 180 degrees from -6.15 0.25  center at -6.15 0 
\plot -6.15 -0.25 -4.65 -0.25 /
\plot -4.1 -0.25 -3.3 -0.25 / 
\setquadratic
\plot  -3.3 0.25  -3.05 .45  -3 0.75 / \plot  -3.3 -0.25  -3.05
-0.45  -3 -0.75 /
\endpicture
\qquad\hbox{and}\qquad
T_i =
\beginpicture
\setcoordinatesystem units <.5cm,.5cm>         
\setplotarea x from -6.5 to 3.5, y from -2 to 2    
\put{${}^i$} at 0 1.2      %
\put{${}^{i+1}$} at 1 1.2      %
\put{$\bullet$} at -3 .75      %
\put{$\bullet$} at -2 .75      %
\put{$\bullet$} at -1 .75      %
\put{$\bullet$} at  0 .75      
\put{$\bullet$} at  1 .75      %
\put{$\bullet$} at  2 .75      %
\put{$\bullet$} at  3 .75      %
\put{$\bullet$} at -3 -.75          %
\put{$\bullet$} at -2 -.75          %
\put{$\bullet$} at -1 -.75          %
\put{$\bullet$} at  0 -.75          
\put{$\bullet$} at  1 -.75          %
\put{$\bullet$} at  2 -.75          %
\put{$\bullet$} at  3 -.75          %
\plot -6.0 1.25 -6.0 -1.25 / 
\plot -5.75 1.25 -5.75  -1.25 / 
\ellipticalarc axes ratio 1:1 360 degrees from -6.0 1.25 center at -5.875 1.25 
\put{$*$} at -5.875 1.25 
\ellipticalarc axes ratio 1:1 180 degrees from -6.0 -1.25 center at -5.875 -1.25
\plot -4.5 1.25 -4.5 -1.25 / \plot -4.25 1.25 -4.25 -1.25 /
\ellipticalarc axes ratio 1:1 360 degrees from -4.5 1.25 center at
-4.375 1.25 \put{$*$} at -4.375 1.25 \ellipticalarc axes ratio 1:1
180 degrees from -4.5 -1.25 center at -4.375 -1.25
\plot 4.5 1.25 4.5 -1.25 / \plot 4.25 1.25 4.25 -1.25 /
\ellipticalarc axes ratio 1:1 360 degrees from 4.5 1.25 center at
4.375 1.25 \put{$*$} at 4.375 1.25 \ellipticalarc axes ratio 1:1 180
degrees from 4.25 -1.25 center at 4.375 -1.25
\plot -3 .75  -3 -.75 / \plot -2 .75  -2 -.75 / \plot -1 .75  -1
-.75 /
\plot  2 .75   2 -.75 / \plot  3 .75   3 -.75 / \setquadratic
\plot  0 -.75  .05 -.45  .4 -0.1 / \plot  .6 0.1  .95 0.45  1 .75 /
\plot 0 .75  .05 .45  .5 0  .95 -0.45  1 -.75 /
\endpicture\ ,
$$
satisfy the relations in~\eqref{BdefrelsI} and~\eqref{Bdefrels2I}.
Pictorially, 
$$Y^{\varepsilon_j}
= \beginpicture
\setcoordinatesystem units <.5cm,.5cm>         
\setplotarea x from -6.3 to 3.5, y from -2 to 2    
\put{${}^j$} at 0 1.45
\put{$\bullet$} at -3 1      %
\put{$\bullet$} at -2 1      %
\put{$\bullet$} at -1 1      %
\put{$\bullet$} at  0 1      
\put{$\bullet$} at  1 1      %
\put{$\bullet$} at  2 1      %
\put{$\bullet$} at  3 1      %
\put{$\bullet$} at -3 -1          %
\put{$\bullet$} at -2 -1          %
\put{$\bullet$} at -1 -1          %
\put{$\bullet$} at  0 -1          
\put{$\bullet$} at  1 -1          %
\put{$\bullet$} at  2 -1          %
\put{$\bullet$} at  3 -1          %
\plot -6.0 1.5 -6.0 -1.5 / 
\plot -5.75 1.5 -5.75  -1.5 / 
\ellipticalarc axes ratio 1:1 360 degrees from -6.0 1.5 center at -5.875 1.5 
\put{$*$} at -5.875 1.5 
\ellipticalarc axes ratio 1:1 180 degrees from -6.0 -1.5 center at -5.875 -1.5
\plot -4.5 1.5 -4.5 0.13 / 
\plot -4.5 -0.13 -4.5 -1.5 / 
\plot -4.25 1.5 -4.25 0.13 / 
\plot -4.25 -0.13 -4.25 -1.5 / 
\ellipticalarc axes ratio 1:1 360 degrees from -4.5 1.5 center at -4.375 1.5 
\put{$*$} at -4.375 1.5 
\ellipticalarc axes ratio 1:1 180 degrees from -4.5 -1.5 center at -4.375 -1.5
\plot 4.5 1.5 4.5 0.13 / 
\plot 4.5 -0.13   4.5 -1.5 / 
\plot 4.25 1.5 4.25  0.13 / 
\plot 4.25 -0.13 4.25 -1.5 / 
\ellipticalarc axes ratio 1:1 360 degrees from 4.5 1.5 center at 4.375 1.5 
\put{$*$} at 4.375 1.5 
\ellipticalarc axes ratio 1:1 180 degrees from 4.25 -1.5 center at 4.375 -1.5
\plot -3 1 -3 0.62 /
\plot -3 0.37 -3 0.13 /
\plot -3 -0.13 -3 -1 /
\plot -2 1 -2 0.62 /
\plot -2 0.37 -2 0.13 /
\plot -2 -0.13 -2 -1 /
\plot -1 1 -1 0.62 /
\plot -1 0.37 -1 0.13 /
\plot -1 -0.13 -1 -1 /
\plot  1 1  1 0.13 /
\plot 1 -0.13  1 -1 / 
\plot  2 1  2 0.13 /
\plot 2 -0.13  2 -1 / 
\plot  3 1   3 0.13 /
\plot 3 -0.13  3 -1 /
\setquadratic
\plot 0 1 -0.05 .7 -0.3 0.5 /
\setlinear
\plot -0.3 0.5 -4.1 0.5 /
\ellipticalarc axes ratio 2:1 180 degrees from -4.65 0.5  center at -4.65 0.25 
\plot -4.65 0 4.65 0 / 
\ellipticalarc axes ratio 2:1 180 degrees from 4.65 -0.5  center at 4.65 -0.25
\plot 4.1 -0.5 3.2 -0.5 / 
\plot 2.2 -0.5  2.8 -0.5 / 
\plot 1.2 -0.5  1.8 -0.5 / 
\plot 0.3 -0.5  0.8 -0.5 /
\setquadratic
\plot 0 -1 0.05 -0.7 0.3 -0.5 /
\endpicture
\qquad\hbox{and}\qquad
X^{\varepsilon_j}
= \beginpicture
\setcoordinatesystem units <.5cm,.5cm>         
\setplotarea x from -7 to 3.5, y from -2 to 2    
\put{${}^j$} at 0 1.45
\put{$\bullet$} at -3 1      %
\put{$\bullet$} at -2 1      %
\put{$\bullet$} at -1 1      %
\put{$\bullet$} at  0 1      
\put{$\bullet$} at  1 1      %
\put{$\bullet$} at  2 1      %
\put{$\bullet$} at  3 1      %
\put{$\bullet$} at -3 -1          %
\put{$\bullet$} at -2 -1          %
\put{$\bullet$} at -1 -1          %
\put{$\bullet$} at  0 -1          
\put{$\bullet$} at  1 -1          %
\put{$\bullet$} at  2 -1          %
\put{$\bullet$} at  3 -1          %
\plot -6.0 1.5 -6.0 0.63 /
\plot -6.0 0.37 -6.0 -1.5 / 
\plot -5.75 1.5  -5.75 0.63 /
\plot -5.75 0.37  -5.75 -1.5 / 
\ellipticalarc axes ratio 1:1 360 degrees from -6.0 1.5 center at -5.875 1.5 
\put{$*$} at -5.875 1.5 
\ellipticalarc axes ratio 1:1 180 degrees from -6.0 -1.5 center at -5.875 -1.5
\plot -4.5 1.5 -4.5 -1.5 /
\plot -4.25 1.5 -4.25 -1.5 / 
\ellipticalarc axes ratio 1:1 360 degrees from -4.5 1.5 center at -4.375 1.5 
\put{$*$} at -4.375 1.5 
\ellipticalarc axes ratio 1:1 180 degrees from -4.5 -1.5 center at -4.375 -1.5
\plot 4.5 1.5 4.5 -0.37 / 
\plot 4.5 -0.63   4.5 -1.5 / 
\plot 4.25 1.5 4.25  -0.37 / 
\plot 4.25 -0.63 4.25 -1.5 / 
\ellipticalarc axes ratio 1:1 360 degrees from 4.5 1.5 center at 4.375 1.5 
\put{$*$} at 4.375 1.5 
\ellipticalarc axes ratio 1:1 180 degrees from 4.25 -1.5 center at 4.375 -1.5
\plot -3 1 -3 -1 /
\plot -2 1 -2 -1 /
\plot -1 1 -1 -1 /
\plot  1 1  1 -0.37 /
\plot 1 -0.62  1 -1 / 
\plot  2 1  2 -0.37 /
\plot 2 -0.62  2 -1 / 
\plot  3 1   3 -0.37 /
\plot 3 -0.62  3 -1 /
\setquadratic
\plot 0 1 -0.05 .7 -0.3 0.5 /
\setlinear
\plot -0.3 0.5 -0.8 0.5 /
\plot -1.2 0.5 -1.8 0.5 /
\plot -2.2 0.5 -2.8 0.5 /
\plot -3.2 0.5 -4.1 0.5 /
\plot -4.65 0.5 -6.15 0.5 /
\ellipticalarc axes ratio 2:1 180 degrees from -6.15 0.5  center at -6.15 0.25 
\plot -5.6 0 -4.65 0 /
\plot -4.1 0 -3.2 0 /
\plot -2.8 0 -2.2 0 /
\plot -1.8 0 -1.2 0 /
\plot -0.8 0 0.8 0 /
\plot 1.2 0 1.8 0 /
\plot 2.2 0 2.8 0 /
\plot 3.2 0 4.1 0 / 
\ellipticalarc axes ratio 2:1 180 degrees from 4.65 -0.5  center at 4.65 -0.25
\plot 4.65 -0.5 0.3 -0.5 /
\setquadratic
\plot 0 -1 0.05 -0.7 0.3 -0.5 /
\endpicture
$$
The pictorial viewpoint is particularly convenient for verifying  that the subgroups
\begin{equation}
X = \langle X^{\varepsilon_1}, \ldots, X^{\varepsilon_n}\rangle
\qquad\hbox{and}\qquad
Y = \langle Y^{\varepsilon_1}, \ldots, Y^{\varepsilon_n}\rangle
\qquad\hbox{are abelian.}
\label{XYabelian}
\end{equation}

The \emph{affine braid groups} $\cB_n$ and $\cB_n^{\vee}$ are the subgroups of $\tilde \cB_n$
given by 
$$
\cB_n = \langle T_0, T_1\ldots, T_n\rangle
\qquad\hbox{and}\qquad
\cB_n^{\vee} = \langle T_0^{\vee}, T_1\ldots, T_n\rangle.
$$  

\begin{thm} \label{Baltpres} 
\item[(a)]The group $\tilde \cB_n$ is presented by generators $T_0^{\vee}, T_1,\ldots, T_n$ and
$Y^{\varepsilon_1}, \ldots, Y^{\varepsilon_n}$ and relations
$$
\beginpicture
\setcoordinatesystem units <1cm,1cm>        
\setplotarea x from -3 to 3, y from -0.5 to 0.5  
\multiput {$\circ$} at 1   0 *2 1 0 /      %
\multiput {$\circ$} at -3   0 *2 1 0 /      
\put {$T_n$}     at 3 0.4   %
\put {$T_{n-1}$}     at 2 0.4   %
\put {$T_{n-2}$} at 1 0.4   %
\put {$T_2$}     at -1 0.4   %
\put {$T_1$}     at -2 0.4   %
\put {$T_0^{\vee}$}     at -3 0.4   
\linethickness=0.75pt                          
\putrule from -2.97 0.045 to -2.03 0.045       
\putrule from -2.97 -0.045 to -2.03 -0.045       %
\putrule from -1.95 0 to -1.05 0              %
\putrule from 1.05 0 to 1.95 0              %
\putrule from 2.03 0.045 to 2.97 0.045       
\putrule from 2.03 -0.045 to 2.97 -0.045       %
\setlinear
\setdashes <2mm,1mm>          %
\putrule from -0.95 0 to 0.95 0  
\endpicture
\qquad \qquad
q^\frac12 \in Z(\tilde \cB_n),
$$
\begin{equation}\label{Ycomm}
Y^{\varepsilon_i}Y^{\varepsilon_j}
= Y^{\varepsilon_j}Y^{\varepsilon_i},
\ \ \hbox{ for $i,j\in \{1, \ldots, n\}$,}
\qquad
T_n Y^{\varepsilon_j} = Y^{\varepsilon_j}T_n,\ \ 
\hbox{for $j\in\{1,\ldots, n-1\}$,}
\end{equation}
\begin{equation}\label{TY1}
T_0^{\vee} Y^{\varepsilon_j} = Y^{\varepsilon_j} T_0^{\vee},\ \ 
\hbox{for $j\in \{2,\ldots, n\}$,}
\end{equation}
\begin{equation}\label{TY2}
Y^{\varepsilon_{i+1}} = T_i^{-1}Y^{\varepsilon_i}T_i^{-1}
\quad\hbox{and}\quad
T_iY^{\varepsilon_j} = Y^{\varepsilon_j}T_i,
\quad\hbox{for $i\in\{1,\ldots, n-1\}$ and $j\not\in\{i, i+1\}$,}
\end{equation}
\item[(b)]The group $\tilde \cB_n$ is presented by generators $T_0, T_1,\ldots, T_n$ and
$X^{\varepsilon_1}, \ldots, X^{\varepsilon_n}$ and relations
$$
\beginpicture
\setcoordinatesystem units <1cm,1cm>        
\setplotarea x from -3 to 3, y from -0.5 to 0.5  
\multiput {$\circ$} at 1   0 *2 1 0 /      %
\multiput {$\circ$} at -3   0 *2 1 0 /      
\put {$T_n$}     at 3 0.3   %
\put {$T_{n-1}$}     at 2 0.3   %
\put {$T_{n-2}$} at 1 0.3   %
\put {$T_2$}     at -1 0.3   %
\put {$T_1$}     at -2 0.3   %
\put {$T_0$}     at -3 0.3   
\linethickness=0.75pt                          
\putrule from -2.97 0.045 to -2.03 0.045       
\putrule from -2.97 -0.045 to -2.03 -0.045       %
\putrule from -1.95 0 to -1.05 0              %
\putrule from 1.05 0 to 1.95 0              %
\putrule from 2.03 0.045 to 2.97 0.045       
\putrule from 2.03 -0.045 to 2.97 -0.045       %
\setlinear
\setdashes <2mm,1mm>          %
\putrule from -0.95 0 to 0.95 0  
\endpicture
\qquad \qquad
q^\frac12 \in Z(\tilde \cB_n),
$$
\begin{equation}\label{Xcomm}
X^{\varepsilon_i}X^{\varepsilon_j}
= X^{\varepsilon_j}X^{\varepsilon_i},
\quad\hbox{for $i,j\in \{1, \ldots, n\}$,}
\qquad
T_n X^{\varepsilon_j} = X^{\varepsilon_j}T_n,\ \ 
\hbox{for $j\in\{1,\ldots, n-1\}$,}
\end{equation}
\begin{equation}\label{TX1}
T_0 X^{\varepsilon_j} = X^{\varepsilon_j} T_0,\ \ 
\hbox{for $j\in\{2,\ldots, n\}$.}
\end{equation}
\begin{equation}\label{TX2}
X^{\varepsilon_{i+1}} = T_iX^{\varepsilon_i}T_i
\quad\hbox{and}\quad
T_iX^{\varepsilon_j} = X^{\varepsilon_j}T_i,
\qquad\hbox{for $i\in\{1,\ldots, n-1\}$ and $j\not\in\{i, i+1\}$.}
\end{equation}
\end{thm}
\begin{proof}~\eqref{BdefrelsI} $\longrightarrow $ (a):
The relations~\eqref{Ycomm} are checked using the pictorial expressions
for $Y^{\varepsilon_j}$ and the additional
relations  in~\eqref{TY1} follow from the definition of $Y^{\varepsilon_j}$ and
the relation $T_0^{\vee} T_1^{-1}T_0 T_1 = T_1^{-1}T_0 T_1T_0^{\vee}.$

(a) $\longrightarrow $~\eqref{BdefrelsI}:  Since
$$T_0 = Y^{\varepsilon_1}T_1^{-1}\cdots T_{n-1}^{-1}T_n^{-1}T_{n-1}^{-1}\cdots T_1^{-1},
$$
the generators in~\eqref{BdefrelsI} can be written in terms of the generators in (a).
We need to show that the relations~\eqref{TY1} and~\eqref{TY2}
imply
$$T_0T_1T_0T_1=T_1T_0T_1T_0,
\quad
T_0^{\vee} T_1^{-1}T_0 T_1 = T_1^{-1}T_0 T_1T_0^{\vee},
\quad\hbox{and}\quad T_0T_i=T_iT_0,\ \ \hbox{for $i = 2,3,\ldots, n$.}
$$
Let
$$T_{s_{\varepsilon_1}} = 
\beginpicture
\setcoordinatesystem units <.5cm,.5cm>         
\setplotarea x from -4 to 5.5, y from -2 to 2    
\put{$\bullet$} at -3 0.75      %
\put{$\bullet$} at -2 0.75      %
\put{$\bullet$} at -1 0.75      %
\put{$\bullet$} at  0 0.75      
\put{$\bullet$} at  1 0.75      %
\put{$\bullet$} at  2 0.75      %
\put{$\bullet$} at  3 0.75      %
\put{$\bullet$} at -3 -0.75          %
\put{$\bullet$} at -2 -0.75          %
\put{$\bullet$} at -1 -0.75          %
\put{$\bullet$} at  0 -0.75          
\put{$\bullet$} at  1 -0.75          %
\put{$\bullet$} at  2 -0.75          %
\put{$\bullet$} at  3 -0.75          %
\plot 4.5 1.25 4.5 0.37 / 
\plot 4.5 0.13   4.5 -1.25 / 
\plot 4.25 1.25 4.25  0.37 / 
\plot 4.25 0.13 4.25 -1.25 / 
\ellipticalarc axes ratio 1:1 360 degrees from 4.5 1.25 center at 4.375 1.25 
\put{$*$} at 4.375 1.25 
\ellipticalarc axes ratio 1:1 180 degrees from 4.25 -1.25 center at 4.375 -1.25
\plot -6.0 1.25 -6.0 -1.25 / 
\plot -5.75 1.25 -5.75  -1.25 / 
\ellipticalarc axes ratio 1:1 360 degrees from -6.0 1.25 center at -5.875 1.25 
\put{$*$} at -5.875 1.25 
\ellipticalarc axes ratio 1:1 180 degrees from -6.0 -1.25 center at -5.875 -1.25
\plot -4.5 1.25 -4.5 -1.25 / 
\plot -4.25 1.25 -4.25 -1.25 /
\ellipticalarc axes ratio 1:1 360 degrees from -4.5 1.25 center at -4.375 1.25 
\put{$*$} at -4.375 1.25 
\ellipticalarc axes ratio 1:1 180 degrees from -4.5 -1.25 center at -4.375 -1.25
\plot -2 0.75 -2 0.38 /
\plot -2 0.12 -2 -0.75 /
\plot -1 0.75 -1 0.38 /
\plot -1 0.12 -1 -0.75 /
\plot 0 0.75 0 0.38 /
\plot 0 0.12 0 -0.75 /
\plot 1 0.75 1 0.38 /
\plot 1 0.12 1 -0.75 /
\plot 2 0.75 2 0.38 /
\plot 2 0.12 2 -0.75 /
\plot 3 0.75 3 0.38 /
\plot 3 0.12 3 -0.75 /
\setlinear
\plot -2.7 0.25  4.65 0.25 / 
\ellipticalarc axes ratio 2:1 180 degrees from 4.65 -0.25  center at 4.65 0 
\plot 4.1 -0.25  3.2 -0.25 /
\plot 2.8 -0.25 2.2 -0.25 /
\plot 1.8 -0.25 1.2 -0.25 /
\plot 0.8 -0.25 0.2 -0.25 /
\plot -0.2 -0.25 -0.8 -0.25 /
\plot -1.2 -0.25 -1.8 -0.25 /
\plot -2.2 -0.25 -2.7 -0.25 /
\setquadratic
\plot  -2.7 0.25  -2.95 .45  -3 0.75 / 
\plot  -2.7 -0.25  -2.95 -0.45  -3 -0.75 /
\endpicture
= T_1\cdots T_n \cdots T_1.$$
Using that $T_iT_{s_{\varepsilon_1}}^{-1} = T_{s_{\varepsilon_1}}^{-1}T_i$, for $i\in\{2,3,\ldots, n\}$,
$$T_0T_i = Y^{\varepsilon_1}T_{s_{\varepsilon_1}}^{-1}T_i 
= Y^{\varepsilon_1}T_iT_{s_{\varepsilon_1}}^{-1}
= T_iY^{\varepsilon_1}T_{s_{\varepsilon_1}}^{-1} = T_iT_0,
\qquad\hbox{for $i\in\{2,\ldots, n\}$.}$$

Using $T_0^{\vee} Y^{\varepsilon_2} = Y^{\varepsilon_2}T_0^{\vee}$ and
$Y^{\varepsilon_2} = T_1^{-1}Y^{\varepsilon_1}T_1^{-1}$ and
$T_0T_i = T_iT_0$ for $i\in\{2,\ldots, n\}$,
\begin{align*}
T_0^{\vee} T_1^{-1}T_0T_1
&= T_0^{\vee} T_1^{-1}Y^{\varepsilon_1} T_1^{-1}\cdots T_n^{-1}\cdots T_1^{-1}T_1 
= T_0 Y^{\varepsilon_2} T_2^{-1}\cdots T_n^{-1}\cdots T_2^{-1} \\
&=  Y^{\varepsilon_2}T_0^{\vee} T_2^{-1}\cdots T_n^{-1}\cdots T_2^{-1} 
=  T_1^{-1}Y^{\varepsilon_1} T_1^{-1}
 T_2^{-1}\cdots T_n^{-1}\cdots T_2^{-1} T_0^{\vee} \\
&= T_1^{-1}T_0T_1T_0^{\vee}.
\end{align*}

Since
$Y^{\varepsilon_1} = T_0T_1\cdots T_n \cdots T_1$
and
$Y^{\varepsilon_2} = T_1^{-1}Y^{\varepsilon_1}T_1^{-1} 
= T_1^{-1}T_0T_1\cdots T_n\cdots T_2$ then
\begin{align*}
Y^{\varepsilon_1+\varepsilon_2} 
&= T_0T_1\cdots T_n\cdots T_2T_0T_1\cdots T_n\cdots T_2 
= T_0T_1T_0T_2\cdots T_n \cdots T_2T_1\cdots T_n\cdots T_2.
\end{align*}
Then $T_1^{-1}Y^{\varepsilon_1+\varepsilon_2}
=T_1^{-1}Y^{\varepsilon_1}Y^{\varepsilon_2}
= T_1^{-1}Y^{\varepsilon_1}T^{-1}Y^{\varepsilon_1}T_1^{-1}
= Y^{\varepsilon_2}Y^{\varepsilon_1}T_1^{-1} = Y^{\varepsilon_2+\varepsilon_1}T_1^{-1}
$ gives
\begin{align*}
T_1Y^{\varepsilon_1+\varepsilon_2}
&= T_1T_0T_1T_0T_2\cdots T_n \cdots T_2 T_1\cdots T_n\cdots T_2 \\
=Y^{\varepsilon_1+\varepsilon_2}T_1
&= T_0T_1T_0 T_2\cdots T_n\cdots T_2 T_{s_{\varphi}}
= T_0T_1T_0 T_{s_{\varphi}} T_2\cdots T_n\cdots T_2 \\
&= T_0T_1T_0T_1T_2\cdots T_n\cdots T_2T_1\cdots T_n\cdots T_2,
\end{align*}
since $T_{s_\varphi} T_i = T_i T_{s_\varphi}$ for $i\in\{2,\ldots, n\}$.
Thus $T_1T_0T_1T_0 = T_0T_1T_0T_1$.

\smallskip\noindent
(b)  Let $\iota\colon \tilde \cB_n\longrightarrow  \tilde \cB_n$ be the automorphism of Proposition
\ref{dualityB}.
Since $\iota(Y^{\varepsilon_i}) = X^{\varepsilon_i}$,  applying the 
automorphism $\iota$ to the presentation in (a) gives the presentation in (b).
\end{proof}

\begin{remark}

Let $T_{s_{\varepsilon_1}} = T_1\cdots T_n\cdots T_1$ and 
\begin{equation}
T_0^\sharp = q^{-\frac12}X^{\varepsilon_1}T_0^{-1}
= q^{-\frac12}(T_0^{\vee})^{-1}T_{s_{\varepsilon_1}}^{-1} T_0^{-1}
=q^{-\frac12} (T_0^{\vee})^{-1}Y^{-\varepsilon_1}
\label{T0sharp}
\end{equation}
Then, as explained in~\cite[proof of Proposition 3.9]{IS03},
$$\beginpicture
\setcoordinatesystem units <1cm,1cm>        
\setplotarea x from -3 to 3, y from -0.5 to 0.5  
\multiput {$\circ$} at 1   0 *2 1 0 /      %
\multiput {$\circ$} at -3   0 *2 1 0 /      
\put {$T_n$}     at 3 0.4   %
\put {$T_{n-1}$}     at 2 0.4   %
\put {$T_{n-2}$} at 1 0.4   %
\put {$T_2$}     at -1 0.4   %
\put {$T_1$}     at -2 0.4   %
\put {$T^\sharp_0$}     at -3 0.4   
\linethickness=0.75pt                          
\putrule from -2.97 0.045 to -2.03 0.045       
\putrule from -2.97 -0.045 to -2.03 -0.045       %
\putrule from -1.95 0 to -1.05 0              %
\putrule from 1.05 0 to 1.95 0              %
\putrule from 2.03 0.045 to 2.97 0.045       
\putrule from 2.03 -0.045 to 2.97 -0.045       %
\setlinear
\setdashes <2mm,1mm>          %
\putrule from -0.95 0 to 0.95 0  
\endpicture
\qquad\hbox{and}\qquad
T_0^{\vee} T_0^\sharp T_0 T_{s_{\varepsilon_1}} = q^{-\frac12},
$$
since 
$$T_0^{\vee} T_0^\sharp T_0 T_{s_{\varepsilon_1}}
= T_0^{\vee} q^{-\frac12}X^{\varepsilon_1}T_{s_{\varepsilon_1}}
=q^{-\frac12} T_0^{\vee} (T_0^{\vee})^{-1}T_{s_{\varepsilon_1}}^{-1}T_{s_{\varepsilon_1}}
= q^{-\frac12}
$$
and
\begin{align*}
T_0^\sharp  T_1 T_0^\sharp T_1 
&= q^{-\frac12}X^{\varepsilon_1}T_0^{-1}T_1 q^{-\frac12}X^{\varepsilon_1}T_0^{-1} T_1 
= q^{-\frac12}X^{\varepsilon_1}T_0^{-1}T_1 q^{-\frac12}X^{\varepsilon_1}T_1T_1^{-1}T_0^{-1} T_1 \\
&= q^{-\frac12}X^{\varepsilon_1}T_0^{-1} q^{-\frac12}X^{\varepsilon_2}T_1^{-1}T_0^{-1} T_1 
= q^{-\frac12}X^{\varepsilon_1} q^{-\frac12}X^{\varepsilon_2}T_0^{-1} T_1^{-1}T_0^{-1} T_1 \\
&= q^{-\frac12}X^{\varepsilon_1} q^{-\frac12}X^{\varepsilon_2}T_1T_0^{-1} T_1^{-1}T_0^{-1} \\
&= T_1 q^{-\frac12}X^{\varepsilon_1} q^{-\frac12}X^{\varepsilon_2}T_0^{-1} T_1^{-1}T_0^{-1} 
= T_1 q^{-\frac12}X^{\varepsilon_1} T_0^{-1} q^{-\frac12}X^{\varepsilon_2} T_1^{-1}T_0^{-1} \\
&= T_1 q^{-\frac12}X^{\varepsilon_1} T_0^{-1} T_1 q^{-\frac12}X^{\varepsilon_1} T_1 T_1^{-1}T_0^{-1} 
= T_1 q^{-\frac12}X^{\varepsilon_1} T_0^{-1} T_1 q^{-\frac12}X^{\varepsilon_1} T_0^{-1} 
= T_1 T_0^\sharp T_1 T_0^\sharp.
\end{align*}

\end{remark}

\begin{remark}
The presentation of the double affine braid group given in~\eqref{BdefrelsI} and
\eqref{Bdefrels2I} is implicit in 
\cite[Cor.\ 4.18]{H06} and quite explicit in~\cite[Theorem 3.11]{IS03}.
\end{remark}

\subsection{The double affine Hecke algebra (DAHA) of type $CC_n$}

As in~\eqref{T0sharp}, let
$$T_0^\sharp = q^{-\frac12}X^{\varepsilon_1}T_0^{-1}
= q^{-\frac12}(T_0^{\vee})^{-1}T_1^{-1}\cdots T_n^{-1}\cdots T_1^{-1}T_0^{-1}
=q^{-\frac12} (T_0^{\vee})^{-1}Y^{-\varepsilon_1},
\quad\hbox{in $\widetilde{\cB}$.}
$$ 
Let $u_0^{\frac12}, u_n^{\frac12}, t_0^{\frac12}, t_n^{\frac12}, t^{\frac12}\in \CC^\times$.
The \emph{double affine Hecke algebra} $\widetilde{H}$ is the quotient of the group algebra of
$\widetilde{\cB}$ by
$$
(T_0^\sharp - u_0^{\frac12})(T_0^\sharp+u_0^{-\frac12}) = 0,
\qquad
(T_0^{\vee} - u_n^{\frac12})(T_0^{\vee} +u_n^{-\frac12}) = 0,
$$
\begin{equation}
(T_0 - t_0^{\frac12})(T_0+t_0^{-\frac12}) = 0,
\qquad
(T_n - t_n^{\frac12})(T_n +t_n^{-\frac12}) = 0,
\quad\hbox{and}
\label{DHeckedefnI}
\end{equation}
$$
(T_i - t^{\frac12})(T_i+t^{-\frac12}) = 0,\quad\hbox{for $i=1,\ldots, n-1$.}
$$

\subsubsection{The translation relations in DAHA}

Recall from~\eqref{XYdefn} that $Y^{\varepsilon_j}$ and $X^{\varepsilon_j}$ are defined as
\begin{equation*}
\begin{array}{lcl}
Y^{\varepsilon_1} = T_0T_1\cdots T_n\cdots T_1
&\qquad
\hbox{and}\qquad 
&Y^{\varepsilon_{j+1}} = T_j^{-1}Y^{\varepsilon_j}T_j^{-1},
\\
X^{\varepsilon_1} = (T_0)^{-1}T_1^{-1}\cdots T_n^{-1}\cdots T_1^{-1},
&\quad\hbox{and}\quad
&X^{\varepsilon_{j+1}}
= T_j X^{\varepsilon_j}T_j.
\end{array}
\end{equation*}
Define (a symbol $\delta$ such that)
$$q^{\frac12} = X^{\frac12\delta} \quad\hbox{and let}\quad
X_i = X^{\varepsilon_i},
\quad\hbox{so that}\quad
q^{\frac{k}{2}} X_1^{\mu_1}\cdots X_n^{\mu_n}
=X^{\frac{k}{2}\delta+\mu_1\varepsilon_1+\cdots+\mu_n\varepsilon_n}=X^\mu
$$
for $\mu = \mu_1\varepsilon_1+\cdots+\mu_n\varepsilon_n+\frac{k}{2}\delta
\in\ZZ\varepsilon_1+\cdots+\ZZ\varepsilon_n + \frac{1}{2}\ZZ \delta$.  
For $\mu = \mu_1\varepsilon_1+\cdots+\mu_n\varepsilon_n+\frac{k}{2}\delta$ define
(the level 0 action of $W_Y$ on $\fa^*_\ZZ+\frac12\ZZ\delta$)
\begin{align*}
s_0\mu
&=  
-\mu_1\varepsilon_1+\mu_2\varepsilon_2+\cdots+\mu_n\varepsilon_n
+\big(\hbox{$\frac{k}{2}$}+\mu_1\big)\delta
\\
s_n\mu
&=  \mu_1\varepsilon_1+\cdots+\mu_{n-1}\varepsilon_{n-1}-\mu_n\varepsilon_n
+\hbox{$\frac{k}{2}$}\delta
\qquad\hbox{and}
\\
s_i\mu
&=  \hbox{$\frac{k}{2}$}\delta+\mu_1\varepsilon_1+\cdots
+\mu_{i-1}\varepsilon_{i-1}+\mu_{i+1}\varepsilon_i+\mu_i\varepsilon_{i+1}
+\mu_{i+2}\varepsilon_{i+2}+\cdots +\mu_n\varepsilon_n,
\end{align*}
for $i\in \{1, \ldots, n-1\}$.
Similarly define (a symbol $K$ so that)
$$q^{\frac12} = Y^{-\frac12 K} \quad\hbox{and let}\quad
Y_i = Y^{\varepsilon_i},
\quad\hbox{so that}\quad
q^{-\frac{k}{2}} Y_1^{\lambda_1}\cdots Y_n^{\lambda_n}
=Y^{\frac{k}{2}K+\lambda_1\varepsilon_1+\cdots+\lambda_n\varepsilon_n}=Y^{\lambda},
$$
for $\lambda 
= \lambda_1\varepsilon_1+\cdots+\lambda_n\varepsilon_n+\frac{k}{2}K
\in \ZZ\varepsilon_1+\cdots+\ZZ\varepsilon_n+\frac{1}{2}\ZZ K$.  
For $\lambda 
= \lambda_1\varepsilon_1 +\cdots + \lambda_n\varepsilon_n+\frac{k}{2}K$ define
(the level 0 action of $W$ on $\fa_\ZZ+\frac12\ZZ K$)
\begin{align}
s_0 \lambda
&=  
-\lambda_1\varepsilon_1+\lambda_2\varepsilon_2+\cdots+\lambda_n\varepsilon_n
+\big(\hbox{$\frac{k}{2}$} + \lambda_1\big)K,
\nonumber \\
s_n\lambda
&=  \lambda_1\varepsilon_1+\cdots+\lambda_{n-1}\varepsilon_{n-1}
-\lambda_n\varepsilon_n +\hbox{$\frac{k}{2}$}K
\qquad\hbox{and}
\label{WveeactionI} 
\\
s_i\lambda
&=  \lambda_1\varepsilon_1 +\cdots
+\lambda_{i-1}\varepsilon_{i-1} +\lambda_{i+1}\varepsilon_i 
+\lambda_i\varepsilon_{i+1}
+\lambda_{i+2}\varepsilon_{i+2}+\cdots +\lambda_n\varepsilon_n
+\hbox{$\frac{k}{2}$} K,
\nonumber 
\end{align}
for $i\in \{1, \ldots, n-1\}$.

\begin{prop} The following relations hold in $\widetilde{H}$.
\begin{align}
T_0X^\mu 
&= X^{s_0\mu}T_0 + 
\frac{ (t_0^{\frac12}-t_0^{-\frac12})+(u_0^{\frac12}-u_0^{-\frac12})q^{\frac12}X^{-\varepsilon_1} 
}{1-qX^{-2\varepsilon_1}}(X^{\mu}-X^{s_0\mu}),
\nonumber \\
T_nX^\mu 
&= X^{s_n\mu}T_n + 
\frac{ (t_n^{\frac12}-t_n^{-\frac12})+(u_n^{\frac12}-u_n^{-\frac12})X^{\varepsilon_n}
}{1-X^{2\varepsilon_n}}(X^{\mu}-X^{s_n\mu}) \quad\hbox{and}
\label{DAHATXI} 
\\
T_i X^\mu 
&= X^{s_i\mu}T_i + \frac{t^{\frac12}-t^{-\frac12}}{1-X^{\varepsilon_i-\varepsilon_{i+1}} }
(X^\mu - X^{s_i\mu})
\qquad\hbox{for $i\in \{1, \ldots, n-1\}$,}
\nonumber
\end{align}
\begin{align}
T_0^{\vee} Y^{\lambda}
&= Y^{s_0 \lambda}T_0^{\vee} + \frac{
((u_n^{\frac12}- u_n^{-\frac12})+(u_0^{\frac12}-u_0^{-\frac12})Y^{-\alpha_0})
}{ 1- Y^{-2\alpha_0} }
(Y^{\lambda}-Y^{s_0 \lambda}),
\nonumber \\
T_nY^{\lambda} &= Y^{s_n\lambda}T_n + 
\frac{\big((t_n^{\frac12}-t_n^{-\frac12})+(t_0^{\frac12}-t_0^{-\frac12})Y^{-\varepsilon_n})}
{1-Y^{-2\varepsilon_n}}(Y^{\lambda}-Y^{s_n\lambda}),
\ \ \hbox{and}
\label{DAHATYI} 
\\
T_iY^{\lambda} &= Y^{s_i\lambda}T_i + (t^{\frac12}-t^{-\frac12})
\frac{Y^{\lambda}-Y^{s_i\lambda}}
{1-Y^{-(\varepsilon_i-\varepsilon_{i+1})}},
\ \ \hbox{for $i\in \{1,\ldots, n-1\}$.}
\nonumber
\end{align}
\end{prop}
\begin{proof}
If $\mu, \nu\in \fh_\ZZ^*$ and 
\begin{align*}
T_iX^\mu 
&= X^{s_i\mu}T_i + 
\frac{\big((t_i^{\frac12}-t_i^{-\frac12})+(u_i^{\frac12}-u_i^{-\frac12})X^{\alpha_i} \big)
}{1-X^{2\alpha_i} }(X^{\mu}-X^{s_i\mu}), \qquad\hbox{and} \\
T_iX^\nu 
&= X^{s_i\nu}T_i + 
\frac{\big((t_i^{\frac12}-t_i^{-\frac12})+(u_i^{\frac12}-u_i^{-\frac12})X^{\alpha_i} \big)
}{1-X^{2\alpha_i} }(X^{\nu}-X^{s_i\nu}),
\end{align*}
then
\begin{align}
T_iX^{\mu+\nu}
&= T_iX^\mu X^\nu
= \left(X^{s_i\mu}T_i + 
\frac{\big((t_i^{\frac12}-t_i^{-\frac12})+(u_i^{\frac12}-u_i^{-\frac12})X^{\alpha_i} \big)
}{1-X^{2\alpha_i} }(X^{\mu}-X^{s_i\mu}) \right) X^{\nu} \nonumber \\
&= X^{s_i\mu}
\left(
X^{s_i\nu}T_i + 
\frac{\big((t_i^{\frac12}-t_i^{-\frac12})+(u_i^{\frac12}-u_i^{-\frac12})X^{\alpha_i} \big)
}{1-X^{2\alpha_i} }(X^{\nu}-X^{s_i\nu}) \right)  \nonumber \\
&\qquad
+\frac{\big((t_i^{\frac12}-t_i^{-\frac12})+(u_i^{\frac12}-u_i^{-\frac12})X^{\alpha_i} \big)
}{1-X^{2\alpha_i} }(X^{\mu}-X^{s_i\mu}) X^{\nu}  \nonumber \\
&= X^{s_i(\mu+\nu)}T_i
+\frac{\big((t_i^{\frac12}-t_i^{-\frac12})+(u_i^{\frac12}-u_i^{-\frac12})X^{\alpha_i} \big)
}{1-X^{2\alpha_i} }(X^{\mu+\nu} +X^{s_i\mu+\nu}-X^{s_i\mu+\nu}-X^{s_i(\mu+\nu)})
\nonumber \\
&= X^{s_i(\mu+\nu)}T_i
+\frac{\big((t_i^{\frac12}-t_i^{-\frac12})+(u_i^{\frac12}-u_i^{-\frac12})X^{\alpha_i} \big)
}{1-X^{2\alpha_i} }(X^{\mu+\nu} - X^{s_i(\mu+\nu)}).
\label{*}
\end{align}

By~\eqref{TX1}, $T_0X^{\varepsilon_j} = X^{\varepsilon_j}T_0$
for $j\in \{2, \dots, n\}$.  Since $q^{\frac12}\in Z(\tilde H)$ then $T_0 q^{\frac12} = q^{\frac12}T_0$.
Following~\cite[p.\,81]{Mac03}:
Since $T_0' = q^{-\frac12}X^{\varepsilon_1}T_0^{-1}$ then
\begin{align*}
T_0X^{\alpha_0} - X^{-\alpha_0}T_0 
&= T_0q^{\frac12}X^{-\varepsilon_1}-q^{-\frac12}X^{\varepsilon_1}T_0 \\
&= T_0q^{\frac12}X^{-\varepsilon_1}-q^{-\frac12}X^{\varepsilon_1}(T_0^{-1}+t_0^{\frac12}-t_0^{-\frac12}) \\
&= (T_0')^{-1}-T_0'-(t_0^{\frac12}-t_0^{-\frac12})q^{-\frac12}X^{\varepsilon_1} \\
&= -(u_0^{\frac12} - u_0^{-\frac12})-(t_0^{\frac12}-t_0^{-\frac12})q^{-\frac12}X^{\varepsilon_1} \\
&= \big((t_0^{\frac12}-t_0^{-\frac12})+(u_0^{\frac12}-u_0^{-\frac12})q^{\frac12}X^{-\varepsilon_1}\big)
\frac{q^{\frac12}X^{-\varepsilon_1}-q^{-\frac12}X^{\varepsilon_1}}{1-qX^{-2\varepsilon_1}} \\
&= \big((t_0^{\frac12}-t_0^{-\frac12})+(u_0^{\frac12}-u_0^{-\frac12})X^{\alpha_0}\big)
\frac{X^{\alpha_0}-X^{-\alpha_0}}{1-X^{2\alpha_0}},
\end{align*}
The elements $\frac12\delta, \alpha_0, \varepsilon_2, \ldots, \varepsilon_n$
form a basis of $\fh_\ZZ^* = \frac12 \ZZ\delta+\ZZ\varepsilon_1+\cdots+\ZZ\varepsilon_n$
and so the first relation in~\eqref{DAHATXI} then follows from~\eqref{*}.

By the last relation in~\eqref{Xcomm}, $T_nX^{\varepsilon_j} = X^{\varepsilon_j} T_n$
for $j\in \{1, \ldots, n-1\}$.
Since
$$X^{-\varepsilon_n}T_n^{-1} = T_n\cdots T_1 T_0 T_1^{-1}\cdots T_n^{-1}
\qquad\hbox{is conjugate to $T_0$}
$$
then, by the second relation in~\eqref{DHeckedefnI},
$X^{-\varepsilon_n}T_n^{-1} 
- (X^{-\varepsilon_n}T_n^{-1})^{-1} = u_n^{\frac12}-u_n^{-\frac12}.$
Thus,
\begin{align*}
T_nX^{\alpha_n} - X^{-\alpha_n}T_n 
&= T_nX^{\varepsilon_n}-X^{-\varepsilon_n}T_n \\
&= T_nX^{\varepsilon_n}-X^{-\varepsilon_n}(T_n^{-1}+t_n^{\frac12}-t_n^{-\frac12}) \\
&= -(u_n^{\frac12} - u_n^{-\frac12})-(t_n^{\frac12}-t_n^{-\frac12})X^{-\varepsilon_n} \\
&= \big((t_n^{\frac12}-t_n^{-\frac12})+(u_n^{\frac12}-u_n^{-\frac12})X^{\varepsilon_n}\big)
\frac{X^{\varepsilon_n}-X^{-\varepsilon_n}}{1-X^{2\varepsilon_n}}.
\end{align*}
The elements $\frac12\delta, \alpha_0, \varepsilon_2, \ldots, \varepsilon_n$
form a basis of $\fh_\ZZ^* = \frac12 \ZZ\delta+\ZZ\varepsilon_1+\cdots+\ZZ\varepsilon_n$
and so the second relation in~\eqref{DAHATXI} then follows from~\eqref{*}.

The proof of part (b) is similar using the following computations.
Since $Y^{\alpha_0} = q^{-\frac12}Y^{-\varepsilon_1}$
and
$T_0^\# = (T_0^{\vee})^{-1}q^{-\frac12}Y^{-\varepsilon_1}$ then
\begin{align*}
T_0^{\vee} Y^{\alpha_0} - Y^{-\alpha_0}T_0^{\vee} 
&= \big((T_0^{\vee})^{-1} + (u_n^{\frac12}-u_n^{-\frac12})\big) q^{-\frac12}Y^{-\varepsilon_1}
-q^{\frac12}Y^{\varepsilon_1}T_0^{\vee} \\
&= (T_0^{\vee})^{-1}q^{-\frac12}Y^{-\varepsilon_1} - q^{\frac12}Y^{\varepsilon_1}T_0^{\vee}
+(u_n^{\frac12} - u_n^{-\frac12})Y^{\alpha_0} \\
&= (u_0^{\frac12}-u_0^{-\frac12}) + ( u_n^{\frac12} - u_n^{-\frac12})Y^{\alpha_0} \\
&=\big( (u_0^{\frac12} - u_0^{-\frac12})Y^{-\alpha_0} + (u_n^{\frac12} - u_n^{-\frac12})\big)
\frac{Y^{\alpha_0} - Y^{-\alpha_0}}{1-Y^{-2\alpha_0}},
\end{align*}
and, since
$T_n^{-1}Y^{\varepsilon_n} = T_n^{-1}T_{n-1}^{-1}\cdots T_1^{-1}T_0T_1\cdots T_n$
is conjugate to $T_0$ then, by the third relation in~\eqref{DHeckedefnI},
$$T_n^{-1}Y^{\varepsilon_n} - (T_n^{-1}Y^{\varepsilon_n})^{-1}
=t_0^{\frac12}- t_0^{-\frac12}, \qquad\hbox{so that}
$$
\begin{align*}
T_nY^{\alpha_n} - Y^{-\alpha_n}T_n 
&= T_nY^{\varepsilon_n}-Y^{-\varepsilon_n}T_n \\
&= (T_n^{-1}+(t_n^{\frac12}-t_n^{-\frac12}))Y^{\varepsilon_n} - Y^{-\varepsilon_n}T_n \\
&= (T_n^{-1}Y^{\varepsilon_n}) - (T_n^{-1}Y^{\varepsilon_n})^{-1} 
+ (t_n^{\frac12}-t_n^{-\frac12})Y^{\varepsilon_n} \\
&= (t_0^{\frac12}-t_0^{-\frac12}) 
+ (t_n^{\frac12}-t_n^{-\frac12})Y^{\varepsilon_n} \\
&= \big( (t_0^{\frac12} - t_0^{-\frac12})Y^{-\varepsilon_n} + (t_n^{\frac12}-t_n^{-\frac12})\big)
\frac{Y^{\varepsilon_n} - Y^{-\varepsilon_n}}{1-Y^{-2\varepsilon_n}}.
\qquad\qquad\qquad\qquad\qquad\qquad\qedhere
\end{align*}
\end{proof}

\subsubsection{Intertwiners for type $CC_n$}

Recall the notations for $c$-functions $c^Y_{\alpha}$ defined in~\eqref{cfcndefn} and for fold functions $F^+_{\alpha_i}$ and $F^-_{\alpha_i}$ defined in~\eqref{Fdefn}, related by
\begin{equation*}
F_{\alpha}^+ = t_{\alpha}^{-\frac12} - c_{\alpha}^Y
\quad\hbox{and}\quad
F_{\alpha}^- = t_{\alpha}^{\frac12} - c_{\alpha}^Y.
\end{equation*}
The \emph{intertwiner operators} are
\begin{equation}
\begin{array}{r@{\;}c@{\;}l} 
\tau_{0} &=&  T_{0}^{\vee} + F^+_{\alpha_i}\\
&=& (T_{0}^{\vee})^{-1} +  F^-_{\alpha_i},
\end{array}
\qquad\text{and}\qquad
\begin{array}{r@{\;}c@{\;}l}
\tau_i &=& T_i + F^+_{\alpha_i} \\
&=& (T_i)^{-1} +  F^-_{\alpha_i},
\end{array}
\qquad\hbox{for $i\in \{1, \ldots, n\}$.}
\label{intdefn}
\end{equation}

\begin{prop} \label{taupastY}
The operators $\tau_0, \ldots, \tau_n$ satisfy the relations
$$
\beginpicture
\setcoordinatesystem units <1cm,1cm>        
\setplotarea x from -3 to 3, y from -0.5 to 0.5  
\multiput {$\circ$} at 1   0 *2 1 0 /      %
\multiput {$\circ$} at -3   0 *2 1 0 /      
\put {$\tau_n$}     at 3 0.3   %
\put {$\tau_{n-1}$}     at 2 0.3   %
\put {$\tau_{n-2}$} at 1 0.3   %
\put {$\tau_2$}     at -1 0.3   %
\put {$\tau_1$}     at -2 0.3   %
\put {$\tau_0$}     at -3 0.3   
\linethickness=0.75pt                          
\putrule from -2.97 0.045 to -2.03 0.045       
\putrule from -2.97 -0.045 to -2.03 -0.045       %
\putrule from -1.95 0 to -1.05 0              %
\putrule from 1.05 0 to 1.95 0              %
\putrule from 2.03 0.045 to 2.97 0.045       
\putrule from 2.03 -0.045 to 2.97 -0.045       %
\setlinear
\setdashes <2mm,1mm>          %
\putrule from -0.95 0 to 0.95 0  
\endpicture
,
\qquad
\tau_i Y^{\lambda} = Y^{s_i\lambda}\tau_i
\qquad \hbox{and}\qquad
(\tau_i)^2 
=c^Y_{-\alpha_i}c^Y_{\alpha_i},
$$
for $i\in \{0, 1, \ldots, n\}$ and $\lambda\in \ZZ\varepsilon_1+\cdots \ZZ\varepsilon_n+\frac12\ZZ K$.
\end{prop}
\begin{proof}
In terms of $c$-functions, the relations in~\eqref{DAHATYI} are written in the form
\begin{equation}
T_{\alpha_i} Y^{\lambda} 
= Y^{s_i\lambda} T_{\alpha_i} + (c_{\alpha_i}^Y-t_{\alpha_i}^{-\frac12}) 
(Y^{\lambda} - Y^{s_i \lambda} )
\label{cfcnTpastY}
\end{equation}
The relation~\eqref{cfcnTpastY} is equivalent to
$
(T_{\alpha_i} +  (t_{\alpha_i}^{-\frac12}- c_{\alpha_i}^Y)) Y^{\lambda}
= Y^{s_i\lambda} (T_{\alpha_i} + (t_{\alpha_i}^{-\frac12} - c_{\alpha_i}^Y))
$
and
$$
\tau_i Y^{\lambda} = Y^{s_i\lambda}\tau_i.
$$
Using
$F^-_{\alpha_i}+F^+_{-\alpha_i} = t^{-\frac12} - c^Y_{\alpha_i} +t^{\frac12} - c^Y_{-\alpha_i} = 0$ gives
\begin{align*}
(\tau_i)^2 
&= \tau_i (T_{\alpha_i}+F^+_{\alpha_i})
= \tau_i T_i + F^+_{s_i\alpha_i} \tau_i 
=(T_i^{-1}+F^-_{\alpha_i})T_i + F^+_{-\alpha_i} \tau_i \\
&=1+F^-_{\alpha_i}T_i + F^+_{-\alpha_i}
(T_i + F^+_{\alpha_i})
=1+0\cdot T_i+(t^{\frac12}_{\alpha_i}-c^Y_{-\alpha_i})(t^{\frac12}_{\alpha_i}-c^Y_{\alpha_i})
\\
&=1+(t_{\alpha}-t^{\frac12}_{\alpha_i}(c^Y_{-\alpha_i}+c^Y_{\alpha_i})+
c^Y_{-\alpha_i}c^Y_{\alpha_i})
=1+t_{\alpha}-t^{\frac12}_{\alpha_i}(t^{\frac12}_{\alpha_i}+t^{-\frac12}_{\alpha_i})
+c^Y_{-\alpha_i}c^Y_{\alpha_i}
\\
&=c^Y_{-\alpha_i}c^Y_{\alpha_i}.
\end{align*}
There are two standard ways to prove the braid relations: by direct check or by using the polynomial representation
and the action of the $\tau_i$ on the $E_\mu$.  For example, checking that
$\tau_0\tau_1\tau_0\tau_1 = \tau_1\tau_0\tau_1\tau_0$ can be checked
by expanding the left side and the right side and confirming that they are equal as in~\cite[(5.5.2) and (5.6.4)]{Mac03}.
Alternatively, use that the polynomial representattion of $\widetilde H$ is faithful and that for a generic $\mu\in \ZZ^n$,
$\tau_0\tau_1\tau_0\tau_1 E_\mu = (\mathrm{const})E_{s_0s_1s_0s_1\mu} = (\mathrm{const})E_{s_1s_0s_1s_0\mu} =
\tau_1\tau_0\tau_1\tau_0E_\mu$.
\end{proof}

\subsection{The hexagon basis lemma}

The following Proposition is used for executing the alcove walk method of computing $E_\mu$
(see~\cite[Proposition 5.3]{GR21} for the type $GL_n$ case).  Let
$$\overline{s_0} = s_1\cdots s_n \cdots s_1\in W_{\mathrm{fin}} 
$$
so that $\overline{s_0}$ is the signed permutation that switches $1$ and $-1$.

\begin{prop} \label{xzsicross}
Let $\mu\in \ZZ^n$ and $w\in W_{\mathrm{fin}}$.  Let $i\in \{1, \ldots, n\}$.  Then
$$X^\mu T_{ws_i} = \begin{cases}
X^\mu T_w T_i, &\hbox{if $\ell(ws_i)>\ell(w)$,} \\
X^\mu T_w T_i^{-1}, &\hbox{if $\ell(ws_i)<\ell(w)$,} 
\end{cases}
\quad\hbox{and}\quad
\begin{cases}
X^\mu T_w(T_0^{\vee})^{-1} = X^\mu X_j T_{w\overline{s_0}}, &\hbox{if $w(1) = j$,} \\
X^\mu T_wT_0^{\vee} = X^\mu X_j^{-1} T_{w\overline{s_0}}, &\hbox{if $w(1) = -j$.} 
\end{cases}
$$
\end{prop}
\begin{proof}
The first equality follows from the fact that
$$\hbox{if $w\in W_{\mathrm{fin}}$ and $\ell(ws_i)>\ell(w)$ then}\quad T_{ws_i} = T_w T_i .
$$
For the second equality:  For $j\in \{1, \ldots, n\}$ let
$$b_j = s_{j-1}\cdots s_2s_1 
\quad\hbox{and}\quad
b_{-j} = s_j\cdots s_n\cdots s_2s_1
\quad\hbox{and}\quad
\overline{s_0} = b_{-j}^{-1}b_j = b_j^{-1}b_{-j} 
$$
If $W'_{\mathrm{fin}}$ is the subgroup of $W_{\mathrm{fin}}$ generated by $\{s_2, \ldots, s_n\}$ then
$b_1, \ldots, b_n, b_{-n}, \ldots, b_{-1}$ are the minimal length representatives of the cosets in $W_{\mathrm{fin}}/W'_{\mathrm{fin}}$.
From the relations in the last diagram of~\eqref{Bdefrels2I},
$$
\hbox{If $z\in W'_{\mathrm{fin}}$\quad then}\quad
\overline{s_0}z = z\overline{s_0}
\quad\hbox{and}\quad
T_zT_0^{\vee} = T_0^{\vee} T_z.
$$
Recall from~\eqref{XYdefn} that
$$X_j
= T_{j-1}\cdots T_1(T_0^{\vee})^{-1}T_1^{-1}\cdots T_n^{-1}\cdots T_j^{-1}
\qquad\hbox{so that}\qquad
X_j = T_{b_j}(T_0^{\vee})^{-1}T_{b_{-j}}^{-1}.
$$
Let $j\in \{1, \ldots, n\}$.  If $w(1) = j$ then $w = b_j z$ with $z\in W'_{\mathrm{fin}}$.  
Using that $w \overline{s_0} = b_j z \overline{s_0}
=b_j \overline{s_0} z = b_jb_j^{-1}b_{-j}z = b_{-j}z$ gives
$$T_w(T_0^{\vee})^{-1} = T_{b_j} T_z (T_0^{\vee})^{-1} = T_{b_j}(T_0^{\vee})^{-1}T_z = X_jT_{b_{-j}}T_z = X_jT_{w \overline{s_0} }.
$$
If $w(1)=-j$ then $w=b_{-j}z$ with $z\in W'_{\mathrm{fin}}$.  
Using that $w \overline{s_0} = b_{-j}z \overline{s_0}
=b_{-j} \overline{s_0} z = b_{-j}b_{-j}^{-1}b_j z = b_j z$ gives
$$T_wT_0^{\vee} = T_{b_{-j}}T_zT_0^{\vee} = T_{b_{-j}}T_0^{\vee} T_z  = X_j^{-1}T_{b_j}T_z = X_j^{-1}T_{w \overline{s_0}}.
$$
\end{proof}

\newpage

\section{Appendix: The CMW Example}\label{sec: examples}
\allowdisplaybreaks

The following Proposition extends Proposition~\ref{0gap} to include the case that $i=0$ and $\mu_m=1$.
This result will allow us to give an expansion of $\widehat{E}^y_\mu$ for the cases treated in 
\cite[Theorem 1.10]{CMW23}.

\begin{prop} \label{0tomgapB} Let $m\in \{1, \ldots, n\}$ and $\beta = -\varepsilon_1-\varepsilon_m+K$.
Let $\mu = (\mu_1, \ldots \ldots, \mu_n)\in \ZZ^n$ be
such that $\mu_1=\cdots=\mu_{m-1}=0$ and $\mu_m=1$ 
and let $\nu =s_0s_1\cdots  s_{m-1}\mu$ so that
$$
\mu = (0,0,\ldots, 0, 1, \mu_{m+1}, \ldots, \mu_n)
\qquad\hbox{and}\qquad
\nu = (0, 0, \ldots, 0, \mu_{m+1}, \ldots, \mu_n).
$$
Let $y\in W_{\mathrm{fin}}$ with 
$$y(1)\prec y(2)\prec \cdots \prec y(m-1),$$
where the order $\prec$ is given by $1\prec 2\prec \cdots \prec n \prec -n \prec \cdots \prec -2 \prec -1$.  
Then
\begin{align*}
\widehat{E}^y_\mu 
&= G_{0,m}^y \widehat{E}^y_\nu
+G_{1,m}^y x_{y(1)} \widehat{E}^{ys_{-1,1}}_\nu
+G_{2,m}^y x_{y(2)} \widehat{E}^{ys_{-2,2}}_\nu
+\cdots
+G_{m,m}^y x_{y(m)} \widehat{E}^{ys_{-m,m}}_\nu
\end{align*}
where, if $j\in \{0, \ldots, m-1\}$ is such that
$y(1)\prec \ldots \prec y(j)\prec y(m)\prec y(j+1)\prec \cdots \prec y(m-1)$ then 
$$
G_{i,m}^y = \begin{cases}
t^{-\frac12(m-2i+2\Card\{r\in \{1, \ldots, i-1\}\ |\ y(r)\prec y(-i)\})} \ev_\nu(F^+_\beta),
&\hbox{if $i\in \{1, \ldots, j-1\}$,} \\
t^{-\frac12(m-2i+2\Card\{r\in \{1, \ldots, i-1\}\ |\ y(r)\prec y(-i)\})} \ev_\nu(F^-_\beta),
&\hbox{if $i\in \{j, \ldots, m-1\}$,} \\
t^{-\frac12(1+m-2i+2\Card\{r\in \{1, \ldots, m-1\}\ |\ y(r)\prec y(-i)\})}, &\hbox{if $i=m$,}
\end{cases}
$$
and if $\alpha_0^\vee = -\varepsilon_1+K$ and $k\in \{1, \ldots, m-1\}$ is minimal such that $y(k)<0$, or $k = 0$ if $y(1), \ldots, y(m-1)$ are all positive, then if $k < j$:
\begin{align*}
G_{0,m}^y 
&= \ev_\nu(F^-_\beta F^+_{\alpha^\vee_0}) t^{-\frac{1}{2}(m-2)} \frac{t^{m-1} - t^{j-1}}{t-1}
+
\ev_\nu(F^+_\beta F^+_{\alpha^\vee_0}) t^{-\frac{1}{2}(m-2)} \frac{t^{j-1} - t^{k}}{t-1} \\
&\qquad\quad 
+
\ev_\nu(F^+_\beta F^-_{\alpha^\vee_0})
t^{-\frac{1}{2}(m-2)} \frac{t^{k} - 1}{t-1}
+
\begin{cases}
\ev_\nu(F^-_{\alpha^\vee_0})t^{\frac12(m-1)}, &\hbox{if $y(m)<0$,} \\
\ev_\nu(F^+_{\alpha^\vee_0})t^{\frac12(m-1)}, &\hbox{if $y(m)>0$,} 
\end{cases}
\end{align*}
and if $k \ge j$:
\begin{align*}
G_{0,m}^y 
&= \ev_\nu(F^-_\beta F^+_{\alpha^\vee_0}) t^{-\frac{1}{2}(m-2)} \frac{t^{m-1} - t^{k}}{t-1}
+
\ev_\nu(F^-_\beta F^-_{\alpha^\vee_0}) t^{-\frac{1}{2}(m-2)} \frac{t^{k} - t^{j-1}}{t-1} \\
&\qquad\quad 
+
\ev_\nu(F^-_\beta F^-_{\alpha^\vee_0})
t^{-\frac{1}{2}(m-2)} \frac{t^{j-1} - 1}{t-1}
+
\begin{cases}
\ev_\nu(F^-_{\alpha^\vee_0})t^{\frac12(m-1)}, &\hbox{if $y(m)<0$,} \\
\ev_\nu(F^+_{\alpha^\vee_0})t^{\frac12(m-1)}, &\hbox{if $y(m)>0$.} 
\end{cases}
\end{align*}
\end{prop}
\begin{proof} 
Using $\mu= (0, \ldots, 0, 1, \mu_{m+1}, \ldots, \mu_n)$ and $s_0\nu = (1, 0, \ldots, 0, \mu_{m+1}, \ldots, \mu_n)$
Proposition \ref{0gap} (across the zero gap) gives
\begin{align*}
\widehat{E}^y_\mu 
&= \widehat{E}^{ys_{m-1}\cdots s_1}_{s_0\nu} + 
\ev_{s_0\nu}(F^-_{\varepsilon^\vee_1-\varepsilon^\vee_m}) 
(
\widehat{E}^{ys_{m-2}\cdots s_1}_{s_0\nu}
+t^{-\frac12}\widehat{E}^{ys_{m-3}\cdots s_1}_{s_0\nu}
+\cdots + t^{-\frac12(m-j-1)} \widehat{E}^{ys_{j-1}\cdots s_1}_{s_0\nu}) 
\\
&\qquad
+\ev_{s_0\nu}(F^+_{\varepsilon^\vee_1-\varepsilon^\vee_m}) 
(t^{-\frac{1}{2}(m-j)} \widehat{E}^{ys_{j-2}\cdots s_1}_{s_0\nu}+\cdots 
+ t^{-\frac12(m-3)} \widehat{E}^{ys_1}_{s_0\nu}
+ t^{-\frac12(m-2)} \widehat{E}^y_{s_0\nu}) \\
&= \widehat{E}^{ys_{m-1}\cdots s_1}_{s_0\nu} + 
\ev_{\nu}(F^-_{-\varepsilon^\vee_1-\varepsilon^\vee_m + K}) 
(
\widehat{E}^{ys_{m-2}\cdots s_1}_{s_0\nu}
+t^{-\frac12}\widehat{E}^{ys_{m-3}\cdots s_1}_{s_0\nu}
+\cdots + t^{-\frac12(m-j-1)} \widehat{E}^{ys_{j-1}\cdots s_1}_{s_0\nu}) 
\\
&\qquad
+\ev_{\nu}(F^+_{-\varepsilon^\vee_1-\varepsilon^\vee_m + K}) 
(t^{-\frac{1}{2}(m-j)} \widehat{E}^{ys_{j-2}\cdots s_1}_{s_0\nu}+\cdots 
+ t^{-\frac12(m-3)} \widehat{E}^{ys_1}_{s_0\nu}
+ t^{-\frac12(m-2)} \widehat{E}^y_{s_0\nu})
\end{align*}
To shorten the notation, let
$$b^\pm = \ev_{\nu}(F^\pm_{-\varepsilon^\vee_1-\varepsilon^\vee_m+K}), \qquad
a^\pm = \ev_\nu(F^\pm_{-\varepsilon_1+K})
\qquad\hbox{and}\qquad
a_m = \begin{cases}
a^-, &\hbox{if $y(m)<0$,} \\
a^+, &\hbox{if $y(m)<0$.}
\end{cases}
$$
Assume that $k<j$.  Using~\eqref{lvl0onaZ} and Proposition~\ref{alcwksteps}, 
\begin{align*}
\widehat{E}^y_\mu 
&= x_{y(m)}\widehat{E}^{ys_{m-1}\cdots s_1s_{-1,1}}_{\nu} 
\\
&\qquad
+ b^- \left( \begin{array}{l}
x_{y(m-1)}\widehat{E}^{ys_{m-2}\cdots s_1s_{-1,1}}_{\nu}
+x_{y(m-2)}t^{-\frac12}\widehat{E}^{ys_{m-3}\cdots s_1s_{-1,1}}_{\nu}
+\cdots 
\\
\quad\cdots + x_{y(j)}t^{-\frac{1}{2}(m-j-1)} \widehat{E}^{ys_{j-1}\cdots s_1s_{-1,1}}_{\nu} 
\end{array}\right)
\\
&\qquad
+ b^+ \left(\begin{array}{l}
x_{y(j-1)}t^{-\frac{1}{2}(m-j)} \widehat{E}^{ys_{j-2}\cdots s_1s_{-1,1}}_{\nu}+\cdots \\
\quad\cdots
+ x_{y(2)}t^{-\frac12(m-3)} \widehat{E}^{ys_1s_{-1,1}}_{\nu}
+ x_{y(1)} t^{-\frac12(m-2)} \widehat{E}^{ys_{-1,1}}_{\nu}
\end{array}\right)
\\
&\qquad
+\left(
\begin{array}{l}
a_m \widehat{E}^{ys_{m-1}\cdots s_1}_{\nu} 
\\
+b^-a^+\widehat{E}^{ys_{m-2}\cdots s_1}_{\nu}
+b^-a^+ t^{-\frac12}\widehat{E}^{ys_{m-3}\cdots s_1}_{\nu}
+\cdots  + b^-a^+ t^{-\frac12(m-j-1)} \widehat{E}^{ys_{j-1}\cdots s_1}_{\nu} 
\\
+
t^{-\frac12 (m-j)} b^+a^+\widehat{E}^{ys_{j-2}\cdots s_1}_{\nu}+\cdots 
+t^{-\frac12(m-k-2)}b^+a^+\widehat{E}^{ys_k\cdots s_1}_\nu
\\
+t^{-\frac12(m-k-1)}b^+a^-\widehat{E}^{ys_{k-1}\cdots s_1}_\nu
+\cdots
+ t^{-\frac12(m-3)} b^+a^- \widehat{E}^{ys_1}_{\nu}
+ t^{-\frac12(m-2)} b^+a^- \widehat{E}^y_{\nu} 
\end{array}
\right),
\end{align*}
where we use the convention that $x_{-i} = x_{i}^{-1}$.  
If $i-1\in \{1, \ldots, m-1\}$ then $ys_{i-1}\cdots s_1s_{-1,1} = ys_{-i,i}s_{i-1}\cdots s_1$ and
\[
\ell(ys_{-i,i}s_{i-1}\cdots s_1)
=
\ell(ys_{-i,i}) + i - 1 - 2\Card\{r\in \{1, \ldots, i-1\}\ |\ y(r)\prec y(-i)\})
\]
so therefore
\begin{align*}
t^{-\frac12(m-2-i+1)} \widehat{E}^{ys_{i-1}\cdots s_1s_{-1,1}}_\nu
&= t^{-\frac12(m-2-i+1)} \widehat{E}^{ys_{-i,i}s_{i-1}\cdots s_1}_\nu
\\
&=t^{-\frac12(m-2-i+1-(i-1-2\Card\{r\in \{1, \ldots, i-1\}\ |\ y(r)\prec y(-i)\}))}\widehat{E}^{ys_{-i,i}}_\nu
\\
&=t^{-\frac12(m-2i+2\Card\{r\in \{1, \ldots, i-1\}\ |\ y(r)\prec y(-i)\})}\widehat{E}^{ys_{-i,i}}_\nu,
\end{align*}
proving the formula for $G^{y}_{i, m}$.  
For the final coefficient, since $\ell(ys_{r-1}\cdots s_1) = r-1+\ell(y)$ then
\[
\widehat{E}^{ys_{r-1}\cdots s_1}_\nu
= t^{\frac12(r-1)}\widehat{E}^{y}_\nu,
\]
and therefore
\begin{multline*}
\left(
\begin{array}{l}
a_m \widehat{E}^{ys_{m-1}\cdots s_1}_{\nu} 
\\
+b^-a^+\widehat{E}^{ys_{m-2}\cdots s_1}_{\nu}
+b^-a^+ t^{-\frac12}\widehat{E}^{ys_{m-3}\cdots s_1}_{\nu}
+\cdots  + b^-a^+ t^{-\frac12(m-j-1)} \widehat{E}^{ys_{j-1}\cdots s_1}_{\nu} 
\\
+
t^{-\frac12 (m-j)} b^+a^+\widehat{E}^{ys_{j-2}\cdots s_1}_{\nu}+\cdots 
+t^{-\frac12(m-k-2)}b^+a^+\widehat{E}^{ys_k\cdots s_1}_\nu
\\
+t^{-\frac12(m-k-1)}b^+a^-\widehat{E}^{ys_{k-1}\cdots s_1}_\nu
+\cdots
+ t^{-\frac12(m-3)} b^+a^- \widehat{E}^{ys_1}_{\nu}
+ t^{-\frac12(m-2)} b^+a^- \widehat{E}^y_{\nu} 
\end{array}
\right)  \\
=
\left(
a_m t^{\frac{1}{2}(m-1)} 
+
b^-a^+ t^{-\frac{1}{2}(m-2)} \frac{t^{m-1} - t^{j-1}}{t-1}
+
b^+a^+ t^{-\frac{1}{2}(m-2)} \frac{t^{j-1} - t^{k}}{t-1}
+
b^+a^-
t^{-\frac{1}{2}(m-2)} \frac{t^{k} - 1}{t-1}
\right) \widehat{E}^{y}_{\nu}.
\end{multline*}
This proves the formula for $G^{y}_{0, m}$ when $k < j$, and the case $k \ge j$ follows from a similar argument.  
\end{proof}

\begin{remark} In the notation of Proposition~\ref{0tomgapB}, explicit formulas for the $G_{i,m}^z$ are provided by substituting
\begin{align*}
\ev_\nu(F^+_\beta)
&=\ev_\nu\Big( \frac{t^{-\frac12}(1-t)}{1-qY_1Y_m}\Big)
=\frac{t^{-\frac12}(1-t)}{1-qt^{2n-1-m-2\ell}t_0t_n}, 
\\
\ev_\nu(F^-_\beta)
&=\ev_\nu\Big( \frac{t^{-\frac12}(1-t)qY_1Y_m}{1-qY_1Y_m}\Big)
= b^+ qt^{2n-1-m-2\ell}t_0t_n,
\\
\ev_\nu(F_{\alpha^\vee_0}^+)
&=t_n^{\frac12}t^{n-1}
\frac{q^{\frac12}t_0^{\frac12}(u_0^{-\frac12}-u_0^{\frac12})+qt^{n-1}t_0t_n^{\frac12}(u_n^{-\frac12}-u_n^{\frac12}) }
{1-qt^{2n-2} t_0t_n},
\\
\ev_\nu(F_{\alpha^\vee_0}^+)
&=\frac{((u_n^{-\frac12}-u_n^{\frac12})
+(u_0^{-\frac12}-u_0^{\frac12}) q^{-\frac12}t^{v_\nu(1)}(t_0^{-\frac12}t_n^{-\frac12}t^{-n})) 
q  t^{-2v_\nu(1)} (t_0t_nt^{2n})  }
{1- q  t^{-2v_\nu(1)} (t_0t_nt^{2n}) },
\end{align*}
where $\ell$ is the number of negative entries in $\nu$.
\end{remark}

To give an example of how Proposition~\ref{0tomgapB} is used to expand $\widehat{E}^y_\mu$ for the cases treated in 
\cite[Theorem 1.10]{CMW23}
let $n=6$ and compute $E^{(2\overline{5}\overline{3}\overline{1}46)}_{\varepsilon_1+\varepsilon_2+\varepsilon_3}=E^{(2\overline{5}\overline{3}\overline{1}46)}_{(1,1,1,0,0,0)}$.  
The process will produce $2\cdot3\cdot4 = 24$ terms after three applications.
The first application of Proposition~\ref{0tomgapB} produces 2 terms.
\begin{align*}
&\widehat{E}^{(2\overline{5}\overline{3}\overline{1}46)}_{(1,1,1,0,0,0)}
=G_{0,1}^{(2\overline{5}\overline{3}\overline{1}46)}\widehat{E}^{(2\overline{5}\overline{3}\overline{1}46)}_{(0,1,1,0,0,0)}
+G_{1,1}^{(2\overline{5}\overline{3}\overline{1}46)}x_2\widehat{E}^{(\overline{2}\overline{5}\overline{3}\overline{1}46)}_{(0,1,1,0,0,0)}
\end{align*}
The second application of Proposition~\ref{0tomgapB} produces 3 terms for each of the existing 2 terms:
\begin{align*}
\widehat{E}^{(2\overline{5}\overline{3}\overline{1}46)}_{(1,1,1,0,0,0)}
&=G_{0,1}^{(2\overline{5}\overline{3}\overline{1}46)}
G_{0,2}^{(2\overline{5}\overline{3}\overline{1}46)}\widehat{E}^{(2\overline{5}\overline{3}\overline{1}46)}_{(0,0,1,0,0,0)}
+G_{0,1}^{(2\overline{5}\overline{3}\overline{1}46)}
G_{1,2}^{(2\overline{5}\overline{3}\overline{1}46)}x_2
\widehat{E}^{(\overline{2}\overline{5}\overline{3}\overline{1}46)}_{(0,0,1,0,0,0)}
\\
&\qquad
+G_{0,1}^{(2\overline{5}\overline{3}\overline{1}46)}
G_{2,2}^{(2\overline{5}\overline{3}\overline{1}46)}x_5^{-1}
\widehat{E}^{(25\overline{3}\overline{1}46)}_{(0,0,1,0,0,0)}
\\
&+G_{1,1}^{(2\overline{5}\overline{3}\overline{1}46)} x_2
G_{0,2}^{(\overline{2}\overline{5}\overline{3}\overline{1}46)}
\widehat{E}^{(\overline{2}\overline{5}\overline{3}\overline{1}46)}_{(0,0,1,0,0,0)}
+G_{1,1}^{(2\overline{5}\overline{3}\overline{1}46)} x_2
G_{1,2}^{(\overline{2}\overline{5}\overline{3}\overline{1}46)} x_2^{-1} 
\widehat{E}^{(2\overline{5}\overline{3}\overline{1}46)}_{(0,0,1,0,0,0)}
\\
&\qquad
+G_{1,1}^{(2\overline{5}\overline{3}\overline{1}46)} x_2
G_{2,2}^{(\overline{2}\overline{5}\overline{3}\overline{1}46)} x_5^{-1}
\widehat{E}^{(\overline{2}5\overline{3}\overline{1}46)}_{(0,0,1,0,0,0)}.
\end{align*}

In preparation for the next application of Proposition~\ref{0tomgapB}, 
arrange the permutations in the subscripts of $\widehat{E}^w_{(0,0,1,0,0,0)}$ to be minimal length
in the coset $wW_{[1,2]}$ where $W_{[1,2]} = \langle s_1\rangle$ (this requires $w(1)\prec w(2)$).
\begin{align*}
\widehat{E}^{(2\overline{5}\overline{3}\overline{1}46)}_{(1,1,1,0,0,0)}
&=G_{0,1}^{(2\overline{5}\overline{3}\overline{1}46)}
G_{0,2}^{(2\overline{5}\overline{3}\overline{1}46)}\widehat{E}^{(2\overline{5}\overline{3}\overline{1}46)}_{(0,0,1,0,0,0)}
+G_{0,1}^{(2\overline{5}\overline{3}\overline{1}46)}
G_{1,2}^{(2\overline{5}\overline{3}\overline{1}46)}x_2
t^{\frac{1}{2}} \widehat{E}^{(\overline{5}\overline{2}\overline{3}\overline{1}46)}_{(0,0,1,0,0,0)}
\\
&\qquad
+G_{0,1}^{(2\overline{5}\overline{3}\overline{1}46)}
G_{2,2}^{(2\overline{5}\overline{3}\overline{1}46)}x_5^{-1}
\widehat{E}^{(25\overline{3}\overline{1}46)}_{(0,0,1,0,0,0)}
\\
&+G_{1,1}^{(2\overline{5}\overline{3}\overline{1}46)} x_2
G_{0,2}^{(\overline{2}\overline{5}\overline{3}\overline{1}46)}
t^{\frac{1}{2}}\widehat{E}^{(\overline{5}\overline{2}\overline{3}\overline{1}46)}_{(0,0,1,0,0,0)}
+G_{1,1}^{(2\overline{5}\overline{3}\overline{1}46)} x_2
G_{1,2}^{(\overline{2}\overline{5}\overline{3}\overline{1}46)} x_2^{-1} 
\widehat{E}^{(2\overline{5}\overline{3}\overline{1}46)}_{(0,0,1,0,0,0)}
\\
&\qquad
+G_{1,1}^{(2\overline{5}\overline{3}\overline{1}46)} x_2
G_{2,2}^{(\overline{2}\overline{5}\overline{3}\overline{1}46)} x_5^{-1} 
t^{\frac{1}{2}}\widehat{E}^{(5\overline{2}\overline{3}\overline{1}46)}_{(0,0,1,0,0,0)}.
\end{align*}
The next application of Proposition~\ref{0tomgapB} produces 4 terms for each of the existing 6 terms.
\begin{align*}\widehat{E}^{(2\overline{5}\overline{3}\overline{1}46)}_{(1,1,1,0,0,0)} =\hspace{-5em}&
\\
&G_{0,1}^{(2\overline{5}\overline{3}\overline{1}46)}
G_{0,2}^{(2\overline{5}\overline{3}\overline{1}46)}
G_{0,3}^{(2\overline{5}\overline{3}\overline{1}46)} 
\widehat{E}^{(2\overline{5}\overline{3}\overline{1}46)}_{(0,0,0,0,0,0)}
+G_{0,1}^{(2\overline{5}\overline{3}\overline{1}46)}
G_{0,2}^{(2\overline{5}\overline{3}\overline{1}46)}
G_{1,3}^{(2\overline{5}\overline{3}\overline{1}46)} x_2 
\widehat{E}^{(\overline{2}\overline{5}\overline{3}\overline{1}46)}_{(0,0,0,0,0,0)}
\\
&\quad+G_{0,1}^{(2\overline{5}\overline{3}\overline{1}46)}
G_{0,2}^{(2\overline{5}\overline{3}\overline{1}46)}
G_{2,3}^{(2\overline{5}\overline{3}\overline{1}46)} x_5^{-1}  
\widehat{E}^{(25\overline{3}\overline{1}46)}_{(0,0,0,0,0,0)}
+G_{0,1}^{(2\overline{5}\overline{3}\overline{1}46)}
G_{0,2}^{(2\overline{5}\overline{3}\overline{1}46)}
G_{3,3}^{(2\overline{5}\overline{3}\overline{1}46)} x_3^{-1} 
\widehat{E}^{(2\overline{5}3\overline{1}46)}_{(0,0,0,0,0,0)}
\\
\\
&+G_{0,1}^{(2\overline{5}\overline{3}\overline{1}46)}
G_{1,2}^{(2\overline{5}\overline{3}\overline{1}46)} x_2t^{\frac12}
G_{0,3}^{(\overline{5}\overline{2}\overline{3}\overline{1}46)} 
\widehat{E}^{(\overline{5}\overline{2}\overline{3}\overline{1}46)}_{(0,0,0,0,0,0)}
+G_{0,1}^{(2\overline{5}\overline{3}\overline{1}46)}
G_{1,2}^{(2\overline{5}\overline{3}\overline{1}46)} x_2t^{\frac12}
G_{1,3}^{(\overline{5}\overline{2}\overline{3}\overline{1}46)} x_5^{-1} 
\widehat{E}^{(5\overline{2}\overline{3}\overline{1}46)}_{(0,0,0,0,0,0)}
\\
&\quad+G_{0,1}^{(2\overline{5}\overline{3}\overline{1}46)}
G_{1,2}^{(2\overline{5}\overline{3}\overline{1}46)} x_2t^{\frac12}
G_{2,3}^{(\overline{5}\overline{2}\overline{3}\overline{1}46)} x_2^{-1} 
\widehat{E}^{(\overline{5}2\overline{3}\overline{1}46)}_{(0,0,0,0,0,0)}
+G_{0,1}^{(2\overline{5}\overline{3}\overline{1}46)}
G_{1,2}^{(2\overline{5}\overline{3}\overline{1}46)} x_2t^{\frac12}
G_{3,3}^{(\overline{5}\overline{2}\overline{3}\overline{1}46)} x_3^{-1} 
\widehat{E}^{(\overline{5}\overline{2}3\overline{1}46)}_{(0,0,0,0,0,0)}
\\
\\
&+G_{0,1}^{(2\overline{5}\overline{3}\overline{1}46)}
G_{2,2}^{(2\overline{5}\overline{3}\overline{1}46)}x_5^{-1}
G_{0,3}^{(25\overline{3}\overline{1}46)}
\widehat{E}^{(25\overline{3}\overline{1}46)}_{(0,0,0,0,0,0)}
+G_{0,1}^{(2\overline{5}\overline{3}\overline{1}46)}
G_{2,2}^{(2\overline{5}\overline{3}\overline{1}46)}x_5^{-1} 
G_{1,3}^{(25\overline{3}\overline{1}46)} x_2
\widehat{E}^{(\overline{2}5\overline{3}\overline{1}46)}_{(0,0,0,0,0,0)}
\\
&\quad+G_{0,1}^{(2\overline{5}\overline{3}\overline{1}46)}
G_{2,2}^{(2\overline{5}\overline{3}\overline{1}46)}x_5^{-1}
G_{2,3}^{(25\overline{3}\overline{1}46)} x_5
\widehat{E}^{(2\overline{5}\overline{3}\overline{1}46)}_{(0,0,0,0,0,0)}
+G_{0,1}^{(2\overline{5}\overline{3}\overline{1}46)}
G_{2,2}^{(2\overline{5}\overline{3}\overline{1}46)}x_5^{-1} 
G_{3,3}^{(25\overline{3}\overline{1}46)} x_3^{-1}
\widehat{E}^{(253\overline{1}46)}_{(0,0,0,0,0,0)}
\\
\\
&+G_{1,1}^{(2\overline{5}\overline{3}\overline{1}46)} x_2
G_{0,2}^{(\overline{2}\overline{5}\overline{3}\overline{1}46)}
t^{\frac{1}{2}}
G_{0, 3}^{(\overline{5}\overline{2}\overline{3}\overline{1}46)}
\widehat{E}^{(\overline{5}\overline{2}\overline{3}\overline{1}46)}_{(0,0,0,0,0,0)}
+ G_{1,1}^{(2\overline{5}\overline{3}\overline{1}46)} x_2
G_{0,2}^{(\overline{2}\overline{5}\overline{3}\overline{1}46)}
t^{\frac{1}{2}}
G_{1, 3}^{(\overline{5}\overline{2}\overline{3}\overline{1}46)}
x_{5}^{-1} \widehat{E}^{(5\overline{2}\overline{3}\overline{1}46)}_{(0,0,0,0,0,0)}
\\
&\quad+G_{1,1}^{(2\overline{5}\overline{3}\overline{1}46)} x_2
G_{0,2}^{(\overline{2}\overline{5}\overline{3}\overline{1}46)}
t^{\frac{1}{2}}
G_{2, 3}^{(\overline{5}\overline{2}\overline{3}\overline{1}46)}
x_{2}^{-1} \widehat{E}^{(\overline{5}2\overline{3}\overline{1}46)}_{(0,0,0,0,0,0)}
+G_{1,1}^{(2\overline{5}\overline{3}\overline{1}46)} x_2
G_{0,2}^{(\overline{2}\overline{5}\overline{3}\overline{1}46)}
t^{\frac{1}{2}}
G_{3, 3}^{(\overline{5}\overline{2}\overline{3}\overline{1}46)}
x_{3}^{-1} \widehat{E}^{(\overline{5}\overline{2}3\overline{1}46)}_{(0,0,0,0,0,0)}
\\
\\
&+G_{1,1}^{(2\overline{5}\overline{3}\overline{1}46)} x_2
G_{1,2}^{(\overline{2}\overline{5}\overline{3}\overline{1}46)} x_2^{-1} 
G_{0, 3}^{(2\overline{5}\overline{3}\overline{1}46)} 
\widehat{E}^{(2\overline{5}\overline{3}\overline{1}46)}_{(0,0,0,0,0,0)}
+G_{1,1}^{(2\overline{5}\overline{3}\overline{1}46)} x_2
G_{1,2}^{(\overline{2}\overline{5}\overline{3}\overline{1}46)} x_2^{-1} 
G_{1, 3}^{(2\overline{5}\overline{3}\overline{1}46)} 
x_{2} \widehat{E}^{(\overline{2}\overline{5}\overline{3}\overline{1}46)}_{(0,0,0,0,0,0)}
\\
&\quad+G_{1,1}^{(2\overline{5}\overline{3}\overline{1}46)} x_2
G_{1,2}^{(\overline{2}\overline{5}\overline{3}\overline{1}46)} x_2^{-1} 
G_{2, 3}^{(2\overline{5}\overline{3}\overline{1}46)} 
x_{5}^{-1} \widehat{E}^{(25\overline{3}\overline{1}46)}_{(0,0,0,0,0,0)}
+G_{1,1}^{(2\overline{5}\overline{3}\overline{1}46)} x_2
G_{1,2}^{(\overline{2}\overline{5}\overline{3}\overline{1}46)} x_2^{-1} 
G_{3, 3}^{(2\overline{5}\overline{3}\overline{1}46)} 
x_{3}^{-1} \widehat{E}^{(2\overline{5}3\overline{1}46)}_{(0,0,0,0,0,0)}
\\
\\
&+G_{1,1}^{(2\overline{5}\overline{3}\overline{1}46)} x_2
G_{2,2}^{(\overline{2}\overline{5}\overline{3}\overline{1}46)} x_5^{-1} 
t^{\frac{1}{2}}
G_{0, 3}^{(5\overline{2}\overline{3}\overline{1}46)}
\widehat{E}^{(5\overline{2}\overline{3}\overline{1}46)}_{(0,0,0,0,0,0)}
+ G_{1,1}^{(2\overline{5}\overline{3}\overline{1}46)} x_2
G_{2,2}^{(\overline{2}\overline{5}\overline{3}\overline{1}46)} x_5^{-1} 
t^{\frac{1}{2}}
G_{1, 3}^{(5\overline{2}\overline{3}\overline{1}46)}
x_{5} \widehat{E}^{(\overline{5}\overline{2}\overline{3}\overline{1}46)}_{(0,0,0,0,0,0)}
\\
&\quad+ G_{1,1}^{(2\overline{5}\overline{3}\overline{1}46)} x_2
G_{2,2}^{(\overline{2}\overline{5}\overline{3}\overline{1}46)} x_5^{-1} 
t^{\frac{1}{2}}
G_{2, 3}^{(5\overline{2}\overline{3}\overline{1}46)}
x_{2}^{-1} \widehat{E}^{(52\overline{3}\overline{1}46)}_{(0,0,0,0,0,0)}
+ G_{1,1}^{(2\overline{5}\overline{3}\overline{1}46)} x_2
G_{2,2}^{(\overline{2}\overline{5}\overline{3}\overline{1}46)} x_5^{-1} 
t^{\frac{1}{2}}
G_{3, 3}^{(5\overline{2}\overline{3}\overline{1}46)}
x_{3}^{-1} \widehat{E}^{(5\overline{2}3\overline{1}46)}_{(0,0,0,0,0,0)}
\end{align*}
The final expansion is obtained by replacing $\widehat{E}^w_{(0,0,0,0,0,0)}$ by $t^{\frac12\ell_s(w)}t_n^{\frac12\ell_d(w)}$.

\end{document}